\documentclass[12pt]{article}
\usepackage{amssymb,amsmath,amsthm,tikz,multirow}
\usepackage{enumerate} 
\usepackage{hyperref} 
\usepackage [autostyle, english = british]{csquotes}
\usepackage [english]{babel}
\usepackage{mathtools}
\usepackage{bm}

\MakeOuterQuote{"}

\newtheorem{theorem}{Theorem}

\newtheorem{proposition}[theorem]{Proposition}
\newtheorem*{theorem*}{Theorem}

\theoremstyle{definition}
\newtheorem*{definition*}{Definition}
\newtheorem*{case*}{Case}
\newtheorem*{subcase*}{Subcase}
\newtheorem*{subsubcase*}{Subsubcase}

\theoremstyle{plain}
\newtheorem{thm}{Theorem}[section]
\newtheorem{lem}[thm]{Lemma}

\newtheorem*{cor}{Corollary}

\theoremstyle{definition}

\theoremstyle{remark}

\numberwithin{equation}{section}

\newcommand{\bvert}{ \vrule width 2pt }

\newcommand{\AVC}{\text{AVC}}

\DeclarePairedDelimiter{\floor}{\lfloor}{\rfloor}

\title{Tilings of the Sphere by Congruent Quadrilaterals with Exactly Two Equal Edges}
\author{Ho Man CHEUNG \\ email: hmcheungae@connect.ust.hk \\[2ex] Hoi Ping LUK \\ email: hoi@connect.ust.hk}

\begin{document}
\maketitle

\begin{abstract} In this paper we give a classification of tilings of the sphere by congruent quadrilaterals with exactly two equal edges. The tilings are the earth map tilings, $(p,q)$-earth map tilings and their flip modifications, and quadrilateral subdivisions of the cube and the triangular prism. We described the ranges of values of the edges and angles for the tile to be geometrically realisable through extrinsic parameters. The symmetry groups of the tilings are also determined.  \\

\end{abstract}

\tableofcontents

\section{Introduction}

The study of edge-to-edge tilings of the sphere by congruent polygons ($n$-gons) have seen fruition in the recent decades. For simplicity, by \textit{tilings} we mean tilings of such kind throughout this paper. As an immediate result from Euler's formula and Dehn-Sommerville formulae, we know that to have tilings $n$ is $3,4,5$, namely the polygons are exactly triangles, quadrilaterals and pentagons. Starting from the work by Sommerville \cite{so} in 1923 on the classification of edge-to-edge tilings of the sphere by congruent triangles, collective effort have been made in the subject. Notably, Ueno and Agaoka gave a complete classification for tilings by triangles in \cite{ua}. Further efforts have been put into study of tilings by quadrilaterals by Ueno and Agaoka in \cite{ua2}, by Akama et al. in \cite{ak}, \cite{ak2}, \cite{as}, \cite{as2}, \cite{avc}, which include quadrilaterals of the types that are equilateral or can be divided into two triangles. On the other hand, Yan et al. are on course to give a complete classification of tilings by pentagons in \cite{gsy}, \cite{ay1}, \cite{wy}, \cite{wy2}, \cite{awy}, \cite{ly}, \cite{wy3}, \cite{yan}, \cite{yan2}. The problems remain open are the tilings by quadrilaterals with exactly two equal edges and those with exactly three equal edges. 

This paper is to give a classification on the tilings by congruent quadrilateral with exactly two equal edges ($\bm{a^2bc}$). For the reason of efficiency and compactness, most of the notations in this paper are inherited from Yan's papers. For instance, we use the product notation for combinations of edges and angles. The notation $\bm{a^2bc}$ denotes two $\bm{a}$-edges, one $\bm{b}$-edge and one $\bm{c}$-edge. Interested readers may refer to \cite{ly} for more details nevertheless the notations adopted in this paper will be explained in due course. We remark that the tilings obtained by Akama in \cite{ak2} can be derived in this paper.

The main result of this paper is stated as follows,
\begin{theorem*} The tilings of the sphere by congruent quadrilaterals with exactly two equal edges are the Earth Map Tilings, the $(\frac{f}{4},4)$-Earth Map Tilings and their flip modifications, and the quadrilateral subdivisions of the cube (or equivalently the octahedron) and of the triangular prism. 
\end{theorem*}

Acknowledgement: we like express our gratitude to Henry Kam Hang Cheng for the discussions and suggestions on this paper.

\section{Basic Techniques}

To understand this paper, we highlight the basic notations about the angles, edges, vertices and the combinatorics in this secion.

\begin{figure}[htp] 
\centering
\begin{tikzpicture}

\draw
	(0,0) -- (1.2,0) -- (1.2,1.2) -- (0,1.2) -- cycle;

\draw[double, line width=0.6]
	(1.2,0) -- (1.2,1.2);

\draw[line width=2]
	(0,0) -- (1.2,0);

\node at (-0.2, -0.2) {$D$};
\node at (-0.2, 1.4) {$A$};
\node at (1.4, 1.4) {$B$};
\node at (1.4, -0.2) {$C$};

\node at (0.2,1) {\small $\alpha$};
\node at (1,0.95) {\small $\beta$};
\node at (1,0.25) {\small $\gamma$};
\node at (0.2,0.25) {\small $\delta$};

\end{tikzpicture}
\caption{Default Quadrilateral}
\label{DefaultQuad}
\end{figure}
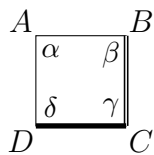

The quadrilateral in this paper, as shown in Figure \ref{DefaultQuad}, has edge combination $\bm{a^2bc}$ and clockwise angle configuration $\alpha,  \beta, \gamma, \delta$ from the top left corner to the bottom left has angles $\beta, \gamma$ next to $\bm{b}$-edge and $\gamma, \delta$ next to $\bm{c}$-edge. Graphically, $\Vert$ denotes the $\bm{b}$-edge whereas $\bvert$ denotes $\bm{c}$-edge. The edges are geodesic arcs, namely parts of the great circles. By Lemma 1 in \cite{gsy}, the quadrilateral is simple. This implies $\bm{a}$ and one of $\bm{b}, \bm{c}$ is $<\pi$. A priori, the four angles $\alpha, \beta, \gamma, \delta$ are not assumed to have distinct values. They will be determined accordingly. The tilings are always assumed to be edge-to-edge and all vertices are assumed to have degree $\ge 3$. All these are implicitly assumed, in particular in the statements of all lemmas and propositions in this paper. We remark that there is another edge arrangement in Figure \ref{FailQuad} for $\bm{a^2bc}$, which however cannot be a tile for any tilings by Proposition \ref{FailQuadProp}. The quadrilateral in Figure \ref{DefaultQuad} is indeed the only tile we need to study. Notice that this tile is symmetric up to the exchange of $\beta, \delta$ and $\bm{b}, \bm{c}$.

\subsection{Notations and Conventions}

\subsubsection*{Angles at a vertex}

A vertex is denoted by $\alpha^a\beta^b\gamma^c\delta^d$ which means it consists of $a$ copies of $\alpha$, $b$ copies of $\beta$, etc. The \textit{angle sum of the vertex} is 
\begin{align*}
a \alpha + b \beta + c \gamma + d \delta = 2\pi,
\end{align*} 
where $a, b, c, d \ge 0$. The vertex angle sums are frequently used so that we simply say, for example, \enquote{by $\beta\gamma\delta$} to mean \enquote{by the angle sum $\beta + \gamma + \delta = 2\pi$ of $\beta\gamma\delta$}. In practice, we only write the angles which appear at a vertex. For example, $\alpha^a\beta^b\gamma^c$ means $a,b,c>0$ and $a+b+c \ge 3$. We remark this slight difference in the use of the notation from that in Yan's papers. 

We use $\alpha\beta\cdots$ to denote a vertex containing one $\alpha$ and one $\beta$. In other words, $\alpha\beta\cdots = \alpha^a\beta^b\gamma^c\delta^d$ with $a\ge1$ and $b\ge1$. The angle combination in $\cdots$ is called the \textit{remainder}. To indicate a vertex without a specific angle, say $\alpha$, we denote it by $\hat{\alpha}\cdots$. We also use $\hat{\bm{a}}\cdots$ or $\hat{\bm{a}}$-vertex to denote a vertex without $\bm{a}$-edge, etc.

Even though we do not assume that the four angles to have distinct values, each of them can be distinguished by its adjacent edges as indicated in Figure \ref{DefaultQuad}. For example, $\alpha$ is the only angle sandwiched by two $\bm{a}$-edges, $\beta$ is the only angle between an $\bm{a}$-edge and a $\bm{b}$-edge, etc. When we mention each angle, its unique position in the default quadrilateral and its unique edge information are implicitly taken into consideration. 

Let $f$ denote the number of tiles. By (12) in \cite{ak}, the angle sum for quadrilateral is 
\begin{align}\label{QuadSum}
 \alpha +  \beta +  \gamma + \delta = (2 + \frac{4}{f})\pi.
\end{align}
The equality means that the area of the quadrilateral is the surface area of $4\pi$ of the (unit) sphere divided by $f$ many quadrilateral tiles.  

The $\AVC$ (\textit{anglewise vertex combination}) is the collection of all vertices in a tiling. For example,
\begin{align*} 
\AVC = \{ \alpha\beta^2, \alpha^2\delta^2, \gamma^4, \alpha\delta^{d}, \beta^2\delta^{d},  \delta^{d} \}.
\end{align*}
The generic notations $d$ in $\alpha\delta^{d}$, $\beta^2\delta^{d}$, and $\delta^{d}$ may or may not be different in value. Indeed we allow the generic notations to take different values at different vertices. Their values will be distinguished accordingly whenever necessary. For instance, when we use the $\AVC$ above to construct tilings as in \eqref{ab2avcf} of Proposition \ref{pqEMTProp}, the values of the generic notations are determined below,
\begin{align*} 
\AVC = \{ \alpha\beta^2,  \alpha^2\delta^2,  \gamma^4, \alpha\delta^{\frac{f+8}{8}},  \beta^2\delta^{ \frac{f-8}{8}}, \delta^{\frac{f}{4}} \}.
\end{align*}

We denote by $\AVC_k$ all the degree $k$ vertices in a tiling. Take the above $\AVC$ as an example, we have $\AVC_3 = \{ \alpha\beta^2 \}$, $\AVC_4 = \{ \alpha^2\delta^2, \gamma^4 \}$ when $f>16$. 

Note that the $\AVC$ lists all the \textit{possible} vertices. As each angle appears at some vertices in a tiling, to ensure this to happen, some vertices must appear, which we call \textit{necessary}, whereas some others may or may not appear, which we call \textit{optional}. The following $\AVC$ encodes such information
\begin{align*} 
\AVC = \{ \alpha\beta^2, \alpha^2\delta^2, \gamma^4 \ \vert \ \alpha\delta^{d}, \beta^2\delta^{d},  \delta^{d} \},
\end{align*}
where the divider $\vert$ separates the necessary vertices before it and the optional vertices after it. Further details of the notation can be seen in \cite{awy}. 

Lastly, $\equiv$ is used instead of $=$ to indicate the exact list of vertices in a tiling. In other words, every vertex listed is necessary. For example, the $(\frac{f}{4}, 4)$-earth map tiling constructed has vertices exactly as follows,
\begin{align*} 
\AVC \equiv \{ \alpha\beta^2,  \alpha^2\delta^2,  \gamma^4, \delta^{\frac{f}{4}} \}.
\end{align*}
Correspondingly, we have $\AVC_3 \equiv \{ \alpha\beta^2 \}$.

Given an $\AVC$, implicitly each vertex has its vertex angle sum equation and hence the $\AVC$ implies a system of simultaneous linear equations in terms of the angles. We also remark that whenever necessary the angle sum of quadrilateral can be included into this system without specifying it in the discussion. For a tiling to exists, there must be a solution to the system. 

\subsubsection*{Counting}

We use $\#$ to denote various total numbers. In particular, $\# \alpha$ denotes the total number of $\alpha$ in a tiling. Since each angle appears exactly once in each tile, we have
\begin{align*}
\# \alpha = \# \beta = \# \gamma = \# \delta = f.
\end{align*}
For $\AVC = \{ \alpha\beta^2,  \alpha^2\delta^2,  \gamma^4, \alpha\delta^{\frac{f+8}{8}},  \beta^2\delta^{ \frac{f-8}{8}}, \delta^{\frac{f}{4}} \}$, we have
\begin{align*}
f &= \# \alpha = \# \alpha\beta^2 + 2\# \alpha^2\delta^2 + \# \alpha\delta^{\frac{f+8}{8}},  \\
f &= \# \beta = 2\# \alpha\beta^2 + 2 \# \beta^2\delta^{ \frac{f-8}{8}}, \\
f &= \# \gamma = 4\# \gamma^4, \\
f &= \# \delta =2 \# \alpha^2\delta^2 + \frac{f+8}{8} \# \alpha\delta^{\frac{f+8}{8}} + \frac{f-8}{8} \#\beta^2\delta^{ \frac{f-8}{8}} + \frac{f}{4} \# \delta^{\frac{f}{4}}, 
\end{align*}
where $\# \alpha\beta^2$ denotes the total number of vertex $\alpha\beta^2$, etc. 

We also use $\#_{\nu} \alpha$ to denote the number of $\alpha$ at a vertex $\nu$. Meanwhile, we use $\#_k \alpha$ to denote the number of $\alpha$ at degree $k$ vertices. For example, in the $\AVC$ above, $\#_3 \alpha =  \# \alpha\beta^2$.

\subsubsection*{$\bm{b}$-Edge and $\bm{c}$-Edge}

The distinguished edges in the default quadrilateral are $\bm{b}, \bm{c}$. The angles adjacent to $\bm{b}$-edge are $\beta, \gamma$ which we call \textit{$\bm{b}$-angles}. The angles adjacent to $\bm{c}$-edge are $\gamma, \delta$ which we call \textit{$\bm{c}$-angles}. A \textit{$\bm{b}$-vertex} has angles $\beta, \gamma$, and a \textit{$\bm{c}$-vertex} has angles $\gamma, \delta$. Meanwhile, a \textit{$\hat{\bm{b}}$-vertex} does not have $\beta, \gamma$ and has only $\alpha, \delta$, and a \textit{$\hat{\bm{c}}$-vertex} does not have angles $\gamma, \delta$ and has only $\alpha, \beta$.

By knowing the \textit{arrangement} of angles at a vertex, we can determine the vertex to a significant extent. For example, the $\bm{b}$-edge $\Vert$ can only appear at a vertex as $\beta \Vert \beta, \beta \Vert \gamma, \gamma \Vert \gamma$. Meanwhile, the $\bm{c}$-edge $\bvert$ can only appear at a vertex as $\gamma \, \bvert \, \gamma, \gamma \, \bvert \, \delta, \delta \, \bvert \, \delta$. This observation leads to the Parity Lemma (Lemma \ref{PaLem}) whereby we establish the list of forbidden vertices and possible vertices in Lemma \ref{ForbVerLem}.

The arrangement of angles at a vertex is essentially circular and hence there are several equivalent presentations. For instance, $\vert \beta \Vert \gamma \, \bvert \, \delta \vert$, $\bvert \, \delta \vert \beta \Vert \gamma \, \bvert$ (rotation), $\vert \delta \, \bvert \, \gamma \Vert \beta \vert$ (reversion) are the same arrangements of vertex $\beta\gamma\delta$.

\subsubsection*{Adjacent angle deduction}

The tiles containing the angles at a vertex can be arranged in various ways. The AAD (\textit{adjacent angle deduction}) is the notation that encodes both the angle arrangement and the tile arrangement at a vertex. For example, all three pictures in Figure \ref{AADEg} are AADs of the arrangement $\bvert \, \delta \vert \alpha \vert \alpha \vert \delta \, \bvert$. These three pictures can be denoted by $\bvert^{\,\gamma} \delta^{\alpha} \vert^{\delta} \alpha^{\beta} \vert^{\beta} \alpha^{\delta} \vert^{\alpha} \delta^{\gamma} \, \bvert$, $\bvert^{\,\gamma} \delta^{\alpha} \vert^{\delta} \alpha^{\beta} \vert^{\delta} \alpha^{\beta} \vert^{\alpha} \delta^{\gamma} \, \bvert$, and $\bvert^{\,\gamma} \delta^{\alpha} \vert^{\beta} \alpha^{\delta} \vert^{\delta} \alpha^{\beta} \vert^{\alpha} \delta^{\gamma} \, \bvert$ respectively.

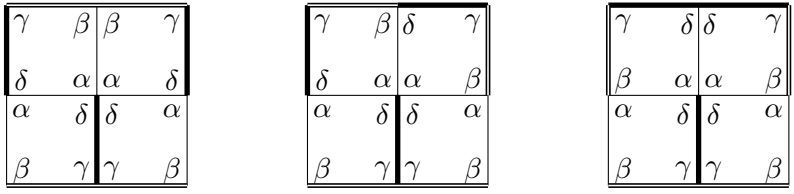
\begin{figure}[htp]
\centering
\begin{tikzpicture}[>=latex,scale=1]

\begin{scope}
\draw
	(1.2,0) -- (1.2, 1.2)-- (0, 1.2) -- (-1.2, 1.2) -- (-1.2, 0)
	(-1.2,0) -- (0,0) -- (0, 1.2)
	(0,0) -- (1.2,0)
	(0, -1.2) -- (1.2,-1.2)
	(1.2,0) -- (1.2, -1.2)
	(-1.2,0) -- (-1.2, -1.2);

\draw[line width=2]
	(0,0) -- (0,-1.2)
	(-1.2,0) -- (-1.2, 1.2)
	(1.2,0) -- (1.2,1.2);

\draw[double, line width=0.6]
	(0,1.2) -- (1.2, 1.2)
	(0,1.2) -- (-1.2,1.2)
	(-1.2,-1.2) -- (0, -1.2)
	(0, -1.2) -- (1.2,-1.2);

\node at (0.2,0.92) {\small $\beta$}; 
\node at (0.2,0.2) {\small $\alpha$}; 
\node at (1,0.2) {\small $\delta$}; 
\node at (1,0.95) {\small $\gamma$}; 

\node at (-0.2,0.2) {\small $\alpha$}; 
\node at (-0.2,0.92) {\small $\beta$}; 
\node at (-1,0.95) {\small $\gamma$}; 
\node at (-1,0.2) {\small $\delta$}; 

\node at (-0.2,-0.25) {\small $\delta$}; 
\node at (-0.2,-1) {\small $\gamma$}; 
\node at (-1,-1) {\small $\beta$}; 
\node at (-1,-0.2) {\small $\alpha$}; 

\node at (0.2,-0.25) {\small $\delta$}; 
\node at (0.2,-1) {\small $\gamma$}; 
\node at (1,-1) {\small $\beta$}; 
\node at (1,-0.2) {\small $\alpha$}; 
\end{scope}

\begin{scope}[xshift = 4 cm]
\draw
	(1.2,0) -- (1.2, 1.2)-- (0, 1.2) -- (-1.2, 1.2) -- (-1.2, 0)
	(-1.2,0) -- (0,0) -- (0, 1.2)
	(0,0) -- (1.2,0)
	(0, -1.2) -- (1.2,-1.2)
	(1.2,0) -- (1.2, -1.2)
	(-1.2,0) -- (-1.2, -1.2);

\draw[line width=2]
	(0,0) -- (0,-1.2)
	(-1.2,0) -- (-1.2, 1.2)
	(0,1.2) -- (1.2, 1.2);

\draw[double, line width=0.6]
	(1.2,0) -- (1.2,1.2)
	(0,1.2) -- (-1.2,1.2)
	(-1.2,-1.2) -- (0, -1.2)
	(0, -1.2) -- (1.2,-1.2);

\node at (0.2,0.2) {\small $\alpha$}; 
\node at (1,0.2) {\small $\beta$}; 
\node at (1,0.95) {\small $\gamma$}; 
\node at (0.15,0.95) {\small $\delta$}; 

\node at (-0.2,0.2) {\small $\alpha$}; 
\node at (-0.2,0.92) {\small $\beta$}; 
\node at (-1,0.95) {\small $\gamma$}; 
\node at (-1,0.2) {\small $\delta$}; 

\node at (-0.2,-0.25) {\small $\delta$}; 
\node at (-0.2,-1) {\small $\gamma$}; 
\node at (-1,-1) {\small $\beta$}; 
\node at (-1,-0.2) {\small $\alpha$}; 

\node at (0.2,-0.25) {\small $\delta$}; 
\node at (0.2,-1) {\small $\gamma$}; 
\node at (1,-1) {\small $\beta$}; 
\node at (1,-0.2) {\small $\alpha$}; 
\end{scope}

\begin{scope}[xshift = 8 cm]
\draw
	(1.2,0) -- (1.2, 1.2)-- (0, 1.2) -- (-1.2, 1.2) -- (-1.2, 0)
	(-1.2,0) -- (0,0) -- (0, 1.2)
	(0,0) -- (1.2,0)
	(0, -1.2) -- (1.2,-1.2)
	(1.2,0) -- (1.2, -1.2)
	(-1.2,0) -- (-1.2, -1.2);

\draw[line width=2]
	(0,0) -- (0,-1.2)
	(0,1.2) -- (-1.2,1.2)
	(0,1.2) -- (1.2, 1.2);

\draw[double, line width=0.6]
	(1.2,0) -- (1.2,1.2)
	(-1.2,0) -- (-1.2, 1.2)
	(-1.2,-1.2) -- (0, -1.2)
	(0, -1.2) -- (1.2,-1.2);

\node at (0.15,0.95) {\small $\delta$}; 
\node at (0.2,0.2) {\small $\alpha$}; 
\node at (1,0.2) {\small $\beta$}; 
\node at (1,0.95) {\small $\gamma$}; 

\node at (-0.2,0.2) {\small $\alpha$}; 
\node at (-0.15,0.95) {\small $\delta$}; 
\node at (-1,0.95) {\small $\gamma$}; 
\node at (-1,0.2) {\small $\beta$}; 

\node at (-0.2,-0.25) {\small $\delta$}; 
\node at (-0.2,-1) {\small $\gamma$}; 
\node at (-1,-1) {\small $\beta$}; 
\node at (-1,-0.2) {\small $\alpha$}; 

\node at (0.2,-0.25) {\small $\delta$}; 
\node at (0.2,-1) {\small $\gamma$}; 
\node at (1,-1) {\small $\beta$}; 
\node at (1,-0.2) {\small $\alpha$}; 
\end{scope}

\end{tikzpicture}
\caption{Adjacent angle deduction}
\label{AADEg}
\end{figure}

Similar to the arrangement of angles at a vertex, an AAD can be rotated and reversed. For example, the AAD of the  second picture in Figure \ref{AADEg} can be written as $\vert^{\delta} \alpha^{\beta} \vert^{\delta} \alpha^{\beta} \vert^{\alpha} \delta^{\gamma} \, \bvert^{\,\gamma} \delta^{\alpha} \vert$ (rotation) and $\bvert^{\,\gamma} \delta^{\alpha} \vert^{\beta} \alpha^{\delta} \vert^{\beta} \alpha^{\delta} \vert^{\alpha} \delta^{\gamma} \, \bvert$ (reversion). 

The use of AAD as a notation can be flexible. For example, we use $\bvert^{} \, \delta^{\alpha} \vert^{} \alpha \vert$ to denote $\bvert^{ \, \gamma} \delta^{\alpha} \vert^{\beta} \alpha^{\delta} \vert$ or $\bvert^{ \, \gamma} \delta^{\alpha} \vert^{\delta} \alpha^{\beta} \vert$. In particular, when the AAD along a $\bm{b}$-edge or $\bm{c}$-edge is uniquely determined and irrelevant to the discussion, for simplicity we may drop the \enquote{upper angle} ($\gamma$) in this example. We may even use $^{\beta} \vert^{} \alpha \vert$ to denote $^{\beta} \vert^{\beta} \alpha \vert$ or $^{\beta} \vert^{\delta} \alpha \vert$. We also use $\alpha \vert \beta \cdots $ to indicate $\vert \alpha \vert \beta \vert $ at a vertex.

From the above examples, we can see that an AAD leads to new AAD and vertices. For instance, $\bvert^{ \, \gamma} \delta^{\alpha} \vert^{\beta} \alpha^{\delta} \vert$ is an AAD at vertex $\alpha\delta\cdots$. This AAD implies an AAD at $\vert^{\beta} \alpha^{\delta} \vert^{\alpha} \beta^{\gamma} \Vert$ at $\alpha\beta\cdots$ (the right half of the third picture in Figure \ref{AADEg}). In general, we have the following \textit{reciprocity property}: an AAD $\lambda^{\theta} \vert ^{\rho} \mu$ at $\lambda\mu \cdots$ implies an AAD at $\theta^{\lambda} \vert^{\mu}\rho$ at $\theta\rho\cdots$ and vice versa.

Though AAD can get complicated, often when it is specified we can streamline the argument. For example, the AAD of $\alpha \vert \alpha$ can be $\alpha^{\beta} \vert^{\beta} \alpha$, $\alpha ^{\beta} \vert^{\delta} \alpha$, $\alpha ^{\delta} \vert^{\delta} \alpha$. If we know that $\beta^2\cdots, \delta^2\cdots$ are not vertices, the AAD of $\alpha \vert \alpha$ must be $\alpha ^{\beta} \vert^{\delta} \alpha$, and hence the AAD is \textit{unique} (up to rotation and reversion). As per the edge configuration in Figure \ref{DefaultQuad}, $\beta\vert\beta, \delta\vert\delta$ indeed have unique AADs $\beta^{\alpha}\vert^{\alpha}\beta, \delta^{\alpha}\vert^{\alpha}\delta$.

The typical applications of AAD are listed below:
\begin{itemize}
\item If $\delta \vert \delta \cdots$ is not a vertex, then $\alpha^{\delta}\vert^{\delta}\alpha\cdots$ is also not a vertex.
\item If $\beta\vert \beta \cdots, \delta \vert \delta \cdots$ are not vertices, then $\alpha \vert \alpha$ has unique AAD $\alpha^{\beta} \vert ^{\delta} \alpha$.
\item If $\beta\vert \beta \cdots, \beta \vert \delta \cdots$ are not vertices, then $\alpha \vert \alpha \vert \alpha\cdots$ cannot be a vertex. In other words, there cannot be an arrangement of three consecutive $\alpha$'s at a vertex.
\end{itemize}
We remark that $\beta^2\cdots$ is not a vertex implies $\beta\vert \beta \cdots$ is not a vertex, etc. The application of AAD sometimes depends on which piece of information we can obtain. 

The advantage of adopting AAD is seen in the above demonstration that the discussion of angle and tile arrangements can be more efficiently and concisely conducted in place of drawing pictures. Luk and Yan invented the notation to substitute tens of pictures in the studies of pentagonal tilings. In this paper, we closely follow the notation used in their papers.

\subsection{Combinatorics}

From Euler formula, two Dehn-Sommerville formulae and a vertex equation,
\begin{align}
&v - e + f = 2,& \
&nf = 2e,& \
&2e = \sum_{k\ge3} kv_k,& \
&v = \sum_{k\ge3} v_k,& 
\end{align}
for $n=4$, we get 
\begin{align}
&v_3 = 8 + \sum_{h\ge4} (h - 4) v_h,& \
&f = 6 + \sum_{h\ge4} (h-3)v_h,& 
&f = 2e_b.&
\end{align}
The last equation which comes from counting the number of $\bm{b}$-edges $e_b$ shows that the number of faces $f$ must be even.

\begin{proposition} \label{CubeProp} In a tiling of the sphere by congruent quadrilaterals, $f=6$ if and only if all vertices are of degree $3$. 
\end{proposition}

\begin{proof} By $f = 6 + \sum_{h\ge4} (h-3)v_h$, $f = 6$ if and only if $\sum_{h\ge4} (h-3)v_h = 0$ for every $h \ge 4$, which is equivalent to all vertices are of degree $3$. 
\end{proof}

\begin{lem} \label{PaLem} (Parity Lemma) If a vertex has a distinguishable edge and $\gamma, \delta$ are the corresponding adjacent angles, then at a vertex $\nu$,
\begin{align*}
\#_{\nu} \gamma + \#_{\nu} \delta \equiv 0 \mod 2.
\end{align*}
In the default quadrilateral in Figure \ref{DefaultQuad}, the $\bm{b}$-angles are $\beta, \gamma$ and the $\bm{c}$-angles are $\gamma, \delta$, then at a vertex $\alpha^a\beta^b\gamma^c\delta^d$,
\begin{align*}
b + c \equiv 0 \mod 2, \\
c  + d \equiv 0 \mod 2,
\end{align*}
and $b, c, d$ must be all even or all odd. 
\end{lem}

\begin{proof} Suppose $\beta, \gamma$ are $\bm{b}$-angles and $\gamma, \delta$ are $\bm{c}$-angles, the case when a vertex without a $\bm{c}$-edge is trivial. When a vertex with a $\bm{b}$-edge or $\bm{c}$-edge, 
\begin{align*}
&\cdots \, \Vert \cdots = \beta  \Vert  \beta, \,\, \beta  \Vert  \gamma, \,\, \gamma  \Vert  \gamma,   \\
&\cdots \, \bvert \cdots = \gamma \, \bvert \, \gamma, \,\, \gamma \, \bvert \, \delta, \,\, \delta \, \bvert \, \delta,
\end{align*}
which implies
\begin{align*}
&b + c \equiv 0 \mod 2,& 
&c + d \equiv 0 \mod 2.&
\end{align*}
When $b, c> 0$, $d=0$ at a vertex $\alpha^a\beta^b\gamma^c\delta^d$, 
\begin{align*}
&b + c \equiv 0 \mod 2, & 
&c \equiv 0 \mod 2, &
\end{align*}
which implies $b \equiv c \equiv 0 \mod 2$. When $c, d > 0$, $b=0$ at a vertex $\alpha^a\beta^b\gamma^c\delta^d$, 
\begin{align*}
&c \equiv 0 \mod 2, &
&c +d \equiv 0 \mod 2,&
\end{align*}
which implies $c \equiv d \equiv 0 \mod 2$. When $b, d > 0$, $c=0$ at a vertex $\alpha^a\beta^b\gamma^c\delta^d$, 
\begin{align*}
&b \equiv 0 \mod 2,& 
&d \equiv 0 \mod 2,&
\end{align*}
which implies $b \equiv d \equiv 0 \mod 2$. When $b, c, d > 0$ at a vertex $\alpha^a\beta^b\gamma^c\delta^d$, (we have $a=0$ by quadrilateral angle sum),
\begin{align*}
&b+c \equiv 0 \mod 2,& 
&c +d \equiv 0 \mod 2,&
\end{align*}
which implies $b \equiv c \equiv d \mod 2$. As $c$ odd implies $b$, $d$ both odd and $c$ even implies $b$, $d$ both even. 
\end{proof}

\begin{lem} \label{BaLem} (Balance Lemma) In a monohedral tiling with each tile having a distinguishable edge $\bm{l}$ and $\theta$, $\varphi$ being the only $\bm{l}$-angles,
\begin{enumerate}
\item $\theta \, \bvert \, \theta \cdots$ is a vertex if and only if $\varphi \, \bvert \, \varphi \cdots$ is also a vertex and when one of them is a vertex then $\theta, \varphi < \pi$. 
\item if $\theta^2\cdots$ is not a vertex, then $\varphi^2\cdots$ is also not a vertex and hence each vertex with $\bm{l}$ has exactly one $\bm{l}$ with one $\theta$ and one $\varphi$.
\item if $\theta\varphi\cdots$ is not a vertex, then every $\bm{l}$-vertex has AAD either $ \theta \, \bvert \, \theta\cdots$ or $ \varphi \, \bvert \, \varphi \cdots$ and both must appear.
\end{enumerate}
\end{lem}

\begin{proof} Suppose $\theta \, \bvert \, \theta\cdots$ is a vertex, then the AAD $\theta^{\varphi} \, \bvert \, ^{\varphi}\theta$ implies that $\varphi \, \bvert \, \varphi \cdots$ is a vertex. Up to the symmetry between $\theta, \varphi$, we have the equivalence. The vertx angle sums imply $2\theta, 2\varphi < 2\pi$ so we get $\theta, \varphi < \pi$. 

When $\theta^2\cdots$ is not a vertex, then $\theta \, \bvert \, \theta \cdots$ is not a vertex and hence $\varphi \, \bvert \, \varphi \cdots$ is also not a vertex. The AAD of $\varphi^2\cdots$ can only be $\cdots \theta\, \bvert \, \varphi \vert \cdots \vert  \varphi \, \bvert \, \theta \cdots $ which contradicts $\theta^2\cdots$ not being a vertex. So $\varphi^2\cdots$ cannot be a vertex and every $\bm{l}$-vertex must be $\theta \, \bvert \, \varphi\cdots$ with exactly one $\theta$ and one $\varphi$. 

When $\theta\varphi\cdots$ is not a vertex, then the $\bm{l}$-vertices have either $\theta$ or $\varphi$ and hence a $\bm{l}$-vertex is either $ \theta \, \bvert \, \theta\cdots$ or $ \varphi \, \bvert \, \varphi \cdots$ is a vertex and both of them must appear.
\end{proof}

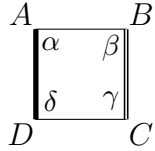
\begin{figure}[htp] 
\centering
\begin{tikzpicture}

\draw
	(0,0) -- (1.2,0) -- (1.2,1.2) -- (0,1.2) -- cycle;

\draw[double, line width=0.6]
	(1.2,0) -- (1.2,1.2);

\draw[line width=2]
	(0,0) -- (0,1.2);

\node at (-0.2, -0.2) {$D$};
\node at (-0.2, 1.4) {$A$};
\node at (1.4, 1.4) {$B$};
\node at (1.4, -0.2) {$C$};

\node at (0.2,1) {\small $\alpha$};
\node at (1,0.95) {\small $\beta$};
\node at (1,0.25) {\small $\gamma$};
\node at (0.2,0.25) {\small $\delta$};

\end{tikzpicture}
\caption{Impossible Quadrilateral}
\label{FailQuad}
\end{figure}

\begin{proposition} \label{FailQuadProp} The quadrilateral in Figure \ref{FailQuad} does not form a tiling. 
\end{proposition}

\begin{proof} Assume the quadrilateral in Figure \ref{FailQuad} is a tile of some tiling. Then $\alpha, \delta$ are $\bm{c}$-angles and $\beta, \gamma$ are $\bm{b}$-angles. Suppose we view $\bm{b}, \bm{c}$ as the same label $\bm{l}$. Then every angle is an $\bm{l}$-angle so Lemma \ref{PaLem} implies that the total number of angles at a vertex must be even. Then the degree of every vertex must be even and $\ge4$, contradicting $v_3 >0$.
\end{proof}

The following results involving congruent quadrilaterals not limited to $\bm{a^2bc}$.

\begin{lem} \label{AlConvexLem} In a tiling of the sphere by congruent quadrilaterals, there is at most one angle with value $\ge\pi$. 
\end{lem}

\begin{proof} Assume the contrary. Suppose the angles are $\alpha, \beta, \gamma, \delta$. Then there are exactly two angles with value $\ge\pi$. Without loss of generality, we may assume $\alpha, \beta \ge \pi$. Then there is no $\alpha^2\cdots$, $\beta^2\cdots$, $\alpha\beta\cdots$. Moreover, all the vertices are in one of the following forms
\begin{align*}
\alpha\gamma^{c_i}\delta^{d_i}, \qquad
\beta\gamma^{c_j}\delta^{d_j}, \qquad
\gamma^{c_k}\delta^{d_k},
\end{align*}
where $c_i, d_i, c_j, d_j, c_k, d_k \ge 0$ and $c_i + d_i , c_j + d_j \ge 2$ and $c_k + d_k\ge3$. As $\alpha, \beta$ must appear at some vertices, $\# \alpha\gamma^{c_i}\delta^{d_i}, \#\beta\gamma^{c_j}\delta^{d_j} > 0$ for some $i, j$. Counting angles $\alpha, \beta$, we get
\begin{align*}
2f &= \# \alpha + \# \beta \\
&= \sum_{i} \#  \alpha\gamma^{c_i}\delta^{d_i} + \sum_{j} \# \beta\gamma^{c_j}\delta^{d_j} \\ 
&<  \sum_{i} (c_i+d_i) \# \alpha\gamma^{c_i}\delta^{d_i} +  \sum_{j} (c_j + d_j) \# \beta\gamma^{c_j}\delta^{d_j} + \sum_{k} (c_k + d_k) \# \gamma^{c_k}\delta^{d_k} \\
&= \# \gamma + \# \delta = 2f,
\end{align*}
a contradiction. 
\end{proof}

\begin{lem} \label{TwoHDAng} In a tiling of the sphere by congruent quadrilaterals, if two angles $\theta_1,\theta_2$ do not appear at any degree $3$ vertex, then there must be a degree $4$ vertex in form of $\theta_i^3\cdots$ ($i=1$ or $2$) or $\theta_i^2\theta_j\cdots$ ($i,j=1,2$), or a degree $5$ vertex in form of $\theta_i^k\theta_j^l$ ($k+l=5$).
\end{lem}

\begin{proof} Counting the other two angles as $\varphi$, then assume the contrary so that $\varphi$ appears at least twice at every degree $4$ vertex and at least once at every degree $5$ vertex. We get $2f \ge 3v_3 + 2v_4 + v_5$. By $f=6 + \sum_{h\ge4} (h-3)v_h$ and $v_3 = 8 + \sum_{h\ge4} (h - 4) v_h$, we get
\begin{align*}
0 \ge 3v_3 + 2v_4 + v_5 - 2f = 12 + \sum_{h\ge6} (h-6)v_h > 0,
\end{align*}
a contradiction.
\end{proof}

\begin{lem} \label{OneHDAng} In a tiling of the sphere by congruent quadrilaterals, if an angle $\theta$ does not appear at any vertex of degree $i$, where $3 \le i \le k$ and $k \ge 4$, then there must be a degree $\ge k+1$ vertex $\theta^{\floor{\frac{k+1}{2}}+1} \cdots$.
\end{lem}

\begin{proof} Assume $\theta$ appears at most $\floor{\frac{k+1}{2}}$ times at every degree $\ge k+1$ vertex. Then $\sum_{h \ge k+1}  \floor{\frac{k+1}{2}} v_h \ge \#_{\ge k+1}  \theta =  \#\theta  = f$ so that
\begin{align*} 
0 &\ge 6 + \sum_{h\ge4} (h-3)v_h -  \sum_{h \ge k+1} \floor*{\frac{k+1}{2}} v_h \\
&\ge 6 + \sum_{h=4}^{k} (h-3)v_h  + \sum_{h \ge k+1} ( k - \floor*{\frac{k+1}{2}} - 2 ) v_h > 0,
\end{align*}
a contradiction.
\end{proof}

\begin{lem} \label{a3Lem} If $\AVC_3 = \{ \theta^3 \}$, then 
\begin{enumerate}
\item $v_4 = 18 + \sum_{h\ge5} (2h-9)v_h + \#_{\ge4}\theta$;
\item $v_4 \ge  18+ \#_{\ge4}\theta$, $f \ge 24$ and hence $\AVC_4 \neq \emptyset$;
\item there is a degree $4$ vertex without $\theta$; 
\item there is a degree $4$ vertex $\varphi^2\cdots$ where $\varphi \neq \theta$ and either there is another different degree $4$ vertex or there is a degree $5$ vertex without $\theta$, $\varphi$; 
\item when $\varphi^4$ is a vertex, where $\varphi \neq \theta$, there must be another degree $4$ vertex; 
\item when $\varphi$, where $\varphi \neq \theta$, appears at least twice in every degree $4$ vertex, then there must be a degree $5$ vertex without $\theta, \varphi$. 
\end{enumerate}
\end{lem}

\begin{proof} Counting $\theta$, we have $f = \# \theta = \#_3 \theta + \#_{\ge4} \theta = 3v_3 + \#_{\ge4} \theta$, then 
\begin{align*}
\#_{\ge4} \theta = f - 3v_3 = ( 6 + \sum_{h\ge4} (h-3)v_h ) - 3( 8 + \sum_{h\ge4} (h - 4) v_h ) = v_4 - 18 - \sum_{h\ge5} (2h-9) v_h,
\end{align*}
which implies $v_4 = 18 + \sum_{h\ge5} (2h-9)v_h + \#_{\ge4}\theta$ and $v_4 >  \#_{\ge4}\theta$. Moreover, as $ \#_{\ge4} \theta \ge 0$, we get $v_4 \ge 18$, so
\begin{align*}
f=6 + \sum_{h\ge4} (h-3)v_h \ge 6 + 18 + \sum_{h\ge5} (h-3) v_h \ge 24.
\end{align*}
As $v_4 > \#_{\ge4}\theta \ge 0$, there must be a degree $4$ vertex without $\theta$. Then at this vertex, one of the other angles must appear at least twice at such a vertex, that is, it must be $\varphi^2\cdots$ for some $\varphi\neq\theta$. Suppose this is the only degree $4$ vertex in $\AVC_4$. Such a vertex must be $\varphi^2\cdots$ for some $\varphi\neq\theta$. So we have $\#_3\varphi = 0$, $\#_4 \varphi \ge 2 v_4$ and $\#\varphi = f = \#_{4} \varphi + \#_{\ge5} \varphi$. By $v_4 = 18 + \sum_{h\ge5} (2h-9)v_h + \#_{\ge4}\theta$, then
\begin{align*}
\#_{\ge5} \varphi = f - \#_{4}\varphi \le f - 2v_4 
=v_5 - 12 - \sum_{h\ge6} (h-6)v_h - \#_{\ge4}\theta,
\end{align*}
which implies $v_5 \ge 12 + \sum_{h\ge6} (h-6)v_h  + \#_{\ge4}\theta + \#_{\ge5} \varphi > 12$. As $v_5 > \#_{\ge5}\theta + \#_{\ge5} \varphi$, $\theta$ and $\varphi$ combined cannot appear at least once at every degree $5$ vertex. 

Assume $\AVC_4 = \{ \varphi^4 \}$ where $\varphi \neq \theta$, then $\# \varphi = f = 4v_4 + \#_{\ge5} \varphi$ such that 
\begin{align*}
\#_{\ge5} \varphi = f - 4v_4 &= 6 + v_4 + \sum_{h\ge5} (h-3)v_h - 4v_4 \\
&=  6 + \sum_{h\ge5} (h-3)v_h - 3( 18 + \sum_{h\ge5} (2h-9)v_h + \#_{\ge5} \theta ) \\
&= -48 - 3\#_{\ge5} \theta -  \sum_{h\ge5} (5h - 24)v_h < 0,
\end{align*}
a contradiction. So there must be another degree $4$ vertex different from $\varphi^4$. 

When $\varphi \neq \theta$ appears at least twice at every degree $4$ vertex, we get $\# \varphi  = f \ge 2v_4 + \#_{\ge5}\varphi$ such that
\begin{align*}
\#_{\ge5}\varphi \le f - 2v_4 &= ( 6 + v_4 + \sum_{h\ge5} (h-3)v_h ) - 2v_4 \\
&= 6 + \sum_{h\ge5} (h-3)v_h - ( 18 + \sum_{h\ge5} (2h-9)v_h + \#_{\ge4} \theta ) \\
&=v_5 -12 - \#_{\ge4} \theta  - \sum_{h\ge 6} (h-6)v_h,
\end{align*}
which implies
\begin{align*}
v_5 \ge 12 + \#_{\ge5}\varphi  +  \#_{\ge4} \theta + \sum_{h\ge 6} (h-6)v_h > \#_{\ge5}\varphi  +  \#_{\ge5} \theta.
\end{align*}
So $v_5 > 0$ and there must be a degree $5$ vertex without $\theta, \varphi$. 
\end{proof}

\begin{lem} \label{ab2Lem} If $\AVC_3 = \{ \theta^2\varphi \}$, then 
\begin{enumerate}
\item $v_4 = 10 + \sum_{h\ge5} (h-5)v_h + \#_{\ge4}\theta$;
\item $v_4 \ge  10+ \#_{\ge4}\theta$, $f \ge 16$, and hence $\AVC_4 \neq \emptyset$;
\item there is a degree $4$ vertex without $\theta$;
\item when there is no $\varphi\cdots$ in $\AVC_i$ for every $4\le i \le k$ where $k \ge 5$, there must be a degree $\ge k+1$ vertex $\varphi^{\floor*{\frac{k}{2}}+1}\cdots$;
\item when $\AVC_4 = \{ \psi^4 \}$, $\psi \neq \theta, \varphi$, there is a degree $5$ vertex with at most of $\psi,\theta$, and there is a degree $5$ vertex without $\theta$;
\item when $\theta$ and $\psi$, where $\psi \neq \varphi, \theta$, together appear four times at every degree $4$, then there is a degree $5$ vertex with at most one of $\theta, \psi$;
\item when $\varphi$ appears at least twice at every degree $4$ vertex, there is a degree $5$ vertex without both $\varphi$ and $\theta$.
\item for $\psi \neq \theta, \varphi$, there is degree $4$ $\psi\cdots$, or there is a degree $5$ $\psi^3\cdots$, or there is a degree $6$ $\psi^5\cdots$, or $\psi^7$ is a vertex. 
\end{enumerate}
\end{lem}

\begin{proof} Counting $\theta$, we have $f = \# \theta = \#_3 \theta + \#_{\ge4} \theta = 2v_3 + \#_{\ge4} \theta$, then
\begin{align*}
\#_{\ge4} \theta = f - 2v_3 = ( 6 + \sum_{h\ge4} (h-3)v_h ) - 2( 8 + \sum_{h\ge4} (h - 4) v_h ) = v_4 - 10 - \sum_{h\ge5} (h-5) v_h,
\end{align*}
which implies $v_4 = 10 + \sum_{h\ge5} (h-5)v_h + \#_{\ge4}\theta$ and $v_4 >  \#_{\ge4}\theta$. Moreover, as $\#_{\ge4} \theta \ge 0$, we get $v_4 \ge 10$, so
\begin{align*}
f=6 + \sum_{h\ge4} (h-3)v_h \ge 6 + 10 + \sum_{h\ge5} (h-3) v_h \ge 16.
\end{align*}
As $v_4 > \#_{\ge4}\theta \ge 0$, there must be a degree $4$ vertex without $\theta$. When $\varphi$ does not appear in every degree $i$ vertex, where $4 \le i \le k$ and $k \ge 5$, assume that $\varphi$ appears at most $\floor*{\frac{k}{2}}$ times at every degree $\ge k+1$ vertex. Then $\#_{\ge k+1} \varphi \le \sum_{h\ge k+1} \floor*{\frac{k}{2}} v_h$ and counting $\varphi$, we have $f=\# \varphi = \#_3 \varphi + \#_{\ge k+1} \varphi \le \frac{f}{2} + \sum_{h\ge k+1} \floor*{\frac{k}{2}} v_h$, then by $v_4 = 10 + \sum_{h\ge5} (h-5)v_h + \#_{\ge4}\theta$
\begin{align*}
0 &\le -\frac{f}{2} + \sum_{h\ge k+1} \floor*{\frac{k}{2}} v_h \\
&= -\frac{1}{2} \left( 6 + v_4 + \sum_{h=5}^{k} (h-3)v_h  + \sum_{h \ge k+1} ( h - 3 - 2\floor*{\frac{k}{2}} )  v_h  \right) \\
&= -\frac{1}{2} \left( 16 +  \sum_{h=5}^{k} (2h-8)v_h + \#_{\ge4}\theta  + \sum_{h \ge k+1} ( 2h - 8 - 2\floor*{\frac{k}{2}} )  v_h  \right), 
\end{align*}
which is $< 0$, a contradiction. So there must be a degree $\ge k+1$ vertex $\varphi^{\floor*{\frac{k}{2}}+1}\cdots$. When $\AVC_4 = \{ \psi^4 \}$ where $\psi \neq \theta, \varphi$, we have $\#_{4}\theta = 0$ and $f = \#\psi = \#_4 \psi + \#_{\ge5} \psi = 4v_4 +  \#_{\ge5} \psi$, then
\begin{align*}
\#_{\ge5} \psi = f - 4v_4 &= 6 + v_4 + \sum_{h\ge5} (h-3)v_h - 4v_4 \\
&= 6 + \sum_{h\ge5} (h-3)v_h - 3( 10 + \sum_{h\ge5} (h-5) v_h + \#_{\ge4}\theta ) \\
&= 2v_5 - 24 - \sum_{h\ge6}2(h-6)v_h - 3\#_{\ge5}\theta,
\end{align*}
so we get
\begin{align*}
2v_5 = \#_{\ge5} \psi  + 3\#_{\ge5}\theta + 24 + \sum_{h\ge6}2(h-6)v_h.
\end{align*}
which implies $2v_5 > \#_5 \psi + \#_{5}\theta$ and $v_5 > \#_5 \theta$. Hence there must be a degree $5$ vertex with at most one of $\psi, \theta$, and there must be a vertex without $\theta$.

When $\theta$ and $\psi$, where $\psi \neq \varphi, \theta$, together appear four times at every degree $4$, let $\xi := \theta, \psi$, so we have $2f = \# \xi = \#_3 \xi + \#_4 \xi + \#_{\ge5} \xi = 2v_3 + 4v_4 + \#_{\ge5}\xi$. Then 
\begin{align*} 
 \#_{\ge5}\xi &= 2f - 2v_3 - 4v_4 \\
&= 2(6 + v_4 + \sum_{h\ge5} (h-3)v_h) - 2( 8 + \sum_{h\ge4} (h - 4) v_h ) - 4v_4 \\
&= 2( ( 6 + \sum_{h\ge5} (h-3)v_h ) - ( 8 + \sum_{h\ge5} (h - 4) v_h ) - ( 10 + \sum_{h\ge5} (h-5)v_h + \#_{\ge4}\theta  ) ) \\
&= 2( -12 - \#_{\ge4}\theta + v_5 - \sum_{h\ge6} (h-6)v_h ) 
\end{align*}
which implies
\begin{align*}
2v_5 =  \#_{\ge5}\xi + 24 + 2\#_{\ge4}\theta + \sum_{h\ge6} 2(h-6)v_h.
\end{align*}
So we get $v_5 > 0$ and $2 v_5 > \#_{\ge5}\xi$ and hence there is a degree $5$ vertex with at most of $\theta,\psi$.

When $\varphi$ appears at least twice at every degree $4$ vertex, we have $f = \#\varphi \ge v_3 + 2v_4 + \#_{\ge5} \varphi$. Then
\begin{align*}
0 &\ge v_3 + 2v_4 + \#_{\ge5} \varphi - f =   v_3 + 2v_4 + \#_{\ge5} \varphi - ( 6 + v_4 + \sum_{h\ge5} (h-3)v_h ) \\
& = \#_{\ge5} \varphi + 8 + \sum_{h\ge4} (h - 4) v_h +  10 + \sum_{h\ge5} (h-5) v_h + \#_{\ge4}\theta  - 6 - \sum_{h\ge5} (h-3)v_h \\
&= \#_{\ge5} \varphi  +  \#_{\ge4}\theta + 12 - v_5 + \sum_{h\ge6} (h-6)v_h
\end{align*}
which implies $v_5  \ge \#_{\ge5} \varphi  +  \#_{\ge4}\theta + 12 + \sum_{h\ge6} (h-6)v_h$. In particular, $v_5  > \#_5\varphi + \#_5\theta$ shows that there is a degree $5$ vertex without both $\varphi$, $\theta$.

When $\psi$ does not appear at any degree $3,4$ vertex, for $\xi \neq \psi$ assume that $\xi$ appears at least three times at each degree $5$ vertex, at least twice at each degree $6$ vertex and at least once at each degree $7$ vertex. Then $\# \xi = 3f \ge 3v_3 + 4v_4 +3v_5 + 2v_6+v_7$, which by $f=6 + \sum_{h\ge4} (h-3)v_h$, $v_3 = 8 + \sum_{h\ge4} (h - 4) v_h$ and $v_4 = 10 + \sum_{h\ge5} (h-5)v_h + \#_{\ge4}\theta$ implies
\begin{align*}
3v_3 + 4v_4 +3v_5 + 2v_6+v_7 - 3f &= 6 + v_4 + 3v_5 + 2v_6+v_7 - \sum_{h\ge5} 3v_h \\
& \ge 16 + \sum_{h\ge8}(h-8)v_h > 0,
\end{align*}
a contradiction.
\end{proof}

\subsection{Anglewise Vertex Combinations ($\AVC$)}

\begin{lem}\label{CountLem} If $\beta\gamma\delta$ is the only vertex with $\beta$ (equivalently $\gamma$ or $\delta$), then the other vertices are $\alpha^a$ where $a\ge3$.
\end{lem}

\begin{proof} By $\beta\cdots = \beta\gamma\delta$, we get $\#\beta = \# \beta\gamma\delta$. Combining with $\#\beta = \# \gamma = \#\delta$, we have $\# \gamma = \# \delta = \# \beta\gamma\delta$, which implies $\beta\gamma\delta$ is the only vertex with $\gamma$ and $\delta$. So the other vertices can only be $\alpha^a$ where $a\ge3$.
\end{proof}

\begin{lem} \label{ForbVerLem} In a tiling of the sphere by congruent quadrilateral $\bm{a^2bc}$ in Figure \ref{DefaultQuad}, the possible vertices are 
\begin{align*}
&\alpha^a, \
\alpha^a\beta^b, \
\alpha^a\beta^b\gamma^c, \
\alpha^a\beta^b\delta^d, \
\alpha^a\gamma^c\delta^d, \
\alpha^a\delta^d; \quad \text{where } a > 0 \text{ and } b,c,d \text{ are even};   \\
&\beta^b, \
\beta^b\gamma^c, \
\beta^b\delta^d, \
\gamma^c, \
\gamma^c\delta^d, \
\delta^d; \quad \text{where } b,c,d \text{ are even};  \\
&\beta^b\gamma^c\delta^d, \qquad \text{where } b,c,d \text{ are all even or all odd};   
\end{align*}
\end{lem}

\begin{proof} By Lemma \ref{PaLem}, when a vertex containing exactly one of $\beta$, $\gamma$, $\delta$, the number must be even, so $\alpha^a\varphi^{\,p}$, $a \ge 0$, $p$ odd, cannot be a vertex. Also by Lemma \ref{PaLem}, when a vertex containing exactly two of $\beta$, $\gamma$, $\delta$, then both $p$ and $q$ must be even, so $\alpha^a\varphi^{\,p}\psi^{\, q}$, $a \ge 0$, $p$ or $q$ odd, cannot be a vertex. Again by Lemma \ref{PaLem}, $b, c, d$ must be all even or all odd in $\beta^b\gamma^c\delta^d$ when $b, c, d > 0$. In $\alpha^a\gamma^c$, $a\ge1$ and by Lemma \ref{PaLem}, $c\ge2$, then there must be a non-empty $\bvert \, \gamma \Vert \cdots \Vert \gamma \, \bvert$ completed by $\alpha \cdots \alpha = \vert \alpha \vert \cdots \vert \alpha \vert$, a contradiction. By quadrilateral angle sum, no vertex contains all four angles. 
\end{proof}

To streamline the arguments, from this point onward, when we give a generic vertex such as $\alpha^a\beta^b\gamma^c$, it will be assumed to satisfy the conditions in Lemma \ref{ForbVerLem} unless otherwise specified. The only vertices that allow odd numbers of $\beta, \gamma, \delta$ is $\beta^b\gamma^c\delta^d$ and $\alpha^a\gamma^c$ is never a vertex.

\begin{cor} The list of degree $3, 4, 5$ vertices are 
\begin{align*}
\AVC_3 &= \{\alpha^3,   \alpha\beta^2,  \alpha\delta^2,  \beta\gamma\delta  \}, \\
\AVC_4 &= \{ \alpha^4,  \beta^4,  \gamma^4,   \delta^4,   \alpha^2\beta^2,   \alpha^2\delta^2,   \beta^2\gamma^2,   \beta^2\delta^2,   \gamma^2\delta^2 \}, \\
\AVC_5 &= \{  \alpha^5,   \alpha\beta^4,   \alpha\delta^4,   \alpha^3\beta^2,  \alpha^3\delta^2,   \beta^3\gamma\delta,  \beta\gamma^3\delta,  \beta\gamma\delta^3,  \alpha\beta^2\gamma^2,  \alpha\beta^2\delta^2,  \alpha\gamma^2\delta^2 \}.
\end{align*}
\end{cor}

\begin{proof} Apply Lemma \ref{ForbVerLem} to degree $3$ vertices of the types $\theta^3$, $\theta^2\varphi$, $\theta\varphi\psi$, 
\begin{align*}
&\theta^3 = \alpha^3,&
&\theta^2\varphi =  \alpha\beta^2, \alpha\delta^2,&
&\theta\varphi\psi = \beta\gamma\delta.&
\end{align*}

Apply Lemma \ref{ForbVerLem} to degree $4$ vertices of the types $\theta^4$, $\theta^3\varphi$, $\theta^2\varphi^2$, $\theta^2\varphi\psi$, 
\begin{align*}
&\theta^4 = \alpha^4,  \beta^4, \gamma^4, \delta^4,&
&\theta^2\varphi^2 = \alpha^2\beta^2,  \alpha^2\delta^2,  \beta^2\gamma^2, \beta^2\delta^2,  \gamma^2\delta^2,&
\end{align*}
whereas the other combinations are impossible. 

Apply Lemma \ref{ForbVerLem} to degree $5$ vertices of the types $\theta^5$, $\theta^4\varphi$, $\theta^3\varphi^2$, $\theta^3\varphi\psi$, $\theta^2\varphi^2\psi$, 
\begin{align*}
&\theta^5 = \alpha^5, \qquad
\theta^4\varphi = \alpha\beta^4,  \alpha\delta^4, \qquad
\theta^3\varphi^2 = \alpha^3\beta^2, \,\, \alpha^3\delta^2, \\
&\theta^3\varphi\psi = \beta^3\gamma\delta,  \beta\gamma^3\delta,  \beta\gamma\delta^3, \qquad
\theta^2\varphi^2\psi = \alpha\beta^2\gamma^2,  \alpha\beta^2\delta^2,  \alpha\gamma^2\delta^2,
\end{align*}
whereas the other combinations are impossible. 
\end{proof}

\subsection{Geometry}

\begin{lem} \label{AngLBLem} In a tiling of the sphere by congruent convex quadrilaterals, for every angle $\theta$ in the quadrilateral we have $\theta > \frac{2\pi}{f}$. 
\end{lem}

\begin{proof} For every angle $\theta$, by convexity, the quadrilateral is contained in the lune defined $\theta$. The area of the lune is $2\theta$ whereas the area of the quadrilateral is $\frac{4\pi}{f}$. Then we get $2\theta > \frac{4\pi}{f}$, which implies $\theta >  \frac{2\pi}{f}$.
\end{proof}

The following proposition is taken from Proposition 4.9 in \cite{ch}.

\begin{proposition}\label{EqAngEqEdge} In a spherical triangle such that none of its angles is $\pi$, two angles are equal if and only if their opposite edges are equal. 
\end{proposition}

\begin{lem} \label{AngEqEdgeEqLem} In a tiling of the sphere by congruent quadrilaterals $\bm{a^2bc}$ in Figure \ref{DefaultQuad}, if $\beta=\delta$ and $\gamma \neq \pi$, then $\bm{b}=\bm{c}$. In other words, there is no tilings for $\beta = \delta$ and $\gamma \neq \pi$.
\end{lem}

\begin{proof}By Lemma \ref{AlConvexLem}, there can be at most one angle $\ge\pi$ in a quadrilateral. When $\beta=\delta$, then $\beta$, $\delta < \pi$. 

When $\alpha$, $\gamma<\pi$, the isosceles triangle $\triangle ABD$ is contained in the quadrilateral with base angles $\beta_1$, $\delta_1$ and their respective complements $\beta_2$, $\delta_2$. Then by $\beta=\delta$ and $\beta_1= \delta_1$, we get
\begin{align*}
\beta_2 = \delta_2,
\end{align*}
which implies $\bm{b}=BC=DC=\bm{c}$.

\begin{figure}[htp]
\centering
\begin{tikzpicture}

\draw
	(0,0) -- (2,0) -- (2,2) -- (0,2) -- cycle;

\draw[dashed]
	(0,0) -- (2,2);

\draw[double, line width=0.6]
	(2,0) -- (2,2);

\draw[line width=2]
	(0,0) -- (2,0);

\node at (-0.2, -0.2) {$D$};
\node at (-0.2, 2.2) {$A$};
\node at (2.2, 2.2) {$B$};
\node at (2.2, -0.2) {$C$};

\node at (0.2,1.8) {\small $\alpha$};
\node at (0.25,0.6) {\small $\delta_1$};
\node at (0.6,0.25) {\small $\delta_2$};
\node at (1.35,1.7) {\small $\beta_1$};
\node at (1.75,1.4) {\small $\beta_2$};
\node at (1.75,0.25) {\small  $\gamma$};


\begin{scope}[xshift=4cm]

\draw
	(0,0) -- (0.6,1.45) -- (2,2) -- (0,2) -- cycle;

\draw[dashed]
	(0,0) -- (2,2);

\draw[double, line width=0.6]
	(0.6,1.45) -- (2,2);

\draw[line width=2]
	(0,0) -- (0.6,1.45);

\node at (-0.2, -0.2) {$D$};
\node at (-0.2, 2.2) {$A$};
\node at (2.2, 2.2) {$B$};
\node at (0.8, 1.25) {$C$};

\node at (0.2,1.8) {\small $\alpha$};
\node at (0.13,0.8) {\small $\delta$};
\node at (0.42,0.58) {\small $\delta'$};
\node at (1.1,1.8) {\small $\beta$};
\node at (1.42,1.52) {\small $\beta'$};
\node at (0.45,1.55) {\small  $\gamma$};

\end{scope}

\begin{scope}[xshift=6.5cm]

\draw
	(0,0) -- (2,0) -- (2,2) -- (1.45,0.6) -- cycle;

\draw[dashed]
	(0,0) -- (2,2);

\draw[double, line width=0.6]
	(2,0) -- (2,2);

\draw[line width=2]
	(0,0) -- (2,0);

\node at (-0.2, -0.2) {$D$};
\node at (1.25,0.8) {$A$};
\node at (2.2, 2.2) {$B$};
\node at (2.2, -0.2) {$C$};

\node at (1.55,0.45) {\small $\alpha$};
\node at (0.6,0.42) {\small $\delta'$};
\node at (0.85,0.16) {\small $\delta$};
\node at (1.6,1.38) {\small $\beta'$};
\node at (1.82,1.1) {\small $\beta$};
\node at (1.8,0.2) {\small  $\gamma$};

\end{scope}

\end{tikzpicture}
\caption{$\beta = \delta$ implies $\bm{b}=\bm{c}$}
\label{Geom1}
\end{figure}
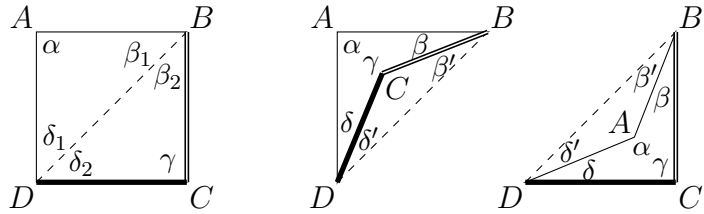

When $\gamma > \pi$, the isosceles triangle $\triangle ABD$ contains the quadrilateral and the complementary triangle $\triangle CBD$ has base angles $\beta'$, $\delta'$. By the base angles of $\triangle ABD$, we have $\beta+\beta' = \delta + \delta'$, which implies $\beta' = \delta'$ in $\triangle CBD$ and hence $\bm{b}=\bm{c}$. 

When $\alpha = \pi$, the quadrilateral is in fact an isosceles $\triangle BCD$. Then $\beta=\delta$ implies $\bm{b}=\bm{c}$.

When $\alpha > \pi$, the triangle $\triangle BCD$ contains the quadrilateral and the complementary triangle $\triangle ABD$ is an isosceles triangle with base angles $\beta'$, $\delta'$. By $\beta = \delta$ and $\beta' = \delta'$, we have the base angles $\beta+\beta'=\delta+\delta'$ in $\triangle BCD$ whereby Proposition \ref{EqAngEqEdge} implies $\bm{b}=\bm{c}$.
\end{proof}

\begin{cor} In a tiling of the sphere by congruent quadrilaterals $\bm{a^2bc}$ and angles $\alpha, \beta, \gamma,  \delta$, there are at least two distinct angles.  
\end{cor}

\begin{proof} By Lemma \ref{AngEqEdgeEqLem}, when $\beta = \delta$ and $\gamma\neq\pi$, we have $\bm{b}=\bm{c}$, a contradiction. So we either have $\gamma = \pi$ and by Lemma \ref{AlConvexLem}, all other angles $<\pi$ or $\beta\neq\delta$. 
\end{proof}

The following lemma is an adaptation of a lemma given by Akama and Van Cleemput in \cite{avc}.

\begin{lem} \label{LunEstLem} In a tiling of the sphere by congruent quadrilaterals $\bm{a^2bc}$ in Figure \ref{DefaultQuad}, if the quadrilateral is convex we have the following inequalities
\begin{align*}
\gamma + \delta < \pi + \beta, \\
\gamma + \beta < \pi + \delta.
\end{align*}
\end{lem}

\begin{proof} By convexity, the quadrilateral can be divided into two triangles. Since every angle in each of the triangles is $< \pi$, so $\bm{a}, \bm{b}, \bm{c} < \pi$. The triangles $\triangle ABD$ and $\triangle BCD$ in the first picture in Figure \ref{Geom1} are contained in the quadrilateral. Let $\beta_1 = \delta_1$ be the base angles of isosceles $\triangle ABD$ and the corresponding complementary angles $\beta_2, \delta_2$ in $\triangle BCD$. Since $\triangle BCD$ is contained in the $\beta_2$-lune and $\delta_2$-lune, the area of $\triangle BCD$ is less than the area of both lunes, which implies
\begin{align*}
&\gamma + \beta_2 + \delta_2 - \pi < 2\beta_2,&  
&\gamma + \beta_2 + \delta_2 - \pi < 2\delta_2,&
\end{align*}
simplifying the above inequalities, and adding $\beta_1$ to the RHS of the first inequality and the LHS of the second, and adding $\delta_1$ to the LHS of the first inequality and the RHS of the second, we get the inequalities intended.
\end{proof}

\begin{lem}\label{albe2Lem} If $\alpha\beta^2$ is a vertex, then $\alpha\beta\cdots = \alpha\beta^2$, and $\alpha\delta^2$, $\alpha^a\beta^b\gamma^c$, $\alpha^a\beta^b\delta^d$, $\alpha^a\gamma^c\delta^d$ cannot be a vertex, and some degree $\ge4$ $\alpha^a\delta^d$ is a vertex. 
\end{lem}

\begin{proof} By $\alpha\beta^2$ and Lemma \ref{ForbVerLem}, $\alpha^a\beta^b = \alpha\beta^2$. Up to symmetry, the AAD of $\alpha\beta^2 = \Vert \beta^{\alpha} \vert^{\beta} \alpha^{\delta} \vert^{\alpha} \beta \Vert$. So some $\alpha\delta\cdots$ must be a vertex. 

When both $\alpha\beta^2$, $\alpha\delta^2$ are vertices, we get the angle formulae $\beta= \delta = \pi - \frac{\alpha}{2}$ and $\gamma=\frac{4\pi}{f}$. By $f\ge6$, $\gamma < \pi$, Lemma \ref{AngEqEdgeEqLem} implies $\bm{b}=\bm{c}$, a contradiction. So $\alpha\delta^2$ cannot be a vertex. By Lemma \ref{ForbVerLem}, $b,c$ are both even at $\alpha^a\beta^b\gamma^c$ and $b,d$ are both even at $\alpha^a\beta^b\delta^d$. By $\alpha\beta^2$, the angle sums of $\alpha^a\beta^b\gamma^c$ and $\alpha^a\beta^b\delta^d$ are $>2\pi$, so they cannot be vertices. By Lemma \ref{ForbVerLem}, $c,d\ge2$ at $\alpha^a\gamma^c\delta^d$, then we have $\alpha + 2\gamma + 2\delta \le 2\pi$, which combined with the angle sum of $\alpha\beta^2$ and divided by $2$ on both sides we get $\alpha+\beta+\gamma+\delta \le 2\pi$, contradicting the quadrilateral angle sum. So $\alpha^a\gamma^c\delta^d$ cannot be a vertex. 
In $\alpha\delta\cdots$, if there is $\beta$ or $\gamma$ in the remainder, as $\alpha^a\beta^b\delta^d$ $\alpha^a\gamma^c\delta^d$ are not vertices, we can only have $\alpha\delta\cdots=\alpha\beta\gamma\delta\cdots$, a contradiction. So there is no $\beta,\gamma$ in the remainder and hence $\alpha\delta\cdots = \alpha^a\delta^d$.
\end{proof}

\section{Vertices}

By Corollary to Lemma \ref{ForbVerLem}, 
\begin{align*}
\AVC_3 &= \{\alpha^3,   \alpha\beta^2,  \alpha\delta^2,  \beta\gamma\delta  \}, \\
\AVC_4 &= \{ \alpha^4,  \beta^4,  \gamma^4,  \delta^4,   \alpha^2\beta^2,  \alpha^2\delta^2,  \beta^2\gamma^2,  \beta^2\delta^2,  \gamma^2\delta^2 \}, \\
\AVC_5 &= \{  \alpha^5,  \alpha\beta^4,  \alpha\delta^4,   \alpha^3\beta^2,   \alpha^3\delta^2,  \beta^3\gamma\delta,  \beta\gamma^3\delta,   \beta\gamma\delta^3,    \alpha\beta^2\gamma^2,   \alpha\beta^2\delta^2,  \alpha\gamma^2\delta^2 \}.
\end{align*}

\begin{proposition} If $\beta\gamma\delta$ is a vertex, then $\AVC$ is given in \eqref{bgdAVC}. 
\end{proposition}

\begin{proof} Up to symmetry between $\beta, \gamma, \delta$ at $\beta\gamma\delta$, it suffices to consider $\beta^2\cdots$. Assume the contrary, then by the symmetry between $\gamma, \delta$ in the vertex angle sum of $\beta\gamma\delta$ and Lemma \ref{PaLem}, suppose $\beta^2\cdots = \beta^2\delta^2\cdots$. The angle sum suggests that $\beta + \delta \le \pi$ which implies $\gamma \ge \pi$. However $\gamma \ge \pi$ implies that $\gamma^2\cdots$ is not a vertex, and then Lemma \ref{BaLem} implies that $\delta^2\cdots$ cannot be a vertex, contradicting $\beta^2\delta^2\cdots$ being a vertex. Hence, $\beta^2\cdots$ has no $\gamma, \delta$ and $\delta^2\cdots$ has no $\beta, \gamma$. In other words, $\beta^2\cdots$ is a $\hat{\bm{c}}$-vertex and $\delta^2\cdots$ is a $\hat{\bm{b}}$-vertex.

The AAD of $\bvert \, \gamma \Vert \cdots$ is $\bvert \, \gamma \Vert \beta \vert \cdots$ or $\bvert \, \gamma \Vert \gamma \, \bvert \cdots$. If $\bvert \, \gamma \Vert \cdots  =$ $\bvert \, \gamma \Vert \beta \vert \cdots$, by $\beta\gamma\delta$ and $\beta^2\cdots$ being a $\hat{\bm{c}}$-vertex, there is no more $\beta$ at the vertex. Apply Lemma \ref{PaLem} on $\bm{c}$, we get $\bvert \, \gamma \Vert \beta \vert \cdots = \vert \delta \, \bvert \, \gamma \Vert \beta \vert$, $\Vert \gamma \, \bvert \, \gamma \Vert \beta \vert \cdots$. The former is indeed the vertex $\beta\gamma\delta$ whereas the latter can only be filled by $\alpha, \gamma$ so it can only be $\alpha^a\beta\gamma^c(a\ge0, c\ge2)$ which by Lemma \ref{ForbVerLem} cannot be a vertex. So $\bvert \, \gamma \Vert \beta \vert \cdots = \beta\gamma\delta$. If $\bvert \, \gamma \Vert \cdots =$ $\bvert \, \gamma \Vert \gamma \, \bvert \cdots$, by $\beta^2\cdots$ being a $\hat{\bm{c}}$-vertex, $\delta^2\cdots$ being a $\hat{\bm{b}}$-vertex and Lemma \ref{ForbVerLem}, we get $\bvert \, \gamma \Vert \gamma \, \bvert \cdots = \gamma^c$. Hence 
\begin{align*}
\gamma\cdots = \beta\gamma\delta, \gamma^c.
\end{align*}

The AAD of $\bvert \, \delta \vert \cdots$ is $\bvert \, \delta \vert \alpha \vert \cdots$, $\bvert \, \delta \vert \beta \Vert \cdots$ or $\bvert \, \delta \vert \delta \, \bvert  \cdots$. If $\bvert \, \delta \vert \cdots =$ $\bvert \, \delta \vert \beta \Vert \cdots$, by $\beta\gamma\delta$, $\beta^2\cdots$ being a $\hat{\bm{c}}$-vertex and $\delta^2\cdots$ being a $\hat{\bm{b}}$-vertex, we get $\bvert \, \delta \vert \beta \Vert \cdots = \Vert \gamma \, \bvert \, \delta \vert \beta \Vert =\beta\gamma\delta$. If $\bvert \, \delta \vert \cdots =$ $\bvert \, \delta \vert \alpha \vert \cdots$ or $\bvert \, \delta \vert \delta \, \bvert  \cdots$, by $\beta\gamma\delta$ and Lemma \ref{ForbVerLem}, we get $\delta^d$, $\alpha^a\delta^d$. Hence
\begin{align*}
\delta\cdots = \beta\gamma\delta, \delta^d, \alpha^a\delta^d.
\end{align*}

By symmetry between $\beta,\delta$ in the default quadrilateral in Figure \ref{DefaultQuad}, we get 
\begin{align*}
\beta\cdots = \beta\gamma\delta, \beta^b, \alpha^a\beta^b.
\end{align*}
Hence the vertices are
\begin{align} \label{begadeV}
\beta\gamma\delta, 
\alpha^a, 
\beta^b, 
\gamma^c, 
\delta^d, 
\alpha^a\beta^b, 
\alpha^a\delta^d. 
\end{align}

By $\beta\gamma\delta$ and quadrilateral angle sum, we get $\alpha = \frac{4\pi}{f}$. We divide the discussion into two cases, $\beta \ge \pi$ and $\beta < \pi$. 

\begin{case*}[$\beta\ge\pi$] By the symmetry of $\beta, \gamma, \delta$ at $\beta\gamma\delta$, it suffices to discuss $\beta \ge \pi$. When $\beta \ge \pi$, $\beta^2\cdots$ cannot be a vertex. The contrapositive of Lemma \ref{OneHDAng} implies that $\beta$ must appear at some degree $3$ or $4$ vertex. However, since every vertex with $\beta$ has exactly one $\beta$, by Lemma \ref{ForbVerLem}, $\beta$ cannot appear at any degree $4$ vertex and indeed $\beta \cdots = \beta\gamma\delta$. Then Lemma \ref{CountLem} implies that the vertices are $\beta\gamma\delta, \alpha^a$. As $\alpha = \frac{4\pi}{f}$, we have $a=\frac{f}{2}$. Hence
\begin{align} \label{bgdAVC}
\AVC = \{ \beta\gamma\delta,  \alpha^{\frac{f}{2}} \}.
\end{align}
\end{case*}

\begin{case*}[$\beta<\pi$] 
When $\beta < \pi$, by $\beta\gamma\delta$, $\beta^2\cdots$ being a $\hat{\bm{c}}$-vertex, $\delta^2\cdots$ being a $\hat{\bm{b}}$-vertex and  Lemma \ref{ForbVerLem} we get $\bvert \, \gamma \Vert \cdots \vert \delta \, \bvert =$ $\bvert \, \gamma \Vert \beta \vert \delta \, \bvert$. So $\gamma\delta\cdots = \beta\gamma\delta$. By $\beta\gamma\delta$ and $\beta < \pi$, we have $\gamma + \delta > \pi$. Then $\max\{\gamma, \delta\} > \frac{\pi}{2}$. From \eqref{begadeV}, we have $\gamma^2\cdots = \gamma^{c\ge4}$, $\gamma\delta\cdots = \beta\gamma\delta$, and $\delta^2\cdots = \delta^d, \alpha^a\delta^d$.

\begin{subcase*}[$\gamma \ge \delta$] Since $\gamma^2\cdots = \gamma^{c\ge4}$, by $\gamma + \delta > \pi$ and $\gamma \ge \delta$, we have $\gamma > \frac{\pi}{2}$ and then $\gamma^4\cdots$ cannot be a vertex, so $\gamma^2\cdots$ cannot be a vertex. By Lemma \ref{BaLem}, $\delta^2\cdots$ cannot be a vertex. So the $\bm{c}$-vertices can only be $\beta\gamma\delta$. By Lemma \ref{CountLem}, the other vertices can only be $\alpha^{a \ge 3}$. Hence we get the same AVC in \eqref{bgdAVC}.
\end{subcase*}
\begin{subcase*}[$\gamma \le \delta$] Since $\delta^2\cdots = \delta^d, \alpha^a\delta^d$, by $\gamma + \delta > \pi$ and $\gamma \le \delta$, we have $\beta\delta\cdots = \beta\gamma\delta$ and $\delta > \frac{\pi}{2}$, so $\delta^4\cdots$ is not a vertex and hence $\delta^2\cdots = \alpha^a\delta^2$. Then the vertices are $\alpha^a, \, \beta^b, \, \alpha^a\beta^b, \, \beta\gamma\delta, \, \gamma^c, \, \alpha^a\delta^2$ from which we know that $\beta\gamma\cdots = \beta\delta\cdots = \beta\gamma\delta$, $\alpha\beta\cdots = \alpha^a\beta^b$, $\alpha\delta\cdots = \alpha^a\delta^2$, $\gamma^2\cdots = \gamma^c$ and $\alpha\gamma\cdots$, $\delta \vert \delta \cdots$, $\alpha^{\delta} \vert^{\delta} \alpha \cdots$ cannot be vertices.

When $\alpha^a\delta^2$ is a vertex, we have $\delta < \pi$. By $\beta\gamma\delta$, $\alpha^a\delta^2$ and $\gamma \le \delta$, we have $\beta + \gamma + \delta = a\alpha + 2\delta$ which implies $\beta - a\alpha =  \delta - \gamma \ge 0$. So by $\beta < \pi$ and $a \ge 1$, we get $\alpha \le \beta < \pi$. Since every angle is $<\pi$, the quadrilateral is convex. By convexity, Lemma \ref{LunEstLem} implies $\gamma + \delta < \pi + \beta$. Then by $\beta\gamma\delta$, we have $\gamma + \delta = 2\pi - \beta$, combining with the inequality, we get $\beta > \frac{\pi}{2}$. So Lemma \ref{ForbVerLem} implies that $\beta^b$, $\alpha^a\beta^{b\ge3}$ cannot be vertices. Then $\alpha^a\beta^b = \alpha^{a\ge1}\beta^2 = \Vert \beta \vert \alpha \cdots \alpha \vert \beta \Vert$ and hence $\beta \vert \beta \cdots$ and $\alpha^{\beta} \vert^{\beta} \alpha\cdots$ cannot be vertices. As a result, we get $\alpha \vert \alpha = \alpha^{\beta} \vert^{\delta} \alpha$ and hence $\alpha \vert \alpha$ has unique AAD $\cdots \vert^{\delta} \alpha^{\beta}  \vert ^{\delta} \alpha^{\beta} \vert \cdots$, by which we get 
\begin{align*}
\alpha^{a}\delta^2 =\, \bvert \delta^{\alpha} \vert^{\beta} \alpha^{\delta} \vert \cdots \vert^{\beta} \alpha^{\delta} \vert^{\alpha} \delta \, \bvert.
\end{align*}
By $\alpha\delta\cdots = \alpha^a\delta^2$ and $\beta\delta\cdots = \beta\gamma\delta$, the AAD at $ \cdots \, ^{\delta} \vert^{\beta} \alpha^{\delta} \vert^{\alpha} \delta \, \bvert$ starting at tiles $T_1, T_2, T_3$ in Figure \ref{vbegade-alabe2} implies $\beta_1\cdots = \beta\gamma\delta$ and $\delta_1\cdots = \alpha^a\delta^2$. Then we get $\gamma_1\cdots = \beta\gamma^2\cdots$ which is not a vertex, a contradiction. So $\alpha^a\delta^2$ cannot be a vertex and hence $\delta \, \bvert  \, \delta \cdots$ cannot be a vertex by which Lemma \ref{BaLem} implies that $\gamma \, \bvert  \, \gamma \cdots$ is also not a vertex. Then $\gamma^c$ is not a vertex and hence $\gamma \, \Vert \, \gamma \cdots$ cannot be a vertex by which Lemma \ref{BaLem} implies that $\alpha^a\beta^2$ cannot be a vertex. So we obtain the same $\AVC$ in \eqref{bgdAVC}.

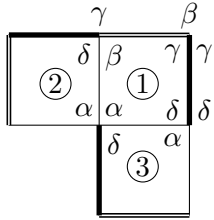
\begin{figure}[htp]
\centering
\begin{tikzpicture}[>=latex,scale=1]

\foreach \a in {0,...,2}
\draw[rotate=90*\a]
	(0,0) -- (0,-1.2) -- (1.2,-1.2) -- (1.2,0);

\draw[]
	(0,0) -- (-1.2,0);

\draw[double, line width=0.6]
	(0,-1.2) -- (1.2,-1.2)
	(-1.2,1.2) -- (-1.2,0) 
	(0,1.2) -- (1.2,1.2);

\draw[line width=2]
	(0,0) -- (0,-1.2)
	(1.2,0) -- (1.2,1.2)
	(0,1.2) -- (-1.2,1.2); 

\node at (0.2,0.2) {\small $\alpha$}; 
\node at (0.2,0.90) {\small $\beta$};
\node at (0.98,0.95) {\small $\gamma$}; 
\node at (0.98,0.2) {\small $\delta$};

\node at (-0.2,0.2) {\small $\alpha$}; 
\node at (-0.2,0.95) {\small $\delta$};

\node at (0.2,-0.25) {\small $\delta$}; 
\node at (0.98,-0.2) {\small $\alpha$};

\node at (0.0,1.45) {\small $\gamma$}; 
\node at (1.4, 0.2) {\small $\delta$};
\node at (1.2, 1.45) {\small $\beta$};
\node at (1.4,0.95) {\small $\gamma$};

\node[inner sep=1,draw,shape=circle] at (0.58,0.55) {\small $1$};
\node[inner sep=1,draw,shape=circle] at (-0.58,0.55) {\small $2$};
\node[inner sep=1,draw,shape=circle] at (0.58,-0.55) {\small $3$};

\end{tikzpicture}
\caption{The AAD of $\alpha^a\delta^2$}
\label{vbegade-alabe2}
\end{figure}
\end{subcase*}
\end{case*}
The tilings of $\AVC$ in \eqref{bgdAVC} are constructed in Figure \ref{begadeEMT}. 
\end{proof}

\begin{proposition}\label{albe2Prop} When $\alpha\beta^2$ and $\gamma^2\delta^2$ are vertices, then we have $f>8$, $\beta > \delta \ge \frac{\pi}{2}$ and $\AVC = \{ \alpha\beta^2, \gamma^2\delta^2, \alpha^2\delta^2, \gamma^c \}$ and no tilings.
\end{proposition}

\begin{proof} By Lemma \ref{albe2Lem}, $\alpha^a\beta^b = \alpha\beta^2$ and some degree $\ge4$ $\alpha^a\delta^d$ must be a vertex and $\alpha\delta^2, \alpha^a\beta^b\gamma^c$, $\alpha^a\beta^b\delta^d, \alpha^a\gamma^c\delta^d$ are not vertices. Since $\gamma^2\delta^2$ is also a vertex, Lemma \ref{ForbVerLem} implies $\gamma^c\delta^d=\gamma^2\delta^2$. If $\beta^2\gamma^2$ is also a vertex, then it implies $\gamma < \pi$ and $\beta=\delta$, by Lemma \ref{AngEqEdgeEqLem} a contradiction. So $\beta^2\gamma^2$ cannot be a vertex. Then if $\beta^b\gamma^c$ is a vertex, we have $b \ge 2, c \ge 2$ and $b+c \ge 6$.

Next we show degree $\ge4$ $\alpha^a\delta^d$, $\beta^b\gamma^c (b \ge 2, c \ge 2, b+c \ge 6)$ cannot both appear as vertices. Assume they both are vertices. Then one of $(b-2), (c-2)$ is positive. Sum up two vertex equations and substitute the angle sums of $\alpha\beta^2, \gamma^2\delta^2$, we get $4\pi = a\alpha + b\beta + c\gamma+d\delta \ge 4\pi + (b-2)\beta + (c-2)\gamma + (d-2)\delta$, a contradiction. Hence the two vertices cannot both appear.

Since some degree $\ge4$ $\alpha^a\delta^d$ is a vertex and $\beta^b\gamma^c$ is not, by Lemma \ref{ForbVerLem} and the discussion above, $\gamma\cdots = \gamma^2\delta^2, \gamma^c, \beta^b\gamma^c\delta^d$. Meanwhile, by $\alpha\beta^2$, $\gamma^2\delta^2$, we get angle formulae
\begin{align*}
\alpha = \frac{8\pi}{f}, \qquad
\beta = \pi - \frac{4\pi}{f}, \qquad
\gamma + \delta = \pi.
\end{align*}
By the angle formulae, $\beta\gamma\delta$ cannnot be a vertex. Meanwhile, by $\gamma^2\delta^2$, we have $f \ge 8$ and by the angle formulae, the quadrilateral is convex. 

In $\beta^b\gamma^c\delta^d$, by Lemma \ref{ForbVerLem}, $b,c,d$ in $\beta^b\gamma^c\delta^d$ are either all even or all odd. By $\gamma^2\delta^2$, $c,d$ cannot be both $\ge2$, so one of them must be $1$. Then $b,c,d$ must be all odd. If $b\ge3$, then we get $\pi = \gamma + \delta \le c\gamma + d\delta \le 2\pi - 3\beta = -\pi + \frac{12\pi}{f}$ which implies $f\le 6$, contradicting $f\ge8$. As $\beta\gamma\delta$ is not a vertex, $\beta^b\gamma^c\delta^d = \beta\gamma^c\delta (c\ge3), \beta\gamma\delta^d(d\ge3)$. When $ \beta\gamma^c\delta (c\ge3)$ is a vertex, we have $2\pi - \frac{4\pi}{f} + 2\gamma = \beta + 3\gamma + \delta \le 2\pi$ which implies $\gamma \le \frac{2\pi}{f}$. However, by convexity, Lemma \ref{AngLBLem} implies that $\gamma>\frac{2\pi}{f}$, a contradiction. So $\beta\gamma^c\delta (c\ge3)$ cannot be a vertex. Apply the same argument to $\beta\gamma\delta^d(d\ge3)$, we get the same contradiction. So $\beta^b\gamma^c\delta^d$ cannot be a vertex. Moreover, by Lemma \ref{ForbVerLem} and the results above, $\gamma\cdots = \gamma^2\delta^2, \gamma^c$.

If $\delta < \frac{\pi}{2}$, combining with $\gamma + \delta = \pi$, we have $\gamma > \frac{\pi}{2}$. Then $\gamma^c$ cannot be a vertex. So $\gamma\cdots = \gamma^2\delta^2$. However, in such $\{ \alpha\beta^2, \gamma^2\delta^2, \alpha^a\delta^d, ... \}$ we have $\#\gamma < \#\delta$, a contradicton. So we must have $\delta \ge  \frac{\pi}{2}$ and $\alpha^a\delta^d (d\ge4)$, $\beta^b\delta^d(d\ge4)$, $\delta^d (d>4)$ cannot be vertices, and $\delta^d=\delta^4$, $\alpha\delta\cdots=\alpha^a\delta^2(a\ge2)$, $\beta^b\delta^d = \beta^b\delta^2$. By $\alpha\beta^2$ and $\alpha^a\delta^2(a\ge2)$, we get $\frac{\pi}{2} \le \delta < \beta = \pi - \frac{4\pi}{f}$ which implies $f > 8$. Then $4\beta = 4\pi - \frac{16\pi}{f} > 2\pi$ which implies $\beta^b$, $\beta^b\delta^2(b\ge4)$ cannot be vertices. Then $\beta^b\delta^d = \beta^2\delta^2$. However, if $\beta^2\delta^2$ is a vertex, we get  
\begin{align*}
\beta^2\delta^2: \quad 
\alpha=\frac{8\pi}{f}, \quad
\beta=\gamma = \pi - \frac{4\pi}{f}, \quad 
\delta=\frac{4\pi}{f}.
\end{align*}
As $\delta \ge \frac{\pi}{2}$, we get $f=8$ which implies $\beta = \gamma = \delta = \frac{\pi}{2}$, contradicting Lemma \ref{AngEqEdgeEqLem}. So $\beta^2\delta^2$ cannot be a vertex and there is no $\beta^b\delta^d$.

So we have $\{ \alpha\beta^2, \gamma^2\delta^2, \alpha^a\delta^2, \delta^4, \gamma^c \, \vert \, \alpha^a \}$, from which $\delta\cdots = \gamma^2\delta^2, \alpha^a\delta^2$ and $\beta\delta\cdots$ is not a vertex. Meanwhile, $\beta^2\cdots = \alpha\beta^2 = \alpha \vert \beta \Vert \beta$. In other words, there is exactly one $\bm{b}$-edge at $\alpha\beta^2$ which implies $\alpha^{\beta}\vert^{\beta}\alpha \cdots$ cannot be a vertex. Since $\beta\delta\cdots$, $\alpha^{\beta} \vert^{\beta} \alpha \cdots$ are not vertice, $\alpha\vert\alpha = \alpha^{\delta}\vert^{\delta}\alpha$ which implies there is no $\alpha\vert\alpha\vert\alpha$ at a vertex. Then $\alpha^a\delta^2 = \alpha^2\delta^2$ and $\alpha^a$ cannot be a vertex. Hence we have 
\begin{align*}
\AVC = \{ \alpha\beta^2, \gamma^2\delta^2, \alpha^2\delta^2, \delta^4, \gamma^c \}.
\end{align*}

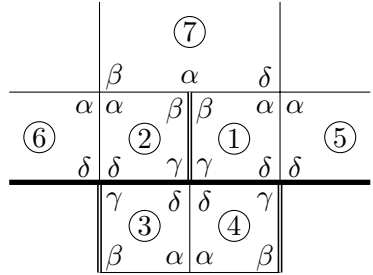
\begin{figure}[htp]
\centering
\begin{tikzpicture}[>=latex,scale=1]

\draw
	(0,1.2) -- (1.2,1.2) -- (1.2,0) 
	(0,1.2) -- (-1.2,1.2) -- (-1.2,0) 
	(0,0) -- (0,-1.2) 
	(-1.2,-1.2) -- (0,-1.2) -- (1.2,-1.2) 
	(-1.2,1.2) -- (-1.2, 2.4) 
	(-1.2,1.2) -- (-2.4, 1.2) 
	(1.2,1.2) -- (1.2, 2.4) 
	(1.2,1.2) -- (2.4, 1.2); 

\draw[double, line width=0.6]
	(0,0) -- (0,1.2) 
	(-1.2,0) -- (-1.2,-1.2) 
	(1.2,0) -- (1.2,-1.2); 

\draw[line width=2]
	(-1.2,0) -- (0,0) -- (1.2,0) 
	(-1.2, 0) -- (-2.4, 0)
	(1.2, 0) -- (2.4, 0);

\node at (0.2,0.2) {\small $\gamma$}; 
\node at (1,0.2) {\small $\delta$}; 
\node at (1,1) {\small $\alpha$}; 
\node at (0.2,0.95) {\small $\beta$}; 

\node at (-0.2,0.2) {\small $\gamma$}; 
\node at (-0.2,0.95) {\small $\beta$}; 
\node at (-1,1) {\small $\alpha$}; 
\node at (-1,0.2) {\small $\delta$};

\node at (-0.2,-0.25) {\small $\delta$}; 
\node at (-0.2,-1) {\small $\alpha$}; 
\node at (-1,-1) {\small $\beta$}; 
\node at (-1,-0.25) {\small $\gamma$}; 

\node at (0.2,-0.25) {\small $\delta$}; 
\node at (0.2,-1) {\small $\alpha$}; 
\node at (1,-1) {\small $\beta$}; 
\node at (1,-0.25) {\small $\gamma$};

\node at (1.4,0.2) {\small $\delta$}; 
\node at (1.4,1) {\small $\alpha$}; 

\node at (-1.4,0.2) {\small $\delta$}; 
\node at (-1.4,1) {\small $\alpha$}; 

\node at (0, 1.4) {\small $\alpha$}; 
\node at (1,1.4) {\small $\delta$}; 
\node at (-1,1.4) {\small $\beta$}; 

\node[inner sep=1,draw,shape=circle] at (0.59,0.58) {\small $1$};
\node[inner sep=1,draw,shape=circle] at (-0.6,0.58) {\small $2$};
\node[inner sep=1,draw,shape=circle] at (-0.6,-0.58) {\small $3$};
\node[inner sep=1,draw,shape=circle] at (0.59,-0.58) {\small $4$};

\node[inner sep=1,draw,shape=circle] at (2,0.6) {\small $5$};
\node[inner sep=1,draw,shape=circle] at (-2,0.6) {\small $6$};
\node[inner sep=1,draw,shape=circle] at (0,2) {\small $7$};

\end{tikzpicture}
\caption{The AAD of $\gamma^2\delta^2$}
\label{AlBe2-Ga2De2-AAD}
\end{figure}

By the unique AAD of $\gamma^2\delta^2$ we determine $T_1, T_2, T_3, T_4$ in Figure \ref{AlBe2-Ga2De2-AAD}. As $\gamma\delta\cdots = \gamma^2\delta^2$, we get $\delta_1\cdots = \delta_2\cdots = \gamma^2\delta^2$, so we further determine $T_5, T_6$. By $\beta^2\cdots = \alpha\beta^2$, we get $\beta_1\beta_2\cdots = \alpha\beta^2$ so we have $\alpha_7$ in $T_7$. By the symmetry about the common edge between $T_1$ and $T_2$, $T_7$ is determined but we get $\beta_7\cdots = \alpha^2\beta\cdots$, which is not a vertex, a contradiction. So $ \gamma^2\delta^2$ cannot be a vertex, contradicting the hypothesis of $\gamma^2\delta^2$ being a vertex. Hence there is no tilings.
\end{proof}

\begin{proposition}\label{Propa3ab2} When $\alpha^3, \alpha\beta^2$ are vertices and $\beta\gamma\delta$ is not a vertex, the $\AVC$ is given in \eqref{a3ab2avcf24}.
\end{proposition}

\begin{proof} By Lemma \ref{albe2Lem}, $\alpha\delta^2$, $\alpha^a\beta^b\gamma^c$, $\alpha^a\beta^b\delta^d$, $\alpha^a\gamma^c\delta^d$ cannot be vertices and some degree $\ge4$  $\alpha^a\delta^d$ must be a vertex. So we have 
\begin{align*}
\AVC_3 = \{ \alpha^3, \alpha\beta^2 \}.
\end{align*}
And the angle formulae are
\begin{align*}
\alpha=\frac{2\pi}{3}, \qquad 
\beta=\frac{2\pi}{3}, \qquad 
\gamma+\delta=\frac{2\pi}{3} + \frac{4\pi}{f}.
\end{align*}
As $\gamma$, $\delta$ do not appear at degree $3$ vertices, and by Lemma \ref{ForbVerLem} there is no degree $5$ $\gamma^c\delta^d$ where $c+d=5$. Then by Lemma \ref{TwoHDAng}, one of $\gamma^4, \gamma^2\delta^2, \delta^4$ must be a vertex. Meanwhile, by Proposition \ref{albe2Prop}, $\gamma^2\delta^2$ cannot be a vertex. So one of $\gamma^4, \delta^4$ must be a vertex and hence and $f\ge8$. We divide the discussion into these two cases. 

\begin{case*}[$\alpha^3,\alpha\beta^2,\gamma^4$] The angle formulae are
\begin{align*}
\alpha=\frac{2\pi}{3}, \qquad 
\beta=\frac{2\pi}{3}, \qquad 
\gamma=\frac{\pi}{2}, \qquad 
\delta=\frac{\pi}{6}+\frac{4\pi}{f},
\end{align*}
with the angle lower bounds $\alpha=\frac{2\pi}{3}, \beta=\frac{2\pi}{3}, \gamma=\frac{\pi}{2},\delta > \frac{\pi}{6}$. By the angle formulae and the angle lower bounds, the number of each angle at a vertex is bounded. Then apply Lemma \ref{ForbVerLem}, we get the other vertices with corresponding $f$ values, $f=12: \delta^4, \gamma^2\delta^2$, $f=24: \alpha^2\delta^2, \beta^2\delta^2, \alpha\delta^4, \delta^6$, $f=36: \beta\gamma\delta^3$, $f=48: \gamma^2\delta^4, \delta^8$, $f=72: \alpha\delta^6$, $f=120: \delta^{10}$. Lemma \ref{albe2Lem} implies that some degree $\ge4$ $\alpha^a\delta^d$ must be a vertex which is absent in $f=12, 36, 48, 120$, so there are no tilings for these $f$ values. Meanwhile, in $f=72: \{ \alpha^3,  \alpha\beta^2,  \gamma^4,  \alpha\delta^6 \}$, by no $\beta \vert \beta$, $\alpha^3$ has unique AAD $\vert^{\delta} \alpha^{\beta} \vert^{\delta} \alpha ^{\beta} \vert^{\delta} \alpha ^{\beta}\vert$ which implies $\beta\delta\cdots$ is a vertex. However, there is no such vertex in the $\AVC$ with $f=72$, so $\alpha^3$ cannot be a vertex and hence $\alpha^2\cdots$ is not a vertex for $f=72$. But the AAD of $\alpha\delta^6$ implies $\delta^{\alpha}\vert^{\alpha}\delta$ which requires $\alpha^2\cdots$ being a vertex, a contradiction. So $\alpha\delta^6$ cannot be a vertex. So we have
\begin{align}
\label{a3ab2avcf24}f&=24, \quad \{ \alpha^3,  \alpha\beta^2,  \gamma^4,   \alpha^2\delta^2, \beta^2\delta^2,   \alpha\delta^4, \delta^6 \}.
\end{align}

The construction of tilings will be explained in Section \ref{SecConst}.

\end{case*}

\begin{case*}[$\alpha^3,\alpha\beta^2,\delta^4$] We may assume that $\gamma^4$ is not a vertex. The angle formulae are
\begin{align*}
\alpha=\frac{2\pi}{3}, \qquad 
\beta=\frac{2\pi}{3}, \qquad 
\gamma=\frac{\pi}{6}+\frac{4\pi}{f}, \qquad 
\delta=\frac{\pi}{2}.
\end{align*}
with the angle lower bounds $\alpha=\frac{2\pi}{3}, \beta=\frac{2\pi}{3}, \gamma>\frac{\pi}{6}, \delta=\frac{\pi}{2}$. By the angle formulae, angle lower bounds and Lemma \ref{ForbVerLem}, we get the other vertices with corresponding $f$ value, $f=12:  \gamma^2\delta^2$, $f=24: \gamma^6, \beta^2\gamma^2$, $f=36: \beta\gamma^3\delta$, $ f=48:  \gamma^4\delta^2, \gamma^8$, $f=120: \gamma^{10}$. By Lemma \ref{albe2Lem}, $\alpha\delta\cdots$ must be a vertex which is absent in all these $f$ values. So $\delta^4$ cannot be a vertex.
\end{case*}

Hence we complete the proof.
\end{proof}

\begin{proposition} When $\alpha^3$ is the only degree $3$ vertex, the $\AVC$ is given in \eqref{a3avcf24}.
\end{proposition}

\begin{proof} By $\alpha^3$, the vertices $\alpha^4,\alpha^5,\alpha^3\beta^2,\alpha^3\delta^2$ cannot appear. Then the corresponding degree $4,5$ vertices are
\begin{align}
\label{a3avc4} \AVC_4 &= \{ \beta^4,  \gamma^4,  \delta^4,  \alpha^2\beta^2, \alpha^2\delta^2,  \beta^2\gamma^2,  \beta^2\delta^2,  \gamma^2\delta^2 \}, \\
\label{a3avc5} \AVC_5 &= \{ \alpha\beta^4,  \alpha\delta^4,  \beta^3\gamma\delta,  \beta\gamma^3\delta,  \beta\gamma\delta^3,  \alpha\beta^2\gamma^2,  \alpha\beta^2\delta^2, \alpha\gamma^2\delta^2 \}.
\end{align}
We will frequently refer to the above results in this proof. Meanwhile, by Lemma \ref{a3Lem}, we have $f\ge24$. Since there is no degree $5$ $\beta^b\gamma^c$ in \eqref{a3avc4}, apply Lemma \ref{TwoHDAng} on $\beta, \gamma$, one of $\beta^4, \gamma^4, \beta^2\gamma^2$ must appear as a vertex. 

\begin{case*}[$\alpha^3,\beta^4$] When $\alpha^3$, $\beta^4$ are vertices, we get $\alpha = \frac{2\pi}{3}$, $\beta = \frac{\pi}{2}$, so $\alpha^2\beta^2$ and $\alpha\beta^4$ cannot be a vertex. By \eqref{a3avc4} and Lemma \ref{a3Lem}, $\AVC_4 \neq \{ \beta^4 \}$, one of $\alpha^2\delta^2$, $\beta^2\gamma^2$, $\beta^2\delta^2$, $\gamma^4$, $\delta^4$, $\gamma^2\delta^2$ must appear.  

\begin{subcase*}[$\alpha^3,\beta^4,\alpha^2\delta^2$] The angle formulae are
\begin{align*}
\alpha=\frac{2\pi}{3}, \qquad 
\beta=\frac{\pi}{2}, \qquad 
\gamma=\frac{\pi}{2}+\frac{4\pi}{f}, \qquad
\delta=\frac{\pi}{3}.
\end{align*}
The angle lower bounds are $\alpha=\frac{2\pi}{3}$, $\beta=\frac{\pi}{2}$, $\gamma>\frac{\pi}{2}$, $\delta=\frac{\pi}{3}$. Then by Lemma \ref{ForbVerLem}, the other vertices are $\alpha\delta^4$, $\delta^6$, $ \gamma^2\delta^2$. We get $\{\alpha^3, \beta^4, \alpha^2\delta^2, \gamma^2\delta^2 \,  \vert \, \alpha\delta^4,  \delta^6 \}$ where $\beta\delta\cdots$ is not a vertex. At $\gamma^2\delta^2$, the AAD is unique as indicated in $T_1, T_2, T_3, T_4$ in Figure \ref{Al3Ga2De2AAD}. From the $\AVC$, we have $\alpha^2\cdots = \alpha^3$. So $\alpha_3\alpha_4\cdots = \alpha^3 = \alpha_3\alpha_4\alpha_5$. Then we determine $T_5$, up to symmetry we have $\delta_5 \cdots = \beta\delta\cdots$ which is not a vertex, a contradiction. So $\gamma^2\delta^2$ cannot be a vertex. Then $\{\alpha^3, \beta^4, \alpha^2\delta^2 \,  \vert \, \alpha\delta^4,  \delta^6 \}$ has no $\gamma\cdots$, a contradiction. So $\alpha^2\delta^2$ cannot be a vertex. 

\begin{figure}[htp]
\centering
\begin{tikzpicture}[>=latex,scale=1]

\draw
	(1.2,0) -- (1.2, 1.2)-- (0, 1.2) -- (-1.2, 1.2) -- (-1.2, 0)
	(-1.2,-1.2) -- (0, -1.2) -- (1.2,-1.2)
	(0,0) -- (0,-1.2);

\draw[line width=2]
	(-1.2,0) -- (0,0) -- (1.2,0);

\draw[double, line width=0.6]
	(0,0) -- (0,1.2)
	(-1.2,0) -- (-1.2, -1.2)
	(1.2,0) -- (1.2, -1.2);

\node at (0.2,0.2) {\small $\gamma$}; 
\node at (1,0.2) {\small $\delta$}; 
\node at (1,1) {\small $\alpha$}; 
\node at (0.2,0.92) {\small $\beta$}; 

\node at (-0.2,0.2) {\small $\gamma$}; 
\node at (-0.2,0.92) {\small $\beta$}; 
\node at (-1,1) {\small $\alpha$}; 
\node at (-1,0.2) {\small $\delta$}; 

\node at (-0.2,-0.25) {\small $\delta$}; 
\node at (-0.2,-1) {\small $\alpha$}; 
\node at (-1,-1) {\small $\beta$}; 
\node at (-1,-0.25) {\small $\gamma$}; 

\node at (0.2,-0.25) {\small $\delta$}; 
\node at (0.2,-1) {\small $\alpha$}; 
\node at (1,-1) {\small $\beta$}; 
\node at (1,-0.25) {\small $\gamma$};

\node at (0,-1.4)  {\small $\alpha$};  
\node at (-1,-1.45)  {\small $\beta$}; 
\node at (1,-1.4)  {\small $\delta$};

\node[inner sep=1,draw,shape=circle] at (0.59,0.58) {\small $1$};
\node[inner sep=1,draw,shape=circle] at (-0.6,0.58) {\small $2$};
\node[inner sep=1,draw,shape=circle] at (-0.6,-0.58) {\small $3$};
\node[inner sep=1,draw,shape=circle] at (0.59,-0.58) {\small $4$};

\node[inner sep=1,draw,shape=circle] at (0,-2) {\small $5$};

\end{tikzpicture}
\caption{The AAD of $\gamma^2\delta^2$}
\label{Al3Ga2De2AAD}
\end{figure}
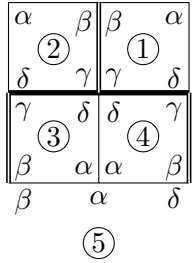
\end{subcase*}

\begin{subcase*}[$\alpha^3,\beta^4,\beta^2\gamma^2$] We may further assume that $\alpha^2\delta^2$ is not a vertex. The angle formulae are
\begin{align*}
\alpha=\frac{2\pi}{3}, \qquad 
\beta=\frac{\pi}{2}, \qquad 
\gamma=\frac{\pi}{2}, \qquad 
\delta=\frac{\pi}{3}+\frac{4\pi}{f}.
\end{align*}
When $f=24$, we get $\beta = \gamma = \delta =\frac{\pi}{2}$, contradicting Lemma \ref{AngEqEdgeEqLem}. So we may assume $f>24$. Then the angle lower bounds are $\alpha=\frac{2\pi}{3}$, $\beta=\frac{\pi}{2}$, $\gamma=\frac{\pi}{2}$, $\delta>\frac{\pi}{3}$. Then by Lemma \ref{ForbVerLem}, the other vertices can only be $\gamma^4$. However, there is no $\delta\cdots$, a contradition. So $\beta^2\gamma^2$ cannot be a vertex. 
\end{subcase*}

\begin{subcase*}[$\alpha^3,\beta^4,\beta^2\delta^2$] We may further assume that $\alpha^2\delta^2$, $\beta^2\gamma^2$ are not a vertices. The angle formulae are
\begin{align*}
\alpha=\frac{2\pi}{3}, \qquad 
\beta=\frac{\pi}{2}, \qquad 
\gamma=\frac{\pi}{3}+\frac{4\pi}{f}, \qquad 
\delta=\frac{\pi}{2}.
\end{align*}
When $f=24$, we get $\gamma < \pi$ and $\beta= \delta =\frac{\pi}{2}$, contradicting Lemma \ref{AngEqEdgeEqLem}. So we may assume $f>24$. Then the angle lower bounds are $\alpha=\frac{2\pi}{3}$, $\beta=\frac{\pi}{2}$, $\gamma>\frac{\pi}{3}$, $\delta=\frac{\pi}{2}$. By the angle lower bounds and Lemma \ref{ForbVerLem}, the other vertices can only be $\delta^4$. However there is no $\gamma\cdots$, a contradition. So $\beta^2\delta^2$ cannot be a vertex. 
\end{subcase*}

\begin{subcase*}[$\alpha^3,\beta^4,\gamma^4$] We may further assume that $\alpha^2\delta^2$, $\beta^2\gamma^2$, $\beta^2\delta^2$ are not a vertices. The angle formulae are
\begin{align*}
\alpha=\frac{2\pi}{3}, \qquad 
\beta=\frac{\pi}{2}, \qquad 
\gamma=\frac{\pi}{2}, \qquad 
\delta=\frac{\pi}{3}+\frac{4\pi}{f},
\end{align*}
which is the same angle formulae as in case $\alpha^3$, $\beta^4$, $\beta^2\gamma^2$, so we get the same conclusion. 
\end{subcase*}

\begin{subcase*}[$\alpha^3,\beta^4,\delta^4$] We may further assume that $\alpha^2\delta^2$, $\beta^2\gamma^2$, $\beta^2\delta^2$, $\gamma^4$ are not a vertices. The angle formulae are
\begin{align*}
\alpha=\frac{2\pi}{3}, \qquad 
\beta=\frac{\pi}{2}, \qquad 
\gamma=\frac{\pi}{3}+\frac{4\pi}{f}, \qquad 
\delta=\frac{\pi}{2},
\end{align*}
which is the same angle formulae as in case $\alpha^3$, $\beta^4$, $\beta^2\delta^2$, so we get the same conclusion. 
\end{subcase*}

\begin{subcase*}[$\alpha^3,\beta^4,\gamma^2\delta^2$] We may further assume that $\alpha^2\delta^2$, $\beta^2\gamma^2$, $\beta^2\delta^2$, $\gamma^4$, $\delta^4$ are not a vertices. The angle formulae are
\begin{align*}
f=24, \qquad
\alpha=\frac{2\pi}{3}, \qquad 
\beta=\frac{\pi}{2}, \qquad 
\gamma + \delta=\pi.
\end{align*}

By the angle formulae, the quadrilateral is convex. Then Lemma \ref{LunEstLem} implies
\begin{align*}
\gamma + \beta < \pi + \delta \quad \Rightarrow \quad \gamma - \delta < \frac{\pi}{2},
\end{align*}
Combining with $\gamma + \delta=\pi$ we get
\begin{align*}
\gamma < \frac{3\pi}{4}, \qquad
\delta > \frac{\pi}{4}. 
\end{align*}
By convexity and Lemma \ref{AngLBLem} and $f=24$, $\gamma > \frac{2\pi}{f} = \frac{\pi}{12}$. Then the angle lower bounds are $\alpha=\frac{2\pi}{3}$, $\beta=\frac{\pi}{2}$, $\gamma >  \frac{\pi}{12}$, $\delta > \frac{\pi}{4}$. By Lemma \ref{ForbVerLem}, the other vertices are 
\begin{align*}
\hat{\gamma}\cdots &= \alpha\delta^4,  \beta^2\delta^4, \delta^6, \delta^8, \\
\gamma\cdots &=\gamma^c,  \beta^2\gamma^c,  \alpha\beta^2\gamma^c,  \beta^b\gamma^c\delta^d,
\end{align*}
where $\beta^b\gamma^c\delta^d$ has $b,c,d$ are all even or all odd. Note that $\beta^2\delta^4$ or $\delta^8$ implies $\delta =  \frac{\pi}{4}$, so they cannot be vertices. Meanwhile, in $\{ \alpha^3, \beta^4, \gamma^2\delta^2 \}$, we have $\# \gamma = \# \delta$. When there is a vertex with more $\gamma$ than $\delta$, then there must be another vertex with more $\delta$ than $\gamma$, and vice versa. 

When $\alpha\delta^4$ or $\delta^6$ is a vertex, the angle formulae become
\begin{align*}
f=24, \qquad 
\alpha=\frac{2\pi}{3}, \qquad 
\beta=\frac{\pi}{2}, \qquad 
\gamma=\frac{2\pi}{3}, \qquad 
\delta=\frac{\pi}{3}.
\end{align*}
By Lemma \ref{ForbVerLem}, the other vertices can only be $\delta^6$ or $\alpha\delta^4$. However, in $\{ \alpha^3, \beta^4, \gamma^2\delta^2, \alpha\delta^4, \delta^6 \}$, we get $\# \gamma < \# \delta$ if any of $\alpha\delta^4, \delta^6$ is a vertex, a contradiction. So $\alpha\delta^4$ and $\delta^6$ cannot be vertices. 

Then $\beta^b\gamma^c\delta^d$ is the only vertex which may have more $\delta$ than $\gamma$. When $\beta^b\gamma^c\delta^d$ is a vertex, by $2\beta + \gamma + \delta = 2\pi$ and $b,c,d$ being all even or all odd, we have $\beta^b\gamma^c\delta^d=\beta\gamma^c\delta^d$. Assume that $\beta\gamma^c\delta^d$ has more $\delta$ than $\gamma$. By $\gamma + \delta = \pi$, $c,d$ cannot be both $\ge2$, then $c=1$ and such $\beta\gamma^c\delta^d = \beta\gamma\delta^d$ where $d=3,5,7$. By $\beta=\frac{\pi}{2}$, $\gamma + \delta = \pi$ and the vertex angle sum, we have $\frac{\pi}{2} = (d-1)\delta  > (d-1)\frac{\pi}{4}$, which implies $d < 3$. So $\beta\gamma\delta^d$ cannot be a vertex and hence no $\beta^b\gamma^c\delta^d$ with more $\delta$ than $\gamma$. Since there is no vertex with more $\delta$ than $\gamma$, every vertex has equal number of $\gamma$ and $\delta$. So $\gamma\cdots = \gamma^2\delta^2$ and we have $
f=24: \{ \alpha^3,  \beta^4,  \gamma^2\delta^2 \}$ which is a subset of the $\AVC$ in subcase ($\alpha^3,\beta^4,\alpha^2\delta^2$). So we get the same contradiction and $\gamma^2\delta^2$ cannot be a vertex. 
\end{subcase*}
Hence this case has no tilings.
\end{case*}

\begin{case*}[$\alpha^3,\gamma^4$] We may assume that $\beta^4$ is not a vertex. Moreover, by the symmetry of $\beta,\delta$ in the default quadrialeteral in Figure \ref{DefaultQuad}, we may also assume $\delta^4$ is not a vertex. When $\alpha^3,\gamma^4$ are vertices, we get $\alpha = \frac{2\pi}{3}$ and $\gamma = \frac{\pi}{2}$.  By Lemma \ref{a3Lem}, $\AVC_4 \neq \{ \gamma^4 \}$, combining with the above results, one of the other degree $4$ vertices must appear. Moreover, when $\beta^4$ is not a vertex, by the symmetry between $\beta, \delta$ in Figure \ref{DefaultQuad}, it suffices to discuss one of $\alpha^2\beta^2, \beta^2\gamma^2$, $\beta^2\delta^2$ being a vertex. 

\begin{subcase*}[$\alpha^3,\gamma^4,\alpha^2\beta^2$] The angle formulae are 
\begin{align*}
\alpha=\frac{2\pi}{3}, \qquad 
\beta=\frac{\pi}{3}, \qquad 
\gamma=\frac{\pi}{2}, \qquad 
\delta=\frac{\pi}{2}+\frac{4\pi}{f},
\end{align*}
The angle lower bounds are $\alpha=\frac{2\pi}{3}$, $\beta=\frac{\pi}{3}$, $\gamma=\frac{\pi}{2}$, $\delta > \frac{\pi}{2}$. Then by Lemma \ref{ForbVerLem}, the other vertices are $\beta^2\delta^2$, $\alpha\beta^4$, $\beta^6$. However, in $\{\alpha^3, \gamma^4, \alpha^2\beta^2, \beta^2\delta^2, \alpha\beta^4, \beta^6 \}$ we get $\# \beta > \# \delta$, a contradiction.  So $\alpha^2\beta^2$ cannot be a vertex.
\end{subcase*}

\begin{subcase*}[$\alpha^3,\gamma^4,\beta^2\gamma^2$] We may further assume that $\alpha^2\beta^2$, $\alpha^2\delta^2$ are not a vertex. The angle formulae are 
\begin{align*}
\alpha=\frac{2\pi}{3}, \qquad 
\beta=\frac{\pi}{2}, \qquad 
\gamma=\frac{\pi}{2}, \qquad 
\delta=\frac{\pi}{3}+\frac{4\pi}{f}.
\end{align*}
When $f=24$, we get $\beta = \gamma = \delta = \frac{\pi}{2}$, contradiction. So we have $f>24$. Then angle lower bounds are $\alpha=\frac{2\pi}{3}$, $\beta=\frac{\pi}{2}$, $\gamma=\frac{\pi}{2}$, $\delta>\frac{\pi}{3}$. By Lemma \ref{ForbVerLem}, the other vertices are $\beta^2\delta^2, \delta^4, \gamma^2\delta^2$, each of which implies $f=24$, a contradiction. Hence there is no vertices other than $\alpha^3, \gamma^4, \beta^2\gamma^2$ where there is no $\delta\cdots$, a contradiction. So $\beta^2\gamma^2$ cannot be a vertex. 
\end{subcase*}

\begin{subcase*}[$\alpha^3,\gamma^4,\beta^2\delta^2$] We may further assume that $\alpha^2\beta^2, \alpha^2\delta^2, \beta^2\gamma^2,\gamma^2\delta^2$ are not a vertex. The angle formulae are 
\begin{align*}
f=24, \qquad
\alpha=\frac{2\pi}{3}, \qquad 
\beta + \delta =\pi, \qquad 
\gamma=\frac{\pi}{2},
\end{align*}
then the quadrilateral is convex. By Lemma \ref{LunEstLem}, $\gamma + \beta < \pi + \delta$ implies $\beta - \delta < \pi - \gamma$. Then by $\beta + \delta =\pi$ and $\gamma=\frac{\pi}{2}$, we get $2\delta > \gamma$ which gives $\delta > \frac{\pi}{4}, \,\, \beta < \frac{3\pi}{4}$. By Lemma \ref{AngLBLem} and $f=24$, we get $\beta > \frac{2\pi}{f} = \frac{\pi}{12}$. So the angle lower bounds are $\alpha=\frac{2\pi}{3}$, $\beta >  \frac{\pi}{12}$, $\gamma=\frac{\pi}{2}$, $\delta > \frac{\pi}{4}$. Then by Lemma \ref{ForbVerLem} and the assumption, the other vertices are 
\begin{align*}
\hat{\beta}\cdots &= \alpha\delta^4,   \delta^6,  \delta^8,   \gamma^2\delta^4, \\
\beta\cdots &= \alpha\beta^b,   \alpha^2\beta^b,    \alpha\beta^b\gamma^2,   \beta^b,   \beta^b\gamma^2,  \beta^b\gamma^c\delta^d,
\end{align*}
where $\beta^b\gamma^c\delta^d$ has $b,c,d$ are all even or all odd. Note that $\delta^8$ or $\gamma^2\delta^4$ implies $\delta = \frac{\pi}{4}$, so they cannot be vertices. Meanwhile, in $\alpha^3, \gamma^4, \beta^2\delta^2$, we have $\# \beta = \# \delta$. When there is a vertex with more $\beta$ than $\delta$, then there must be another vertex with more $\delta$ than $\beta$, and vice versa. 

When $\alpha\delta^4$ or $\delta^6$ is a vertex, the angle formulae are
\begin{align*}
f=24, \qquad
\alpha=\frac{2\pi}{3}, \qquad 
\beta=\frac{2\pi}{3}, \qquad 
\gamma=\frac{\pi}{2}, \qquad 
\delta=\frac{\pi}{3}.
\end{align*}
Then by Lemma \ref{ForbVerLem} and the assumption, the other vertices can only be $\delta^6$ or $\alpha\delta^4$. However, in $\{ \alpha^3, \gamma^4, \beta^2\delta^2, \alpha\delta^4, \delta^6 \}$ we get $\# \beta < \# \delta$ if any of $\alpha\delta^4, \delta^6$ is a vertex, a contradiction. So $\alpha\delta^4$ and $\delta^6$ cannot be vertices.

Then $\beta^b\gamma^c\delta^d$ is the only vertex which may have more $\delta$ than $\beta$. When $\beta^b\gamma^c\delta^d$ is a vertex, by $\beta + 2\gamma + \delta = 2\pi$ and $b,c,d$ being all even or all odd, we have $\beta^b\gamma^c\delta^d=\beta^b\gamma\delta^d$. Assume that $\beta^b\gamma\delta^d$ has more $\delta$ than $\beta$. By $\beta + \delta = \pi$, we know that $b,d$ cannot be both $\ge2$, then $b=1$ and such $\beta^b\gamma\delta^d = \beta\gamma\delta^d$ where $d=3,5,7$. By $\gamma=\frac{\pi}{2}$, $\beta + \delta = \pi$ and the vertex angle sum, we have $\frac{\pi}{2} = (d-1)\delta >  (d-1)\frac{\pi}{4}$, which implies $d < 3$. So $\beta\gamma\delta^d$ cannot be a vertex and hence no $\beta^b\gamma^c\delta^d$ with more $\delta$ than $\beta$. Since there is no vertex with more $\delta$ than $\beta$, every vertex has equal number of $\beta$ and $\delta$ and such that $\alpha\beta^b, \alpha^2\beta^b, \alpha\beta^b\gamma^2, \beta^b, \beta^b\gamma^2$ cannot be vertices. So $\beta\cdots = \beta^2\delta^2$ and we have 
\begin{align}\label{a3avcf24}
f=24, \quad \AVC = \{ \alpha^3,  \gamma^4, \beta^2\delta^2 \}.
\end{align}
The tiling is uniquely constructed in Figure \ref{Tf24CubeSD}. 
\end{subcase*}

\end{case*}

\begin{case*}[$\alpha^3,\beta^2\gamma^2$] We may assume $\beta^4,\gamma^4,\delta^4$ are not vertices. Moreover, when $\alpha^2\delta^2$ is a vertex, there is no solutions. So $\alpha^2\delta^2$ cannot be a vertex. By Lemma \ref{a3Lem}, either we have one of $\alpha^2\beta^2,\beta^2\delta^2,\gamma^2\delta^2$ as a vertex or there is a degree $5$ vertex without $\alpha,\beta$. For the latter, however, there is no such degree $5$ vertex in \eqref{a3avc5}. So one of $\alpha^2\beta^2, \beta^2\delta^2, \gamma^2\delta^2$ must appear as a vertex. 

\begin{subcase*}[$\alpha^3,\beta^2\gamma^2,\alpha^2\beta^2$] The angle formulae are
\begin{align*}
\alpha=\frac{2\pi}{3}, \qquad 
\beta=\frac{\pi}{3}, \qquad 
\gamma=\frac{2\pi}{3}, \qquad 
\delta=\frac{\pi}{3}+\frac{4\pi}{f}.
\end{align*}
The angle lower bounds are $\alpha=\frac{2\pi}{3}$, $\beta=\frac{\pi}{3}$, $\gamma=\frac{2\pi}{3}$, $\delta>\frac{\pi}{3}$. By $f \ge 24$, $\beta^2\delta^2$ which implies $f=12$ is not a vertex. Then by Lemma \ref{ForbVerLem} and the assumption, the other vertices are $\alpha\beta^4,\beta^6$. However, there is no $\delta\cdots$, a contradiction. So $\alpha^2\beta^2$ cannot be a vertex. 
\end{subcase*}

\begin{subcase*}[$\alpha^3,\beta^2\gamma^2,\beta^2\delta^2$] We may further assume that $\alpha^2\beta^2$ is not a vertex. Then we either have another degree $4$ vertex without $\beta^2\cdots$, that is, $\gamma^2\delta^2$, or by Lemma \ref{a3Lem} there is a degree $5$ vertex without $\alpha$, $\beta$. When $\gamma^2\delta^2$ is also a vertex, we get $\beta = \gamma = \delta = \frac{\pi}{2}$, contradicting Lemma \ref{AngEqEdgeEqLem}. So there must be a degree $5$ vertex without $\alpha$, $\beta$, which is however impossible by \eqref{a3avc5}. So $\beta^2\delta^2$ cannot be a vertex.
\end{subcase*}

\begin{subcase*}[$\alpha^3,\beta^2\gamma^2,\gamma^2\delta^2$] We may further assume that $\alpha^2\beta^2$, $\beta^2\delta^2$ are not vertices. Then we either have another degree $4$ vertex without $\gamma^2\cdots$ or by Lemma \ref{a3Lem} there is a degree $5$ vertex without $\alpha,\gamma$. By \eqref{a3avc4} and the assumption, there is no such degree $4$ vertex, and by \eqref{a3avc5} there is no such degree $5$ vertex either. So $\gamma^2\delta^2$ cannot be a vertex.
\end{subcase*}
\end{case*}
Therefore we complete the proof.
\end{proof}

\begin{proposition}\label{pqEMTProp} When $\alpha\beta^2$ is the only degree $3$ vertex and $\gamma^2\delta^2$ is not a vertex, then the $\AVC$ is given in \eqref{ab2avcf16} and \eqref{ab2avcf}.
\end{proposition}

\begin{proof} By Lemma \ref{albe2Lem}, $\alpha\beta\cdots = \alpha\beta^2$ and some degree $\ge4$ $\alpha^a\delta^d$ is a vertex and $\alpha\delta^2$, $\alpha^a\gamma^c\delta^d$, $\alpha^a\beta^b\gamma^c$, $\alpha^a\beta^b\delta^d$ cannot be a vertex. The corresponding degree $4,5$ vertices are
\begin{align}
\label{ab2avc4} \AVC_4 &= \{ \alpha^4,  \beta^4,  \gamma^4,  \delta^4, \alpha^2\delta^2,  \beta^2\gamma^2, \beta^2\delta^2 \}, \\
\label{ab2avc5} \AVC_5 &= \{  \alpha^5,  \alpha\delta^4,  \alpha^3\delta^2,  \beta^3\gamma\delta,  \beta\gamma^3\delta,  \beta\gamma\delta^3 \}.
\end{align}
Moreover, as $\gamma$, $\delta$ do not appear in degree $3$ vertices and every degree $5$ vertex has $\alpha$ or $\beta$, Lemma \ref{TwoHDAng} implies that one of $\gamma^4,\delta^4$ from \eqref{ab2avc4} must appear.

\begin{case*}[$\alpha\beta^2, \gamma^4$] By Lemma \ref{ForbVerLem}, degree $5$ $\delta^3\cdots = \alpha\delta^4, \beta\gamma\delta^3$, degree $6$ $\delta^5\cdots = \delta^6$, and $\delta^7$ cannot be a vertex. Apply Lemma \ref{ab2Lem} on $\delta$, one of $\delta^4, \alpha^2\delta^2, \beta^2\delta^2$, $ \alpha\delta^4, \beta\gamma\delta^3, \delta^6$ must be a vertex. 

\begin{subcase*}[$\alpha\beta^2,  \gamma^4, \delta^4$]The angle formulae are
\begin{align*}
\alpha=\frac{8\pi}{f}, \qquad 
\beta=\pi - \frac{4\pi}{f}, \qquad 
\gamma=\frac{\pi}{2}, \qquad 
\delta=\frac{\pi}{2}.
\end{align*}

From the angle formulae, by $f\ge16$, the angle inequalities are $\beta > \gamma = \delta \ge \alpha$ and the angle lower bounds are $\alpha = \frac{8\pi}{f}$, $\beta \ge \frac{3\pi}{4}$, $\gamma=\frac{\pi}{2}$, $\delta=\frac{\pi}{2}$. By the angle formulae, the angle lower bounds and Lemma \ref{ForbVerLem} and the assumption,
\begin{align*}
\hat{\alpha}\cdots &= \gamma^4,   \delta^4. \\
\alpha\cdots &= \alpha\beta^2,  \alpha^a,  \alpha^a\delta^2.
\end{align*}
So we have $\{ \alpha\beta^2, \gamma^4, \delta^4,  \alpha^a, \alpha^a\delta^2  \}$. Since $\beta\delta\cdots$ and $\alpha^{\beta}\vert^{\beta}\alpha\cdots$ are not vertices, there is no $\alpha\vert\alpha\vert\alpha$ at a vertex and hence $\alpha^a\delta^2= \alpha^2\delta^2$ and $\alpha^a$ cannot be a vertex. Then we have
\begin{align}\label{ab2avcf16}
f=16: \quad \AVC = \{ \alpha\beta^2,   \alpha^2\delta^2,  \gamma^4,  \delta^4  \}.
\end{align}
The tiling for the above is uniquely constructed in Figure \ref{pqEMT1624}.

Next we assume that $\delta^4$ is not a vertex, then in the unique AAD of $\alpha^2\delta^2$, we have $\alpha^{\delta} \vert ^{\delta} \alpha$ which implies $\cdots \, \bvert \delta \vert \delta \, \bvert \cdots$ must be a vertex. However in the $\AVC$, $\cdots \, \bvert \delta \vert \delta \, \bvert \cdots$ can only be $\delta^4$, contradicting $\delta^4$ not being a vertex. Hence Figure \ref{pqEMT1624} is the only tiling for $f=16$. 
\end{subcase*}

\begin{subcase*}[$\alpha\beta^2,\gamma^4,\alpha^2\delta^2$]We may further assume no $\delta^4$. The angle formulae are
\begin{align*}
\alpha=\pi - \frac{8\pi}{f}, \qquad 
\beta=\frac{\pi}{2}+\frac{4\pi}{f}, \qquad 
\gamma=\frac{\pi}{2}, \qquad 
\delta=\frac{8\pi}{f}.
\end{align*}

From the angle formulae, when $f=16$, we get $\alpha = \gamma = \delta = \frac{\pi}{2}$ and $\beta = \frac{3\pi}{4}$ and hence a subset of \eqref{ab2avcf16}. It suffices to discuss $f>16$. 

For $f>16$, the angle inequalities are $2\gamma > \alpha, \beta > \gamma > \delta$. The angle lower bounds are $\alpha \ge \frac{\pi}{2}$, $\beta > \frac{\pi}{2}$, $\gamma = \frac{\pi}{2}$, $\delta > 0$. By the angle formulae, the angle lower bounds and Lemma \ref{ForbVerLem}, 
\begin{align*}
\hat{\delta}\cdots &= \gamma^4,   \alpha\beta^2; \\
\delta\cdots &= \alpha^2\delta^2,   \alpha\delta^d,   \beta^2\delta^d,   \beta\gamma\delta^d,   \gamma^2\delta^d,   \delta^d.
\end{align*}
So the other vertices can only be $\alpha\delta^d, \beta^2\delta^d,  \beta\gamma\delta^d,  \gamma^2\delta^d,  \delta^d$.

\begin{figure}[htp]
\centering
\begin{tikzpicture}[>=latex,scale=1]

\draw[]
	(0,1.2) -- (1.2,1.2) 
	(0,0) -- (0,1.2) 
	(0,1.2) -- (-1.2,1.2) 
	(0,-1.2) -- (1.2,-1.2)-- (1.2,0) 
	(1.2,1.2) -- (1.2,2.4) 
	(1.2, 0) -- (2.4,0) -- (2.4, 1.2); 

\draw[double, line width=0.6]
	(0,0) -- (1.2, 0) 
	(-1.2,0) -- (-1.2,1.2) 
	(0,2.4) -- (1.2,2.4) 
	(1.2,1.2) -- (2.4, 1.2); 

\draw[line width=2]
	(1.2,0) -- (1.2,1.2) 
	(0,0) -- (0,-1.2) 
	(0,0) -- (-1.2,0) 
	(0,1.2) -- (0,2.4); 

\node at (0.2,1) {\small $\alpha$};
\node at (0.2,0.2) {\small $\beta$}; 
\node at (1,1) {\small $\delta$}; 
\node at (1,0.2) {\small $\gamma$};

\node at (0.2,-0.25) {\small $\gamma$}; 
\node at (0.98,-0.3) {\small $\beta$};

\node at (-0.35,-0.25) {\small $\ddots$};

\node at (-0.2,0.25) {\small $\delta$}; 
\node at (-0.2,1) {\small $\alpha$};

\node at (1, 1.42) {\small $\alpha$}; 
\node at (0.2,1.45) {\small $\delta$};

\node at (1.4, 0.2) {\small $\delta$}; 
\node at (1.4,0.95) {\small $\gamma$};

\node[inner sep=1,draw,shape=circle] at (0.6,0.6) {\small $1$};
\node[inner sep=1,draw,shape=circle] at (0.6,-0.6) {\small $2$};
\node[inner sep=1,draw,shape=circle] at (-0.6,0.6) {\small $3$};
\node[inner sep=1,draw,shape=circle] at (0.6,1.8) {\small $4$};
\node[inner sep=1,draw,shape=circle] at (1.8,0.6) {\small $5$};

\begin{scope}[xshift=6.5cm, yshift=0cm]

\draw[]
	(0,0) -- (-1.2,0) -- (-1.2,1.2) 
	(0,0) -- (1.2, 0) -- (1.2,1.2) 
	(0.5,2.4) -- (1.7,2.4) 
	(0,1.2) -- (0.5,2.4) 
	(0,1.2) -- (-0.5,2.4) 
	(-0.5,2.4) -- (-1.7,2.4) 
	(-1.2, 0) -- (-2.4,0) -- (-2.4, 1.2) 
	(1.2, 0) -- (2.4,0) -- (2.4, 1.2); 

\draw[double, line width=0.6]
	(0,0) -- (0,1.2) 
	(1.2,1.2) -- (1.7,2.4) 
	(-1.2,1.2) -- (-1.7,2.4); 

\draw[line width=2]
	(0,1.2) -- (-1.2,1.2) 
	(0,1.2) -- (1.2,1.2) 
	(-1.2,1.2) -- (-2.4, 1.2) 
	(1.2,1.2) -- (2.4, 1.2); 

\node at (-0.2,0.2) {\small $\beta$}; 
\node at (-0.2,0.95) {\small $\gamma$};
\node at (-1,0.95) {\small $\delta$};
\node at (-1,0.2) {\small $\alpha$};

\node at (0.2,0.95) {\small $\gamma$};
\node at (0.2,0.2) {\small $\beta$}; 
\node at (1,0.95) {\small $\delta$}; 
\node at (1,0.2) {\small $\alpha$};

\node at (1.1, 1.42) {\small $\gamma$}; 
\node at (0.3,1.45) {\small $\delta$};

\node at (0.04,2) {\small $\cdots$};

\node at (-1.1, 1.42) {\small $\gamma$}; 
\node at (-0.3,1.45) {\small $\delta$};

\node at (-1.4, 0.2) {\small $\alpha$}; 
\node at (-1.4,0.95) {\small $\delta$};

\node at (-1.6,1.6) {\small $\vdots$};

\node at (1.4, 0.2) {\small $\alpha$}; 
\node at (1.4,0.95) {\small $\delta$}; 

\node at (1.6,1.6) {\small $\vdots$};


\node[inner sep=1,draw,shape=circle] at (0.6,0.6) {\small $2$};
\node[inner sep=1,draw,shape=circle] at (-0.6,0.6) {\small $1$};
\node[inner sep=1,draw,shape=circle] at (0.8,1.8) {\small $3$};
\node[inner sep=1,draw,shape=circle] at (-0.8,1.8) {\small $4$};
\node[inner sep=1,draw,shape=circle] at (-1.8,0.6) {\small $5$};
\node[inner sep=1,draw,shape=circle] at (1.8,0.6) {\small $6$};
\node[inner sep=1,draw,shape=circle] at (-0.6,-0.6) {\small $7$};

\end{scope}

\end{tikzpicture}
\caption{The AAD of $\beta\gamma\delta^d$ and $\gamma^2\delta^d$}
\label{albe2-begade3-ga2ded}
\end{figure}
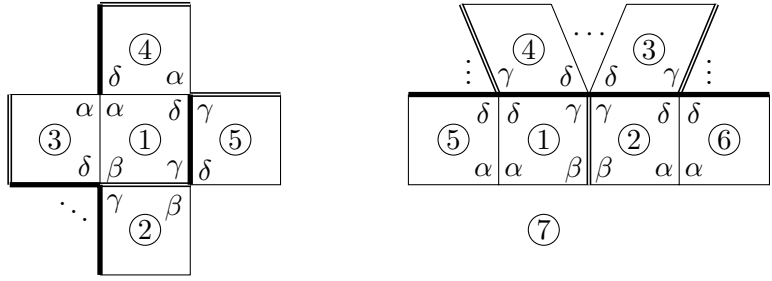

Since $\beta\gamma\delta^d = \vert \beta \Vert \gamma \, \bvert \, \delta \cdots \delta \vert$, by the unique AAD we determine $T_1, T_2, T_3$ in the first picture in Figure \ref{albe2-begade3-ga2ded}. By $\alpha^2\cdots = \alpha^2\delta^2$, we get $\alpha_1\alpha_3\cdots = \alpha^2\delta^2$ so that we further determine $T_4$. By $\beta\gamma\cdots = \beta\gamma\delta^d$, $\beta_2\gamma_1\cdots = \beta\gamma\delta^d$ so that we further determine $T_5$. At $T_4$, we get $\alpha_4\cdots = \alpha\gamma\delta\cdots$, which is not a vertex, a contradiction. So $\beta\gamma\delta^d$ cannot be a vertex. 

Since $\gamma^2\delta^d =$ $\bvert \, \gamma \Vert \gamma \, \bvert \, \delta \cdots \delta$, by the unique AAD we determine $T_1, T_2, T_3, T_4$ in the second picture in Figure \ref{albe2-begade3-ga2ded}. As $\beta\gamma\delta^d$ can no longer be a vertex, we get $\gamma\delta\cdots = \gamma^2\delta^d$ such that we further determine $T_5, T_6$. By $\beta^2\cdots = \alpha\beta^2, \beta^2\delta^d$, we have $\beta_1\beta_2\cdots = \alpha\beta^2, \beta^2\delta^d$. When $\beta_1\beta_2\cdots = \alpha\beta^2$, we determine $T_7$ and the picture is symmetric with respect to the common edge between $T_1, T_2$. Up to symmetry, we have $\alpha_1\alpha_5\cdots = \alpha_1\alpha_5\beta_7\cdots$ which contradicts $\alpha^2\beta\cdots$ not being a vertex. So $\beta_1\beta_2\cdots = \beta^2\delta^d$. Since $\beta^2\delta^d = \vert \beta \Vert \beta \vert \delta \, \bvert \, \cdots \, \bvert \, \delta \vert$, the unique AAD determines $T_7$ (adjacent to $T_1$) and we get $\alpha_1\alpha_5\cdots = \alpha_1\alpha_5\alpha_7\cdots$ which contradicts $\alpha^3\cdots$ not being a vertex. Then $\beta_1\beta_2\cdots$ cannot be a vertex and so $\gamma^2\delta^d$ cannot be a vertex. 

By the angle formulae, we have $\alpha\delta^d = \alpha\delta^{\frac{f+8}{8}}$, $\beta^2\delta^d =\beta^2\delta^{ \frac{f-8}{8}}$ and $\delta^d = \delta^{\frac{f}{4}}$. Hence we get
\begin{align}\label{ab2avcf} 
\AVC = \{ \alpha\beta^2,   \alpha^2\delta^2,  \gamma^4,   \alpha\delta^{\frac{f+8}{8}},   \beta^2\delta^{ \frac{f-8}{8}},  \delta^{\frac{f}{4}} \}.
\end{align}
For the above $\AVC$, the tilings will be constructed in the next section. 
\end{subcase*}

\begin{subcase*}[$\alpha\beta^2, \gamma^4, \beta^2\delta^2$] We may further assume no $\delta^4$, $\alpha^2\delta^2$. The angle formulae are
\begin{align*}
\alpha=\frac{\pi}{2}+\frac{4\pi}{f}, \qquad 
\beta=\frac{3\pi}{4} - \frac{2\pi}{f}, \qquad 
\gamma=\frac{\pi}{2}, \qquad 
\delta=\frac{\pi}{4}+\frac{2\pi}{f}.
\end{align*}

By $f\ge16$, the angle lower bounds are $\alpha > \frac{\pi}{2}$, $\beta \ge \frac{5\pi}{8}$, $\gamma=\frac{\pi}{2}$, $\delta > \frac{\pi}{4}$. Then by $f\ge16$, the angle lower bounds and Lemma \ref{ForbVerLem}, the other vertices are $f=24: \alpha\delta^4, \delta^6$. So we get $\{ \alpha\beta^2,   \gamma^4,   \beta^2\delta^2 \, \vert \, \alpha\delta^4,  \delta^6 \}$. Since $\beta^2\delta^2 = \Vert \beta^{\alpha} \vert^{\alpha} \delta \, \bvert \, \delta^{\alpha} \vert^{\alpha} \beta \Vert$, the AAD implies that $\alpha \vert \alpha \cdots$ has to be a vertex, contradicting $\alpha^2\cdots$ is not a vertex. So $\beta^2\delta^2$ cannot be a vertex. 
\end{subcase*}

\begin{subcase*}[$\alpha\beta^2,\gamma^4,\alpha\delta^4$] We may further assume no $\delta^4$, $\alpha^2\delta^2$, $\beta^2\delta^2$. The angle formulae are
\begin{align*}
\alpha=\frac{16\pi}{f}, \qquad 
\beta=\pi - \frac{8\pi}{f}, \qquad 
\gamma=\frac{\pi}{2}, \qquad 
\delta=\frac{\pi}{2} - \frac{4\pi}{f}.
\end{align*}
And we have $\beta = 2\delta$. By $f\ge16$, the angle lower bounds are $\alpha > 0$, $\beta \ge \frac{\pi}{2}$, $\gamma=\frac{\pi}{2}$, $\delta \ge \frac{\pi}{4}$. In $\alpha\cdots$, by Lemma \ref{ForbVerLem}, $\alpha\cdots = \alpha^a, \alpha^a\beta^b, \alpha^a\beta^b\gamma^c$, $\alpha^a\beta^b\delta^d$, $\alpha\gamma^c\delta^d$, $\alpha^a\delta^d$, where $a > 0$ and $b,c,d$ are even. Then $\alpha^a\beta^b\gamma^c$, $\alpha^a\beta^b\delta^d$, $\alpha\gamma^c\delta^d$ have vertex angle sums $ > 2\pi$ and $\alpha^a\beta^b=\alpha\beta^2$. Meanwhile, as $\alpha^2\delta^2$ is excluded by the assumption, $\alpha^a\delta^d = \alpha^{a\ge3}\delta^2, \alpha\delta^4$. So
\begin{align*}
\alpha\cdots = \alpha\beta^2,   \alpha\delta^4,    \alpha^{a\ge3}\delta^2,   \alpha^{a>3}.
\end{align*}
The $\hat{\alpha}$-vertices other than $\gamma^4$ with corresponding $f$ values are $f = 16: \beta^4, \beta^2\gamma^2, \beta^2\delta^4,  \gamma^2\delta^4, \delta^8$, $f = 20: \beta\gamma\delta^3$, $f = 24: \delta^6$. When the other $\hat{\alpha}$-vertices appear, combining with $f=16, 20, 24$, we get $\alpha = \pi, \frac{4\pi}{5}, \frac{2\pi}{3}$ and $\delta =\frac{\pi}{4}, \frac{3\pi}{10}, \frac{\pi}{3}$ respectively, then $\alpha^{a>3}$ and $\alpha^3\delta^2$ have angle sums $>2\pi$ so they cannot be vertices. Hence $\alpha\cdots = \alpha\beta^2, \alpha\delta^4$. However, the AAD shows that $\alpha\delta^4$ contains $\delta^{\alpha}\vert^{\alpha}\delta$ which implies $\alpha^2\cdots$ being a vertex, a contradiction. So none of the $\hat{\alpha}$-vertices other than $\gamma^4$ can be a vertex and we have $\{ \alpha\beta^2, \gamma^4, \alpha\delta^4,  \alpha^a, \alpha^a\delta^2 \}$. Since $\beta\delta\cdots$, $\alpha^{\beta} \vert^{\beta} \alpha \cdots$ are not vertices, there is no $\alpha\vert\alpha\vert\alpha$ at a vertex. So $\alpha^a$ is not a vertex and $\alpha^a\delta^2 = \alpha^2\delta^2$ which is however not a vertex under our assumption. Meanwhile, $\alpha\delta^4 = \vert \alpha \vert \delta \, \bvert \, \delta \vert \delta \, \bvert \, \delta \vert$ in which $\delta \vert \delta = \delta^{\alpha} \vert^{\alpha} \delta$ implies $\alpha^2\cdots$ has to be a vertex, which is not in $\{ \alpha\beta^2,  \gamma^4, \alpha\delta^4 \}$. Then $\alpha\delta^4$ cannot be a vertex. 
\end{subcase*}

\begin{subcase*}[$ \alpha\beta^2, \gamma^4,\beta\gamma\delta^3$] We may further assume no $\delta^4, \alpha^2\delta^2, \beta^2\delta^2, \alpha\delta^4$. The angle formulae are
\begin{align*}
\alpha=\frac{\pi}{2}+\frac{6\pi}{f}, \qquad 
\beta=\frac{3\pi}{4} - \frac{3\pi}{f}, \qquad 
\gamma=\frac{\pi}{2}, \qquad 
\delta=\frac{\pi}{4}+\frac{\pi}{f}.
\end{align*}
By $f\ge 16$, the angle inequalities are $\alpha > \frac{\pi}{2}$, $\beta \ge \frac{9\pi}{16}$, $\gamma = \frac{\pi}{2}$, $\delta > \frac{\pi}{4}$. By Lemma \ref{ForbVerLem} and the assumption, the other vertices can only be $\delta^6$, which implies $\AVC_4 = \{ \gamma^4 \}$. However, by Lemma \ref{ab2Lem}, there is a degree $5$ vertex with at most one of $\beta, \gamma$, which does not exist under the our assumption, a contradiction. Hence this case has no tilings. 
\end{subcase*}

\begin{subcase*}[$\alpha\beta^2, \gamma^4, \delta^{6}$] We may assume no $\delta^4, \alpha^2\delta^2, \beta^2\delta^2, \alpha\delta^4, \beta\gamma\delta^3$. The angle formulae are
\begin{align*}
\alpha=\frac{\pi}{3}+\frac{8\pi}{f}, \qquad 
\beta=\frac{5\pi}{6} - \frac{4\pi}{f}, \qquad 
\gamma=\frac{\pi}{2}, \qquad 
\delta=\frac{\pi}{3}.
\end{align*}
By Lemma \ref{albe2Lem}, some degree $\ge4$ $\alpha^a\delta^d$ is a vertex. Under our assumption, $\alpha^2\delta^2, \alpha\delta^4$ are not vertices, then by the angle formulae $\alpha^a\delta^d = \alpha^3\delta^2$ must be a vertex. When $\alpha^3\delta^2$ is a vertex, the angle values are 
\begin{align*}
f=72, \qquad
\alpha=\frac{4\pi}{9},\qquad
\beta=\frac{7\pi}{9}, \qquad 
\gamma=\frac{\pi}{2}, \qquad 
\delta=\frac{\pi}{3}.
\end{align*}
Then there is no other vertex. We have $\AVC = \{ \alpha\beta^2, \gamma^4,  \alpha^3\delta^2, \delta^{6} \}$ which implies $\beta^2\cdots = \alpha\beta^2$ and $\beta\delta\cdots$ is not a vertex. By the unique AAD of $\alpha\beta^2$ and $\alpha^{\beta}\vert^{\beta}\alpha\cdots$ is not a vertex. This implies that $\alpha\vert\alpha\vert\alpha \cdots$ cannot be a vertex. However the AAD of $\alpha^3\delta^2$ has to be $\bvert \, \delta \vert \alpha\vert\alpha\vert \alpha \vert \delta \, \bvert$, a contradiction. Then $\alpha^3\delta^2$ cannot be a vertex, a contradiction. Hence this case has no tilings. 
\end{subcase*}
Therefore we complete the discussion of this case.
\end{case*}

\begin{case*}[$\alpha\beta^2, \delta^4$] We may further assume no $\gamma^4$. So \eqref{ab2avc4} becomes $\AVC_4=\{ \alpha^4, \beta^4,  \delta^4, \alpha^2\delta^2,  \beta^2\gamma^2, \beta^2\delta^2 \}$. Apply Lemma \ref{ab2Lem} on $\gamma$, one of $\beta^2\gamma^2, \beta\gamma^3\delta, \gamma^6$ must be a vertex.

\begin{subcase*}[$\alpha\beta^2,\delta^4,\beta^2\gamma^2$] When $\delta^4$ is a vertex, $\alpha\delta^4$ cannot be a vertex. The angle formulae are
\begin{align*}
\alpha=\frac{\pi}{2} + \frac{4\pi}{f}, \qquad 
\beta=\frac{3\pi}{4} - \frac{2\pi}{f}, \qquad 
\gamma=\frac{\pi}{4}+\frac{2\pi}{f}, \qquad 
\delta=\frac{\pi}{2}.
\end{align*}
By $f\ge16$, the angle lower bounds are $\alpha > \frac{\pi}{2}$, $\beta \ge \frac{5\pi}{8}$, $\gamma > \frac{\pi}{4}$, $\delta = \frac{\pi}{2}$. By Lemma \ref{ForbVerLem}, the other vertices can only be $f=24: \gamma^6$.
So we get $\{ \alpha\beta^2,  \beta^2\gamma^2,  \delta^4,  \gamma^6 \}$ in which $\#\alpha < \#\beta$, a contradiction. So $\beta^2\gamma^2$ cannot be a vertex. 
\end{subcase*}

\begin{subcase*}[$\alpha\beta^2,\delta^4,\beta\gamma^3\delta$]We may further assume $\beta^2\gamma^2$ are not vertices. The angle formulae are
\begin{align*}
\alpha=\frac{\pi}{2}+\frac{6\pi}{f}, \qquad 
\beta=\frac{3\pi}{4} - \frac{3\pi}{f}, \qquad 
\gamma=\frac{\pi}{4}+\frac{\pi}{f}, \qquad 
\delta=\frac{\pi}{2}.
\end{align*}
By $f\ge16$, the angle lower bounds are $\alpha > \frac{\pi}{2}$, $\beta \ge \frac{9\pi}{16}$, $\gamma > \frac{\pi}{4}$, $\delta=\frac{\pi}{2}$. By the angle lower bounds and Lemma \ref{ForbVerLem}, there is no other vertex. So we get $\{ \alpha\beta^2, \delta^4,  \beta\gamma^3\delta \}$ in which $\#\alpha < \# \beta$. So $\beta\gamma^3\delta$ cannot be a vertex. 
\end{subcase*}

\begin{subcase*}[$\alpha\beta^2, \delta^4, \gamma^{6}$] We may further assume $\beta^2\gamma^2$, $\beta\gamma^3\delta$ are not vertices. The angle formulae are 
\begin{align*} 
\alpha=\frac{\pi}{3}+\frac{8\pi}{f}, \qquad 
\beta=\frac{5\pi}{6}-\frac{4\pi}{f}, \qquad 
\gamma=\frac{\pi}{3}, \qquad 
\delta=\frac{\pi}{2}.
\end{align*}
Since some degree $\ge4$ $\alpha^a\delta^d$ is a vertex, the angle formulae implies such $\alpha^a\delta^d=\alpha^2\delta^2$ is a vertex. When $\alpha^2\delta^2$ is a vertex, we get
\begin{align*} 
f=48, \qquad
\alpha=\frac{\pi}{2}, \qquad 
\beta=\frac{3\pi}{4}, \qquad 
\gamma=\frac{\pi}{3}, \qquad 
\delta=\frac{\pi}{2}.
\end{align*}
And the other vertex is $\alpha^4$. So we have $\AVC = \{ \alpha\beta^2, \alpha^4, \alpha^2\delta^2, \delta^4, \gamma^{6} \}$ which implies $\beta^2\cdots=\alpha\beta^2$ and $\beta\delta\cdots$ is not a vertex. By the unique AAD of $\alpha\beta^2$, $\beta^{\alpha}\vert^{\alpha}\beta\cdots$ is not a vertex so $\alpha \vert \alpha = \alpha^{\delta} \vert^{\delta}\alpha$ and hence $\alpha \vert \alpha \vert \alpha \cdots$ cannot be a vertex. Then $\alpha^4$ is not a vertex and $\AVC = \{ \alpha\beta^2, \alpha^2\delta^2, \delta^4, \gamma^{6} \}$. By the $\AVC$, we know $\alpha\beta \cdots = \beta^2\cdots = \alpha\beta^2$ and $\alpha^2\cdots = \alpha\delta\cdots = \alpha^2\delta^2$ and $\delta^3\cdots = \delta^4$. Then AAD of $\gamma^6$ in Figure \ref{ga6AAD} uniquely determines $T_1, ..., T_6$. Up to symmetry, we may assume $T_7$ as it is, then we determine $T_8, ..., T_{13}$ starting from $\beta_7$ and $T_{14}, ..., T_{16}$ starting from $\delta_7$. At $T_{17}$, one side of $\alpha_{17}$ there must be $\alpha^2\beta\cdots$ which is not a vertex, a contradiction. Hence this subcase has no tiling. 

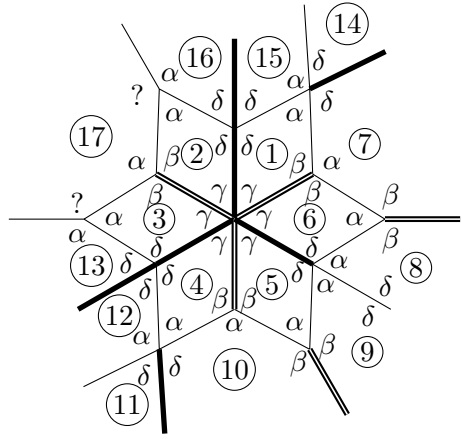
\begin{figure}

\centering
\begin{tikzpicture}[>=latex,scale=1]

\begin{scope}[]

\foreach \a in {0,...,2}{

\draw[rotate=120*\a]
	(30:1.2) -- (60:2)
	(120:2) -- (90:1.2) -- (60:2)
	(150:1.2) -- (120:2);

\draw[rotate=120*\a, double, line width=0.8]
	(30:0) -- (30:1.2);

\draw[rotate=120*\a, line width=2]
	(90:0) -- (90:1.2);

}

\foreach \a in {0,1}{
\draw[rotate=120*\a, line width=2]
	(90:1.2) -- (90:2.4);
}

\foreach \b in {1,2}{
\draw[rotate=60*\b]
	(60:2) -- (60:3);	
}

\foreach \b in {4,5}{
\draw[rotate=60*\b, double, line width=0.8]
	(60:2) -- (60:3);	
}

\draw[]
	(330:1.2) -- (330:2.4)
	(60:2) -- (72:3)
	(240:2) -- (228:3);
	
\draw[line width=2]
	(60:2) -- (48:3)
	(240:2) -- (252:3);

\node at (60:0.4) {\small $\gamma$}; 
\node at (40:1.1) {\small $\beta$}; 
\node at (80:1) {\small $\delta$}; 
\node at (60:1.6) {\small $\alpha$}; 

\node at (120:0.4) {\small $\gamma$}; 
\node at (100:1) {\small $\delta$}; 
\node at (136:1.15) {\small $\beta$}; 
\node at (120:1.6) {\small $\alpha$}; 

\node at (180:0.4) {\small $\gamma$}; 
\node at (164:1.1) {\small $\beta$}; 
\node at (200:1.1) {\small $\delta$}; 
\node at (180:1.6) {\small $\alpha$}; 

\node at (240:0.4) {\small $\gamma$}; 
\node at (220:1.15) {\small $\delta$}; 
\node at (260:1.1) {\small $\beta$}; 
\node at (240:1.6) {\small $\alpha$}; 

\node at (300:0.4) {\small $\gamma$}; 
\node at (280:1.1) {\small $\beta$}; 
\node at (320:1.1) {\small $\delta$}; 
\node at (300:1.6) {\small $\alpha$}; 

\node at (0:0.4) {\small $\gamma$}; 
\node at (340:1.1) {\small $\delta$}; 
\node at (18:1.1) {\small $\beta$}; 
\node at (0:1.6) {\small $\alpha$}; 

\node at (82:1.6) {\small $\delta$}; 

\node at (62:2.4) {\small $\delta$}; 
\node at (66:2.0) {\small $\alpha$}; 

\node at (98:1.6) {\small $\delta$}; 
\node at (114:2.1) {\small $\alpha$}; 

\node at (30:1.5) {\small $\alpha$}; 
\node at (8:2.1) {\small $\beta$}; 
\node at (52:2.0) {\small $\delta$}; 

\node at (338:1.5) {\small $\alpha$}; 
\node at (336:2.2) {\small $\delta$}; 
\node at (352:2.1) {\small $\beta$}; 

\node at (324:1.5) {\small $\alpha$}; 
\node at (324:2.2) {\small $\delta$}; 
\node at (306:2.1) {\small $\beta$}; 

\node at (270:1.4) {\small $\alpha$}; 
\node at (294:2.1) {\small $\beta$}; 
\node at (248:2.1) {\small $\delta$}; 

\node at (240:2.4) {\small $\delta$}; 

\node at (232:2.0) {\small $\alpha$}; 
\node at (218:1.5) {\small $\delta$}; 

\node at (202:1.55) {\small $\delta$}; 
\node at (186:2.1) {\small $\alpha$}; 

\node at (150:1.5) {\small $\alpha$}; 
\node at (174:2.1) {\small $?$}; 
\node at (128:2.1) {\small $?$}; 

\node[inner sep=1,draw,shape=circle] at (60:1) {\small $1$}; 
\node[inner sep=1,draw,shape=circle] at (120:1) {\small $2$}; 
\node[inner sep=1,draw,shape=circle] at (180:1) {\small $3$}; 
\node[inner sep=1,draw,shape=circle] at (240:1) {\small $4$}; 
\node[inner sep=1,draw,shape=circle] at (300:1) {\small $5$}; 
\node[inner sep=1,draw,shape=circle] at (360:1) {\small $6$}; 

\node[inner sep=1,draw,shape=circle] at (30:2) {\small $7$}; 
\node[inner sep=1,draw,shape=circle] at (345:2.5) {\small $8$}; 
\node[inner sep=1,draw,shape=circle] at (315:2.5) {\small $9$}; 
\node[inner sep=1,draw,shape=circle] at (270:2) {\small $10$}; 
\node[inner sep=1,draw,shape=circle] at (240:2.85) {\small $11$}; 
\node[inner sep=1,draw,shape=circle] at (220:2.0) {\small $12$}; 
\node[inner sep=1,draw,shape=circle] at (198:2.0) {\small $13$}; 

\node[inner sep=1,draw,shape=circle] at (60:3) {\small $14$}; 
\node[inner sep=1,draw,shape=circle] at (78:2.2) {\small $15$}; 
\node[inner sep=1,draw,shape=circle] at (102:2.2) {\small $16$}; 

\node[inner sep=1,draw,shape=circle] at (150:2.2) {\small $17$}; 

\end{scope}

\end{tikzpicture}
\caption{The AAD of $\gamma^6$.}
\label{ga6AAD}

\end{figure}

\end{subcase*}

\end{case*}
Therefore we complete the proof.
\end{proof}

\section{Tilings}

In the first part of this section, we demonstrate the construction of tilings from the $\AVC$s obtained in the previous section and in the second part we explain how they are related via flip modifications. The idea of flip modification was first introduced by Yan et al. Interested reader may refer to in \cite{wy2} for more examples.

\subsection{Construction} \label{SecConst}

We summarise the $\AVC$s from \eqref{bgdAVC}, \eqref{a3ab2avcf24}, \eqref{a3avcf24}, \eqref{ab2avcf16}, \eqref{ab2avcf}. in the previous section as follows
\begin{enumerate}
\item $\{ \beta\gamma\delta, \alpha^{\frac{f}{2}} \}$,
\item $\{ \alpha\beta^2,  \alpha^2\delta^2,  \gamma^4, \delta^4  \}$,
\item $\{ \alpha^3,  \gamma^4, \beta^2\delta^2 \}.$
\item $\{ \alpha^3, \alpha\beta^2,  \alpha^2\delta^2,  \beta^2\delta^2, \gamma^4,  \alpha\delta^4,  \delta^6 \}$,
\item $\{ \alpha\beta^2, \alpha^2\delta^2, \gamma^4, \alpha\delta^{\frac{f+8}{8}}, \beta^2\delta^{ \frac{f-8}{8}},  \delta^{\frac{f}{4}} \}$.
\end{enumerate}
It is apparent that some of the $\AVC$s is a subset or a special case of some subset of another $\AVC$. The relation between tilings constructed from each of these $\AVC$s will become transparent in the construction below.

\begin{case*}[$\AVC = \{ \beta\gamma\delta, \alpha^{\frac{f}{2}} \}$]

\begin{figure}[htp]
	\centering
	\begin{tikzpicture}[>=latex,scale=0.75]
	\foreach \a in {0,1,2}
	\draw[rotate=120*\a]
	(0,0) -- (-30:1.2)
	(30:2.4) -- (30:3.6);
	
	\foreach \b in {0,1,2}
	\draw[double, line width=0.6, rotate=120*\b]
	(-30:1.2) -- (30:2.4);
	
	\foreach \c in {0,1,2}
	\draw[line width=2, rotate=120*\c]
	(90:1.2) -- (30:2.4);

	
	\node[shift={(30:0.2)}] at (0,0) {\small $\alpha$};
	\node[shift={(150:0.2)}] at (0,0) {\small $\alpha$};
	\node[shift={(270:0.2)}] at (0,0) {\small $\alpha$};
	
	
	\node[shift={(0.05,0.2)}] at (30:2.4) {\small $\delta$};
	\node[shift={(-0.35,-0.25)}] at (30:2.4) {\small $\gamma$};
	\node[shift={(0.06,-0.25)}] at (30:2.4) {\small $\beta$};
	
	\node[shift={(0.0,0.2)}] at (90:1.2) {\small $\gamma$};
	\node[shift={(-0.2,-0.25)}] at (90:1.2) {\small $\beta$};
	\node[shift={(0.15,-0.25)}] at (90:1.2) {\small $\delta$};
	
	\node[shift={(0.1,0.2)}] at (150:2.4) {\small $\beta$};
	\node[shift={(-0.08,-0.25)}] at (150:2.4) {\small $\delta$};
	\node[shift={(0.35,-0.25)}] at (150:2.4) {\small $\gamma$};	
	
	\node[shift={(0.05,0.25)}] at (210:1.15) {\small $\delta$};
	\node[shift={(-0.25,-0.05)}] at (210:1.2) {\small $\gamma$};
	\node[shift={(0.28,-0.05)}] at (210:1.2) {\small $\beta$};
	
	\node[shift={(0,0.3)}] at (270:2.4) {\small $\gamma$};
	\node[shift={(-0.2,-0.15)}] at (270:2.4) {\small $\beta$};
	\node[shift={(0.2,-0.15)}] at (270:2.4) {\small $\delta$};
	

	\node[shift={(-0.05,0.25)}] at (330:1.2) {\small $\beta$};
	\node[shift={(0.25,-0.05)}] at (330:1.2) {\small $\gamma$};
	\node[shift={(-0.28,-0.1)}] at (330:1.2) {\small $\delta$};

	\node[shift={(-30:2.2)}] at (0,0) {\small $\alpha$};
	\node[shift={(-150:2.2)}] at (0,0) {\small $\alpha$};
	\node[shift={(-270:2.2)}] at (0,0) {\small $\alpha$};

	\node[inner sep=1,draw,shape=circle] at (270:1) {\small $1$}; 
	\node[inner sep=1,draw,shape=circle] at (30:1) {\small $2$}; 
	\node[inner sep=1,draw,shape=circle] at (150:1) {\small $3$}; 
	\node[inner sep=1,draw,shape=circle] at (330:2.2) {\small $4$}; 
	\node[inner sep=1,draw,shape=circle] at (90:2.2) {\small $5$}; 
	\node[inner sep=1,draw,shape=circle] at (210:2.2) {\small $6$}; 
	
	\end{tikzpicture}
	\caption{Tiling of $f=6$, $\AVC = \{ \alpha^3, \beta\gamma\delta \}$}
	\label{Tf6}
\end{figure}

When $f=6$, starting at this vertex the tiling given by the cube is uniquely constructed in Figure \ref{Tf6}. 

We construct the tilings with $\AVC = \{ \beta\gamma\delta, \alpha^a \}$. When $a>3$, Figure \ref{begadeEMT} shows the Earth Map Tilings and their time zones. There are four time zones indicated, each consists of two tiles, namely $(T_1, T_4), (T_2, T_5), (T_3, T_6), (T_7, T_8)$.

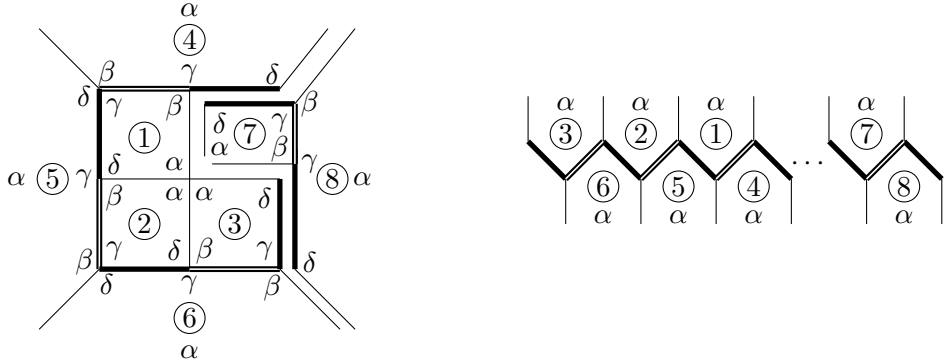
\begin{figure}[htp]
\centering
\begin{tikzpicture}[>=latex,scale=1]

\foreach \a in {0,...,3} 
\draw[rotate=90*\a]
	(0,0) -- (0,1.2);

\foreach \a in {0,...,2} 
{
\draw[rotate=90*\a, double, line width=0.8]
	(0,1.2) -- (-1.2,1.2); 

\draw[rotate=90*\a, line width=2]
	(-1.2,1.2)-- (-1.2, 0);
}

\foreach \a in {0,...,1} 
\draw[rotate=90*\a]
	(-1.2, 1.2) -- (-2,2);

\foreach \a in {3} 
{
\draw
	(0.2,0.3) -- (0.2, 1)	
	(1.4, -1.2) -- (2.2,-2)
	(1.4, 0.2) -- (0.3, 0.2)
	(1.4,1) -- (2.2,2)
	(1.2, 1.2) -- (1.84, 2)
	(1.2,-1.2) -- (2, -2);

\draw[line width=2]
 	(0.2, 1) -- (1.4, 1)
	(1.4, 0.4) -- (1.4, -1.2)
	(0,1.2) -- (1.2,1.2); 

\draw[double, line width=0.8]
	(1.4, 1) -- (1.4, 0.2); 

}

\node at (-0.2,0.2) {\small $\alpha$}; 
\node at (-0.2,0.95) {\small $\beta$}; 
\node at (-1,0.95) {\small $\gamma$}; 
\node at (-1,0.2) {\small $\delta$}; 

\node at (-0.2,-0.2) {\small $\alpha$}; 
\node at (-1,-0.25) {\small $\beta$}; 
\node at (-1,-0.95) {\small $\gamma$}; 
\node at (-0.2,-0.95) {\small $\delta$}; 

\node at (0.2,-0.2) {\small $\alpha$}; 
\node at (0.2,-0.95) {\small $\beta$}; 
\node at (1,-0.95) {\small $\gamma$}; 
\node at (1,-0.25) {\small $\delta$}; 

\node at (0,2.25) {\small $\alpha$}; 
\node at (0,1.4) {\small $\gamma$}; 
\node at (-1.1,1.4) {\small $\beta$}; 
\node at (1.1,1.4) {\small $\delta$}; 

\node at (-2.3,0) {\small $\alpha$}; 
\node at (-1.4,0) {\small $\gamma$}; 
\node at (-1.4,1.1) {\small $\delta$}; 
\node at (-1.4,-1.1) {\small $\beta$}; 

\node at (0, -2.3) {\small $\alpha$}; 
\node at (0, -1.4) {\small $\gamma$}; 
\node at (-1.1,-1.4) {\small $\delta$}; 
\node at (1.1,-1.45) {\small $\beta$}; 

\node at (0.4,0.4) {\small $\alpha$}; 
\node at (0.4,0.75) {\small $\delta$}; 
\node at (1.2,0.75) {\small $\gamma$}; 
\node at (1.2,0.4) {\small $\beta$}; 

\node at (2.3,0) {\small $\alpha$}; 
\node at (1.6,0.2) {\small $\gamma$}; 
\node at (1.6,-1.1) {\small $\delta$}; 
\node at (1.6,1) {\small $\beta$}; 

\node[inner sep=1,draw,shape=circle] at (-0.6,0.55) {\small $1$}; 
\node[inner sep=1,draw,shape=circle] at (-0.6,-0.6) {\small $2$}; 
\node[inner sep=1,draw,shape=circle] at (0.6,-0.6) {\small $3$}; 

\node[inner sep=1,draw,shape=circle] at (0,1.85) {\small $4$}; 
\node[inner sep=1,draw,shape=circle] at (-1.82, 0) {\small $5$}; 
\node[inner sep=1,draw,shape=circle] at (0, -1.85) {\small $6$}; 
\node[inner sep=1,draw,shape=circle] at (0.8,0.6) {\small $7$}; 
\node[inner sep=1,draw,shape=circle] at (1.9, 0) {\small $8$}; 

\begin{scope}[xshift=5cm]

\foreach \b in {0,...,3} {
\begin{scope}[xshift=1*\b cm]
\draw[line width=2]
	(0,0) -- (-0.5,0.5);
\draw 
	(-0.5,0.5)-- (-0.5, 1.1)
	(0,0) -- (0,-0.6);

\end{scope}
}

\foreach \b in {0,...,2} {
\begin{scope}[xshift=1*\b cm]
\draw[double, line width=0.8]
	(0,0) -- (0.5,0.5);

\end{scope}
}


\foreach \c in {4,5} {
\begin{scope}[xshift=1*\c cm]
\draw[line width=2]
	(0,0) -- (-0.5,0.5);
\draw 
	(-0.5,0.5) -- (-0.5, 1.1)
	(0,0) -- (0,-0.6);
\end{scope}
}

\foreach \c in {4} {
\begin{scope}[xshift=1*\c cm]
\draw[double, line width=0.8]
	(0,0) -- (0.5,0.5);
\end{scope}
}

\node at (3.25,0.2) {\small $\cdots$}; 

\node at (0,1) {\small $\alpha$}; 
\node at (1,1) {\small $\alpha$}; 
\node at (2,1) {\small $\alpha$}; 
\node at (4,1) {\small $\alpha$}; 

\node at (0.5,-0.5) {\small $\alpha$}; 
\node at (1.5,-0.5) {\small $\alpha$}; 
\node at (2.5,-0.5) {\small $\alpha$}; 
\node at (4.5,-0.5) {\small $\alpha$}; 

\node[inner sep=1,draw,shape=circle] at (0,0.6) {\small $3$}; 
\node[inner sep=1,draw,shape=circle] at (1,0.6) {\small $2$}; 
\node[inner sep=1,draw,shape=circle] at (2,0.6) {\small $1$}; 

\node[inner sep=1,draw,shape=circle] at (0.5,-0.1) {\small $6$}; 
\node[inner sep=1,draw,shape=circle] at (1.5,-0.1) {\small $5$}; 
\node[inner sep=1,draw,shape=circle] at (2.5,-0.1) {\small $4$}; 

\node[inner sep=1,draw,shape=circle] at (4,0.6) {\small $7$}; 
\node[inner sep=1,draw,shape=circle] at (4.5,-0.1) {\small $8$}; 

\end{scope}

\end{tikzpicture}
\caption{Earth Map Tilings: $\AVC \equiv \{ \beta\gamma\delta, \alpha^a \}$}
\label{begadeEMT}
\end{figure}

\end{case*}

\begin{case*}[$\AVC = \{ \alpha\beta^2, \alpha^2\delta^2,  \gamma^4,  \delta^4 \}$] From the $\AVC$, we have $\alpha^2\cdots = \alpha^2\delta^2$, $\alpha\beta\cdots = \beta^2\cdots = \alpha\beta^2$, $\gamma^2\cdots = \gamma^4$. By the unique AAD of $\delta^4$, we detemine $T_1, T_2, T_3, T_4$ in Figure \ref{pqEMT1624}. Then by $\alpha^2\cdots = \alpha^2\delta^2$, we get $\alpha_1\alpha_2\cdots = \alpha_3\alpha_4\cdots = \alpha^2\delta^2$. So we determine $T_5, T_6, T_7, T_8$. Meanwhile, by $\alpha\beta\cdots = \alpha\beta^2$, we have $\alpha_5\beta_1\cdots = \alpha_6\beta_2\cdots = \alpha_7\beta_3\cdots = \alpha_8\beta_4\cdots = \alpha\beta^2$, so we determine $T_9, T_{10}, T_{11}, T_{12}$. Then we have $\alpha_9\beta_5\cdots = \alpha_{10}\beta_6 = \alpha_{11}\beta_7\cdots = \alpha_{12}\beta_{8}\cdots = \alpha\beta^2$, so we determine $T_{13}, T_{14}, T_{15}, T_{16}$. Hence we obtain the unique tiling with $\AVC \equiv \{ \alpha\beta^2, \alpha^2\delta^2, \gamma^4, \delta^4  \}$. 

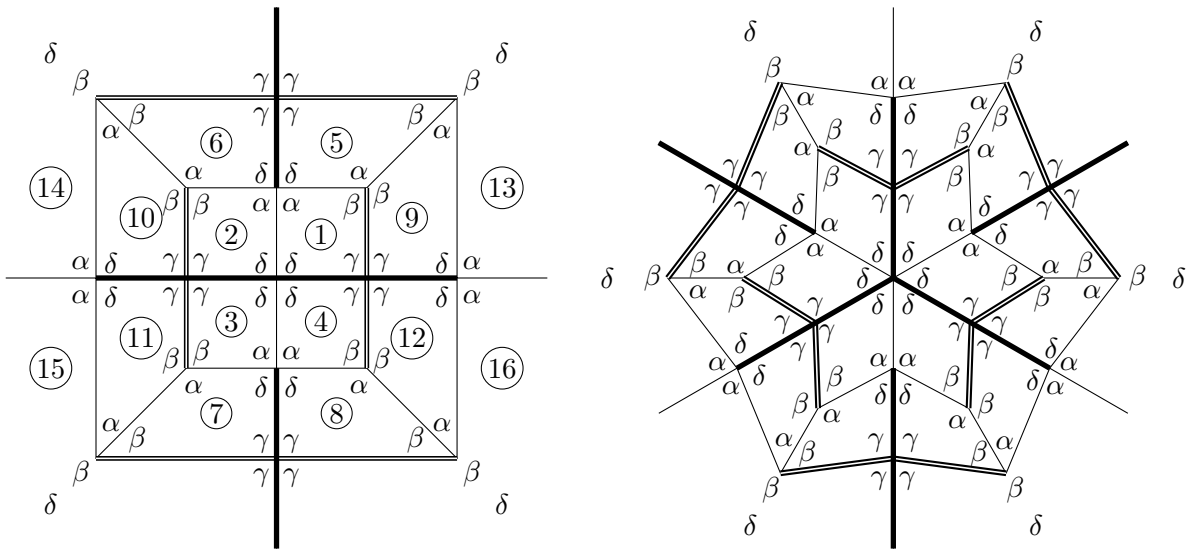
\begin{figure}[htp]
	\centering
	\begin{tikzpicture}[>=latex,scale=1]
	\foreach \a in {0,...,3}
	\draw[rotate=90*\a]
	(0,0) -- (3.6,0)
	(1.2,0) -- (1.2,1.2) -- (0,1.2)
	(2.4,0) -- (2.4,2.4) -- (0,2.4)
	(1.2,1.2) -- (2.4,2.4);
	
	\foreach \b in {0,1}
	\draw[rotate=180*\b, double, line width=0.6]
	(-2.4,2.4) -- (2.4,2.4) 
	(1.2,-1.2) -- (1.2,1.2);

	\foreach \c in {0,1}
	\draw[rotate=180*\c, line width=2]
	(0,0) -- (2.4,0)
	(0,1.2) -- (0,3.6);
	
	
	\node[shift={(0.2,0.2)}] at (0,0) {\small $\delta$};
	\node[shift={(-0.2,0.2)}] at (0,0) {\small $\delta$};
	\node[shift={(-0.2,-0.25)}] at (0,0) {\small $\delta$};
	\node[shift={(0.2,-0.25)}] at (0,0) {\small $\delta$};

	
	
	\node[shift={(0.2,0.2)}] at (1.2,0) {\small $\gamma$};
	\node[shift={(-0.2,0.2)}] at (1.2,0) {\small $\gamma$};
	\node[shift={(-0.2,-0.25)}] at (1.2,0) {\small $\gamma$};
	\node[shift={(0.2,-0.25)}] at (1.2,0) {\small $\gamma$};
	
	\node[shift={(0.2,-0.1)}] at (1.2,1.2) {\small $\beta$};
	\node[shift={(-0.1,0.2)}] at (1.2,1.2) {\small $\alpha$};
	\node[shift={(-0.2,-0.25)}] at (1.2,1.2) {\small $\beta$};
	
	\node[shift={(0.2,0.2)}] at (0,1.2) {\small $\delta$};
	\node[shift={(-0.2,0.2)}] at (0,1.2) {\small $\delta$};
	\node[shift={(-0.2,-0.22)}] at (0,1.2) {\small $\alpha$};
	\node[shift={(0.2,-0.22)}] at (0,1.2) {\small $\alpha$};

	\node[shift={(0.2,-0.25)}] at (-1.2,1.2) {\small $\beta$};
	\node[shift={(0.1,0.2)}] at (-1.2,1.2) {\small $\alpha$};
	\node[shift={(-0.2,-0.1)}] at (-1.2,1.1) {\small $\beta$};
	
	\node[shift={(0.2,0.2)}] at (-1.2,0) {\small $\gamma$};
	\node[shift={(-0.2,0.2)}] at (-1.2,0) {\small $\gamma$};
	\node[shift={(-0.2,-0.25)}] at (-1.2,0) {\small $\gamma$};
	\node[shift={(0.2,-0.25)}] at (-1.2,0) {\small $\gamma$};

	\node[shift={(0.2,0.2)}] at (-1.2,-1.2) {\small $\beta$};
	\node[shift={(-0.2,0.1)}] at (-1.2,-1.2) {\small $\beta$};	
	\node[shift={(0.1,-0.25)}] at (-1.2,-1.2) {\small $\alpha$};

	\node[shift={(0.2,0.2)}] at (0,-1.2) {\small $\alpha$};
	\node[shift={(-0.2,0.2)}] at (0,-1.2) {\small $\alpha$};
	\node[shift={(-0.2,-0.25)}] at (0,-1.2) {\small $\delta$};
	\node[shift={(0.2,-0.25)}] at (0,-1.2) {\small $\delta$};
	
	\node[shift={(0.2,0.1)}] at (1.2,-1.2) {\small $\beta$};		
	\node[shift={(-0.2,0.2)}] at (1.2,-1.2) {\small $\beta$};
	\node[shift={(-0.1,-0.25)}] at (1.2,-1.2) {\small $\alpha$};

	
	\node[shift={(0.2,0.2)}] at (2.4,0) {\small $\alpha$};
	\node[shift={(-0.2,0.2)}] at (2.4,0) {\small $\delta$};
	\node[shift={(-0.2,-0.25)}] at (2.4,0) {\small $\delta$};
	\node[shift={(0.2,-0.25)}] at (2.4,0) {\small $\alpha$};
	
	\node[shift={(0.2,0.2)}] at (2.4,2.4) {\small $\beta$};
	\node[shift={(-0.55,-0.25)}] at (2.4,2.4) {\small $\beta$};	
	\node[shift={(-0.2,-0.45)}] at (2.4,2.4) {\small $\alpha$};	
	
	\node[shift={(0.2,0.2)}] at (0,2.4) {\small $\gamma$};
	\node[shift={(-0.2,0.2)}] at (0,2.4) {\small $\gamma$};
	\node[shift={(-0.2,-0.25)}] at (0,2.4) {\small $\gamma$};
	\node[shift={(0.2,-0.25)}] at (0,2.4) {\small $\gamma$};

	\node[shift={(0.55,-0.25)}] at (-2.4,2.4) {\small $\beta$};
	\node[shift={(-0.2,0.2)}] at (-2.4,2.4) {\small $\beta$};
	\node[shift={(0.2,-0.45)}] at (-2.4,2.4) {\small $\alpha$};

	\node[shift={(0.2,0.2)}] at (-2.4,0) {\small $\delta$};
	\node[shift={(-0.2,0.2)}] at (-2.4,0) {\small $\alpha$};
	\node[shift={(-0.2,-0.25)}] at (-2.4,0) {\small $\alpha$};
	\node[shift={(0.2,-0.25)}] at (-2.4,0) {\small $\delta$};
	
	\node[shift={(0.55,0.25)}] at (-2.4,-2.4) {\small $\beta$};
	\node[shift={(0.2,0.45)}] at (-2.4,-2.4) {\small $\alpha$};	
	\node[shift={(-0.2,-0.2)}] at (-2.4,-2.4) {\small $\beta$};
		
	\node[shift={(0.2,0.2)}] at (0,-2.4) {\small $\gamma$};
	\node[shift={(-0.2,0.2)}] at (0,-2.4) {\small $\gamma$};
	\node[shift={(-0.2,-0.25)}] at (0,-2.4) {\small $\gamma$};
	\node[shift={(0.2,-0.25)}] at (0,-2.4) {\small $\gamma$};
	
	\node[shift={(0.2,-0.2)}] at (2.4,-2.4) {\small $\beta$};	
	\node[shift={(-0.2,0.45)}] at (2.4,-2.4) {\small $\alpha$};	
	\node[shift={(-0.55,0.25)}] at (2.4,-2.4) {\small $\beta$};

	
	\node[shift={(3,3)}] at (0,0) {\small $\delta$};
	\node[shift={(-3,3)}] at (0,0) {\small $\delta$};
	\node[shift={(-3,-3)}] at (0,0) {\small $\delta$};
	\node[shift={(3,-3)}] at (0,0) {\small $\delta$};

	\node[inner sep=1,draw,shape=circle] at (0.59,0.58) {\small $1$};
	\node[inner sep=1,draw,shape=circle] at (-0.6,0.58) {\small $2$};
	\node[inner sep=1,draw,shape=circle] at (-0.6,-0.58) {\small $3$};
	\node[inner sep=1,draw,shape=circle] at (0.59,-0.58) {\small $4$};

	\node[inner sep=1,draw,shape=circle] at (0.8,1.8) {\small $5$};
	\node[inner sep=1,draw,shape=circle] at (-0.8,1.8) {\small $6$};	
	\node[inner sep=1,draw,shape=circle] at (-0.8,-1.8) {\small $7$};
	\node[inner sep=1,draw,shape=circle] at (0.8,-1.8) {\small $8$};

	\node[inner sep=1,draw,shape=circle] at (1.8,0.8) {\small $9$};
	\node[inner sep=1,draw,shape=circle] at (-1.8,0.8) {\small $10$};
	\node[inner sep=1,draw,shape=circle] at (-1.8,-0.8) {\small $11$};
	\node[inner sep=1,draw,shape=circle] at (1.8,-0.8) {\small $12$};

	\node[inner sep=1,draw,shape=circle] at (3,1.2) {\small $13$};
	\node[inner sep=1,draw,shape=circle] at (-3,1.2) {\small $14$};
	\node[inner sep=1,draw,shape=circle] at (-3,-1.2) {\small $15$};
	\node[inner sep=1,draw,shape=circle] at (3,-1.2) {\small $16$};

\begin{scope}[xshift=8.2cm]

\foreach \a in {0,...,5}{
\draw[rotate=60*\a]
	(90:0) -- (90:1.2) -- (90:2.4)
	(60:2) -- (60:3) 
	(90:2.4) -- (90: 3.6);

}

\foreach \a in {0,...,2}{

\draw[rotate=120*\a]
	(30:1.2) -- (60:2)
	(150:1.2) -- (120:2)
	(60:3) -- (90:2.4) -- (120:3);

\draw[rotate=120*\a, double, line width=0.8]
	(120:2) -- (90:1.2) -- (60:2)
	(0:3) -- (30:2.4) -- (60:3);

\draw[rotate=120*\a, line width=2]
	(90:0) -- (90:1.2) -- (90:2.4)
	(30:1.2) -- (30:2.4) -- (30:3.6);

}

\node at (60:0.4) {\small $\delta$}; 
\node at (40:1.1) {\small $\alpha$}; 
\node at (80:1) {\small $\gamma$}; 
\node at (58:1.55) {\small $\beta$}; 

\node at (120:0.4) {\small $\delta$}; 
\node at (100:1) {\small $\gamma$}; 
\node at (140:1.1) {\small $\alpha$}; 
\node at (122:1.55) {\small $\beta$}; 

\node at (180:0.4) {\small $\delta$}; 
\node at (160:1.1) {\small $\alpha$}; 
\node at (200:1.1) {\small $\gamma$}; 
\node at (180:1.55) {\small $\beta$}; 

\node at (240:0.4) {\small $\delta$}; 
\node at (220:1.15) {\small $\gamma$}; 
\node at (260:1.1) {\small $\alpha$}; 
\node at (240:1.55) {\small $\beta$}; 

\node at (300:0.4) {\small $\delta$}; 
\node at (280:1.1) {\small $\alpha$}; 
\node at (320:1.1) {\small $\gamma$}; 
\node at (300:1.55) {\small $\beta$}; 

\node at (0:0.4) {\small $\delta$}; 
\node at (340:1.1) {\small $\gamma$}; 
\node at (20:1.1) {\small $\alpha$}; 
\node at (0:1.55) {\small $\beta$}; 

\node at (38:1.55) {\small $\delta$}; 
\node at (36:2.2) {\small $\gamma$}; 
\node at (54:2.1) {\small $\alpha$}; 
\node at (56:2.55) {\small $\beta$}; 

\node at (82:1.6) {\small $\gamma$}; 
\node at (64:2.1) {\small $\beta$}; 
\node at (64:2.6) {\small $\alpha$}; 
\node at (84:2.2) {\small $\delta$}; 

\node at (98:1.6) {\small $\gamma$}; 
\node at (96:2.2) {\small $\delta$}; 
\node at (112:2.1) {\small $\beta$}; 
\node at (116:2.65) {\small $\alpha$}; 

\node at (142:1.6) {\small $\delta$}; 
\node at (125:2.1) {\small $\alpha$}; 
\node at (125:2.6) {\small $\beta$}; 
\node at (144:2.2) {\small $\gamma$}; 

\node at (159:1.6) {\small $\delta$}; 
\node at (156:2.2) {\small $\gamma$}; 
\node at (176:2.1) {\small $\alpha$}; 
\node at (176:2.6) {\small $\beta$}; 

\node at (203:1.6) {\small $\gamma$}; 
\node at (188:2.1) {\small $\beta$}; 
\node at (184:2.6) {\small $\alpha$}; 
\node at (203:2.2) {\small $\delta$}; 

\node at (218:1.6) {\small $\gamma$}; 
\node at (216:2.2) {\small $\delta$}; 
\node at (234:2.1) {\small $\beta$}; 
\node at (236:2.6) {\small $\alpha$}; 

\node at (264:1.55) {\small $\delta$}; 
\node at (246:2) {\small $\alpha$}; 
\node at (244:2.6) {\small $\beta$}; 
\node at (264:2.2) {\small $\gamma$}; 

\node at (276:1.55) {\small $\delta$}; 
\node at (276:2.2) {\small $\gamma$}; 
\node at (294:2.1) {\small $\alpha$}; 
\node at (296:2.6) {\small $\beta$}; 

\node at (322:1.55) {\small $\gamma$}; 
\node at (304:2.65) {\small $\alpha$}; 
\node at (306:2.1) {\small $\beta$}; 
\node at (324:2.3) {\small $\delta$}; 

\node at (336:1.55) {\small $\gamma$}; 
\node at (336:2.3) {\small $\delta$}; 
\node at (352:2) {\small $\beta$}; 
\node at (356:2.65) {\small $\alpha$}; 

\node at (22:1.55) {\small $\delta$}; 
\node at (24:2.2) {\small $\gamma$}; 
\node at (4:2.1) {\small $\alpha$}; 
\node at (4:2.55) {\small $\beta$}; 

\node at (60:3.25) {\small $\beta$}; 
\node at (34:2.6) {\small $\gamma$}; 
\node at (86:2.6) {\small $\alpha$}; 
\node at (60:3.8) {\small $\delta$}; 

\node at (94:2.6) {\small $\alpha$}; 
\node at (120:3.2) {\small $\beta$}; 
\node at (146:2.6) {\small $\gamma$}; 
\node at (120:3.8) {\small $\delta$}; 

\node at (180:3.2) {\small $\beta$}; 
\node at (155:2.65) {\small $\gamma$}; 
\node at (180:3.8) {\small $\delta$}; 
\node at (206:2.6) {\small $\alpha$}; 

\node at (214:2.6) {\small $\alpha$}; 
\node at (240:3.25) {\small $\beta$}; 
\node at (240:3.8) {\small $\delta$}; 
\node at (266:2.7) {\small $\gamma$}; 

\node at (274:2.7) {\small $\gamma$}; 
\node at (300:3.25) {\small $\beta$}; 
\node at (300:3.8) {\small $\delta$}; 
\node at (326:2.6) {\small $\alpha$}; 

\node at (334:2.6) {\small $\alpha$}; 
\node at (0:3.25) {\small $\beta$}; 
\node at (0:3.8) {\small $\delta$}; 
\node at (24:2.6) {\small $\gamma$}; 

\end{scope}

	\end{tikzpicture}
	\caption{$(4,4)$-Earth Map Tiling of $f=16$, $\AVC \equiv \{ \alpha\beta^2, \alpha^2\delta^2, \gamma^4, \delta^{4} \}$ and $(6,4)$-Earth Map Tiling of $f=24$, $\AVC \equiv \{ \alpha\beta^2, \alpha^2\delta^2, \gamma^4, \delta^{6} \}$}
	\label{pqEMT1624}
\end{figure}

In light of the unique construction starting at $\delta^4$ in Figure \ref{pqEMT1624}, we uniquely obtain the tilings of $\AVC \equiv \{ \alpha\beta^2, \alpha^2\delta^2, \gamma^4, \delta^{\frac{f}{4}} \}$. We call such tilings \textit{$(p, q)$-Earth Map Tilings}, where $p=\frac{f}{4}$ is the degree of the polar vertices $\delta^{\frac{f}{4}}$ and $q=4$ is the degree of the adjacent rings of vertices, $\alpha^2\delta^2$ and $\gamma^4$, to the polar vertices, for the reason explained in Figure \ref{pqEMT} (numbering not related). A time zone consists of eight tiles, composed in the shape as indicated in $T_1, \cdots, T_8$ in Figure \ref{pqEMT}. Two time zones combined we get the tiling of $f=16, \AVC \equiv \{ \alpha\beta^2, \alpha^2\delta^2, \gamma^4, \delta^4 \}$ in Figure \ref{pqEMT1624}. Three time zones combined we get the tiling of $f=24, \AVC \equiv \{ \alpha\beta^2, \alpha^2\delta^2, \gamma^4, \delta^6 \}$ in Proposition \ref{Propa3ab2}. In general, tilings of such kind consist of $k=\frac{f}{8}$ time zones for $k$ in Figure \ref{pqEMT}. 

\begin{figure}[htp]
	\centering


\begin{tikzpicture}[>=latex,scale=1]

\foreach \b in {0,1,2,3,4,6,7,8}{

	\begin{scope}[xshift=1.2*\b cm] 
	\draw[]
	(0,0) -- (0,-1.2)	
	(0.6, -1.2) -- (0.6, -2.4)
	(1.2, -2.4) -- (1.2, -3.6);

	\end{scope}
}

\foreach \b in {0,2,4,6, 8}{

	\begin{scope}[xshift=1.2*\b cm] 
	\draw[]
	(0,-1.2) -- (0.6, -1.2) 
	(0.6, -2.4) -- (1.2,-2.4);
	\end{scope}

}

\foreach \b in {0,2,6}{

	\begin{scope}[xshift=1.2*\b cm] 
	\draw[]
	(1.8, -1.2) -- (2.4, -1.2)
	(1.2, -2.4) -- (1.8, -2.4);
	\end{scope}

}


\foreach \b in {0, 2, 6}{

	\begin{scope}[xshift=1.2*\b cm] 
	
	\draw[double, line width=0.8]
	(0.6, -1.2) -- (1.2, -1.2) -- (1.8, -1.2)
	(1.8, -2.4) -- (2.4, -2.4) -- (3.0, -2.4);

	\end{scope}
}

\foreach \b in {0, 2, 6}{

	\begin{scope}[xshift=1.2*\b cm] 
	
	\draw[line width=2]
	(1.2, 0) -- (1.2, -1.2)
	(1.2, -1.2) -- (1.2, -2.4)
	(2.4, -1.2) -- (2.4, -2.4)
	(2.4, -2.4) -- (2.4, -3.6);
	\end{scope}
}

\node at (6.5,-1.8) {\small $\cdots$};

\node at (0.6,0) {\small $\delta$}; 
\node at (0.2,-1) {\small $\alpha$}; 
\node at (0.6,-1) {\small $\beta$}; 
\node at (1,-1) {\small $\gamma$}; 

\node at (1.8,0) {\small $\delta$}; 
\node at (1.4,-1) {\small $\gamma$}; 
\node at (1.8,-1) {\small $\beta$}; 
\node at (2.2,-1) {\small $\alpha$}; 

\node at (0.75,-2.25) {\small $\alpha$}; 
\node at (0.7,-1.45) {\small $\beta$}; 
\node at (1,-1.4) {\small $\gamma$}; 
\node at (1,-2.2) {\small $\delta$}; 

\node at (1.4,-1.4) {\small $\gamma$}; 
\node at (1.65,-1.45) {\small $\beta$}; 
\node at (1.35,-2.2) {\small $\delta$}; 
\node at (1.65,-2.25) {\small $\alpha$}; 

\node at (1.95,-1.35) {\small $\alpha$}; 
\node at (1.95,-2.2) {\small $\beta$}; 
\node at (2.25,-2.2) {\small $\gamma$}; 
\node at (2.25,-1.4) {\small $\delta$}; 

\node at (2.85,-1.35) {\small $\alpha$}; 
\node at (2.85,-2.2) {\small $\beta$}; 
\node at (2.55,-2.2) {\small $\gamma$}; 
\node at (2.55,-1.4) {\small $\delta$}; 

\node at (1.35,-2.55) {\small $\alpha$}; 
\node at (1.8,-2.62) {\small $\beta$}; 
\node at (2.25,-2.65) {\small $\gamma$}; 
\node at (1.8,-3.6) {\small $\delta$}; 

\node at (3.45,-2.55) {\small $\alpha$}; 
\node at (3,-2.62) {\small $\beta$}; 
\node at (2.55,-2.65) {\small $\gamma$}; 
\node at (3,-3.6) {\small $\delta$}; 




\node[inner sep=1,draw,shape=circle] at (0.6,-0.5) {\small $1$}; 
\node[inner sep=1,draw,shape=circle] at (1.8,-0.5) {\small $2$}; 
\node[inner sep=1,draw,shape=circle] at (0.9,-1.8) {\small $3$}; 
\node[inner sep=1,draw,shape=circle] at (1.5,-1.8) {\small $4$}; 
\node[inner sep=1,draw,shape=circle] at (2.1,-1.8) {\small $5$}; 
\node[inner sep=1,draw,shape=circle] at (2.7,-1.8) {\small $6$}; 
\node[inner sep=1,draw,shape=circle] at (1.8,-3.05) {\small $7$}; 
\node[inner sep=1,draw,shape=circle] at (3,-3.05) {\small $8$}; 

\node[inner sep=1,draw,shape=circle] at (3,-0.5) {\small $9$}; 
\node[inner sep=0.5,draw,shape=circle] at (4.2,-0.5) {\small $10$}; 

\node[inner sep=0.5,draw,shape=circle] at (3.3,-1.8) {\small $11$}; 
\node[inner sep=0.5,draw,shape=circle] at (3.9,-1.8) {\small $12$}; 
\node[inner sep=0.5,draw,shape=circle] at (4.5,-1.8) {\small $13$}; 
\node[inner sep=0.5,draw,shape=circle] at (5.1,-1.8) {\small $14$}; 

\node[inner sep=0.5,draw,shape=circle] at (4.2,-3) {\small $15$}; 
\node[inner sep=0.5,draw,shape=circle] at (5.4,-3) {\small $16$}; 

\node[inner sep=0.5,draw,shape=circle] at (7.8,-0.5) {\small $k_1$}; 
\node[inner sep=0.5,draw,shape=circle] at (9.0,-0.5) {\small $k_2$}; 
\node[inner sep=0.5,draw,shape=circle] at (8.1,-1.8) {\small $k_3$}; 
\node[inner sep=0.5,draw,shape=circle] at (8.7,-1.8) {\small $k_4$}; 
\node[inner sep=0.5,draw,shape=circle] at (9.3,-1.8) {\small $k_5$}; 
\node[inner sep=0.5,draw,shape=circle] at (9.9,-1.8) {\small $k_6$}; 

\node[inner sep=0.5,draw,shape=circle] at (9.0,-3) {\small $k_7$}; 
\node[inner sep=0.5,draw,shape=circle] at (10.2,-3) {\small $k_8$}; 

\end{tikzpicture}
\caption{$(\frac{f}{4}, 4)$-Earth Map Tiling of $\AVC \equiv \{ \alpha\beta^2, \alpha^2\delta^2, \gamma^4, \delta^{\frac{f}{4}} \}$: Time Zone Decomposition}
\label{pqEMT}
\end{figure}
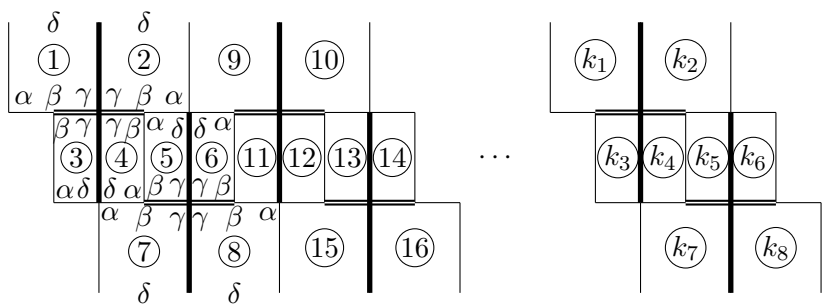

\end{case*}

\begin{case*}[$\AVC = \{\alpha^3, \alpha\beta^2, \alpha^2\delta^2, \beta^2\delta^2, \gamma^4, \alpha\delta^4, \delta^6 \}$] When $\delta^6$ is a vertex, from the previous case we will get the $(6,4)$-earth map tiling of $\AVC \equiv \{  \alpha\beta^2, \alpha^2\delta^2,  \gamma^4, \delta^6 \}$ in Figure \ref{pqEMT1624}. Next we may assume $\delta^6$ is not a vertex. Suppose $\alpha^3$ is a vertex, then Figure \ref{a3AAD} shows the uniquely determined the neighbourhood of $\alpha^3$ up to symmetry. In $T_6, T_7$, by $\alpha^2\cdots = \alpha^3, \alpha^2\delta^2$, there are two choices for $\alpha_6\alpha_7\cdots$. \\

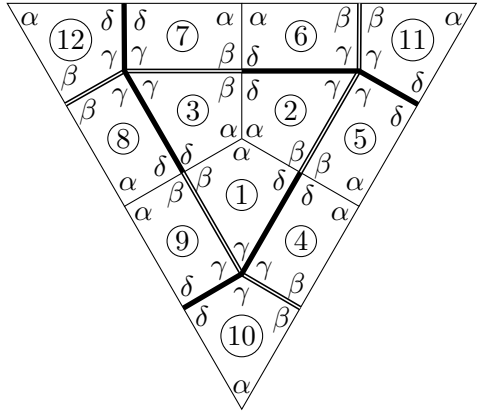
\begin{figure}[htp]
	\centering
	\begin{tikzpicture}[>=latex,scale=0.75]
	\foreach \a in {0,1,2}
	\draw[rotate=120*\a]
	(0,0)-- (-30:2.4) -- (30:4.8) -- (90:2.4)
	(-30:1.2) -- (30:2.4) -- (90:1.2)
	(11:3.18) -- (30:2.4) -- (49:3.18);
	
	\foreach \b in {0,1,2}
	\draw[double, line width=0.6, rotate=120*\b]
	(-30:1.2) -- (30:2.4) -- (49:3.18);
	
	\foreach \c in {0,1,2}
	\draw[line width=2, rotate=120*\c]
	(90:1.2) -- (30:2.4) -- (11:3.18);

	
	\node[shift={(30:0.2)}] at (0,0) {\small $\alpha$};
	\node[shift={(150:0.2)}] at (0,0) {\small $\alpha$};
	\node[shift={(270:0.2)}] at (0,0) {\small $\alpha$};
	
	
	\node[shift={(0.2,0.15)}] at (30:2.4) {\small $\gamma$};
	\node[shift={(-0.2,0.2)}] at (30:2.4) {\small $\gamma$};
	\node[shift={(-0.35,-0.25)}] at (30:2.4) {\small $\gamma$};
	\node[shift={(0.06,-0.35)}] at (30:2.4) {\small $\gamma$};
	
	\node[shift={(0.15,0.2)}] at (90:1.2) {\small $\delta$};
	\node[shift={(-0.2,0.2)}] at (90:1.2) {\small $\beta$};
	\node[shift={(-0.2,-0.25)}] at (90:1.2) {\small $\beta$};
	\node[shift={(0.15,-0.25)}] at (90:1.2) {\small $\delta$};
	
	\node[shift={(0.2,0.2)}] at (150:2.4) {\small $\gamma$};
	\node[shift={(-0.2,0.15)}] at (150:2.4) {\small $\gamma$};
	\node[shift={(-0.06,-0.35)}] at (150:2.4) {\small $\gamma$};
	\node[shift={(0.35,-0.25)}] at (150:2.4) {\small $\gamma$};	
	
	\node[shift={(0.05,0.25)}] at (210:1.2) {\small $\delta$};
	\node[shift={(-0.3,0.09)}] at (210:1.2) {\small $\delta$};
	\node[shift={(-0.1,-0.3)}] at (210:1.2) {\small $\beta$};
	\node[shift={(0.28,-0.1)}] at (210:1.2) {\small $\beta$};
	
	\node[shift={(0,0.3)}] at (270:2.4) {\small $\gamma$};
	\node[shift={(-0.3,0.05)}] at (270:2.4) {\small $\gamma$};
	\node[shift={(0,-0.3)}] at (270:2.4) {\small $\gamma$};
	\node[shift={(0.3,0.05)}] at (270:2.4) {\small $\gamma$};
	
	\node[shift={(0.3,0.09)}] at (330:1.2) {\small $\beta$};
	\node[shift={(-0.05,0.25)}] at (330:1.2) {\small $\beta$};
	\node[shift={(-0.28,-0.1)}] at (330:1.2) {\small $\delta$};
	\node[shift={(0.1,-0.3)}] at (330:1.2) {\small $\delta$};
	
	
	\node[shift={(0,0.3)}] at (11:3.18) {\small $\delta$};
	\node[shift={(-0.28,-0.15)}] at (11:3.18) {\small $\delta$};
	
	\node[shift={(-0.3,-0.2)}] at (30:4.8) {\small $\alpha$};
	
	\node[shift={(0.22,-0.22)}] at (49:3.18) {\small $\beta$};
	\node[shift={(-0.22,-0.22)}] at (49:3.18) {\small $\beta$};
	
	\node[shift={(0.22,-0.2)}] at (90:2.4) {\small $\alpha$};
	\node[shift={(-0.22,-0.2)}] at (90:2.4) {\small $\alpha$};
	
	\node[shift={(0.22,-0.22)}] at (131:3.18) {\small $\delta$};
	\node[shift={(-0.22,-0.22)}] at (131:3.18) {\small $\delta$};
	
	\node[shift={(0.3,-0.2)}] at (150:4.8) {\small $\alpha$};
	
	\node[shift={(0.05,0.35)}] at (169:3.18) {\small $\beta$};
	\node[shift={(0.3,-0.1)}] at (169:3.18) {\small $\beta$};
	
	\node[shift={(0.05,0.3)}] at (210:2.4) {\small $\alpha$};
	\node[shift={(0.25,-0.1)}] at (210:2.4) {\small $\alpha$};
	
	\node[shift={(0.05,0.3)}] at (251:3.18) {\small $\delta$};
	\node[shift={(0.25,-0.1)}] at (251:3.18) {\small $\delta$};
	
	\node[shift={(0,0.3)}] at (270:4.8) {\small $\alpha$};
	
	\node[shift={(-0.05,0.3)}] at (289:3.18) {\small $\beta$};
	\node[shift={(-0.25,-0.15)}] at (289:3.18) {\small $\beta$};
	
	\node[shift={(-0.05,0.3)}] at (330:2.4) {\small $\alpha$};
	\node[shift={(-0.25,-0.1)}] at (330:2.4) {\small $\alpha$};

	\node[inner sep=1,draw,shape=circle] at (270:1) {\small $1$}; 
	\node[inner sep=1,draw,shape=circle] at (30:1) {\small $2$}; 
	\node[inner sep=1,draw,shape=circle] at (150:1) {\small $3$}; 
	\node[inner sep=1,draw,shape=circle] at (300:2.1) {\small $4$}; 
	\node[inner sep=1,draw,shape=circle] at (0:2.1) {\small $5$}; 
	\node[inner sep=1,draw,shape=circle] at (60:2.1) {\small $6$}; 
	\node[inner sep=1,draw,shape=circle] at (120:2.1) {\small $7$}; 
	\node[inner sep=1,draw,shape=circle] at (180:2.1) {\small $8$}; 
	\node[inner sep=1,draw,shape=circle] at (240:2.1) {\small $9$}; 
	\node[inner sep=1,draw,shape=circle] at (270:3.5) {\small $10$}; 
	\node[inner sep=1,draw,shape=circle] at (30:3.5) {\small $11$}; 
	\node[inner sep=1,draw,shape=circle] at (150:3.5) {\small $12$}; 

	\end{tikzpicture}
	\caption{AAD of $\alpha^3$.}
	\label{a3AAD}
\end{figure}

When $\alpha_6\alpha_7\cdots = \alpha^3$, then $T_{13}, T_{14}, T_{15}$ in Figure \ref{Tf24a3-1} are determined. If $\beta_{8}\beta_{12}\cdots = \alpha\beta^2$, then either $\alpha_8\alpha_9\cdots = \alpha^2\beta\cdots$ or $\alpha_{12}\alpha_{15}\cdots = \alpha^2\beta\cdots$ which is not a vertex, a contradiction. So $\beta_{8}\beta_{12}\cdots = \beta^2\delta^2$. Then we determine $T_{16}, T_{17}, T_{18}$. By the same token, $\beta_4\beta_{10}\cdots =\beta^2\delta^2$ and we determine $T_{19}, T_{20}, T_{21}$. Lastly, by $\gamma\cdots = \gamma^4$, we get $\gamma_{13}\cdots = \gamma_{16}\cdots = \gamma_{19}\cdots = \gamma^4$ and we determine $T_{22}, T_{23}, T_{24}$ and hence the tiling of $f=24, \AVC \equiv \{\alpha^3, \beta^2\delta^2, \gamma^4 \}$. This tiling is the same as the tiling in Figure \ref{Tf24CubeSD} which will be further explained. 

\begin{figure}[htp]
	\centering
	\begin{tikzpicture}[>=latex,scale=0.85]
	\foreach \a in {0,1,2}
	\draw[rotate=120*\a]
	(0,0)-- (-30:2.4) -- (30:4.8) -- (90:2.4)
	(-30:1.2) -- (30:2.4) -- (90:1.2)
	(11:3.18) -- (30:2.4) -- (49:3.18)
	(49:3.18) -- (90:3.84) -- (131:3.18)
	(30:7.56) -- (90:3.84) -- (150:7.56)
	(30:4.8) -- (30:9);
	
	\foreach \c in {0,1,2}
	\draw[double, line width=0.6, rotate=120*\c]
	(-30:1.2) -- (30:2.4) -- (49:3.18)
	(30:7.56) -- (90:3.84) -- (131:3.18);
	
	\foreach \b in {0,1,2}
	\draw[line width=2, rotate=120*\b]
	(90:1.2) -- (30:2.4) -- (11:3.18)
	(49:3.18)  -- (90:3.84) -- (150:7.56);

	
	\node[shift={(30:0.2)}] at (0,0) {\small $\alpha$};
	\node[shift={(150:0.2)}] at (0,0) {\small $\alpha$};
	\node[shift={(270:0.2)}] at (0,0) {\small $\alpha$};
	
	
	\node[shift={(0.2,0.15)}] at (30:2.4) {\small $\gamma$};
	\node[shift={(-0.2,0.2)}] at (30:2.4) {\small $\gamma$};
	\node[shift={(-0.35,-0.25)}] at (30:2.4) {\small $\gamma$};
	\node[shift={(0.06,-0.35)}] at (30:2.4) {\small $\gamma$};
	
	\node[shift={(0.15,0.2)}] at (90:1.2) {\small $\delta$};
	\node[shift={(-0.2,0.2)}] at (90:1.2) {\small $\beta$};
	\node[shift={(-0.2,-0.25)}] at (90:1.2) {\small $\beta$};
	\node[shift={(0.15,-0.25)}] at (90:1.2) {\small $\delta$};
	
	\node[shift={(0.2,0.2)}] at (150:2.4) {\small $\gamma$};
	\node[shift={(-0.2,0.15)}] at (150:2.4) {\small $\gamma$};
	\node[shift={(-0.06,-0.35)}] at (150:2.4) {\small $\gamma$};
	\node[shift={(0.35,-0.25)}] at (150:2.4) {\small $\gamma$};	
	
	\node[shift={(0.05,0.25)}] at (210:1.2) {\small $\delta$};
	\node[shift={(-0.3,0.09)}] at (210:1.2) {\small $\delta$};
	\node[shift={(-0.1,-0.3)}] at (210:1.2) {\small $\beta$};
	\node[shift={(0.28,-0.1)}] at (210:1.2) {\small $\beta$};
	
	\node[shift={(0,0.3)}] at (270:2.4) {\small $\gamma$};
	\node[shift={(-0.3,0.05)}] at (270:2.4) {\small $\gamma$};
	\node[shift={(0,-0.3)}] at (270:2.4) {\small $\gamma$};
	\node[shift={(0.3,0.05)}] at (270:2.4) {\small $\gamma$};
	
	\node[shift={(0.3,0.09)}] at (330:1.2) {\small $\beta$};
	\node[shift={(-0.05,0.25)}] at (330:1.2) {\small $\beta$};
	\node[shift={(-0.28,-0.1)}] at (330:1.2) {\small $\delta$};
	\node[shift={(0.1,-0.3)}] at (330:1.2) {\small $\delta$};

	
	\node[shift={(0,0.3)}] at (11:3.18) {\small $\delta$};
	\node[shift={(-0.28,-0.15)}] at (11:3.18) {\small $\delta$};
	\node[shift={(-0.12,-0.5)}] at (11:3.18) {\small $\beta$};
	\node[shift={(0.2,-0.05)}] at (11:3.18) {\small $\beta$};

	\node[shift={(-0.3,-0.2)}] at (30:4.8) {\small $\alpha$};
	\node[shift={(105:0.2)}] at (30:4.8) {\small $\alpha$};
	\node[shift={(-45:0.2)}] at (30:4.8) {\small $\alpha$};
	
	\node[shift={(0.05,0.22)}] at (49:3.18) {\small $\delta$};
	\node[shift={(-0.65,0.22)}] at (49:3.18) {\small $\delta$};
	\node[shift={(0.22,-0.22)}] at (49:3.18) {\small $\beta$};
	\node[shift={(-0.22,-0.22)}] at (49:3.18) {\small $\beta$};
	
	\node[shift={(0.22,-0.2)}] at (90:2.4) {\small $\alpha$};
	\node[shift={(-0.22,-0.2)}] at (90:2.4) {\small $\alpha$};
	\node[shift={(0,0.2)}] at (90:2.4) {\small $\alpha$};
	
	\node[shift={(0.65,0.22)}] at (131:3.18) {\small $\beta$};
	\node[shift={(-0.05,0.22)}] at (131:3.18) {\small $\beta$};
	\node[shift={(0.22,-0.22)}] at (131:3.18) {\small $\delta$};
	\node[shift={(-0.22,-0.22)}] at (131:3.18) {\small $\delta$};

	\node[shift={(0.3,-0.2)}] at (150:4.8) {\small $\alpha$};
	\node[shift={(75:0.2)}] at (150:4.8) {\small $\alpha$};
	\node[shift={(225:0.2)}] at (150:4.8) {\small $\alpha$};
	
	\node[shift={(0.05,0.35)}] at (169:3.18) {\small $\beta$};
	\node[shift={(0.3,-0.1)}] at (169:3.18) {\small $\beta$};
	\node[shift={(0.12,-0.5)}] at (169:3.18) {\small $\delta$};
	\node[shift={(-0.2,-0.05)}] at (169:3.18) {\small $\delta$};
	
	\node[shift={(0.05,0.3)}] at (210:2.4) {\small $\alpha$};
	\node[shift={(0.25,-0.1)}] at (210:2.4) {\small $\alpha$};
	\node[shift={(210:0.2)}] at (210:2.4) {\small $\alpha$};
	
	\node[shift={(0.05,0.3)}] at (251:3.18) {\small $\delta$};
	\node[shift={(-0.5,0.5)}] at (251:3.18) {\small $\beta$};
	\node[shift={(-0.15,-0.2)}] at (251:3.18) {\small $\beta$};
	\node[shift={(0.25,-0.1)}] at (251:3.18) {\small $\delta$};

	\node[shift={(0,0.3)}] at (270:4.8) {\small $\alpha$};
	\node[shift={(195:0.2)}] at (270:4.8) {\small $\alpha$};
	\node[shift={(-5:0.2)}] at (270:4.8) {\small $\alpha$};
	
	\node[shift={(0.45,0.45)}] at (289:3.18) {\small $\delta$};
	\node[shift={(-0.05,0.3)}] at (289:3.18) {\small $\beta$};
	\node[shift={(-0.25,-0.15)}] at (289:3.18) {\small $\beta$};
	\node[shift={(0.15,-0.15)}] at (289:3.18) {\small $\delta$};
	
	\node[shift={(-0.05,0.3)}] at (330:2.4) {\small $\alpha$};
	\node[shift={(-0.25,-0.1)}] at (330:2.4) {\small $\alpha$};	
	\node[shift={(330:0.2)}] at (330:2.4) {\small $\alpha$};
	
	
	\node[shift={(105:0.25)}] at (30:7.56) {\small $\beta$};
	\node[shift={(-0.6,-0.22)}] at (30:7.56) {\small $\beta$};
	\node[shift={(-45:0.25)}] at (30:7.56) {\small $\delta$};
	\node[shift={(-0.45,-0.45)}] at (30:7.56) {\small $\delta$};

	\node[shift={(90:0.25)}] at (90:3.84) {\small $\gamma$};
	\node[shift={(-0.6,-0.22)}] at (90:3.84) {\small $\gamma$};
	\node[shift={(0,-0.3)}] at (90:3.84) {\small $\gamma$};
	\node[shift={(0.6,-0.22)}] at (90:3.84) {\small $\gamma$};

	\node[shift={(0.7,-0.2)}] at (150:7.56) {\small $\delta$};
	\node[shift={(75:0.25)}] at (150:7.56) {\small $\delta$};
	\node[shift={(225:0.25)}] at (150:7.56) {\small $\beta$};
	\node[shift={(0.55,-0.55)}] at (150:7.56) {\small $\beta$};
	
	\node[shift={(30:0.25)}] at (210:3.84) {\small $\gamma$};
	\node[shift={(-0.15,0.5)}] at (210:3.84) {\small $\gamma$};
	\node[shift={(210:0.25)}] at (210:3.84) {\small $\gamma$};
	\node[shift={(0.45,-0.45)}] at (210:3.84) {\small $\gamma$};

	\node[shift={(0.17,0.65)}] at (270:7.56) {\small $\beta$};
	\node[shift={(-0.17,0.65)}] at (270:7.56) {\small $\delta$};
	\node[shift={(195:0.25)}] at (270:7.56) {\small $\delta$};
	\node[shift={(-15:0.25)}] at (270:7.56) {\small $\beta$};

	\node[shift={(0.15,0.5)}] at (330:3.84) {\small $\gamma$};
	\node[shift={(330:0.25)}] at (330:3.84) {\small $\gamma$};
	\node[shift={(150:0.25)}] at (330:3.84) {\small $\gamma$};
	\node[shift={(-0.45,-0.45)}] at (330:3.84) {\small $\gamma$};
	
	\node[shift={(90:5.1)}] at (0,0) {\small $\alpha$};
	\node[shift={(210:5)}] at (0,0) {\small $\alpha$};
	\node[shift={(330:5)}] at (0,0) {\small $\alpha$};

	\node[inner sep=1,draw,shape=circle] at (270:1) {\small $1$}; 
	\node[inner sep=1,draw,shape=circle] at (30:1) {\small $2$}; 
	\node[inner sep=1,draw,shape=circle] at (150:1) {\small $3$}; 
	\node[inner sep=1,draw,shape=circle] at (300:2.1) {\small $4$}; 
	\node[inner sep=1,draw,shape=circle] at (0:2.1) {\small $5$}; 
	\node[inner sep=1,draw,shape=circle] at (60:2.1) {\small $6$}; 
	\node[inner sep=1,draw,shape=circle] at (120:2.1) {\small $7$}; 
	\node[inner sep=1,draw,shape=circle] at (180:2.1) {\small $8$}; 
	\node[inner sep=1,draw,shape=circle] at (240:2.1) {\small $9$}; 
	\node[inner sep=1,draw,shape=circle] at (270:3.5) {\small $10$}; 
	\node[inner sep=1,draw,shape=circle] at (30:3.5) {\small $11$}; 
	\node[inner sep=1,draw,shape=circle] at (150:3.5) {\small $12$}; 
	\node[inner sep=1,draw,shape=circle] at (90:3.1) {\small $13$}; 
	\node[inner sep=1,draw,shape=circle] at (45:4.3) {\small $14$}; 
	\node[inner sep=1,draw,shape=circle] at (135:4.3) {\small $15$}; 
	\node[inner sep=1,draw,shape=circle] at (210:3.1) {\small $16$}; 
	\node[inner sep=1,draw,shape=circle] at (165:4.4) {\small $17$}; 
	\node[inner sep=1,draw,shape=circle] at (255:4.4) {\small $18$}; 
	\node[inner sep=1,draw,shape=circle] at (330:3.1) {\small $19$}; 
	\node[inner sep=1,draw,shape=circle] at (285:4.4) {\small $20$}; 
	\node[inner sep=1,draw,shape=circle] at (15:4.4) {\small $21$}; 
	\node[inner sep=1,draw,shape=circle] at (90:5) {\small $22$}; 
	\node[inner sep=1,draw,shape=circle] at (210:5) {\small $23$}; 
	\node[inner sep=1,draw,shape=circle] at (330:5) {\small $24$}; 

	\end{tikzpicture}
	\caption{The Tiling of $f=24$, $\AVC \equiv \{\alpha^3, \beta^2\delta^2, \gamma^4 \}$}
	\label{Tf24a3-1}
\end{figure}
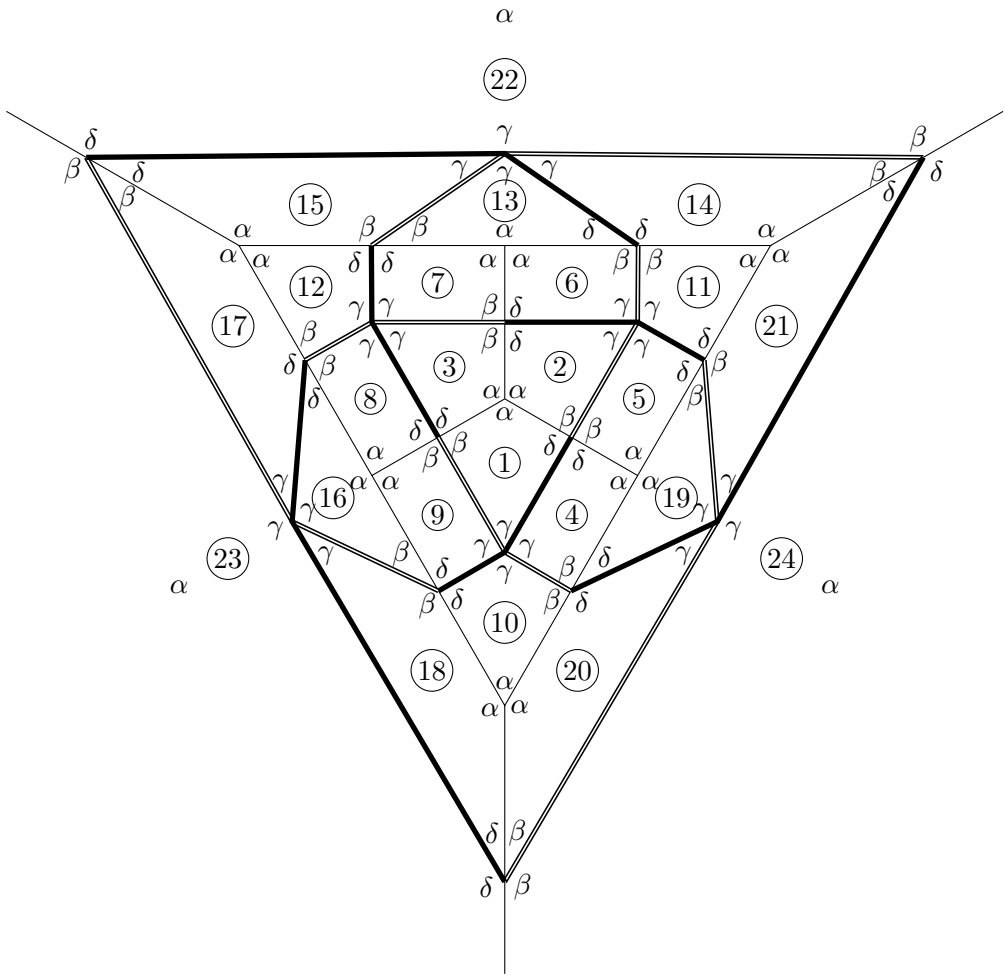

When $\alpha_6\alpha_7\cdots = \alpha^2\delta^2$, then $T_{13}, T_{14}, T_{15}$ in Figure \ref{Tf24a3-2} are determined. If $\beta_{4}\beta_{10}\cdots = \beta^2\delta^2$, we get $\alpha_4\alpha_5\cdots = \alpha^3$ and hence $\alpha_{15}\delta_5\delta_{11}\cdots = \alpha\beta\delta^2\cdots$ which is not a vertex, a contradiction. So $\beta_{4}\beta_{10}\cdots = \alpha\beta^2$ and we determine $T_{16}, T_{17}, T_{18}$. By the same token, $\beta_{8}\beta_{12}\cdots = \alpha\beta^2$ and we determine $T_{19}, T_{20}, T_{21}$. Lastly, by $\gamma\cdots = \gamma^4$, we get $\gamma_{13}\cdots = \gamma_{16}\cdots = \gamma_{19}\cdots = \gamma^4$ and we determine $T_{22}, T_{23}, T_{24}$ and hence the tiling of $f=24, \AVC \equiv \{\alpha^3, \alpha\beta^2, \alpha^2\delta^2, \beta^2\delta^2, \gamma^4 \}$. This is in fact the quadrilateral subdivision of the triangular prism in the third picture in Figure \ref{QuadSubd}.

\begin{figure}[htp]
	\centering
	\begin{tikzpicture}[>=latex,scale=0.85]
	\foreach \a in {0,1,2}
	\draw[rotate=120*\a]
	(0,0)-- (-30:2.4) -- (30:4.8) -- (90:2.4)
	(-30:1.2) -- (30:2.4) -- (90:1.2)
	(11:3.18) -- (30:2.4) -- (49:3.18)
	(131:3.18) -- (116:4.32) -- (40:7.2)
	(150:4.8) -- (160:7.2) -- (116:4.32)
	(90:2.4) -- (40:7.2)
	(116:4.32) -- (108:5.5);
	
	\foreach \c in {0,1,2}
	\draw[double, line width=0.6, rotate=120*\c]
	(-30:1.2) -- (30:2.4) -- (49:3.18)	
	(30:4.8) -- (40:7.2) -- (116:4.32);
	
	\foreach \b in {0,1,2}
	\draw[line width=2, rotate=120*\b]
	(90:1.2) -- (30:2.4) -- (11:3.18)
	(90:2.4) -- (40:7.2) -- (-4:4.32);

	
	\node[shift={(30:0.2)}] at (0,0) {\small $\alpha$};
	\node[shift={(150:0.2)}] at (0,0) {\small $\alpha$};
	\node[shift={(270:0.2)}] at (0,0) {\small $\alpha$};
	
	
	\node[shift={(0.2,0.15)}] at (30:2.4) {\small $\gamma$};
	\node[shift={(-0.2,0.2)}] at (30:2.4) {\small $\gamma$};
	\node[shift={(-0.35,-0.25)}] at (30:2.4) {\small $\gamma$};
	\node[shift={(0.06,-0.35)}] at (30:2.4) {\small $\gamma$};
	
	\node[shift={(0.15,0.2)}] at (90:1.2) {\small $\delta$};
	\node[shift={(-0.2,0.2)}] at (90:1.2) {\small $\beta$};
	\node[shift={(-0.2,-0.25)}] at (90:1.2) {\small $\beta$};
	\node[shift={(0.15,-0.25)}] at (90:1.2) {\small $\delta$};
	
	\node[shift={(0.2,0.2)}] at (150:2.4) {\small $\gamma$};
	\node[shift={(-0.2,0.15)}] at (150:2.4) {\small $\gamma$};
	\node[shift={(-0.06,-0.35)}] at (150:2.4) {\small $\gamma$};
	\node[shift={(0.35,-0.25)}] at (150:2.4) {\small $\gamma$};	
	
	\node[shift={(0.05,0.25)}] at (210:1.2) {\small $\delta$};
	\node[shift={(-0.3,0.09)}] at (210:1.2) {\small $\delta$};
	\node[shift={(-0.1,-0.3)}] at (210:1.2) {\small $\beta$};
	\node[shift={(0.28,-0.1)}] at (210:1.2) {\small $\beta$};
	
	\node[shift={(0,0.3)}] at (270:2.4) {\small $\gamma$};
	\node[shift={(-0.3,0.05)}] at (270:2.4) {\small $\gamma$};
	\node[shift={(0,-0.3)}] at (270:2.4) {\small $\gamma$};
	\node[shift={(0.3,0.05)}] at (270:2.4) {\small $\gamma$};
	
	\node[shift={(0.3,0.09)}] at (330:1.2) {\small $\beta$};
	\node[shift={(-0.05,0.25)}] at (330:1.2) {\small $\beta$};
	\node[shift={(-0.28,-0.1)}] at (330:1.2) {\small $\delta$};
	\node[shift={(0.1,-0.3)}] at (330:1.2) {\small $\delta$};

	
	\node[shift={(275:0.3)}] at (12:3.20) {\small $\alpha$};	
	\node[shift={(20:0.3)}] at (10:3.15) {\small $\alpha$};
	\node[shift={(0,0.3)}] at (11:3.18) {\small $\delta$};
	\node[shift={(-0.28,-0.15)}] at (11:3.18) {\small $\delta$};
	
	\node[shift={(-0.3,-0.2)}] at (30:4.8) {\small $\alpha$};
	\node[shift={(-45:0.2)}] at (30:4.8) {\small $\beta$};
	\node[shift={(105:0.2)}] at (30:4.8) {\small $\beta$};
	
	\node[shift={(0.22,-0.22)}] at (49:3.18) {\small $\beta$};
	\node[shift={(-0.22,-0.22)}] at (49:3.18) {\small $\beta$};
	\node[shift={(90:0.15)}] at (49:3.18) {\small $\alpha$};

	\node[shift={(0.22,-0.2)}] at (90:2.4) {\small $\alpha$};
	\node[shift={(-0.22,-0.2)}] at (90:2.4) {\small $\alpha$};
	\node[shift={(10:0.9)}] at (90:2.4) {\small $\delta$};	
	\node[shift={(95:0.2)}] at (90:2.4) {\small $\delta$};

	\node[shift={(35:0.3)}] at (131:3.18) {\small $\alpha$};	
	\node[shift={(140:0.3)}] at (131:3.18) {\small $\alpha$};
	\node[shift={(0.22,-0.22)}] at (131:3.18) {\small $\delta$};
	\node[shift={(-0.22,-0.22)}] at (131:3.18) {\small $\delta$};
	
	\node[shift={(0.3,-0.2)}] at (150:4.8) {\small $\alpha$};
	\node[shift={(75:0.2)}] at (150:4.8) {\small $\beta$};
	\node[shift={(-0.12,-0.25)}] at (150:4.8) {\small $\beta$};
	
	\node[shift={(0.05,0.35)}] at (169:3.18) {\small $\beta$};
	\node[shift={(0.3,-0.1)}] at (169:3.18) {\small $\beta$};
	\node[shift={(210:0.15)}] at (169:3.18) {\small $\alpha$};
	
	\node[shift={(0.05,0.3)}] at (210:2.4) {\small $\alpha$};
	\node[shift={(0.25,-0.1)}] at (210:2.4) {\small $\alpha$};
	\node[shift={(130:0.9)}] at (210:2.4) {\small $\delta$};	
	\node[shift={(215:0.2)}] at (212:2.4) {\small $\delta$};

	\node[shift={(175:0.3)}] at (250:3.12) {\small $\alpha$};	
	\node[shift={(260:0.3)}] at (250:3.15) {\small $\alpha$};
	\node[shift={(0.05,0.3)}] at (251:3.18) {\small $\delta$};
	\node[shift={(0.25,-0.1)}] at (251:3.18) {\small $\delta$};

	\node[shift={(0,0.3)}] at (270:4.8) {\small $\alpha$};
	\node[shift={(195:0.2)}] at (270:4.8) {\small $\beta$};
	\node[shift={(0.25,0)}] at (270:4.8) {\small $\beta$};

	\node[shift={(-0.05,0.3)}] at (289:3.18) {\small $\beta$};
	\node[shift={(-0.25,-0.15)}] at (289:3.18) {\small $\beta$};
	\node[shift={(330:0.15)}] at (289:3.18) {\small $\alpha$};
	
	\node[shift={(-0.05,0.3)}] at (330:2.4) {\small $\alpha$};
	\node[shift={(-0.25,-0.1)}] at (330:2.4) {\small $\alpha$};	
	\node[shift={(250:0.9)}] at (330:2.4) {\small $\delta$};	
	\node[shift={(335:0.2)}] at (330:2.4) {\small $\delta$};
	
	
	\node[shift={(40:0.2)}] at (40:7.2) {\small $\gamma$};
	\node[shift={(194:1.3)}] at (40:7.2) {\small $\gamma$};
	\node[shift={(222:0.8)}] at (40:7.2) {\small $\gamma$};
	\node[shift={(248:1.3)}] at (40:7.2) {\small $\gamma$};
	
	\node[shift={(47:0.3)}] at (116:4.32) {\small $\beta$};
	\node[shift={(115:0.25)}] at (116:4.32) {\small $\delta$};
	\node[shift={(235:0.3)}] at (116:4.32) {\small $\delta$};
	\node[shift={(295:0.3)}] at (116:4.32) {\small $\beta$};
	
	\node[shift={(160:0.2)}] at (160:7.2) {\small $\gamma$};
	\node[shift={(314:1.3)}] at (160:7.2) {\small $\gamma$};
	\node[shift={(342:0.8)}] at (160:7.2) {\small $\gamma$};
	\node[shift={(8:1.3)}] at (160:7.2) {\small $\gamma$};
	
	\node[shift={(0.15,0.25)}] at (236:4.32) {\small $\beta$};
	\node[shift={(-0.35,0.1)}] at (236:4.32) {\small $\beta$};
	\node[shift={(-0.09,-0.25)}] at (236:4.32) {\small $\delta$};
	\node[shift={((0.33,-0.06))}] at (236:4.32) {\small $\delta$};
	
	\node[shift={(0.05,-0.24)}] at (280:7.2) {\small $\gamma$};
	\node[shift={(74:1.3)}] at (280:7.2) {\small $\gamma$};
	\node[shift={(102:0.8)}] at (280:7.2) {\small $\gamma$};
	\node[shift={(128:1.3)}] at (280:7.2) {\small $\gamma$};
	
	\node[shift={(0.05,-0.3)}] at (356:4.32) {\small $\beta$};
	\node[shift={(-0.3,-0.12)}] at (356:4.32) {\small $\beta$};
	\node[shift={(-0.05,0.3)}] at (356:4.32) {\small $\delta$};
	\node[shift={(0.19,0.15)}] at (356:4.32) {\small $\delta$};
	
	
	\node[shift={(40:7)}] at (0,0) {\small $\alpha$};
	\node[shift={(160:7)}] at (0,0) {\small $\alpha$};
	\node[shift={(280:7)}] at (0,0) {\small $\alpha$};

	\node[inner sep=1,draw,shape=circle] at (270:1) {\small $1$}; 
	\node[inner sep=1,draw,shape=circle] at (30:1) {\small $2$}; 
	\node[inner sep=1,draw,shape=circle] at (150:1) {\small $3$}; 
	\node[inner sep=1,draw,shape=circle] at (300:2.1) {\small $4$}; 
	\node[inner sep=1,draw,shape=circle] at (0:2.1) {\small $5$}; 
	\node[inner sep=1,draw,shape=circle] at (60:2.1) {\small $6$}; 
	\node[inner sep=1,draw,shape=circle] at (120:2.1) {\small $7$}; 
	\node[inner sep=1,draw,shape=circle] at (180:2.1) {\small $8$}; 
	\node[inner sep=1,draw,shape=circle] at (240:2.1) {\small $9$}; 
	\node[inner sep=1,draw,shape=circle] at (270:3.5) {\small $10$}; 
	\node[inner sep=1,draw,shape=circle] at (30:3.5) {\small $11$}; 
	\node[inner sep=1,draw,shape=circle] at (150:3.5) {\small $12$}; 
	\node[inner sep=1,draw,shape=circle] at (90:3.3) {\small $13$}; 
	\node[inner sep=1,draw,shape=circle] at (45:4.3) {\small $14$}; 
	\node[inner sep=1,draw,shape=circle] at (15:4.2) {\small $15$}; 
	\node[inner sep=1,draw,shape=circle] at (285:4.2) {\small $16$}; 
	\node[inner sep=1,draw,shape=circle] at (330:3.4) {\small $17$}; 
	\node[inner sep=1,draw,shape=circle] at (255:4.2) {\small $18$}; 
	\node[inner sep=1,draw,shape=circle] at (165:4.2) {\small $19$}; 
	\node[inner sep=1,draw,shape=circle] at (205:3.3) {\small $20$}; 
	\node[inner sep=1,draw,shape=circle] at (135:4.2) {\small $21$}; 
	\node[inner sep=1,draw,shape=circle] at (90:4.8) {\small $22$}; 
	\node[inner sep=1,draw,shape=circle] at (330:4.8) {\small $23$}; 
	\node[inner sep=1,draw,shape=circle] at (205:4.8) {\small $24$}; 

	\end{tikzpicture}
	\caption{The Tiling $f=24$, $\AVC \equiv \{\alpha^3, \alpha\beta^2, \alpha^2\delta^2,  \beta^2\delta^2, \gamma^4 \}$}
	\label{Tf24a3-2}
\end{figure}
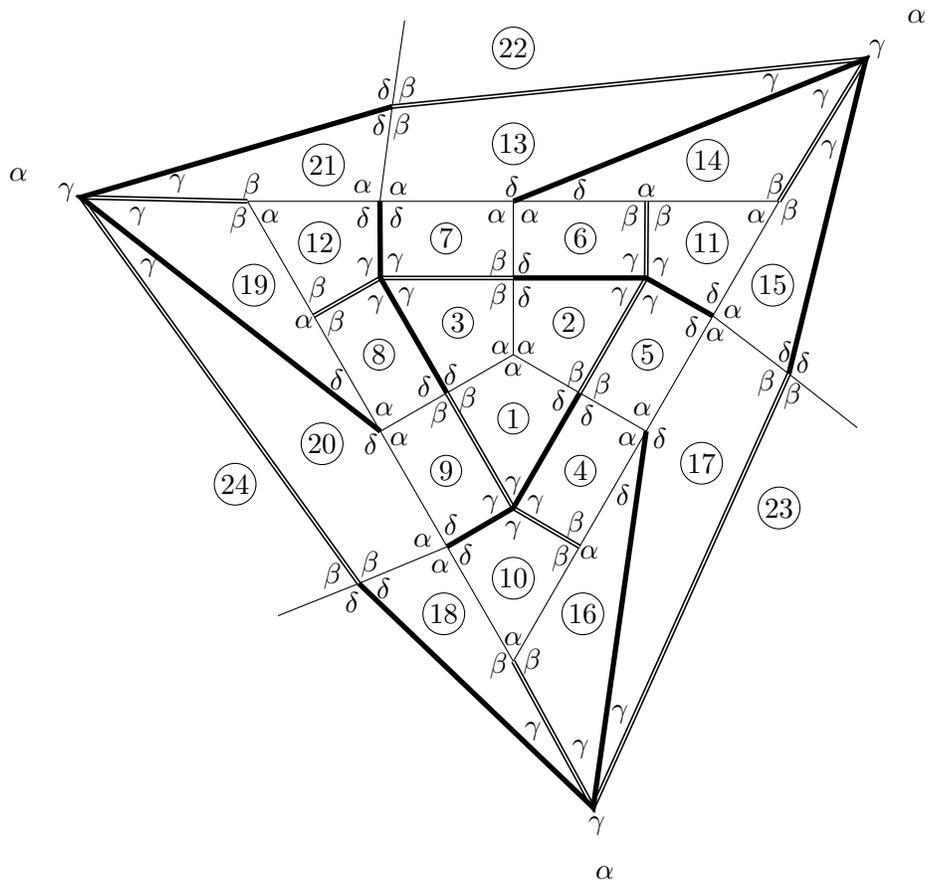

Next we may assume no $\alpha^3$. Then $\AVC = \{ \alpha\beta^2, \alpha^2\delta^2, \beta^2\delta^2, \gamma^4, \alpha\delta^4 \}$. Then $\alpha^2\cdots = \alpha^2\delta^2$ and $\delta^3\cdots = \alpha\delta^4$. If $\beta^2\delta^2$ is a vertex, starting at this vertex we determine $T_1, T_2, T_3, T_4$ in Figure \ref{AADbe2de2}. By $\gamma_3\cdots = \gamma_4\cdots = \gamma^4$, $T_5, T_6$ are determined. By $\alpha_1\alpha_4\cdots =  \alpha_2\alpha_{3}\cdots = \alpha^2\cdots = \alpha^2\delta^2$, $T_7, T_8, T_9, T_{10}$ are determined. By $\gamma_7\cdots = \gamma_8\cdots = \gamma_9\cdots = \gamma_{10}\cdots = \gamma^4$, $T_{11}, T_{12}, T_{13}, T_{14}$ are determined and at these four vertices we have $\delta^4\cdots = \alpha\delta^4$ must be a vertex. Up to symmetry, $\alpha\delta^4$ uniquely determines $T_{15}$ which is translated in Figure \ref{Tf24b2d2}. Meanwhile, $\gamma_{15}\cdots = \gamma_1\cdots = \gamma_2\cdots = \gamma^4$ so $T_{16}, T_{17}, T_{18}, T_{19}, T_{20}$ are determined. By $\alpha_{18}\alpha_{19}\cdots = \alpha^2\cdots  =\alpha^2\delta^2$, $T_{21}, T_{22}$ are determined. Finally, $\gamma_{21}\cdots = \gamma^4$ we determine $T_{23}, T_{24}$. Hence we obtain the tiling of $f=24, \AVC \equiv \{ \alpha\beta^2, \alpha^2\delta^2, \beta^2\delta^2, \gamma^4, \alpha\delta^4 \}$ in Figure \ref{Tf24b2d2}.

\begin{figure}[htp]
	\centering
	\begin{tikzpicture}[>=latex,scale=1]
	\foreach \a in {0,...,3}
	\draw[rotate=90*\a]
	(0,0) -- (4.8,0)
	(3.6,0) -- (3.6,3.6) -- (0,3.6) 
	(1.2,0) -- (1.2,1.2) -- (0,1.2)
	(2.4,0) -- (2.4,2.4) -- (0,2.4)
	(1.2,1.2) -- (3.6,3.6);
	
	\foreach \b in {0,1}
	\draw[double, line width=0.6,rotate=180*\b]
	(0,1.2) -- (0,-1.2)
	(2.4,2.4) -- (-2.4,2.4)
	(1.2,0) -- (3.6,0)
	(0,3.6) -- (0,4.8);
	
	\foreach \c in {1,3}
	\draw[line width=2,rotate=90*\c]
	(0,1.2) -- (0,-1.2)
	(2.4,2.4) -- (-2.4,2.4)
	(1.2,0) -- (3.6,0)
	(0,3.6) -- (0,4.8);

	
	\node[shift={(0.2,0.2)}] at (0,0) {\small $\gamma$};
	\node[shift={(-0.2,0.2)}] at (0,0) {\small $\gamma$};
	\node[shift={(-0.2,-0.25)}] at (0,0) {\small $\gamma$};
	\node[shift={(0.2,-0.25)}] at (0,0) {\small $\gamma$};

	
	
	\node[shift={(0.2,0.2)}] at (1.2,0) {\small $\beta$};
	\node[shift={(-0.2,0.2)}] at (1.2,0) {\small $\delta$};
	\node[shift={(-0.2,-0.25)}] at (1.2,0) {\small $\delta$};
	\node[shift={(0.2,-0.25)}] at (1.2,0) {\small $\beta$};
	
	\node[shift={(0.2,-0.1)}] at (1.2,1.2) {\small $\alpha$};
	\node[shift={(-0.1,0.2)}] at (1.2,1.2) {\small $\alpha$};
	\node[shift={(-0.2,-0.22)}] at (1.2,1.2) {\small $\alpha$};
	
	\node[shift={(0.2,0.2)}] at (0,1.2) {\small $\delta$};
	\node[shift={(-0.2,0.2)}] at (0,1.2) {\small $\delta$};
	\node[shift={(-0.2,-0.25)}] at (0,1.2) {\small $\beta$};
	\node[shift={(0.2,-0.25)}] at (0,1.2) {\small $\beta$};
	
	\node[shift={(0.2,-0.22)}] at (-1.2,1.2) {\small $\alpha$};
	\node[shift={(0.1,0.2)}] at (-1.2,1.2) {\small $\alpha$};
	\node[shift={(-0.2,-0.1)}] at (-1.2,1.2) {\small $\alpha$};
	
	\node[shift={(0.2,0.2)}] at (-1.2,0) {\small $\delta$};
	\node[shift={(-0.2,0.2)}] at (-1.2,0) {\small $\beta$};
	\node[shift={(-0.2,-0.25)}] at (-1.2,0) {\small $\beta$};
	\node[shift={(0.2,-0.25)}] at (-1.2,0) {\small $\delta$};
	
	\node[shift={(0.2,0.22)}] at (-1.2,-1.2) {\small $\alpha$};
	\node[shift={(0.1,-0.2)}] at (-1.2,-1.2) {\small $\alpha$};
	\node[shift={(-0.2,0.1)}] at (-1.2,-1.2) {\small $\alpha$};
	
	\node[shift={(0.2,0.2)}] at (0,-1.2) {\small $\beta$};
	\node[shift={(-0.2,0.2)}] at (0,-1.2) {\small $\beta$};
	\node[shift={(-0.2,-0.25)}] at (0,-1.2) {\small $\delta$};
	\node[shift={(0.2,-0.25)}] at (0,-1.2) {\small $\delta$};
	
	\node[shift={(-0.2,0.22)}] at (1.2,-1.2) {\small $\alpha$};
	\node[shift={(-0.1,-0.2)}] at (1.2,-1.2) {\small $\alpha$};
	\node[shift={(0.2,0.1)}] at (1.2,-1.2) {\small $\alpha$};

	
	\node[shift={(0.2,0.2)}] at (2.4,0) {\small $\gamma$};
	\node[shift={(-0.2,0.2)}] at (2.4,0) {\small $\gamma$};
	\node[shift={(-0.2,-0.25)}] at (2.4,0) {\small $\gamma$};
	\node[shift={(0.2,-0.25)}] at (2.4,0) {\small $\gamma$};

	\node[shift={(0.2,-0.1)}] at (2.4,2.4) {\small $\delta$};
	\node[shift={(-0.1,0.25)}] at (2.4,2.4) {\small $\beta$};
	\node[shift={(-0.55,-0.25)}] at (2.4,2.4) {\small $\beta$};	
	\node[shift={(-0.2,-0.45)}] at (2.4,2.4) {\small $\delta$};	
	
	\node[shift={(0.2,0.2)}] at (0,2.4) {\small $\gamma$};
	\node[shift={(-0.2,0.2)}] at (0,2.4) {\small $\gamma$};
	\node[shift={(-0.2,-0.25)}] at (0,2.4) {\small $\gamma$};
	\node[shift={(0.2,-0.25)}] at (0,2.4) {\small $\gamma$};
	
	\node[shift={(0.55,-0.25)}] at (-2.4,2.4) {\small $\beta$};	
	\node[shift={(0.1,0.25)}] at (-2.4,2.4) {\small $\beta$};
	\node[shift={(0.2,-0.45)}] at (-2.4,2.4) {\small $\delta$};
	\node[shift={(-0.2,-0.1)}] at (-2.4,2.4) {\small $\delta$};
	
	\node[shift={(0.2,0.2)}] at (-2.4,0) {\small $\gamma$};
	\node[shift={(-0.2,0.2)}] at (-2.4,0) {\small $\gamma$};
	\node[shift={(-0.2,-0.25)}] at (-2.4,0) {\small $\gamma$};
	\node[shift={(0.2,-0.25)}] at (-2.4,0) {\small $\gamma$};
	
	\node[shift={(0.55,0.25)}] at (-2.4,-2.4) {\small $\beta$};
	\node[shift={(0.2,0.45)}] at (-2.4,-2.4) {\small $\delta$};
	\node[shift={(-0.2,0.1)}] at (-2.4,-2.4) {\small $\delta$};	
	\node[shift={(0.1,-0.25)}] at (-2.4,-2.4) {\small $\beta$};
	
	\node[shift={(0.2,0.2)}] at (0,-2.4) {\small $\gamma$};
	\node[shift={(-0.2,0.2)}] at (0,-2.4) {\small $\gamma$};
	\node[shift={(-0.2,-0.25)}] at (0,-2.4) {\small $\gamma$};
	\node[shift={(0.2,-0.25)}] at (0,-2.4) {\small $\gamma$};
	
	\node[shift={(0.2,0.1)}] at (2.4,-2.4) {\small $\delta$};	
	\node[shift={(-0.2,0.45)}] at (2.4,-2.4) {\small $\delta$};
	\node[shift={(-0.55,0.25)}] at (2.4,-2.4) {\small $\beta$};
	\node[shift={(-0.1,-0.25)}] at (2.4,-2.4) {\small $\beta$};
	
	
	\node[shift={(0.2,0.2)}] at (3.6,0) {\small $\delta$};
	\node[shift={(-0.2,0.2)}] at (3.6,0) {\small $\beta$};
	\node[shift={(-0.2,-0.25)}] at (3.6,0) {\small $\beta$};
	\node[shift={(0.2,-0.25)}] at (3.6,0) {\small $\delta$};
	
	\node[shift={(0.2,0.2)}] at (3.6,3.6) {\small $\alpha$};
	\node[shift={(-0.5,-0.22)}] at (3.6,3.6) {\small $\alpha$};
	\node[shift={(-0.2,-0.5)}] at (3.6,3.6) {\small $\alpha$};
	
	\node[shift={(0.2,0.2)}] at (0,3.6) {\small $\beta$};
	\node[shift={(-0.2,0.2)}] at (0,3.6) {\small $\beta$};
	\node[shift={(-0.2,-0.25)}] at (0,3.6) {\small $\delta$};
	\node[shift={(0.2,-0.25)}] at (0,3.6) {\small $\delta$};
	
	\node[shift={(-0.2,0.2)}] at (-3.6,3.6) {\small $\alpha$};
	\node[shift={(0.5,-0.22)}] at (-3.6,3.6) {\small $\alpha$};
	\node[shift={(0.2,-0.5)}] at (-3.6,3.6) {\small $\alpha$};
	
	\node[shift={(0.2,0.2)}] at (-3.6,0) {\small $\beta$};
	\node[shift={(-0.2,0.2)}] at (-3.6,0) {\small $\delta$};
	\node[shift={(-0.2,-0.25)}] at (-3.6,0) {\small $\delta$};
	\node[shift={(0.2,-0.25)}] at (-3.6,0) {\small $\beta$};

	\node[shift={(-0.2,-0.2)}] at (-3.6,-3.6) {\small $\alpha$};
	\node[shift={(0.5,0.2)}] at (-3.6,-3.6) {\small $\alpha$};
	\node[shift={(0.2,0.5)}] at (-3.6,-3.6) {\small $\alpha$};
	
	\node[shift={(0.2,0.2)}] at (0,-3.6) {\small $\delta$};
	\node[shift={(-0.2,0.2)}] at (0,-3.6) {\small $\delta$};
	\node[shift={(-0.2,-0.25)}] at (0,-3.6) {\small $\beta$};
	\node[shift={(0.2,-0.25)}] at (0,-3.6) {\small $\beta$};
	
	\node[shift={(0.2,-0.2)}] at (3.6,-3.6) {\small $\alpha$};
	\node[shift={(-0.5,0.2)}] at (3.6,-3.6) {\small $\alpha$};
	\node[shift={(-0.2,0.5)}] at (3.6,-3.6) {\small $\alpha$};
	
	
	\node[shift={(4.2,4.2)}] at (0,0) {\small $\gamma$};
	\node[shift={(-4.2,4.2)}] at (0,0) {\small $\gamma$};
	\node[shift={(-4.2,-4.2)}] at (0,0) {\small $\gamma$};
	\node[shift={(4.2,-4.2)}] at (0,0) {\small $\gamma$};

	\node[inner sep=1,draw,shape=circle] at (0.59,0.58) {\small $1$};
	\node[inner sep=1,draw,shape=circle] at (-0.6,0.58) {\small $2$};
	\node[inner sep=1,draw,shape=circle] at (-0.6,-0.58) {\small $3$};
	\node[inner sep=1,draw,shape=circle] at (0.59,-0.58) {\small $4$};

	\node[inner sep=1,draw,shape=circle] at (0.8,1.8) {\small $5$};
	\node[inner sep=1,draw,shape=circle] at (-0.8,1.8) {\small $6$};	
	\node[inner sep=1,draw,shape=circle] at (-0.8,-1.8) {\small $7$};
	\node[inner sep=1,draw,shape=circle] at (0.8,-1.8) {\small $8$};

	\node[inner sep=1,draw,shape=circle] at (1.8,0.8) {\small $9$};
	\node[inner sep=1,draw,shape=circle] at (-1.8,0.8) {\small $10$};
	\node[inner sep=1,draw,shape=circle] at (-1.8,-0.8) {\small $11$};
	\node[inner sep=1,draw,shape=circle] at (1.8,-0.8) {\small $12$};

	\node[inner sep=1,draw,shape=circle] at (3,1.2) {\small $13$};
	\node[inner sep=1,draw,shape=circle] at (1.2,3) {\small $14$};
	\node[inner sep=1,draw,shape=circle] at (-1.2,3) {\small $15$};
	\node[inner sep=1,draw,shape=circle] at (-3,1.2) {\small $16$};

	\node[inner sep=1,draw,shape=circle] at (-3,-1.2) {\small $17$};
	\node[inner sep=1,draw,shape=circle] at (-1.2,-3) {\small $18$};
	\node[inner sep=1,draw,shape=circle] at (1.2,-3) {\small $19$};
	\node[inner sep=1,draw,shape=circle] at (3,-1.2) {\small $20$};

	\node[inner sep=1,draw,shape=circle] at (4.2,1.6) {\small $21$};
	\node[inner sep=1,draw,shape=circle] at (-4.2,1.6) {\small $22$};
	\node[inner sep=1,draw,shape=circle] at (-4.2,-1.6) {\small $23$};
	\node[inner sep=1,draw,shape=circle] at (4.2,-1.6) {\small $24$};

	\end{tikzpicture}
	\caption{Tiling of $f=24$, $\AVC\equiv\{\alpha^3, \beta^2\delta^2, \gamma^4\}$: Cube Subdivision}
	\label{Tf24CubeSD}
\end{figure}
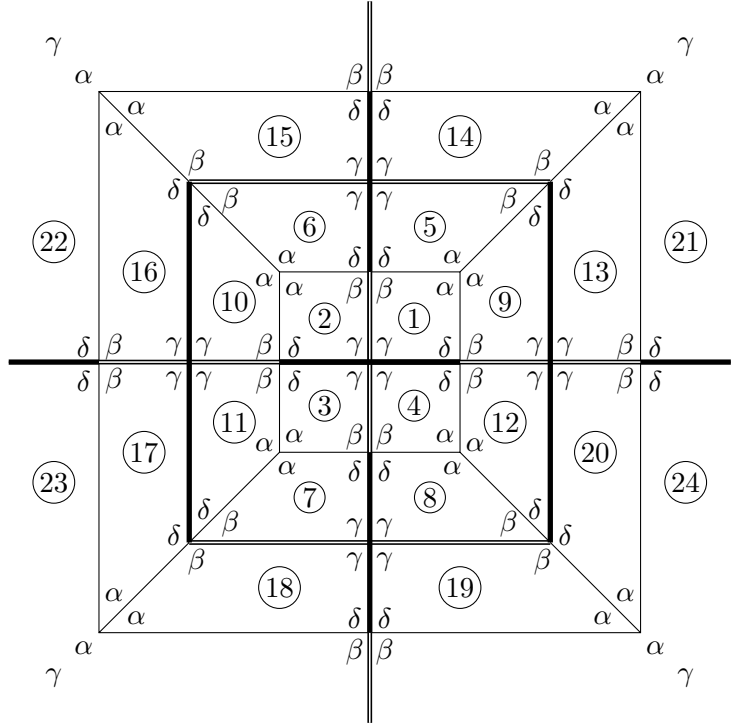

\begin{figure}[htp]
	\centering
	\begin{tikzpicture}[>=latex,scale=1]

\begin{scope}[xshift=0cm, yshift=0cm]


	\coordinate (O) at (0,0);
	\coordinate[shift={(30:2.4)}] (A1) at (O);
	\coordinate[shift={(150:2.4)}] (A2) at (O);
	\coordinate[shift={(270:2.4)}] (A3) at (O);	
	
	\coordinate[shift={(30:2.4)}] (A4) at (A3);
	\coordinate[shift={(150:2.4)}] (A5) at (A3);
	
	\coordinate[shift={(150:2.4)}] (A6) at (A1);
	
	\coordinate[shift={(30:1.2)}] (OA1) at (O);
	\coordinate[shift={(150:1.2)}] (OA2) at (O);
	\coordinate[shift={(150:1.2)}] (A1A6) at (A1);
	\coordinate[shift={(30:1.2)}] (A2A6) at (A2);
	
	\coordinate[shift={(270:1.2)}] (OA3) at (O);
	\coordinate[shift={(30:1.2)}] (A3A4) at (A3);
	\coordinate[shift={(270:1.2)}] (A1A4) at (A1);
	
	\coordinate[shift={(150:1.2)}] (A3A5) at (A3);
	\coordinate[shift={(270:1.2)}] (A2A5) at (A2);

	\draw 
	(O) -- (A1)
	(O) -- (A2)
	(O) -- (A3)
	(A1) -- (A6) -- (A2)
	(A3) -- (A4) -- (A1)
	(A3) -- (A5) -- (A2);
	
	\draw[double, line width=0.6] 
	(OA1) -- (A2A6)
	(OA3) -- (A1A4)
	(OA2) -- (A3A5);

	\draw[line width=2]
	(OA2) -- (A1A6)
	(OA1) -- (A3A4)	
	(OA3) -- (A2A5);

\end{scope}

	\begin{scope}[xshift=2.8 cm]

	\coordinate (A1) at (0,0);
	\coordinate (A2) at (0.6*4.5,0.2*4.5);
	\coordinate (A3) at (1*4.5,0);
	\coordinate (A4) at (0.4*4.5,-0.2*4.5);
	\coordinate (B1) at (0.5*4.5,0.5*4.5);
	\coordinate (B2) at (0.5*4.5,-0.5*4.5);
	\coordinate (A1B1) at (0.25*4.5,0.25*4.5);
	\coordinate (A1A4) at (0.2*4.5,-0.1*4.5);
	\coordinate (A4B1) at (0.45*4.5,0.15*4.5);
	\coordinate (A1A4B1) at (0.3*4.5,0.1*4.5);
	\coordinate (A3B1) at (0.75*4.5,0.25*4.5);
	\coordinate (A3A4) at (0.7*4.5,-0.1*4.5);
	\coordinate (A3A4B1) at (0.6333*4.5,0.1*4.5);
	\coordinate (A3B2) at (0.75*4.5,-0.25*4.5);
	\coordinate (A4B2) at (0.45*4.5,-0.35*4.5);
	\coordinate (A3A4B2) at (0.6333*4.5,-0.2333*4.5);
	\coordinate (A1B2) at (0.25*4.5,-0.25*4.5);
	\coordinate (A1A4B2) at (0.3*4.5,-0.2333*4.5);

	\draw (A1) -- (A4) -- (B1);
	\draw (A1) -- (A4) -- (B2);
	\draw (A3) -- (A4) -- (B1);
	\draw (A3) -- (A4) -- (B2);
	\draw (B1) -- (A1) -- (B2) -- (A3) --cycle;

	\draw 
	(A1A4B1) -- (A1B1)
	(A1A4B1) -- (A1A4)
	(A1A4B1) -- (A4B1);
	
	\draw
	(A3A4B1) -- (A3B1)
	(A3A4B1) -- (A3A4)
	(A3A4B1) -- (A4B1);
	
	\draw
	(A3A4B2) -- (A3B2)
	(A3A4B2) -- (A4B2)
	(A3A4B2) -- (A3A4);

	\draw
	(A1A4B2) -- (A1B2)
	(A1A4B2) -- (A4B2)
	(A1A4B2) -- (A1A4);

	\draw[double, line width=0.6]
	(B1) -- (A1B1)
	(B1) -- (A3B1)
	(A1) -- (A1A4)
	(A3) -- (A3A4)
	(B2) -- (A1B2)
	(B2) -- (A3B2)
	(A4) -- (A4B1)
	(A4) -- (A4B2);

	\draw[line width=2]
	(B1) -- (A4B1)
	(A4) -- (A1A4)
	(A1) -- (A1B1)
	(A3) -- (A3B1)
	(A4) -- (A3A4)
	(B2) -- (A4B2)
	(A1) -- (A1B2)
	(A3) -- (A3B2);
	\end{scope}
	
	\begin{scope}[xshift = 9.4 cm, yshift=1cm]
		
	\draw 
	(0,0) -- (0,-3)
	(1.4,1.4) -- (1.4,-1.6)
	(-1.4,1.4) -- (-1.4,-1.6)
	(-1.4,1.4) -- (0,0) -- (1.4,1.4)--cycle
	(-1.4,-1.6) -- (0,-3) -- (1.4,-1.6);
	
	\draw
	(0,0.96) -- (0,1.4)
	(0,0.96) -- (-0.7,0.7)
	(0,0.96) -- (0.7,0.7);
	
	\draw
	(-0.7,0.7) -- (-1.4,-1.6)
	(0,0) -- (-0.7,-2.3)
	(0.7,0.7) -- (0,-3)
	(1.4,1.4) -- (0.7,-2.3);
	
	\draw 
	(-1.4,0.4) -- (0,-2)
	(0,-1) -- (1.4,-0.6);

	\draw[double, line width=0.6]
	(0,1.4) -- (-1.4,1.4)
	(-0.7,0.7) -- (0,0)
	(0.7,0.7) -- (1.4,1.4)
	(-1.4,1.4) -- (-1.4,0.4)
	(-1.4,-0.6) -- (-1.4,-1.6)
	(0,0) -- (0,-1)
	(0,-2) -- (0,-3)
	(1.4,1.4) -- (1.4,0.4)
	(1.4,-0.6) -- (1.4,-1.6)
	(-1.4,-1.6) -- (-0.7,-2.3)
	(0,-3) -- (0.7,-2.3);

	\draw[line width=2]
	(-1.4,1.4) -- (-0.7,0.7)
	(0,0) -- (0.7,0.7)
	(0,1.4) -- (1.4,1.4)
	(-0.7,-2.3) -- (0,-3)
	(0.7,-2.3) -- (1.4,-1.6)
	(-1.4,-1.6) -- (-1,-0.2857)
	(-0.4,-1.314) -- (0,0)
	(0.4,-0.8857) -- (0,-3)
	(1,-0.7143) -- (1.4,1.4);
	
	\end{scope}
	
	\end{tikzpicture}
	\caption{Quadrilateral Subdivisions of the Cube, the Octahedron and the Triangular Prism.}
	\label{QuadSubd}
\end{figure}
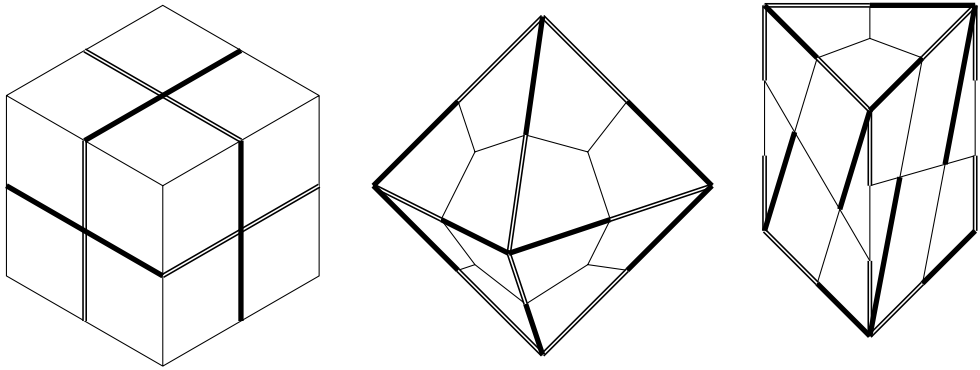

From the $\AVC$, we already know the tiling which is given by Figure \ref{Tf24a3-1}. The discussion here is to give unequivocal perspective to show that the tiling is indeed a quadrangle subdivision of the cube. We have $\alpha\cdots = \alpha^3, \beta\cdots = \delta\cdots = \beta^2\delta^2, \gamma\cdots = \gamma^4$. At $\gamma^4$, the unique AAD determines tiles $T_1, T_2, T_3, T_4$ in Figure \ref{Tf24CubeSD}. As $\beta\cdots = \beta^2\delta^2$, we get $\beta_1\beta_2\cdots = \beta_3\beta_4\cdots = \beta^2\delta^2$, so we determine $T_5, T_6, T_7, T_8$. Meanwhile, $\delta_1\delta_4\cdots = \delta_2\delta_3\cdots = \beta^2\delta^2$, so we determine $T_9, T_{10}, T_{11}, T_{12}$. Then by $\beta_5\cdots = \beta_6\cdots = \beta_7\cdots = \beta_8\cdots = \beta^2\delta^2$, so we determine $T_{13}, T_{14}, T_{15}, T_{16}, T_{17}, T_{18}, T_{19}, T_{20}$. Lastly $\beta_{13}\cdots = \beta_{16}\cdots = \beta_{17}\cdots = \beta_{20}\cdots = \beta^2\delta^2$, so we determine $T_{21}, T_{22}, T_{23}, T_{24}$. Hence we get a tiling with $f=24$. It is indeed the quadrilateral subdivision of the cube (or the octahedron) as indicated in the first picture (second picture) in Figure \ref{QuadSubd}.

\begin{figure}[htp]
	\centering
	\begin{tikzpicture}[>=latex,scale=1]
	
	\draw 
	(0,1.2) -- (0,-3.6);
	
	\draw 
	(1.2,1.2) -- (1.2,-2.4)
	(2.4,0) -- (2.4,-2.4);
	
	\draw
	(0,1.2) -- (1.2,1.2)
	(0,0) -- (1.2,0)
	(0,-1.2) -- (1.2,-1.2)
	(0,-2.4) -- (1.2,-2.4);

	\draw 
	(2.4,-2.4) -- (1.2,0)
	(2.4,-2.4) -- (1.2,-2.4)
	(2.4,-2.4) -- (0,-3.6);
	
	\draw 
	(3.6,-3) -- (2.4,0)
	(3.6,-3) -- (0,-3.6);
	
	\draw 
	(2.4,0) -- (1.2,1.2);
	
	\begin{scope}[xscale=-1, yscale=1]
	\draw 
	(1.2,1.2) -- (1.2,-2.4)
	(2.4,0) -- (2.4,-2.4);
	
	\draw
	(0,1.2) -- (1.2,1.2)
	(0,0) -- (1.2,0)
	(0,-1.2) -- (1.2,-1.2)
	(0,-2.4) -- (1.2,-2.4);

	\draw 
	(2.4,-2.4) -- (1.2,0)
	(2.4,-2.4) -- (1.2,-2.4)
	(2.4,-2.4) -- (0,-3.6);
	
	\draw 
	(3.6,-3) -- (2.4,0)
	(3.6,-3) -- (0,-3.6);
	
	\draw 
	(2.4,0) -- (1.2,1.2);
	\end{scope}
	
	\draw[double, line width=0.6]
	(0,0) -- (0,1.2)
	(-1.2,-1.2) -- (1.2,-1.2)
	(2.4,0) -- (2.4,-2.4) -- (1.2,-2.4)
	(-2.4,0) -- (-2.4,-2.4) -- (-1.2,-2.4);
	
	\draw[line width=2]
	(-1.2,1.2) -- (1.2,1.2)
	(0,0) -- (0,-2.4)
	(1.2,0) -- (2.4,-2.4) -- (0,-3.6)
	(-1.2,0) -- (-2.4,-2.4) -- (0,-3.6);

	
	
	\node[shift={(0.2,-0.25)}] at (0,1.2) {\small $\gamma$};
	
	\node[shift={(0.2,0.2)}] at (0,0) {\small $\beta$};
	\node[shift={(0.2,-0.25)}] at (0,0) {\small $\delta$};
	
	\node[shift={(0.2,0.2)}] at (0,-1.2) {\small $\gamma$};
	\node[shift={(0.2,-0.25)}] at (0,-1.2) {\small $\gamma$};
	
	\node[shift={(0.2,0.2)}] at (0,-2.4) {\small $\delta$};
	\node[shift={(0.2,-0.22)}] at (0,-2.4) {\small $\alpha$};
	
	\node[shift={(0.2,0.35)}] at (0,-3.6) {\small $\delta$};	
	\node[shift={(1.2,0.38)}] at (0,-3.6) {\small $\delta$};
	
	\node[shift={(-0.3,0.2)}] at (3.6,-3) {\small $\alpha$};	
	
	\node[shift={(-0.2,-0.15)}] at (2.4,0) {\small $\beta$};
	\node[shift={(0.2,-1)}] at (2.4,0) {\small $\beta$};
	
	\node[shift={(-0.2,-0.25)}] at (1.2,1.2) {\small $\delta$};
	\node[shift={(0.2,-0.48)}] at (1.2,1.2) {\small $\alpha$};
	
	\node[shift={(0.2,0.1)}] at (1.2,0) {\small $\delta$};
	\node[shift={(-0.2,0.2)}] at (1.2,0) {\small $\alpha$};
	\node[shift={(-0.2,-0.22)}] at (1.2,0) {\small $\alpha$};
	\node[shift={(0.2,-0.85)}] at (1.2,0) {\small $\delta$};
	
	\node[shift={(0.2,0)}] at (1.2,-1.2) {\small $\alpha$};
	\node[shift={(-0.2,0.2)}] at (1.2,-1.2) {\small $\beta$};
	\node[shift={(-0.2,-0.25)}] at (1.2,-1.2) {\small $\beta$};
	
	\node[shift={(0.2,0.2)}] at (1.2,-2.4) {\small $\beta$};
	\node[shift={(-0.2,0.2)}] at (1.2,-2.4) {\small $\alpha$};
	\node[shift={(0,-0.25)}] at (1.2,-2.4) {\small $\beta$};
	
	\node[shift={(0.2,0)}] at (2.4,-2.4) {\small $\gamma$};
	\node[shift={(-0.2,0.7)}] at (2.4,-2.4) {\small $\gamma$};
	\node[shift={(-0.4,0.2)}] at (2.4,-2.4) {\small $\gamma$};
	\node[shift={(-0.9,-0.23)}] at (2.4,-2.4) {\small $\gamma$};
	
	\begin{scope}[xscale=-1, yscale=1]	
	
	\node[shift={(-0.2,-0.25)}] at (0,1.2) {\small $\gamma$};
	
	\node[shift={(-0.2,0.2)}] at (0,0) {\small $\beta$};
	\node[shift={(-0.2,-0.25)}] at (0,0) {\small $\delta$};
	
	\node[shift={(-0.2,0.2)}] at (0,-1.2) {\small $\gamma$};
	\node[shift={(-0.2,-0.25)}] at (0,-1.2) {\small $\gamma$};
	
	\node[shift={(-0.2,0.2)}] at (0,-2.4) {\small $\delta$};
	\node[shift={(-0.2,-0.22)}] at (0,-2.4) {\small $\alpha$};
	
	\node[shift={(-0.2,0.35)}] at (0,-3.6) {\small $\delta$};	
	\node[shift={(-1.2,0.38)}] at (0,-3.6) {\small $\delta$};
	
	\node[shift={(0.3,0.2)}] at (3.6,-3) {\small $\alpha$};	
	
	\node[shift={(0.2,-0.15)}] at (2.4,0) {\small $\beta$};
	\node[shift={(-0.2,-1)}] at (2.4,0) {\small $\beta$};
	
	\node[shift={(0.2,-0.25)}] at (1.2,1.2) {\small $\delta$};
	\node[shift={(-0.2,-0.48)}] at (1.2,1.2) {\small $\alpha$};
	
	\node[shift={(-0.2,0.1)}] at (1.2,0) {\small $\delta$};
	\node[shift={(0.2,0.2)}] at (1.2,0) {\small $\alpha$};
	\node[shift={(0.2,-0.22)}] at (1.2,0) {\small $\alpha$};
	\node[shift={(-0.2,-0.85)}] at (1.2,0) {\small $\delta$};
	
	\node[shift={(-0.2,0)}] at (1.2,-1.2) {\small $\alpha$};
	\node[shift={(0.2,0.2)}] at (1.2,-1.2) {\small $\beta$};
	\node[shift={(0.2,-0.25)}] at (1.2,-1.2) {\small $\beta$};
	
	\node[shift={(-0.2,0.2)}] at (1.2,-2.4) {\small $\beta$};
	\node[shift={(0.2,0.2)}] at (1.2,-2.4) {\small $\alpha$};
	\node[shift={(0,-0.25)}] at (1.2,-2.4) {\small $\beta$};
	
	\node[shift={(-0.2,0)}] at (2.4,-2.4) {\small $\gamma$};
	\node[shift={(0.17,0.7)}] at (2.4,-2.4) {\small $\gamma$};
	\node[shift={(0.4,0.2)}] at (2.4,-2.4) {\small $\gamma$};
	\node[shift={(0.9,-0.23)}] at (2.4,-2.4) {\small $\gamma$};
	\end{scope}
	
	
	\node[inner sep=0,draw,shape=circle, shift={(-0.6,-0.6)}] at (0,1.2) {\small $1$};
	\node[inner sep=0,draw,shape=circle, shift={(0.6,-0.6)}] at (0,1.2) {\small $2$};

	\node[inner sep=0,draw,shape=circle, shift={(0.6,0.6)}] at (0,-1.2) {\small $3$};
	\node[inner sep=0,draw,shape=circle, shift={(-0.6,0.6)}] at (0,-1.2) {\small $4$};
	\node[inner sep=0,draw,shape=circle, shift={(-0.6,-0.6)}] at (0,-1.2) {\small $5$};
	\node[inner sep=0,draw,shape=circle, shift={(0.6,-0.6)}] at (0,-1.2) {\small $6$};
	
	\node[inner sep=0,draw,shape=circle, shift={(0.75,0.6)}] at (-2.4,-2.4) {\small $8$};
	\node[inner sep=0,draw,shape=circle, shift={(0.5,1.8)}] at (-2.4,-2.4) {\small $7$};
	\node[inner sep=0,draw,shape=circle, shift={(-0.5,-0.3)}] at (-2.4,-2.4) {\small $14$};
	\node[inner sep=0,draw,shape=circle, shift={(1.8,-0.45)}] at (-2.4,-2.4) {\small $13$};
	
	\node[inner sep=0,draw,shape=circle, shift={(-0.75,0.6)}] at (2.4,-2.4) {\small $10$};
	\node[inner sep=0,draw,shape=circle, shift={(-0.5,1.8)}] at (2.4,-2.4) {\small $9$};
	\node[inner sep=0,draw,shape=circle, shift={(0.5,-0.3)}] at (2.4,-2.4) {\small $11$};
	\node[inner sep=0,draw,shape=circle, shift={(-1.8,-0.45)}] at (2.4,-2.4) {\small $12$};
	
	\end{tikzpicture}
	\caption{AAD of $\beta^2\delta^2$.}
	\label{AADbe2de2}
\end{figure}
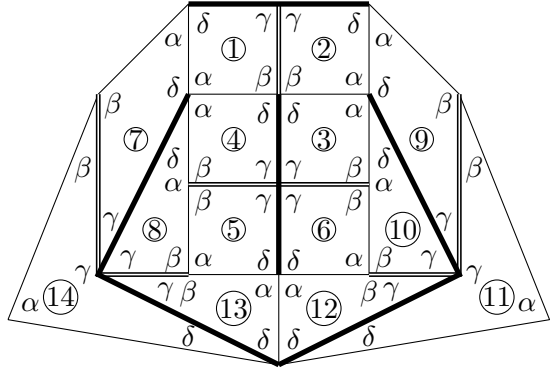

\begin{figure}[htp]
	\centering
	\begin{tikzpicture}[>=latex]

	\begin{scope}[xscale=-1]

	\draw[] 
	(0,0) -- (0,1.2) -- (-1.6, 1.8) 
	(0,0) -- (1.6, -0.6)
	(-1.6, 1.8) -- (1.6, 1.8) -- (2.8, 1.8) -- (2.8, 0.6)
	(-1.6, -0.6) -- (1.6, -0.6)
	(1.6, 1.8) -- (1.6, 3.0)
	(-1.6, -0.6)  -- (-1.6, -1.8)
	(-1.6, 1.8) -- (-2.8, 1.8) -- (-2.8, 3.0)
 	(-2.8, 0.6) -- (-2.8, 1.8)
	(2.8, 0.6) -- (2.8, -0.6) -- (2.8, -1.8)
	(1.6, -0.6) -- (2.8, -0.6)
	(-1.6, -0.6) -- (-2.8, -0.6) -- (-2.8, 0.6)
	(2.8, -1.8) -- (5.8, 3.8)
	(-2.8, 3.0) -- (-5.8, -2.6) 
	(1.6, 3.0) -- (1.6, 3.8) -- (5.8,3.8)
	(-1.6, -1.8) -- (-1.6, -2.6) -- (-5.8, -2.6)
	(1.6, 3.8) -- (0.6, 4.2)
	(-1.6, -2.6) -- (-0.6, -3.0);

	\draw[double, line width=0.6]
	(0,1.2) -- (1.6, 0.6) -- (2.8, 0.6)
	(0,0) -- (-1.6, 0.6) -- (-2.8, 0.6)
	(-1.6, 3.0) -- (1.6, 3.0)
	(-1.6, -1.8) -- (1.6, -1.8)
	(1.6, -1.8) -- (2.8, -1.8)
	(-1.6, 3.0) -- (-2.8, 3.0)
	(2.8, 1.8) -- (4.6, 3.0) -- (5.8, 3.8)
	(-2.8, -0.6) -- (-4.6,-1.8) -- (-5.8, -2.6);

	\draw[line width=2]
	(-1.6, 1.8) -- (-1.6, 0.6)
	(1.6, -0.6) -- (1.6, 0.6)
	(1.6, 0.6) -- (1.6, 1.8)
	(-1.6, -0.6) -- (-1.6,0.6)
	(-1.6, 1.8) -- (-1.6, 3.0)
	(1.6, -0.6) -- (1.6, -1.8)
	(1.6, 3.0) -- (4.6, 3.0)
	(-1.6, -1.8) -- (-4.6, -1.8)
	(2.8, -0.6) -- (4.6, 3.0)
	(-2.8, 1.8) -- (-4.6, -1.8)
	(-1.6, 3.0) -- (-1.6, 3.6)
	(1.6, -1.8) -- (1.6, -2.4)
	(1.6, 3.8) -- (1.6, 4.4)
	(-1.6, -2.6) -- (-1.6, -3.2);

	\draw[dotted, line width=2]
	(-1.6, -3.2) -- (-1.6, -3.8)
	(1.6, -2.4) -- (1.6, -3.0)
	(-1.6, 3.6) -- (-1.6, 4.2)
	(1.6, 4.4) -- (1.6, 5.0);

	\draw[dotted, line width=0.75]
	(0.6, 4.2) -- (0.2, 4.4)
	(-0.6, -3.0) -- (-0.2, -3.2);

	\node at (1.4, 3.7) {\small $\alpha$}; 
	\node at (1.4, 3.2) {\small $\beta$}; 
	\node at (-1.4, 3.2) {\small $\gamma$}; 
	\node at (-1.4, -3.0) {\small $\delta$}; 

	\node at (1.4, 1.95) {\small $\alpha$}; 
	\node at (1.4, 2.75) {\small $\beta$}; 
	\node at (-1.4, 2.75) {\small $\gamma$}; 
	\node at (-1.4, 2.0) {\small $\delta$}; 
	
	\node at (1.8, 1.95) {\small $\alpha$}; 
	\node at (2.7, 1.95) {\small $\beta$}; 
	\node at (3.9, 2.75) {\small $\gamma$}; 
	\node at (1.8, 2.75) {\small $\delta$}; 
	
	\node at (1.8, 3.65) {\small $\alpha$}; 
	\node at (5.2, 3.6) {\small $\beta$}; 
	\node at (4.6, 3.2) {\small $\gamma$}; 
	\node at (1.8, 3.2) {\small $\delta$}; 
	
	\node at (2.92, -1.3) {\small $\alpha$}; 
	\node at (5.3, 3.2) {\small $\beta$}; 
	\node at (3.0, -0.6) {\small $\delta$}; 
	\node at (4.75, 2.8) {\small $\gamma$}; 

	\node at (3.0, 0.6) {\small $\alpha$}; 
	\node at (2.95, 1.6) {\small $\beta$}; 
	\node at (4.1, 2.4) {\small $\gamma$}; 
	\node at (2.95, 0.1) {\small $\delta$}; 

	\node at (-1.4, -2.5) {\small $\alpha$}; 
	\node at (-1.4, -2.1) {\small $\beta$}; 
	\node at (1.4, -2.0) {\small $\gamma$}; 
	\node at (1.4, 4.1) {\small $\delta$}; 

	\node at (5.9, 4.0) {\small $\alpha$}; 
	\node at (2.8, -2.05) {\small $\beta$}; 
	\node at (1.8, -2.0) {\small $\gamma$}; 
	\node at (1.8, 4.0) {\small $\delta$}; 

	\node at (-0.8, 1.65) {\small $\alpha$}; 
	\node at (0.0, 1.4) {\small $\beta$}; 
	\node at (1.4, 0.9) {\small $\gamma$}; 
	\node at (1.4, 1.6) {\small $\delta$}; 

	\node at (2.6, 1.6) {\small $\alpha$}; 
	\node at (2.6, 0.8) {\small $\beta$}; 
	\node at (1.8, 0.8) {\small $\gamma$}; 
	\node at (1.8, 1.6) {\small $\delta$}; 

	\node at (0.2, 0.1) {\small $\alpha$}; 
	\node at (0.2, 0.8) {\small $\beta$}; 
	\node at (1.4, 0.4) {\small $\gamma$}; 
	\node at (1.4, -0.3) {\small $\delta$}; 

	\node at (2.6, -0.4) {\small $\alpha$}; 
	\node at (2.6, 0.35) {\small $\beta$}; 
	\node at (1.8, 0.4) {\small $\gamma$}; 
	\node at (1.8, -0.4) {\small $\delta$}; 

	\node at (2.6, -0.8) {\small $\alpha$}; 
	\node at (2.6, -1.6) {\small $\beta$}; 
	\node at (1.8, -1.6) {\small $\gamma$}; 
	\node at (1.8, -0.8) {\small $\delta$}; 

	\node at (-1.4, -0.8) {\small $\alpha$}; 
	\node at (-1.4, -1.6) {\small $\beta$}; 
	\node at (1.4, -1.6) {\small $\gamma$}; 
	\node at (1.4, -0.8) {\small $\delta$}; 

	\node at (0.8, -0.45) {\small $\alpha$}; 
	\node at (0.0, -0.3) {\small $\beta$}; 
	\node at (-1.4, 0.3) {\small $\gamma$}; 
	\node at (-1.4, -0.4) {\small $\delta$}; 

	\node at (-0.2, 1.0) {\small $\alpha$}; 
	\node at (-0.2, 0.3) {\small $\beta$}; 
	\node at (-1.4, 0.7) {\small $\gamma$}; 
	\node at (-1.4, 1.5) {\small $\delta$}; 

	\node at (-5.9, -2.8) {\small $\alpha$}; 
	\node at (-2.8, 3.2) {\small $\beta$}; 
	\node at (-1.8, 3.2) {\small $\gamma$}; 
	\node at (-1.8, -2.8) {\small $\delta$}; 

	\node at (-2.6, 1.6) {\small $\alpha$}; 
	\node at (-2.6, 0.8) {\small $\beta$}; 
	\node at (-1.8, 0.8) {\small $\gamma$}; 
	\node at (-1.8, 1.6) {\small $\delta$}; 

	\node at (-2.6, 2.0) {\small $\alpha$}; 
	\node at (-2.6, 2.75) {\small $\beta$}; 
	\node at (-1.8, 2.75) {\small $\gamma$}; 
	\node at (-1.8, 2.0) {\small $\delta$}; 

	\node at (-2.6, -0.4) {\small $\alpha$}; 
	\node at (-2.6, 0.3) {\small $\beta$}; 
	\node at (-1.8, 0.3) {\small $\gamma$}; 
	\node at (-1.8, -0.4) {\small $\delta$}; 

	\node at (-2.95, 2.4) {\small $\alpha$}; 
	\node at (-5.3, -2.0) {\small $\beta$}; 
	\node at (-4.8, -1.6) {\small $\gamma$}; 
	\node at (-3.0, 1.8) {\small $\delta$}; 

	\node at (-2.95, 0.6) {\small $\alpha$}; 
	\node at (-3.0, -0.5) {\small $\beta$}; 
	\node at (-4.0, -1.1) {\small $\gamma$}; 
	\node at (-2.9, 1.2) {\small $\delta$}; 

	\node at (-1.8, -0.8) {\small $\alpha$}; 
	\node at (-2.75, -0.85) {\small $\beta$}; 
	\node at (-3.8, -1.6) {\small $\gamma$}; 
	\node at (-1.8, -1.6) {\small $\delta$}; 

	\node at (-1.8, -2.4) {\small $\alpha$}; 
	\node at (-5.0, -2.4) {\small $\beta$}; 
	\node at (-4.5, -2.0) {\small $\gamma$}; 
	\node at (-1.8, -2.0) {\small $\delta$}; 

	\node[inner sep=1,draw,shape=circle] at (0.0, 3.6) {\small $1$}; 
	\node[inner sep=1,draw,shape=circle] at (0.0, 2.4) {\small $2$}; 
	\node[inner sep=1,draw,shape=circle] at (2.4, 2.4) {\small $3$}; 
	\node[inner sep=1,draw,shape=circle] at (3.2, 3.4) {\small $4$}; 
	\node[inner sep=1,draw,shape=circle] at (4.05, 1.2) {\small $5$}; 
	\node[inner sep=1,draw,shape=circle] at (3.2, 1.2) {\small $6$}; 
	\node[inner sep=1,draw,shape=circle] at (0.0, -2.4) {\small $7$}; 
	\node[inner sep=1,draw,shape=circle] at (2.2, -2.4) {\small $8$}; 
	
	\node[inner sep=1,draw,shape=circle] at (0.8, 1.4) {\small $9$}; 
	\node[inner sep=1,draw,shape=circle] at (2.2, 1.2) {\small $10$}; 
	\node[inner sep=1,draw,shape=circle] at (0.8, 0.3) {\small $11$}; 
	\node[inner sep=1,draw,shape=circle] at (2.2, 0.0) {\small $12$}; 
	\node[inner sep=1,draw,shape=circle] at (2.2, -1.2) {\small $13$}; 
	\node[inner sep=1,draw,shape=circle] at (0.0, -1.2) {\small $14$}; 
	\node[inner sep=1,draw,shape=circle] at (-0.8, -0.2) {\small $15$}; 
	\node[inner sep=1,draw,shape=circle] at (-0.8, 0.9) {\small $16$}; 

	\node[inner sep=1,draw,shape=circle] at (-2.2, 0.0) {\small $17$}; 
	\node[inner sep=1,draw,shape=circle] at (-2.2, 1.2) {\small $18$}; 
	\node[inner sep=1,draw,shape=circle] at (-2.2, 2.4) {\small $19$}; 
	\node[inner sep=1,draw,shape=circle] at (-2.2, 3.6) {\small $20$}; 

	\node[inner sep=0.5,draw,shape=circle] at (-4.05, 0.0) {\small $21$}; 
	\node[inner sep=1,draw,shape=circle] at (-3.2, 0.0) {\small $22$}; 

	\node[inner sep=1,draw,shape=circle] at (-2.4, -1.2) {\small $23$}; 
	\node[inner sep=1,draw,shape=circle] at (-3.2, -2.2) {\small $24$}; 

	\end{scope}

	\end{tikzpicture}
	\caption{Tiling of $f=24$, $\AVC \equiv \{ \alpha\beta^2, \alpha^2\delta^2, \beta^2\delta^2, \gamma^4, \alpha\delta^4 \}$.}
	\label{Tf24b2d2}
\end{figure}

Next we may assume $\beta^2\delta^2$ is not a vertex and we have $\AVC = \{ \alpha\beta^2, \alpha^2\delta^2, \gamma^4, \alpha\delta^4 \}$. If $\alpha^2\delta^2$ is a vertex, then in $\alpha\beta^2, \alpha^2\delta^2, \gamma^4$, we have $\# \alpha > \# \delta$ so $\alpha\delta^4$ must be a vertex. Meanwhile, if $\alpha\delta^4$ is a vertex, the AAD implies $\alpha^2\cdots$ must be a vertex, so $\alpha^2\delta^2$. Then both $\alpha^2\delta^2, \alpha\delta^4$ must be vertices. Starting at $\alpha\delta^4$, by $\alpha\beta\cdots = \alpha\beta^2, \gamma\cdots = \gamma^4$ we determine $T_1, \cdots T_{10}$ in Figure \ref{Tf24a2d2ad4}. By $\beta^2\cdots = \alpha\beta^2$, $\beta_2\beta_3\cdots = \alpha\beta^2$ so we also determine $T_{11}$. Then we have $\alpha_2\cdots = \alpha\delta^3\cdots = \alpha\delta^4$. So $T_{12}$ is determined. By $\gamma_{11}\cdots = \gamma^4$, we determine $T_{13}, T_{14}$. By the same token, we also determine $T_{15}, T_{16}$. By $\alpha_{10}\alpha_{12}\cdots = \alpha^2\delta^2$, we determine $T_{17}, T_{18}$. Then by $\alpha_{9}\beta_{18}\cdots = \alpha\beta^2$, $T_{19}$ is also determined. By the same token, $T_{20}, T_{21}, T_{22}$ are determined. By $\alpha_{14}\beta_{17}\cdots = \alpha_{15}\beta_{20}\cdots = \alpha\beta^2$, $T_{23}, T_{24}$ are determined. Hence we obtain the tiling for $f=24, \AVC = \{ \alpha\beta^2, \alpha^2\delta^2, \gamma^4, \alpha\delta^4 \}$.

\begin{figure}[htp]
	\centering
	\begin{tikzpicture}[>=latex,scale=0.95]

	\begin{scope}[xscale=-1] 

	\draw[] 
	(0, -1.2) -- (0,0) -- (1.2,0) -- (1.2, 1.2)
	(1.2, 1.2) -- (1.2, 2.4) -- (0, 2.4)
	(0, -1.2) -- (0, -2.4) -- (1.2, -2.4)
	(0,0) -- (-1.2,0.6) -- (-1.2, 1.8)
	(1.2,0) -- (2.4, -0.6) -- (2.4, -1.8)
	(0, 2.4) -- (-1.2, 3.0) -- (-1.2, 1.8)
	(1.2, -2.4) -- (2.4, -3.0) -- (2.4, -1.8) 
	(0, 2.4) -- (2.4, 4.2)
	(1.2, -2.4) -- (-1.2, -4.2)
	(2.4, -0.6) -- (3.6, 1.8)
	(-1.2, 0.6) -- (-2.4, -1.8)
	(2.4, 4.2) -- (4.8, 4.2) -- (3.6, 1.8)
	(-1.2, -4.2) -- (-3.6, -4.2) -- (-2.4, -1.8)
	(2.4, 4.2) -- (2.4, 4.8) 
	(-1.2, -4.2) -- (-1.2, -4.8);

	\draw[double, line width=0.6]
	(0,1.2) -- (1.2,1.2)
	(0,-1.2) -- (1.2,-1.2)
	(1.2, 2.4) -- (2.4, 2.4) -- (3.6, 1.8)
	(0,-2.4) -- (-1.2, -2.4) -- (-2.4, -1.8) 
	(0,1.2) -- (-1.2,1.8) 
	(1.2,-1.2) -- (2.4, -1.8)
	(2.4, -3.0) -- (4.8, -4.2)
	(-1.2, 3.0) -- (-3.6, 4.2)
	(-3.6, -4.2) -- (-3.6, 4.2)
	(4.8, 4.2) -- (4.8, -4.2);

	\draw[line width=2]
	(0,0) -- (0,1.2) -- (0,2.4)
	(1.2,0) -- (1.2, -1.2) -- (1.2, -2.4)
	(0,0)  -- (-1.2, -2.4)
	(1.2,0) -- (2.4, 2.4) 
	(2.4, 2.4) -- (2.4, 4.2)
	(-1.2, -2.4) -- (-1.2, -4.2)
	(2.4, -0.6) -- (4.8, -4.2)
	(-1.2, 0.6) -- (-3.6, 4.2)
	(-3.6, 4.2) -- (2.4, 4.2)
	(-1.2, -4.2) -- (4.8, -4.2);

	\draw[dotted, line width=0.6]
	(2.4, 4.8) -- (2.4, 5.2)
	(-1.2, -4.8) -- (-1.2, -5.2);

	\node at (0.2, -0.2) {\small $\alpha$}; 
	\node at (0.2, -0.95) {\small $\beta$}; 
	\node at (1, -0.95) {\small $\gamma$}; 
	\node at (1, -0.2) {\small $\delta$}; 

	\node at (0.2, 0.2) {\small $\delta$}; 
	\node at (1, 0.2) {\small $\alpha$}; 
	\node at (1, 0.9) {\small $\beta$}; 
	\node at (0.2, 0.95) {\small $\gamma$}; 

	\node at (1, 2.2) {\small $\alpha$}; 
	\node at (1, 1.4) {\small $\beta$}; 
	\node at (0.2, 1.4) {\small $\gamma$}; 
	\node at (0.2, 2.2) {\small $\delta$}; 

	\node at (-1.0, 2.7) {\small $\alpha$}; 
	\node at (-1.0, 2.0) {\small $\beta$}; 
	\node at (-0.2, 1.5) {\small $\gamma$}; 
	\node at (-0.2, 2.3) {\small $\delta$}; 

	\node at (-1.0, 0.7) {\small $\alpha$}; 
	\node at (-1.0, 1.4) {\small $\beta$}; 
	\node at (-0.2, 1) {\small $\gamma$}; 
	\node at (-0.2, 0.3) {\small $\delta$}; 

	\node at (-1.1, 0.3) {\small $\alpha$}; 
	\node at (-2.1, -1.6) {\small $\beta$}; 
	\node at (-1.25, -2.1) {\small $\gamma$}; 
	\node at (-0.25, -0.1) {\small $\delta$}; 

	\node at (-0.2, -1.2) {\small $\alpha$}; 
	\node at (-0.15, -2.2) {\small $\beta$}; 
	\node at (-0.9, -2.2) {\small $\gamma$}; 
	\node at (-0.15, -0.6) {\small $\delta$}; 

	\node at (0.2, -2.2) {\small $\alpha$}; 
	\node at (1, -1.45) {\small $\gamma$}; 
	\node at (0.2, -1.45) {\small $\beta$}; 
	\node at (1, -2.2) {\small $\delta$}; 

	\node at (2.25, -2.75) {\small $\alpha$}; 
	\node at (2.2, -2.0) {\small $\beta$}; 
	\node at (1.4, -1.55) {\small $\gamma$}; 
	\node at (1.4, -2.2) {\small $\delta$}; 

	\node at (2.2, -0.7) {\small $\alpha$}; 
	\node at (2.2, -1.4) {\small $\beta$}; 
	\node at (1.4, -1.1) {\small $\gamma$}; 
	\node at (1.4, -0.3) {\small $\delta$}; 

	\node at (1.35, 1.2) {\small $\alpha$}; 
	\node at (1.4, 2.15) {\small $\beta$}; 
	\node at (2.1, 2.15) {\small $\gamma$}; 
	\node at (1.3, 0.5) {\small $\delta$}; 
	
	\node at (3.3, 1.6) {\small $\beta$}; 
	\node at (2.45, 2.1) {\small $\gamma$}; 
	\node at (1.45, 0.1) {\small $\delta$}; 
	\node at (2.3, -0.4) {\small $\alpha$}; 

	\node at (0.45, 2.55) {\small $\alpha$}; 
	\node at (1.2, 2.6) {\small $\beta$}; 
	\node at (2.2, 2.6) {\small $\gamma$}; 
	\node at (2.2, 3.8) {\small $\delta$}; 

	\node at (4.5, 4.0) {\small $\alpha$}; 
	\node at (3.6, 2.1) {\small $\beta$}; 
	\node at (2.6, 2.5) {\small $\gamma$}; 
	\node at (2.6, 4.0) {\small $\delta$}; 

	\node at (-3.3, -4.0) {\small $\alpha$}; 
	\node at (-2.3, -2.2) {\small $\beta$}; 
	\node at (-1.4, -2.6) {\small $\gamma$}; 
	\node at (-1.4, -4.0) {\small $\delta$}; 

	\node at (0.8, -2.55) {\small $\alpha$}; 
	\node at (0.0, -2.65) {\small $\beta$}; 
	\node at (-1.0, -2.65) {\small $\gamma$}; 
	\node at (-1.0, -3.8) {\small $\delta$}; 

	\node at (3.8, 1.75) {\small $\alpha$}; 
	\node at (4.6, 3.4) {\small $\beta$}; 
	\node at (2.6, -0.5) {\small $\delta$}; 
	\node at (4.6, -3.6) {\small $\gamma$}; 

	\node at (2.6, -1.85) {\small $\alpha$}; 
	\node at (2.55, -2.75) {\small $\beta$}; 
	\node at (4.1, -3.6) {\small $\gamma$}; 
	\node at (2.55, -1.15) {\small $\delta$}; 

	\node at (1.2, -2.65) {\small $\alpha$}; 
	\node at (2.3, -3.3) {\small $\beta$}; 
	\node at (3.9, -4.0) {\small $\gamma$}; 
	\node at (-0.7, -4.0) {\small $\delta$}; 

	\node at (-2.6, -1.8) {\small $\alpha$}; 
	\node at (-3.4, -3.5) {\small $\beta$}; 
	\node at (-3.4, 3.4) {\small $\gamma$}; 
	\node at (-1.4, 0.55) {\small $\delta$}; 

	\node at (-1.4, 1.8) {\small $\alpha$}; 
	\node at (-1.4, 2.8) {\small $\beta$}; 
	\node at (-3.0, 3.6) {\small $\gamma$}; 
	\node at (-1.3, 1.1) {\small $\delta$}; 

	\node at (0.0, 2.6) {\small $\alpha$}; 
	\node at (-1.0, 3.2) {\small $\beta$}; 
	\node at (-2.8, 4.0) {\small $\gamma$}; 
	\node at (1.8, 4.0) {\small $\delta$}; 

	\node at (2.6, 4.4) {\small $\alpha$}; 
	\node at (5.0, 4.4) {\small $\beta$}; 
	\node at (5.0, -4.4) {\small $\gamma$}; 
	\node at (-1.0, -4.4) {\small $\delta$}; 
	
	\node at (-1.4, -4.4) {\small $\alpha$}; 
	\node at (-3.8, -4.4) {\small $\beta$}; 
	\node at (-3.8, 4.4) {\small $\gamma$}; 
	\node at (2.2, 4.4) {\small $\delta$}; 

	\node[inner sep=1,draw,shape=circle] at (0.6, -0.6) {\small $1$}; 
	\node[inner sep=1,draw,shape=circle] at (0.6, 0.6) {\small $2$}; 
	\node[inner sep=1,draw,shape=circle] at (0.6, 1.8) {\small $3$}; 
	\node[inner sep=1,draw,shape=circle] at (-0.6, 2.1) {\small $4$}; 
	\node[inner sep=1,draw,shape=circle] at (-0.6, 0.95) {\small $5$}; 
	\node[inner sep=1,draw,shape=circle] at (-1.1, -0.8) {\small $6$}; 
	\node[inner sep=1,draw,shape=circle] at (-0.35, -1.6) {\small $7$}; 
	\node[inner sep=1,draw,shape=circle] at (0.6, -1.8) {\small $8$}; 
	\node[inner sep=1,draw,shape=circle] at (1.8, -2.1) {\small $9$}; 
	\node[inner sep=1,draw,shape=circle] at (1.8, -0.9) {\small $10$}; 
	\node[inner sep=1,draw,shape=circle] at (1.6, 1.6) {\small $11$}; 
	\node[inner sep=1,draw,shape=circle] at (2.4, 0.8) {\small $12$}; 
	\node[inner sep=1,draw,shape=circle] at (1.6, 3.0) {\small $13$}; 
	\node[inner sep=1,draw,shape=circle] at (3.4, 3.0) {\small $14$}; 
	\node[inner sep=1,draw,shape=circle] at (-2.15, -3.0) {\small $15$}; 
	\node[inner sep=1,draw,shape=circle] at (-0.35, -3.0) {\small $16$}; 
	\node[inner sep=1,draw,shape=circle] at (4.0, 0.0) {\small $17$}; 
	\node[inner sep=1,draw,shape=circle] at (3.0, -2.6) {\small $18$}; 
	\node[inner sep=1,draw,shape=circle] at (1.8, -3.6) {\small $19$}; 
	\node[inner sep=1,draw,shape=circle] at (-2.8, 0.0) {\small $20$}; 
	\node[inner sep=1,draw,shape=circle] at (-1.8, 2.5) {\small $21$}; 
	\node[inner sep=1,draw,shape=circle] at (-0.6, 3.6) {\small $22$}; 
	\node[inner sep=1,draw,shape=circle] at (1.8, -4.8) {\small $23$}; 
	\node[inner sep=1,draw,shape=circle] at (-0.6, 4.8) {\small $24$}; 

	\end{scope}

	\begin{scope}[xshift=9 cm, xscale=-1]

	\draw[] 
	(0, -1.2) -- (0,0) -- (1.2,0) -- (1.2, 1.2)
	(1.2, 1.2) -- (1.2, 2.4) -- (0, 2.4)
	(0, -1.2) -- (0, -2.4) -- (1.2, -2.4)
	(0, 2.4) -- (-1.2, 3.0) 
	(1.2, -2.4) -- (2.4, -3.0) 
	(0, 2.4) -- (2.4, 4.2)
	(1.2, -2.4) -- (-1.2, -4.2)
	(2.4, 4.2) -- (3.6, 4.2) 
	(-1.2, -4.2) -- (-2.4, -4.2) 
	(2.4, 4.2) -- (2.4, 4.8) 
	(-1.2, -4.2) -- (-1.2, -4.8);

	\draw[double, line width=0.6]
	(0,1.2) -- (1.2,1.2)
	(0,-1.2) -- (1.2,-1.2)
	(1.2, 2.4) -- (2.4, 2.4) 
	(0,-2.4) -- (-1.2, -2.4) 
	;

	\draw[line width=2]
	(0,1.8) -- (0,2.4)
	(1.2, -1.8) -- (1.2, -2.4)
	(2.4, 3.4) -- (2.4, 4.2)
	(-1.2, -3.4) -- (-1.2, -4.2)
	;

	\draw[dotted, line width=0.6]
	(2.4, 4.8) -- (2.4, 5.2)
	(-1.2, -4.8) -- (-1.2, -5.2);

	\node at (0.2, -0.2) {\small $\alpha$}; 
	\node at (0.2, -0.95) {\small $\beta$}; 

	\node at (1, 0.2) {\small $\alpha$}; 
	\node at (1, 0.9) {\small $\beta$}; 

	\node at (1, 2.2) {\small $\alpha$}; 
	\node at (1, 1.4) {\small $\beta$}; 
	\node at (0.2, 2.2) {\small $\delta$}; 

	\node at (-0.2, 2.3) {\small $\delta$}; 



	\node at (-0.2, -1.2) {\small $\alpha$}; 
	\node at (-0.15, -2.2) {\small $\beta$}; 

	\node at (0.2, -2.2) {\small $\alpha$}; 
	\node at (0.2, -1.45) {\small $\beta$}; 
	\node at (1, -2.2) {\small $\delta$}; 

	\node at (1.4, -2.2) {\small $\delta$}; 


	\node at (1.35, 1.2) {\small $\alpha$}; 
	\node at (1.4, 2.15) {\small $\beta$}; 
	

	\node at (0.45, 2.55) {\small $\alpha$}; 
	\node at (1.2, 2.6) {\small $\beta$}; 
	\node at (2.2, 3.8) {\small $\delta$}; 

	\node at (2.6, 4.0) {\small $\delta$}; 

	\node at (-1.4, -4.0) {\small $\delta$}; 

	\node at (0.8, -2.55) {\small $\alpha$}; 
	\node at (0.0, -2.65) {\small $\beta$}; 
	\node at (-1.0, -3.8) {\small $\delta$}; 



	\node at (1.2, -2.65) {\small $\alpha$}; 



	\node at (0.0, 2.6) {\small $\alpha$}; 

	\node at (2.6, 4.4) {\small $\alpha$}; 
	
	\node at (-1.4, -4.4) {\small $\alpha$}; 

	\node at (3.0, 0.0) {\small $R_1$}; 
	\node at (-1.8, 0.0) {\small $R_2$}; 

	\node at (1.7, 0.0) {\small $\delta^{\frac{f+8}{8}}$}; 
	\node at (-0.6, -4.4) {\small $\delta^{\frac{f-8}{8}}$}; 
	\node at (-0.5, 0.0) {\small $\delta^{\frac{f+8}{8}}$}; 
	\node at (1.8, 4.4) {\small $\delta^{\frac{f-8}{8}}$}; 

	\node[inner sep=1,draw,shape=circle] at (0.6, -0.6) {\small $1$}; 
	\node[inner sep=1,draw,shape=circle] at (0.6, 0.6) {\small $2$}; 
	\node[inner sep=1,draw,shape=circle] at (0.6, 1.8) {\small $3$}; 
	\node[inner sep=1,draw,shape=circle] at (0.6, -1.8) {\small $8$}; 
	\node[inner sep=1,draw,shape=circle] at (1.6, 3.0) {\small $13$}; 
	\node[inner sep=1,draw,shape=circle] at (-0.35, -3.0) {\small $16$}; 

	\end{scope}
	
	\end{tikzpicture}
	\caption{Tiling of $f=24$, $\AVC \equiv \{ \alpha\beta^2, \alpha^2\delta^2, \gamma^4, \alpha\delta^4 \}$}
	\label{Tf24a2d2ad4}
	
\end{figure}

\end{case*}

\begin{case*}[$\AVC = \{\alpha\beta^2, \alpha^2\delta^2, \gamma^4, \alpha\delta^{\frac{f+8}{8}}, \beta^2\delta^{\frac{f-8}{8}}, \delta^{\frac{f}{4}} \}$] When $\delta^{\frac{f}{4}}$ is a vertex, we get the $(\frac{f}{4},4)$-earth map tiling of $\AVC \equiv \{  \alpha\beta^2, \alpha^2\delta^2,  \gamma^4, \delta^{\frac{f}{4}} \}$. Then we may assume $\delta^{\frac{f}{4}}$ is not a vertex. When $\beta^2\delta^{\frac{f-8}{8}}$ is a vertex, we know that $\frac{f-8}{8} \ge 2$, that is, $f\ge24$. When $f=24$, we have seen the tiling in Figure \ref{Tf24b2d2} obtained by the unique AAD of $\beta^2\delta^2$ in Figure \ref{AADbe2de2}. When $f>24$, the same AAD applies to $\beta^2\delta^{\frac{f-8}{8}}$ as indicated in the second picture of Figure \ref{AADbe2ded} where $R$ can be viewed as cutting out $T_3, T_4, T_5, T_6$ in Figure \ref{AADbe2de2}, along the common edge of $T_{12}, T_{13}$ and filled by $\frac{f-24}{16}$ copies of zones which are given by $T_{k_1}, ... T_{k_8}$ plus half a time zone given by $T_{k_1}, T_{k_2}, T_{k_3}, T_{k_4}$ in Figure \ref{pqEMT}. For such region $R$, we may express it as $R=R(d, d-2)$ where $d, d-2$ are the multiplicities of $\delta$ in the correponding vertices in the first picture of Figure \ref{AADbe2ded}. So a tiling is uniquely determined for each $f>24$ and $f \equiv 0 \mod 8$ as in Figure \ref{b2dFamily} where $R_1 = R_2 = R(\frac{f-8}{8}, \frac{f-24}{8})$. \\

\begin{figure}[htp]
	\centering
	\begin{tikzpicture}[>=latex,scale=0.7]

	\begin{scope}[]

\foreach \a in {0,1,2,3,4,5,6,7,8,9}
\draw[rotate=36*\a]
(90:3) -- (126:3);

\coordinate (O) at (0,0);
\coordinate (A) at (90:0.8);
\coordinate (B) at (270:0.8);
\coordinate (L) at (162:1.8);
\coordinate (R) at (342:1.8);

\coordinate (A1) at (90:3);
\coordinate (A2) at (126:3);
\coordinate (A3) at (162:3);
\coordinate (A4) at (198:3);
\coordinate (A5) at (234:3);
\coordinate (A6) at (270:3);
\coordinate (A7) at (306:3);
\coordinate (A8) at (342:3);
\coordinate (A9) at (18:3);
\coordinate (A10) at (54:3);

\coordinate[shift={(-18:0.9)}] (A3L) at (A3);
\coordinate[shift={(89:0.9)}] (A5L) at (A5);
\coordinate[shift={(43:0.9)}] (A5B) at (A5);

\coordinate[shift={(198:0.9)}] (A9L) at (A9);
\coordinate[shift={(91:0.9)}] (A7L) at (A7);
\coordinate[shift={(137:0.9)}] (A7B) at (A7);

\draw
(A5) -- (A5B)
(A7) -- (A7B);


\draw[double, line width=0.6] 
(A3) -- (A3L)
(A9) -- (A9L);


\draw[line width=2]
(A5) -- (A5L)
(A7) -- (A7L);

\node[shift={(0.02,-0.35)}] at (A1) {\small $\delta^{d}$};

\node[shift={(306:0.2)}] at (A2) {\small $\alpha$};

\node[shift={(0.3,0.15)}] at (A3) {\small $\beta$};
\node[shift={(0.2,-0.35)}] at (A3) {\small $\beta$};

\node[shift={(18:0.2)}] at (A4) {\small $\alpha$};

\node[shift={(-0.2,0.65)}] at (A5) {\small $\delta$};
\node[shift={(0.2,0.5)}] at (A5) {\small $\delta$};
\node[shift={(0.4,0.08)}] at (A5) {\small $\alpha$};

\node[shift={(-0.02,0.4)}] at (A6) {\small $\delta^{d-2}$};

\node[shift={(0.2,0.65)}] at (A7) {\small $\delta$};
\node[shift={(-0.2,0.5)}] at (A7) {\small $\delta$};
\node[shift={(-0.4,0.08)}] at (A7) {\small $\alpha$};

\node[shift={(162:0.2)}] at (A8) {\small $\alpha$};

\node[shift={(-0.4,0.12)}] at (A9) {\small $\beta$};
\node[shift={(-0.2,-0.35)}] at (A9) {\small $\beta$};

\node[shift={(234:0.2)}] at (A10) {\small $\alpha$};

\node at (0,0) {\small $R$};

\end{scope}

\begin{scope}[xshift=10 cm, yshift=0 cm]

\foreach \a in {0,1,2,3,4,5,6,7,8,9}
\draw[rotate=36*\a]
(90:3) -- (126:3);

\coordinate (O) at (0,0);
\coordinate (A) at (90:0.8);
\coordinate (B) at (270:0.8);
\coordinate (L) at (162:1.8);
\coordinate (R) at (342:1.8);

\coordinate (A1) at (90:3);
\coordinate (A2) at (126:3);
\coordinate (A3) at (162:3);
\coordinate (A4) at (198:3);
\coordinate (A5) at (234:3);
\coordinate (A6) at (270:3);
\coordinate (A7) at (306:3);
\coordinate (A8) at (342:3);
\coordinate (A9) at (18:3);
\coordinate (A10) at (54:3);

\coordinate[shift={(-18:0.9)}] (A3L) at (A3);
\coordinate[shift={(89:0.9)}] (A5L) at (A5);
\coordinate[shift={(43:0.9)}] (A5B) at (A5);

\coordinate[shift={(198:0.9)}] (A9L) at (A9);
\coordinate[shift={(91:0.9)}] (A7L) at (A7);
\coordinate[shift={(137:0.9)}] (A7B) at (A7);


\draw
(A5) -- (A5B)
(A7) -- (A7B)
(A10) -- (54:4.3)
(A1) -- (90:4.3)
(A2) -- (126:4.3)
(54:4.3) -- (90:4.3) -- (126:4.3)
(54:3) -- (18:4.1) -- (-18:3)
(18:4.1) -- (30:5) -- (54:4.3)
(18:4.1) to[out=300,in=-30] (A6)
(30:5) to[out=-60,in=80] (-18:5) to[out=240,in=-30] (A6)
(126:3) -- (162:4.1) -- (198:3)
(162:4.1) -- (150:5) -- (126:4.3)
(162:4.1) to[out=240,in=210] (A6)
(150:5) to[out=240,in=100] (198:5) to[out=300,in=210] (A6);


\draw[double, line width=0.6] 
(A3) -- (A3L)
(A9) -- (A9L);


\draw[line width=2]
(A5) -- (A5L)
(A7) -- (A7L);

\node[shift={(0.02,-0.4)}] at (A1) {\small $\delta^{\frac{f-8}{8}}$};

\node[shift={(306:0.2)}] at (A2) {\small $\alpha$};

\node[shift={(0.3,0.15)}] at (A3) {\small $\beta$};
\node[shift={(0.2,-0.35)}] at (A3) {\small $\beta$};

\node[shift={(18:0.2)}] at (A4) {\small $\alpha$};

\node[shift={(-0.2,0.65)}] at (A5) {\small $\delta$};
\node[shift={(0.2,0.5)}] at (A5) {\small $\delta$};
\node[shift={(0.4,0.08)}] at (A5) {\small $\alpha$};

\node[shift={(-0.02,0.4)}] at (A6) {\small $\delta^{\frac{f-24}{8}}$};

\node[shift={(0.2,0.65)}] at (A7) {\small $\delta$};
\node[shift={(-0.2,0.5)}] at (A7) {\small $\delta$};
\node[shift={(-0.4,0.08)}] at (A7) {\small $\alpha$};

\node[shift={(162:0.2)}] at (A8) {\small $\alpha$};

\node[shift={(-0.4,0.15)}] at (A9) {\small $\beta$};
\node[shift={(-0.2,-0.35)}] at (A9) {\small $\beta$};

\node[shift={(234:0.2)}] at (A10) {\small $\alpha$};


\draw[double, line width=0.6]
(A1) -- (90:4.3)
(30:5) -- (18:4.1) -- (A8)
(150:5) -- (162:4.1) -- (A4);

\draw[line width=2]
(126:4.3) -- (90:4.3) -- (54:4.3)
(A10) -- (18:4.1) to[out=300,in=-30] (A6)
(A2) -- (162:4.1) to[out=240,in=210] (A6);


\node[shift={(0.2,0)}] at (18:4.1) {\small $\gamma$};
\node[shift={(-0.1,0.25)}] at (18:4.1) {\small $\gamma$};
\node[shift={(-0.25,-0.15)}] at (18:4.1) {\small $\gamma$};
\node[shift={(-0.03,-0.5)}] at (18:4.1) {\small $\gamma$};


\node[shift={(0.2,-0.3)}] at (90:4.3) {\small $\gamma$};

\node[shift={(0.2,0.2)}] at (A1) {\small $\beta$};

\node[shift={(0.7,-0.65)}] at (A10) {\small $\delta$};
\node[shift={(0.3,0.1)}] at (A10) {\small $\delta$};
\node[shift={(-0.05,0.25)}] at (A10) {\small $\alpha$};

\node[shift={(18:0.2)}] at (A9) {\small $\alpha$};

\node[shift={(0.15,0.7)}] at (A8) {\small $\beta$};
\node[shift={(-20:0.2)}] at (A8) {\small $\beta$};

\node[shift={(-54:0.2)}] at (A7) {\small $\alpha$};

\node[shift={(0.7,0.03)}] at (A6) {\small $\delta$};
\node[shift={(1.5,-0.08)}] at (A6) {\small $\delta$};

\node[shift={(-0.2,0.05)}] at (-18:5) {\small $\alpha$};

\node[shift={(-0.3,-0.15)}] at (30:5) {\small $\beta$};
\node[shift={(0.03,-0.5)}] at (30:5) {\small $\beta$};

\node[shift={(0.1,-0.3)}] at (54:4.3) {\small $\alpha$};
\node[shift={(-0.35,-0.15)}] at (54:4.3) {\small $\delta$};


\node[shift={(-0.2,0)}] at (162:4.1) {\small $\gamma$};
\node[shift={(0.1,0.25)}] at (162:4.1) {\small $\gamma$};
\node[shift={(0.25,-0.15)}] at (162:4.1) {\small $\gamma$};
\node[shift={(0.03,-0.5)}] at (162:4.1) {\small $\gamma$};


\node[shift={(-0.2,-0.3)}] at (90:4.3) {\small $\gamma$};

\node[shift={(-0.2,0.2)}] at (A1) {\small $\beta$};

\node[shift={(-0.7,-0.65)}] at (A2) {\small $\delta$};
\node[shift={(-0.3,0.1)}] at (A2) {\small $\delta$};
\node[shift={(0.05,0.25)}] at (A2) {\small $\alpha$};

\node[shift={(162:0.2)}] at (A3) {\small $\alpha$};

\node[shift={(-0.15,0.8)}] at (A4) {\small $\beta$};
\node[shift={(200:0.2)}] at (A4) {\small $\beta$};

\node[shift={(234:0.2)}] at (A5) {\small $\alpha$};

\node[shift={(-0.7,0.03)}] at (A6) {\small $\delta$};
\node[shift={(-1.55,-0.08)}] at (A6) {\small $\delta$};

\node[shift={(0.2,0.05)}] at (198:5) {\small $\alpha$};

\node[shift={(0.3,-0.15)}] at (150:5) {\small $\beta$};
\node[shift={(-0.05,-0.5)}] at (150:5) {\small $\beta$};

\node[shift={(-0.1,-0.3)}] at (126:4.3) {\small $\alpha$};
\node[shift={(0.35,-0.15)}] at (126:4.3) {\small $\delta$};







\node at (0,0) {\small $R$};


\end{scope}

\end{tikzpicture}

\caption{Region $R$ and AAD of $\beta^{2}\delta^{\frac{f-8}{8}}$}
\label{AADbe2ded}
\end{figure}

\begin{figure}[htp]
	\centering
	\begin{tikzpicture}[>=latex,scale=1]

	\begin{scope}[xscale=-1]

	\draw[] 
	(0,0) -- (0,1.2) -- (-1.6, 1.8) 
	(0,0) -- (1.6, -0.6)
	(-1.6, 1.8) -- (1.6, 1.8) -- (2.8, 1.8) -- (2.8, 0.6)
	(-1.6, -0.6) -- (1.6, -0.6)
	(1.6, 1.8) -- (1.6, 3.0)
	(-1.6, -0.6)  -- (-1.6, -1.8)
	
	(-1.6, 1.8) -- (-2.8, 2.2) -- (-2.8, 3.0) 
 	(-2.8, 0.6) -- (-2.8, 1.4) -- (-1.6, 1.8) 

	(2.8, 0.6) -- (2.8, -0.2) 
	(1.6, -0.6) -- (2.8, -0.2) 
	(2.8, -1.0) -- (2.8, -1.8) 
	(1.6, -0.6) -- (2.8, -1.0) 

	(-1.6, -0.6) -- (-2.8, -0.6) -- (-2.8, 0.6)
	(2.8, -1.8) -- (5.8, 3.8)
	(-2.8, 3.0) -- (-5.8, -2.6) 
	(1.6, 3.0) -- (1.6, 3.8) -- (5.8,3.8)
	(-1.6, -1.8) -- (-1.6, -2.6) -- (-5.8, -2.6)
	(1.6, 3.8) -- (0.6, 4.2)
	(-1.6, -2.6) -- (-0.6, -3.0);

	\draw[double, line width=0.6]
	(0,1.2) -- (1.6, 0.6) -- (2.8, 0.6)
	(0,0) -- (-1.6, 0.6) -- (-2.8, 0.6)
	(-1.6, 3.0) -- (1.6, 3.0)
	(-1.6, -1.8) -- (1.6, -1.8)
	(1.6, -1.8) -- (2.8, -1.8)
	(-1.6, 3.0) -- (-2.8, 3.0)
	(2.8, 1.8) -- (3.4, 2.2)  
	(5.2, 3.4) -- (5.8, 3.8) 
	(-2.8, -0.6) -- (-3.4,-1.0) 
	(-5.2, -2.2) -- (-5.8, -2.6); 

	\draw[line width=2]
	(-1.6, 1.8) -- (-1.6, 0.6)
	(1.6, -0.6) -- (1.6, 0.6)
	(1.6, 0.6) -- (1.6, 1.8)
	(-1.6, -0.6) -- (-1.6,0.6)
	(-1.6, 1.8) -- (-1.6, 3.0)
	(1.6, -0.6) -- (1.6, -1.8)


	(-1.6, 3.0) -- (-1.6, 3.6)
	(1.6, -1.8) -- (1.6, -2.4)
	(1.6, 3.8) -- (1.6, 4.4)
	(-1.6, -2.6) -- (-1.6, -3.2);

	\draw[dotted, line width=2]
	(-1.6, -3.2) -- (-1.6, -3.8)
	(1.6, -2.4) -- (1.6, -3.0)
	(-1.6, 3.6) -- (-1.6, 4.2)
	(1.6, 4.4) -- (1.6, 5.0);

	\draw[dotted, line width=0.75]
	(0.6, 4.2) -- (0.2, 4.4)
	(-0.6, -3.0) -- (-0.2, -3.2);

	\node at (1.4, 3.7) {\small $\alpha$}; 
	\node at (1.4, 3.2) {\small $\beta$}; 
	\node at (-1.4, 3.2) {\small $\gamma$}; 
	\node at (-1.4, -3.0) {\small $\delta$}; 

	\node at (1.4, 1.95) {\small $\alpha$}; 
	\node at (1.4, 2.75) {\small $\beta$}; 
	\node at (-1.4, 2.75) {\small $\gamma$}; 
	\node at (-1.4, 2.0) {\small $\delta$}; 
	
	\node at (1.8, 1.95) {\small $\alpha$}; 
	\node at (2.7, 1.95) {\small $\beta$}; 
	
	\node at (1.8, 3.65) {\small $\alpha$}; 
	\node at (5.2, 3.6) {\small $\beta$}; 
	
	\node at (2.92, -1.3) {\small $\alpha$}; 
	\node at (5.3, 3.2) {\small $\beta$}; 

	\node at (3.0, 0.6) {\small $\alpha$}; 
	\node at (2.95, 1.6) {\small $\beta$}; 

	\node at (-1.4, -2.5) {\small $\alpha$}; 
	\node at (-1.4, -2.1) {\small $\beta$}; 
	\node at (1.4, -2.0) {\small $\gamma$}; 
	\node at (1.4, 4.1) {\small $\delta$}; 

	\node at (5.9, 4.0) {\small $\alpha$}; 
	\node at (2.8, -2.05) {\small $\beta$}; 
	\node at (1.8, -2.0) {\small $\gamma$}; 
	\node at (1.8, 4.0) {\small $\delta$}; 

	\node at (-0.8, 1.65) {\small $\alpha$}; 
	\node at (0.0, 1.4) {\small $\beta$}; 
	\node at (1.4, 0.9) {\small $\gamma$}; 
	\node at (1.4, 1.6) {\small $\delta$}; 

	\node at (2.6, 1.6) {\small $\alpha$}; 
	\node at (2.6, 0.8) {\small $\beta$}; 
	\node at (1.8, 0.8) {\small $\gamma$}; 
	\node at (1.8, 1.6) {\small $\delta$}; 

	\node at (0.2, 0.1) {\small $\alpha$}; 
	\node at (0.2, 0.8) {\small $\beta$}; 
	\node at (1.4, 0.4) {\small $\gamma$}; 
	\node at (1.4, -0.3) {\small $\delta$}; 

	\node at (2.6, -0.1) {\small $\alpha$}; 
	\node at (2.6, 0.35) {\small $\beta$}; 
	\node at (1.8, 0.4) {\small $\gamma$}; 
	\node at (1.8, -0.3) {\small $\delta$}; 

	\node at (2.6, -1.1) {\small $\alpha$}; 
	\node at (2.6, -1.6) {\small $\beta$}; 
	\node at (1.8, -1.6) {\small $\gamma$}; 
	\node at (1.8, -0.9) {\small $\delta$}; 

	\node at (-1.4, -0.8) {\small $\alpha$}; 
	\node at (-1.4, -1.6) {\small $\beta$}; 
	\node at (1.4, -1.6) {\small $\gamma$}; 
	\node at (1.4, -0.8) {\small $\delta$}; 

	\node at (0.8, -0.45) {\small $\alpha$}; 
	\node at (0.0, -0.3) {\small $\beta$}; 
	\node at (-1.4, 0.3) {\small $\gamma$}; 
	\node at (-1.4, -0.4) {\small $\delta$}; 

	\node at (-0.2, 1.0) {\small $\alpha$}; 
	\node at (-0.2, 0.3) {\small $\beta$}; 
	\node at (-1.4, 0.7) {\small $\gamma$}; 
	\node at (-1.4, 1.5) {\small $\delta$}; 

	\node at (-5.9, -2.8) {\small $\alpha$}; 
	\node at (-2.8, 3.2) {\small $\beta$}; 
	\node at (-1.8, 3.2) {\small $\gamma$}; 
	\node at (-1.8, -2.8) {\small $\delta$}; 

	\node at (-2.6, 1.3) {\small $\alpha$}; 
	\node at (-2.6, 0.8) {\small $\beta$}; 
	\node at (-1.8, 0.8) {\small $\gamma$}; 
	\node at (-1.8, 1.5) {\small $\delta$}; 

	\node at (-2.6, 2.3) {\small $\alpha$}; 
	\node at (-2.6, 2.75) {\small $\beta$}; 
	\node at (-1.8, 2.75) {\small $\gamma$}; 
	\node at (-1.8, 2.1) {\small $\delta$}; 

	\node at (-2.6, -0.4) {\small $\alpha$}; 
	\node at (-2.6, 0.3) {\small $\beta$}; 
	\node at (-1.8, 0.3) {\small $\gamma$}; 
	\node at (-1.8, -0.4) {\small $\delta$}; 

	\node at (-2.95, 2.4) {\small $\alpha$}; 
	\node at (-5.3, -2.0) {\small $\beta$}; 

	\node at (-2.95, 0.6) {\small $\alpha$}; 
	\node at (-3.0, -0.5) {\small $\beta$}; 

	\node at (-1.8, -0.8) {\small $\alpha$}; 
	\node at (-2.75, -0.85) {\small $\beta$}; 

	\node at (-1.8, -2.4) {\small $\alpha$}; 
	\node at (-5.0, -2.4) {\small $\beta$}; 

	\node at (4.2, 2.8) {\small $R_1$}; 
	\node at (-4.0, -1.6) {\small $R_2$}; 

	\node[inner sep=1,draw,shape=circle] at (0.0, 3.6) {\small $1$}; 
	\node[inner sep=1,draw,shape=circle] at (0.0, 2.4) {\small $2$}; 
	\node[inner sep=1,draw,shape=circle] at (0.0, -2.4) {\small $7$}; 
	\node[inner sep=1,draw,shape=circle] at (2.2, -2.4) {\small $8$}; 
	
	\node[inner sep=1,draw,shape=circle] at (0.8, 1.4) {\small $9$}; 
	\node[inner sep=1,draw,shape=circle] at (2.2, 1.2) {\small $10$}; 
	\node[inner sep=1,draw,shape=circle] at (0.8, 0.3) {\small $11$}; 
	\node[inner sep=1,draw,shape=circle] at (2.2, 0.0) {\small $12$}; 
	\node[inner sep=1,draw,shape=circle] at (2.2, -1.2) {\small $13$}; 
	\node[inner sep=1,draw,shape=circle] at (0.0, -1.2) {\small $14$}; 
	\node[inner sep=1,draw,shape=circle] at (-0.8, -0.2) {\small $15$}; 
	\node[inner sep=1,draw,shape=circle] at (-0.8, 0.9) {\small $16$}; 

	\node[inner sep=1,draw,shape=circle] at (-2.2, 0.0) {\small $17$}; 
	\node[inner sep=1,draw,shape=circle] at (-2.2, 1.2) {\small $18$}; 
	\node[inner sep=1,draw,shape=circle] at (-2.2, 2.4) {\small $19$}; 
	\node[inner sep=1,draw,shape=circle] at (-2.2, 3.6) {\small $20$}; 



	\end{scope}

	\end{tikzpicture}
	\caption{Tiling of $\AVC \equiv \{ \alpha\beta^2, \alpha^2\delta^2, \gamma^4, \beta^2\delta^{\frac{f-8}{8}},  \alpha\delta^{\frac{f+8}{8}} \}$}
	\label{b2dFamily}
\end{figure}
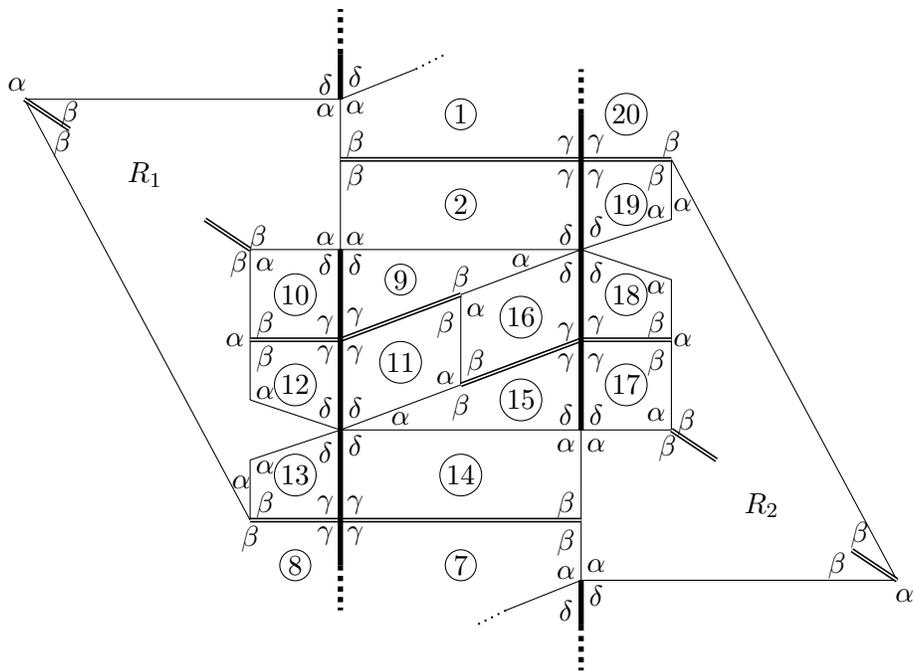

When $\delta^{\frac{f}{4}}$, $\beta^2\delta^{\frac{f-8}{8}}$ are not vertices, $\alpha\delta^{\frac{f+8}{8}}$ is the only high degree vertex and $\beta^2\cdots = \alpha\beta^2$. From $\alpha\delta^{\frac{f+8}{8}}$, we know that $f \ge 24$. When $f=24$, we have seen the tiling obtained by $\alpha\delta^4$ in the first picture of Figure \ref{Tf24a2d2ad4}. Indeed the same AAD applies to $\alpha\delta^{\frac{f+8}{8}}$ and for each $f\ge 24$ and $f \equiv 0 \mod 8$ we uniquely determine a tiling of $\AVC \equiv \{ \alpha\beta^2, \alpha^2\delta^2, \gamma^4, \alpha\delta^{\frac{f+8}{8}} \}$ as indicated in the second picture of Figure \ref{Tf24a2d2ad4} where $R_1 = R_2 = R(\frac{f+8}{8}, \frac{f-8}{8})$. 
\end{case*}

\subsection{Relations via Flip Modifications}

First we explain the relation between the tilings in Figure \ref{Tf24a3-1} and \ref{Tf24a3-2} via flip modifications. We express the AAD of Figure \ref{a3AAD} as an $\alpha^3$-disk in the first picture of Figure \ref{a3DiskFlip}. Flipping along the dahsed line we obtain its flip modification. The tiling in Figure \ref{Tf24a3-1} can be obtained by gluing two copies of the $\alpha^3$-disk along the five edges of each of the neighbouring disks and then the along other ten edges in the first row of Figure \ref{a3RelnFlip}. The tiling in Figure \ref{Tf24a3-2} can be obtained by gluing an $\alpha^3$-disk and its flip modification along the five edges of each of the neighbouring disks and then along other ten edges in the second row of Figure \ref{a3RelnFlip}. \\

\begin{figure}[htp]
	\centering
	\begin{tikzpicture}[>=latex,scale=0.7]

	\foreach \b in {0}
	{
	\begin{scope}[]
	\draw[red, line width=2, dashed]
	(105:4) -- (285:4);

	\foreach \a in {0,1,2,3,4,5,6,7,8,9,10,11}
	\draw[rotate=30*\a]
	(90:3) -- (120:3);
	
	\coordinate (O) at (0,0);
	\coordinate (A1) at (90:3);
	\coordinate (A2) at (120:3);
	\coordinate (A3) at (150:3);
	\coordinate (A4) at (180:3);
	\coordinate (A5) at (210:3);
	\coordinate (A6) at (240:3);
	\coordinate (A7) at (270:3);
	\coordinate (A8) at (300:3);
	\coordinate (A9) at (330:3);
	\coordinate (A10) at (360:3);
	\coordinate (A11) at (30:3);
	\coordinate (A12) at (60:3);
	
	\coordinate (A1h) at (90:1.5);	
	\coordinate (A5h) at (210:1.5);
	\coordinate (A9h) at (330:1.5);
	

	\coordinate (B1) at (30:1.607695);
	\coordinate (B2) at (150:1.607695);
	\coordinate (B3) at (270:1.607695);	
	
	\draw 
	(O) -- (A1)
	(O) -- (A5)
	(O) -- (A9)
	(A1h) -- (A4)
	(A1h) -- (A10)
	(A5h) -- (A2)
	(A5h) -- (A8)
	(A9h) -- (A6)
	(A9h) -- (A12);
	
	
	\draw[]
	(A1) -- (A2) -- (A3) -- (A4) -- (A5) -- (A6) -- (A7);
	
	\draw[]
	(A7) -- (A8) -- (A9) -- (A10) -- (A11) -- (A12) -- (A1);

	\draw[double, line width=0.6]
	(A1h) -- (A4)
	(A5h) -- (A8)
	(A9h) -- (A12);
	
	\draw[line width=2]
	(A1h) -- (A10)
	(A5h) -- (A2)
	(A9h) -- (A6);
	
	
	\node[shift={(30:0.2)}] at (0,0) {\small $\alpha$};
	\node[shift={(150:0.2)}] at (0,0) {\small $\alpha$};
	\node[shift={(270:0.2)}] at (0,0) {\small $\alpha$};
	
	
	\node[shift={(0.25,0.15)}] at (B1) {\small $\gamma$};
	\node[shift={(-0.2,0.32)}] at (B1) {\small $\gamma$};
	\node[shift={(-0.22,-0.22)}] at (B1) {\small $\gamma$};
	\node[shift={(0.2,-0.4)}] at (B1) {\small $\gamma$};

	\node[shift={(0.2,0.15)}] at (A1h) {\small $\delta$};
	\node[shift={(-0.2,0.15)}] at (A1h) {\small $\beta$};
	\node[shift={(-0.2,-0.35)}] at (A1h) {\small $\beta$};
	\node[shift={(0.2,-0.35)}] at (A1h) {\small $\delta$};
	
	\node[shift={(-0.25,0.15)}] at (B2) {\small $\gamma$};
	\node[shift={(0.2,0.32)}] at (B2) {\small $\gamma$};
	\node[shift={(0.22,-0.22)}] at (B2) {\small $\gamma$};
	\node[shift={(-0.2,-0.4)}] at (B2) {\small $\gamma$};

	\node[shift={(0.2,0.35)}] at (A5h) {\small $\delta$};
	\node[shift={(-0.2,0.15)}] at (A5h) {\small $\delta$};
	\node[shift={(-0.05,-0.3)}] at (A5h) {\small $\beta$};
	\node[shift={(0.42,0)}] at (A5h) {\small $\beta$};

	\node[shift={(0.4,0)}] at (B3) {\small $\gamma$};
	\node[shift={(0,0.3)}] at (B3) {\small $\gamma$};
	\node[shift={(-0.4,0)}] at (B3) {\small $\gamma$};
	\node[shift={(0,-0.3)}] at (B3) {\small $\gamma$};

	\node[shift={(-0.2,0.38)}] at (A9h) {\small $\beta$};
	\node[shift={(0.2,0.15)}] at (A9h) {\small $\beta$};
	\node[shift={(0.05,-0.3)}] at (A9h) {\small $\delta$};
	\node[shift={(-0.42,0)}] at (A9h) {\small $\delta$};

	
	\node[shift={(0.2,-0.25)}] at (A1) {\small $\alpha$};
	\node[shift={(-0.2,-0.25)}] at (A1) {\small $\alpha$};
	
	\node[shift={(0.2,-0.25)}] at (A2) {\small $\delta$};
	\node[shift={(-0.2,-0.45)}] at (A2) {\small $\delta$};
	
	\node[shift={(330:0.2)}] at (A3) {\small $\alpha$};

	\node[shift={(0.25,0.35)}] at (A4) {\small $\beta$};
	\node[shift={(0.25,-0.15)}] at (A4) {\small $\beta$};
	
	\node[shift={(0.14,0.3)}] at (A5) {\small $\alpha$};
	\node[shift={(0.3,-0.05)}] at (A5) {\small $\alpha$};
	
	\node[shift={(0,0.3)}] at (A6) {\small $\delta$};
	\node[shift={(0.45,0.05)}] at (A6) {\small $\delta$};
	
	\node[shift={(90:0.2)}] at (A7) {\small $\alpha$};
	
	\node[shift={(0,0.3)}] at (A8) {\small $\beta$};
	\node[shift={(-0.5,0.05)}] at (A8) {\small $\beta$};
	
	\node[shift={(-0.14,0.3)}] at (A9) {\small $\alpha$};
	\node[shift={(-0.3,-0.05)}] at (A9) {\small $\alpha$};
	
	\node[shift={(-0.25,0.35)}] at (A10) {\small $\delta$};
	\node[shift={(-0.25,-0.15)}] at (A10) {\small $\delta$};
	
	\node[shift={(210:0.2)}] at (A11) {\small $\alpha$};
	
	\node[shift={(-0.2,-0.25)}] at (A12) {\small $\beta$};
	\node[shift={(0.15,-0.45)}] at (A12) {\small $\beta$};
	
	\end{scope}
	}

\begin{scope}[xshift=8cm, yshift=0 cm]
	\foreach \a in {0,1,2,3,4,5,6,7,8,9,10,11}
	\draw[rotate=30*\a]
	(90:3) -- (120:3);
	
	\coordinate (O) at (0,0);
	\coordinate (A1) at (90:3);
	\coordinate (A2) at (120:3);
	\coordinate (A3) at (150:3);
	\coordinate (A4) at (180:3);
	\coordinate (A5) at (210:3);
	\coordinate (A6) at (240:3);
	\coordinate (A7) at (270:3);
	\coordinate (A8) at (300:3);
	\coordinate (A9) at (330:3);
	\coordinate (A10) at (360:3);
	\coordinate (A11) at (30:3);
	\coordinate (A12) at (60:3);

	\coordinate (A4h) at (180:1.5);
	\coordinate (A8h) at (300:1.5);
	\coordinate (A12h) at (60:1.5);


	\coordinate (B1) at (120:1.607695);
	\coordinate (B2) at (240:1.607695);
	\coordinate (B3) at (0:1.607695);

	\draw
	(O) -- (A4)
	(O) -- (A8)
	(O) -- (A12)
	(A4h) -- (A1)
	(A4h) -- (A7)
	(A8h) -- (A5)
	(A8h) -- (A11)
	(A12h) -- (A3)
	(A12h) -- (A9);
	

	\draw[]
	(A7) -- (A8) -- (A9) -- (A10) -- (A11) -- (A12) -- (A1);
	
	\draw[]
	(A1) -- (A2) -- (A3) -- (A4) -- (A5) -- (A6) -- (A7);
	
	
	\draw[double, line width=0.6] 
	(A4h) -- (A7)
	(A8h) -- (A11)
	(A12h) -- (A3);
	
	\draw[line width=2] 
	(A4h) -- (A1)
	(A8h) -- (A5)
	(A12h) -- (A9);
	
	
	\node[shift={(120:0.2)}] at (0,0) {\small $\alpha$};
	\node[shift={(240:0.2)}] at (0,0) {\small $\alpha$};
	\node[shift={(0:0.2)}] at (0,0) {\small $\alpha$};

	\node[shift={(0:0.3)}] at (B3) {\small $\gamma$};
	\node[shift={(90:0.4)}] at (B3) {\small $\gamma$};
	\node[shift={(180:0.3)}] at (B3) {\small $\gamma$};
	\node[shift={(270:0.4)}] at (B3) {\small $\gamma$};

	\node[shift={(10:0.25)}] at (A12h) {\small $\delta$};
	\node[shift={(110:0.25)}] at (A12h) {\small $\beta$};
	\node[shift={(218:0.4)}] at (A12h) {\small $\beta$};
	\node[shift={(270:0.35)}] at (A12h) {\small $\delta$};
	
	\node[shift={(120:0.3)}] at (B1) {\small $\gamma$};
	\node[shift={(210:0.4)}] at (B1) {\small $\gamma$};
	\node[shift={(300:0.3)}] at (B1) {\small $\gamma$};
	\node[shift={(30:0.4)}] at (B1) {\small $\gamma$};

	\node[shift={(120:0.25)}] at (A4h) {\small $\delta$};
	\node[shift={(240:0.3)}] at (A4h) {\small $\beta$};
	\node[shift={(330:0.45)}] at (A4h) {\small $\beta$};
	\node[shift={(35:0.35)}] at (A4h) {\small $\delta$};

	\node[shift={(240:0.3)}] at (B2) {\small $\gamma$};
	\node[shift={(330:0.4)}] at (B2) {\small $\gamma$};
	\node[shift={(60:0.3)}] at (B2) {\small $\gamma$};
	\node[shift={(150:0.4)}] at (B2) {\small $\gamma$};
	
	\node[shift={(250:0.25)}] at (A8h) {\small $\delta$};
	\node[shift={(-5:0.28)}] at (A8h) {\small $\beta$};
	\node[shift={(88:0.4)}] at (A8h) {\small $\beta$};
	\node[shift={(150:0.4)}] at (A8h) {\small $\delta$};

	
	\node[shift={(0.08,-0.3)}] at (A1) {\small $\delta$};
	\node[shift={(-0.38,-0.3)}] at (A1) {\small $\delta$};
	
	\node[shift={(-60:0.2)}] at (A2) {\small $\alpha$};
	
	\node[shift={(0.45,0.2)}] at (A3) {\small $\beta$};
	\node[shift={(0.1,-0.3)}] at (A3) {\small $\beta$};
	
	\node[shift={(0.28,0.2)}] at (A4) {\small $\alpha$};
	\node[shift={(0.28,-0.25)}] at (A4) {\small $\alpha$};
	
	\node[shift={(0.45,-0.2)}] at (A5) {\small $\delta$};
	\node[shift={(0.1,0.3)}] at (A5) {\small $\delta$};
	
	\node[shift={(60:0.2)}] at (A6) {\small $\alpha$};
	
	\node[shift={(0.1,0.33)}] at (A7) {\small $\beta$};
	\node[shift={(-0.4,0.33)}] at (A7) {\small $\beta$};
	
	\node[shift={(-0.3,0.1)}] at (A8) {\small $\alpha$};
	\node[shift={(0.04,0.3)}] at (A8) {\small $\alpha$};
	
	\node[shift={(-0.28,0.03)}] at (A9) {\small $\delta$};
	\node[shift={(-0.02,0.4)}] at (A9) {\small $\delta$};
	
	\node[shift={(180:0.2)}] at (A10) {\small $\alpha$};

	\node[shift={(-0.28,-0.06)}] at (A11) {\small $\beta$};
	\node[shift={(-0.06,-0.43)}] at (A11) {\small $\beta$};

	\node[shift={(-0.3,-0.1)}] at (A12) {\small $\alpha$};
	\node[shift={(0.04,-0.3)}] at (A12) {\small $\alpha$};
	
	\end{scope}

\end{tikzpicture}
\caption{$\alpha^3$-Disk and Its Flip Modification}
\label{a3DiskFlip}
\end{figure}
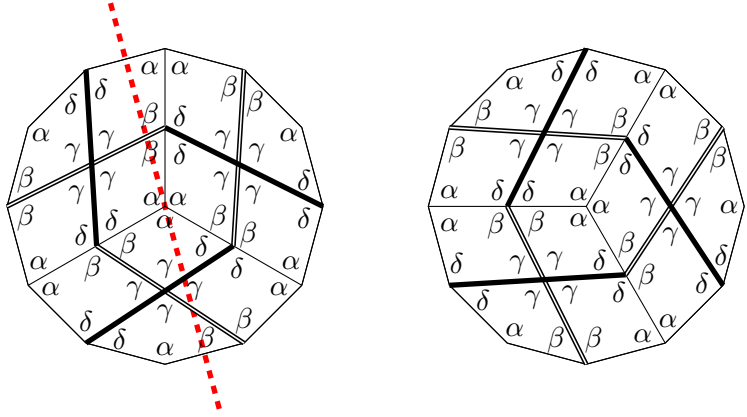

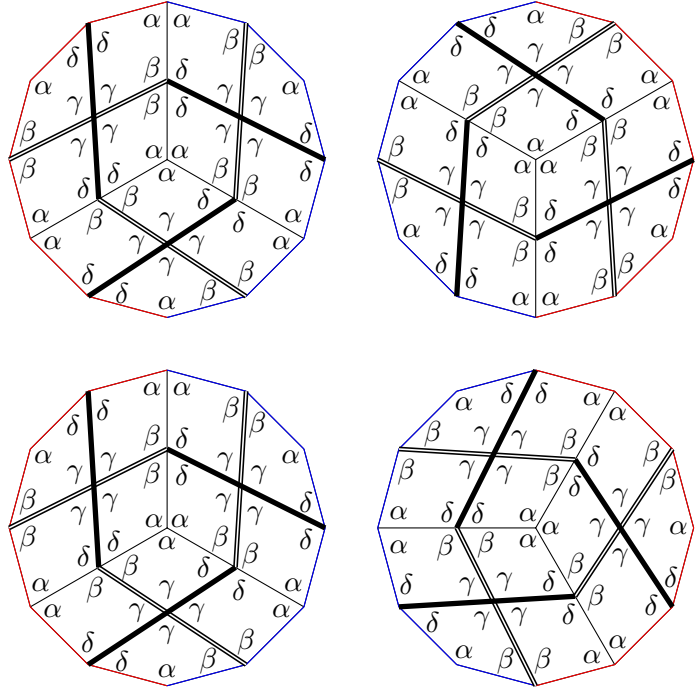
\begin{figure}[htp]
	\centering
	\begin{tikzpicture}[>=latex,scale=0.7]

	\foreach \b in {0,1}
	{
	\begin{scope}[xshift=0cm, yshift=-7*\b cm]
	\foreach \a in {0,1,2,3,4,5,6,7,8,9,10,11}
	\draw[rotate=30*\a]
	(90:3) -- (120:3);
	
	\coordinate (O) at (0,0);
	\coordinate (A1) at (90:3);
	\coordinate (A2) at (120:3);
	\coordinate (A3) at (150:3);
	\coordinate (A4) at (180:3);
	\coordinate (A5) at (210:3);
	\coordinate (A6) at (240:3);
	\coordinate (A7) at (270:3);
	\coordinate (A8) at (300:3);
	\coordinate (A9) at (330:3);
	\coordinate (A10) at (360:3);
	\coordinate (A11) at (30:3);
	\coordinate (A12) at (60:3);
	
	\coordinate (A1h) at (90:1.5);	
	\coordinate (A5h) at (210:1.5);
	\coordinate (A9h) at (330:1.5);
	

	\coordinate (B1) at (30:1.607695);
	\coordinate (B2) at (150:1.607695);
	\coordinate (B3) at (270:1.607695);	
	
	\draw 
	(O) -- (A1)
	(O) -- (A5)
	(O) -- (A9)
	(A1h) -- (A4)
	(A1h) -- (A10)
	(A5h) -- (A2)
	(A5h) -- (A8)
	(A9h) -- (A6)
	(A9h) -- (A12);
	
	
	\draw[red]
	(A1) -- (A2) -- (A3) -- (A4) -- (A5) -- (A6) -- (A7);
	
	\draw[blue]
	(A7) -- (A8) -- (A9) -- (A10) -- (A11) -- (A12) -- (A1);

	\draw[double, line width=0.6]
	(A1h) -- (A4)
	(A5h) -- (A8)
	(A9h) -- (A12);
	
	\draw[line width=2]
	(A1h) -- (A10)
	(A5h) -- (A2)
	(A9h) -- (A6);
	
	
	\node[shift={(30:0.2)}] at (0,0) {\small $\alpha$};
	\node[shift={(150:0.2)}] at (0,0) {\small $\alpha$};
	\node[shift={(270:0.2)}] at (0,0) {\small $\alpha$};
	
	
	\node[shift={(0.25,0.15)}] at (B1) {\small $\gamma$};
	\node[shift={(-0.2,0.32)}] at (B1) {\small $\gamma$};
	\node[shift={(-0.22,-0.22)}] at (B1) {\small $\gamma$};
	\node[shift={(0.2,-0.4)}] at (B1) {\small $\gamma$};

	\node[shift={(0.2,0.15)}] at (A1h) {\small $\delta$};
	\node[shift={(-0.2,0.15)}] at (A1h) {\small $\beta$};
	\node[shift={(-0.2,-0.35)}] at (A1h) {\small $\beta$};
	\node[shift={(0.2,-0.35)}] at (A1h) {\small $\delta$};
	
	\node[shift={(-0.25,0.15)}] at (B2) {\small $\gamma$};
	\node[shift={(0.2,0.32)}] at (B2) {\small $\gamma$};
	\node[shift={(0.22,-0.22)}] at (B2) {\small $\gamma$};
	\node[shift={(-0.2,-0.4)}] at (B2) {\small $\gamma$};

	\node[shift={(0.2,0.35)}] at (A5h) {\small $\delta$};
	\node[shift={(-0.2,0.15)}] at (A5h) {\small $\delta$};
	\node[shift={(-0.05,-0.3)}] at (A5h) {\small $\beta$};
	\node[shift={(0.42,0)}] at (A5h) {\small $\beta$};

	\node[shift={(0.4,0)}] at (B3) {\small $\gamma$};
	\node[shift={(0,0.3)}] at (B3) {\small $\gamma$};
	\node[shift={(-0.4,0)}] at (B3) {\small $\gamma$};
	\node[shift={(0,-0.3)}] at (B3) {\small $\gamma$};

	\node[shift={(-0.2,0.38)}] at (A9h) {\small $\beta$};
	\node[shift={(0.2,0.15)}] at (A9h) {\small $\beta$};
	\node[shift={(0.05,-0.3)}] at (A9h) {\small $\delta$};
	\node[shift={(-0.42,0)}] at (A9h) {\small $\delta$};

	
	\node[shift={(0.2,-0.25)}] at (A1) {\small $\alpha$};
	\node[shift={(-0.2,-0.25)}] at (A1) {\small $\alpha$};
	
	\node[shift={(0.2,-0.25)}] at (A2) {\small $\delta$};
	\node[shift={(-0.2,-0.45)}] at (A2) {\small $\delta$};
	
	\node[shift={(330:0.2)}] at (A3) {\small $\alpha$};

	\node[shift={(0.25,0.35)}] at (A4) {\small $\beta$};
	\node[shift={(0.25,-0.15)}] at (A4) {\small $\beta$};
	
	\node[shift={(0.14,0.3)}] at (A5) {\small $\alpha$};
	\node[shift={(0.3,-0.05)}] at (A5) {\small $\alpha$};
	
	\node[shift={(0,0.3)}] at (A6) {\small $\delta$};
	\node[shift={(0.45,0.05)}] at (A6) {\small $\delta$};
	
	\node[shift={(90:0.2)}] at (A7) {\small $\alpha$};
	
	\node[shift={(0,0.3)}] at (A8) {\small $\beta$};
	\node[shift={(-0.5,0.05)}] at (A8) {\small $\beta$};
	
	\node[shift={(-0.14,0.3)}] at (A9) {\small $\alpha$};
	\node[shift={(-0.3,-0.05)}] at (A9) {\small $\alpha$};
	
	\node[shift={(-0.25,0.35)}] at (A10) {\small $\delta$};
	\node[shift={(-0.25,-0.15)}] at (A10) {\small $\delta$};
	
	\node[shift={(210:0.2)}] at (A11) {\small $\alpha$};
	
	\node[shift={(-0.2,-0.25)}] at (A12) {\small $\beta$};
	\node[shift={(0.15,-0.45)}] at (A12) {\small $\beta$};
	
	\end{scope}
	}

	\begin{scope}[xshift=7cm, yshift=0 cm]
	\foreach \a in {0,1,2,3,4,5,6,7,8,9,10,11}
	\draw[rotate=30*\a]
	(90:3) -- (120:3);

	\coordinate (O) at (0,0);
	\coordinate (A1) at (90:3);
	\coordinate (A2) at (120:3);
	\coordinate (A3) at (150:3);
	\coordinate (A4) at (180:3);
	\coordinate (A5) at (210:3);
	\coordinate (A6) at (240:3);
	\coordinate (A7) at (270:3);
	\coordinate (A8) at (300:3);
	\coordinate (A9) at (330:3);
	\coordinate (A10) at (360:3);
	\coordinate (A11) at (30:3);
	\coordinate (A12) at (60:3);

	\coordinate (A3h) at (150:1.5);
	\coordinate (A7h) at (270:1.5);
	\coordinate (A11h) at (30:1.5);	
	

	\coordinate (B1) at (90:1.607695);
	\coordinate (B2) at (210:1.607695);
	\coordinate (B3) at (330:1.607695);	
	
	\draw
	(O) -- (A3)
	(O) -- (A7)
	(O) -- (A11)
	(A3h) -- (A6)
	(A3h) -- (A12)
	(A7h) -- (A4)
	(A7h) -- (A10)
	(A11h) -- (A2)
	(A11h) -- (A8);

	
	\draw[red]
	(A7) -- (A8) -- (A9) -- (A10) -- (A11) -- (A12) -- (A1);
	
	\draw[blue]
	(A1) -- (A2) -- (A3) -- (A4) -- (A5) -- (A6) -- (A7);
	
	\draw[double, line width=0.6]
	(A3h) -- (A12)
	(A7h) -- (A4)
	(A11h) -- (A8);
	
	\draw[line width=2]
	(A3h) -- (A6)
	(A7h) -- (A10)
	(A11h) -- (A2);
	
	
	\node[shift={(90:0.2)}] at (0,0) {\small $\alpha$};
	\node[shift={(210:0.2)}] at (0,0) {\small $\alpha$};
	\node[shift={(330:0.2)}] at (0,0) {\small $\alpha$};
	
	
	\node[shift={(-0.2,-0.38)}] at (A11h) {\small $\beta$};
	\node[shift={(0.2,-0.15)}] at (A11h) {\small $\beta$};
	\node[shift={(0.05,0.3)}] at (A11h) {\small $\delta$};
	\node[shift={(-0.42,0)}] at (A11h) {\small $\delta$};
	
	\node[shift={(0.4,0)}] at (B1) {\small $\gamma$};
	\node[shift={(0,0.3)}] at (B1) {\small $\gamma$};
	\node[shift={(-0.4,0)}] at (B1) {\small $\gamma$};
	\node[shift={(0,-0.3)}] at (B1) {\small $\gamma$};
	
	\node[shift={(0.2,-0.35)}] at (A3h) {\small $\delta$};
	\node[shift={(-0.2,-0.15)}] at (A3h) {\small $\delta$};
	\node[shift={(0.05,0.3)}] at (A3h) {\small $\beta$};
	\node[shift={(0.42,0)}] at (A3h) {\small $\beta$};
	
	\node[shift={(-0.25,-0.2)}] at (B2) {\small $\gamma$};
	\node[shift={(0.2,-0.37)}] at (B2) {\small $\gamma$};
	\node[shift={(0.22,0.17)}] at (B2) {\small $\gamma$};
	\node[shift={(-0.2,0.35)}] at (B2) {\small $\gamma$};
	
	\node[shift={(0.2,-0.18)}] at (A7h) {\small $\delta$};
	\node[shift={(-0.2,-0.2)}] at (A7h) {\small $\beta$};
	\node[shift={(-0.2,0.38)}] at (A7h) {\small $\beta$};
	\node[shift={(0.2,0.38)}] at (A7h) {\small $\delta$};
	
	\node[shift={(0.25,-0.2)}] at (B3) {\small $\gamma$};
	\node[shift={(-0.2,-0.37)}] at (B3) {\small $\gamma$};
	\node[shift={(-0.22,0.17)}] at (B3) {\small $\gamma$};
	\node[shift={(0.2,0.35)}] at (B3) {\small $\gamma$};
	
	
	\node[shift={(270:0.2)}] at (A1) {\small $\alpha$};

	\node[shift={(0,-0.3)}] at (A2) {\small $\delta$};
	\node[shift={(0.45,-0.05)}] at (A2) {\small $\delta$};
	
	\node[shift={(0.14,-0.3)}] at (A3) {\small $\alpha$};
	\node[shift={(0.3,0.05)}] at (A3) {\small $\alpha$};

	\node[shift={(0.3,-0.42)}] at (A4) {\small $\beta$};
	\node[shift={(0.25,0.15)}] at (A4) {\small $\beta$};

	\node[shift={(30:0.2)}] at (A5) {\small $\alpha$};
	
	\node[shift={(0.2,0.25)}] at (A6) {\small $\delta$};
	\node[shift={(-0.2,0.45)}] at (A6) {\small $\delta$};
	
	\node[shift={(0.2,0.25)}] at (A7) {\small $\alpha$};
	\node[shift={(-0.2,0.25)}] at (A7) {\small $\alpha$};

	\node[shift={(-0.2,0.2)}] at (A8) {\small $\beta$};
	\node[shift={(0.15,0.4)}] at (A8) {\small $\beta$};

	\node[shift={(150:0.2)}] at (A9) {\small $\alpha$};

	\node[shift={(-0.25,-0.38)}] at (A10) {\small $\delta$};
	\node[shift={(-0.25,0.12)}] at (A10) {\small $\delta$};
	
	\node[shift={(-0.14,-0.3)}] at (A11) {\small $\alpha$};
	\node[shift={(-0.3,0.05)}] at (A11) {\small $\alpha$};

	\node[shift={(0,-0.3)}] at (A12) {\small $\beta$};
	\node[shift={(-0.5,-0.1)}] at (A12) {\small $\beta$};
	

	\end{scope}
	
	\begin{scope}[xshift=7cm, yshift=-7 cm]
	\foreach \a in {0,1,2,3,4,5,6,7,8,9,10,11}
	\draw[rotate=30*\a]
	(90:3) -- (120:3);
	
	\coordinate (O) at (0,0);
	\coordinate (A1) at (90:3);
	\coordinate (A2) at (120:3);
	\coordinate (A3) at (150:3);
	\coordinate (A4) at (180:3);
	\coordinate (A5) at (210:3);
	\coordinate (A6) at (240:3);
	\coordinate (A7) at (270:3);
	\coordinate (A8) at (300:3);
	\coordinate (A9) at (330:3);
	\coordinate (A10) at (360:3);
	\coordinate (A11) at (30:3);
	\coordinate (A12) at (60:3);

	\coordinate (A4h) at (180:1.5);
	\coordinate (A8h) at (300:1.5);
	\coordinate (A12h) at (60:1.5);


	\coordinate (B1) at (120:1.607695);
	\coordinate (B2) at (240:1.607695);
	\coordinate (B3) at (0:1.607695);

	\draw
	(O) -- (A4)
	(O) -- (A8)
	(O) -- (A12)
	(A4h) -- (A1)
	(A4h) -- (A7)
	(A8h) -- (A5)
	(A8h) -- (A11)
	(A12h) -- (A3)
	(A12h) -- (A9);
	

	\draw[red]
	(A7) -- (A8) -- (A9) -- (A10) -- (A11) -- (A12) -- (A1);
	
	\draw[blue]
	(A1) -- (A2) -- (A3) -- (A4) -- (A5) -- (A6) -- (A7);
	
	
	\draw[double, line width=0.6] 
	(A4h) -- (A7)
	(A8h) -- (A11)
	(A12h) -- (A3);
	
	\draw[line width=2] 
	(A4h) -- (A1)
	(A8h) -- (A5)
	(A12h) -- (A9);
	
	
	\node[shift={(120:0.2)}] at (0,0) {\small $\alpha$};
	\node[shift={(240:0.2)}] at (0,0) {\small $\alpha$};
	\node[shift={(0:0.2)}] at (0,0) {\small $\alpha$};

	\node[shift={(0:0.3)}] at (B3) {\small $\gamma$};
	\node[shift={(90:0.4)}] at (B3) {\small $\gamma$};
	\node[shift={(180:0.3)}] at (B3) {\small $\gamma$};
	\node[shift={(270:0.4)}] at (B3) {\small $\gamma$};

	\node[shift={(10:0.25)}] at (A12h) {\small $\delta$};
	\node[shift={(110:0.25)}] at (A12h) {\small $\beta$};
	\node[shift={(218:0.4)}] at (A12h) {\small $\beta$};
	\node[shift={(270:0.35)}] at (A12h) {\small $\delta$};
	
	\node[shift={(120:0.3)}] at (B1) {\small $\gamma$};
	\node[shift={(210:0.4)}] at (B1) {\small $\gamma$};
	\node[shift={(300:0.3)}] at (B1) {\small $\gamma$};
	\node[shift={(30:0.4)}] at (B1) {\small $\gamma$};

	\node[shift={(120:0.25)}] at (A4h) {\small $\delta$};
	\node[shift={(240:0.3)}] at (A4h) {\small $\beta$};
	\node[shift={(330:0.45)}] at (A4h) {\small $\beta$};
	\node[shift={(35:0.35)}] at (A4h) {\small $\delta$};

	\node[shift={(240:0.3)}] at (B2) {\small $\gamma$};
	\node[shift={(330:0.4)}] at (B2) {\small $\gamma$};
	\node[shift={(60:0.3)}] at (B2) {\small $\gamma$};
	\node[shift={(150:0.4)}] at (B2) {\small $\gamma$};
	
	\node[shift={(250:0.25)}] at (A8h) {\small $\delta$};
	\node[shift={(-5:0.28)}] at (A8h) {\small $\beta$};
	\node[shift={(88:0.4)}] at (A8h) {\small $\beta$};
	\node[shift={(150:0.4)}] at (A8h) {\small $\delta$};

	
	\node[shift={(0.08,-0.3)}] at (A1) {\small $\delta$};
	\node[shift={(-0.38,-0.3)}] at (A1) {\small $\delta$};
	
	\node[shift={(-60:0.2)}] at (A2) {\small $\alpha$};
	
	\node[shift={(0.45,0.2)}] at (A3) {\small $\beta$};
	\node[shift={(0.1,-0.3)}] at (A3) {\small $\beta$};
	
	\node[shift={(0.28,0.2)}] at (A4) {\small $\alpha$};
	\node[shift={(0.28,-0.25)}] at (A4) {\small $\alpha$};
	
	\node[shift={(0.45,-0.2)}] at (A5) {\small $\delta$};
	\node[shift={(0.1,0.3)}] at (A5) {\small $\delta$};
	
	\node[shift={(60:0.2)}] at (A6) {\small $\alpha$};
	
	\node[shift={(0.1,0.33)}] at (A7) {\small $\beta$};
	\node[shift={(-0.4,0.33)}] at (A7) {\small $\beta$};
	
	\node[shift={(-0.3,0.1)}] at (A8) {\small $\alpha$};
	\node[shift={(0.04,0.3)}] at (A8) {\small $\alpha$};
	
	\node[shift={(-0.28,0.03)}] at (A9) {\small $\delta$};
	\node[shift={(-0.02,0.4)}] at (A9) {\small $\delta$};
	
	\node[shift={(180:0.2)}] at (A10) {\small $\alpha$};

	\node[shift={(-0.28,-0.06)}] at (A11) {\small $\beta$};
	\node[shift={(-0.06,-0.43)}] at (A11) {\small $\beta$};

	\node[shift={(-0.3,-0.1)}] at (A12) {\small $\alpha$};
	\node[shift={(0.04,-0.3)}] at (A12) {\small $\alpha$};
	
	\end{scope}
\end{tikzpicture}
\caption{Relation between Two Tilings}
\label{a3RelnFlip}
\end{figure}

Next we first provide a different and more useful perspective that explains the relation between $(6,4)$-earth map tiling in Figure \ref{pqEMT1624} and the tilings in Figure \ref{Tf24b2d2} and Figure \ref{Tf24a2d2ad4}. The time zone in Figure \ref{pqEMT} can be viewed as a disk in the first picture of Figure \ref{TZFlip}. Flipping the disk along the dashed line, we get the modification in the second picture of Figure \ref{TZFlip}. The tilings mentioned are obtained by gluing different combinations of such disks. For instance, the $(6,4)$-earth map tiling is obtained in the first row of Figure \ref{RelnViaFlipEg} by gluing along the five edges from each of the neighbouring disks and then gluing the remaining five edges from the first disk with those from the last disk. Likewise, the other two tilings are obtained in the second and the third row of Figure \ref{RelnViaFlipEg}.

\begin{figure}[htp]
	\centering
	\begin{tikzpicture}[>=latex,scale=0.7]

\begin{scope}[]
			\draw[red, line width=2, dashed]
			(108:4) -- (288:4);
			\end{scope}

\foreach \b in {0}
			{
			\begin{scope}[]			
			\foreach \a in {0,1,2,3,4,5,6,7,8,9}
			\draw[rotate=36*\a]
			(90:3) -- (126:3);
			
			\coordinate (O) at (0,0);
			\coordinate (A) at (90:0.8);
			\coordinate (B) at (270:0.8);
			\coordinate (L) at (162:1.8);
			\coordinate (R) at (342:1.8);
			
			\coordinate (A1) at (90:3);
			\coordinate (A2) at (126:3);
			\coordinate (A3) at (162:3);
			\coordinate (A4) at (198:3);
			\coordinate (A5) at (234:3);
			\coordinate (A6) at (270:3);
			\coordinate (A7) at (306:3);
			\coordinate (A8) at (342:3);
			\coordinate (A9) at (18:3);
			\coordinate (A10) at (54:3);
			
			\draw 
			(54:3) -- (A)--(B) -- (234:3)
			(L) -- (90:3)
			(L) -- (162:3)
			(L) -- (234:3)
			(L) -- (A)
			(R) -- (270:3)
			(R) -- (342:3)
			(R) -- (54:3)
			(R) -- (B);
			
			
			\draw[]
			(A6) -- (A7) -- (A8) -- (A9) -- (A10) -- (A1);
			
			\draw[]
			(A1) -- (A2) -- (A3) -- (A4) -- (A5) -- (A6);

			
			\draw[double, line width=0.6] 
			(A3) -- (L) -- (A)
			(A8) -- (R) -- (B);
			
			
			\draw[line width=2]
			(A1) -- (L) -- (A5)
			(A10) -- (R) -- (A6);
			
			
			\node[shift={(0.45,0.25)}] at (L) {\small $\gamma$};
			\node[shift={(-0.1,0.25)}] at (L) {\small $\gamma$};
			\node[shift={(-0.2,-0.25)}] at (L) {\small $\gamma$};
			\node[shift={(0.2,-0.25)}] at (L) {\small $\gamma$};
			
			
			\node[shift={(0.02,-0.3)}] at (A1) {\small $\delta$};
			\node[shift={(-0.45,-0.32)}] at (A1) {\small $\delta$};
			
			\node[shift={(306:0.2)}] at (A2) {\small $\alpha$};
			
			\node[shift={(0.3,0.15)}] at (A3) {\small $\beta$};
			\node[shift={(0.2,-0.35)}] at (A3) {\small $\beta$};
			
			\node[shift={(18:0.2)}] at (A4) {\small $\alpha$};
			
			\node[shift={(-0.2,0.65)}] at (A5) {\small $\delta$};
			\node[shift={(0.2,0.5)}] at (A5) {\small $\delta$};
			\node[shift={(0.4,0.08)}] at (A5) {\small $\alpha$};
			
			
			\node[shift={(0,0.22)}] at (A) {\small $\beta$};
			\node[shift={(-0.2,-0.3)}] at (A) {\small $\beta$};
			\node[shift={(0.2,-0.1)}] at (A) {\small $\alpha$};
			
			
			
			\node[shift={(-0.45,-0.3)}] at (R) {\small $\gamma$};
			\node[shift={(0.1,-0.3)}] at (R) {\small $\gamma$};
			\node[shift={(0.2,0.2)}] at (R) {\small $\gamma$};
			\node[shift={(-0.2,0.2)}] at (R) {\small $\gamma$};

			
			\node[shift={(-0.02,0.3)}] at (A6) {\small $\delta$};
			\node[shift={(0.45,0.32)}] at (A6) {\small $\delta$};
			
			\node[shift={(126:0.2)}] at (A7) {\small $\alpha$};
			
			\node[shift={(-0.3,-0.15)}] at (A8) {\small $\beta$};
			\node[shift={(-0.2,0.35)}] at (A8) {\small $\beta$};
			
			\node[shift={(198:0.2)}] at (A9) {\small $\alpha$};
			
			\node[shift={(0.2,-0.65)}] at (A10) {\small $\delta$};
			\node[shift={(-0.2,-0.5)}] at (A10) {\small $\delta$};
			\node[shift={(-0.4,-0.08)}] at (A10) {\small $\alpha$};

			
			\node[shift={(0,-0.25)}] at (B) {\small $\beta$};
			\node[shift={(0.2,0.25)}] at (B) {\small $\beta$};
			\node[shift={(-0.2,0.1)}] at (B) {\small $\alpha$};

			\end{scope}
			}

			\begin{scope}[xshift=8 cm, yshift=0 cm]
			\foreach \a in {0,1,2,3,4,5,6,7,8,9}
			\draw[rotate=36*\a]
			(90:3) -- (126:3);
			
			\coordinate (O) at (0,0);
			\coordinate (A) at (126:0.8);
			\coordinate (B) at (306:0.8);
			\coordinate (L) at (234:1.8);
			\coordinate (R) at (54:1.8);

			\coordinate (A1) at (90:3);
			\coordinate (A2) at (126:3);
			\coordinate (A3) at (162:3);
			\coordinate (A4) at (198:3);
			\coordinate (A5) at (234:3);
			\coordinate (A6) at (270:3);
			\coordinate (A7) at (306:3);
			\coordinate (A8) at (342:3);
			\coordinate (A9) at (18:3);
			\coordinate (A10) at (54:3);
			\draw 
			(A3) -- (A)--(B) -- (A8)
			(R) -- (A2)
			(R) -- (A10)
			(R) -- (A8)
			(R) -- (A)
			(L) -- (A7)
			(L) -- (A5)
			(L) -- (A3)
			(L) -- (B);
			
			
			\draw[]
			(A6) -- (A7) -- (A8) -- (A9) -- (A10) -- (A1);
			
			\draw[]
			(A1) -- (A2) -- (A3) -- (A4) -- (A5) -- (A6);

			
			\draw[double, line width=0.6] 
			(A10) -- (R) -- (A)
			(A5) -- (L) -- (B);
			
			
			\draw[line width=2]
			(A3) -- (L) -- (A7)
			(A2) -- (R) -- (A8);
			
			
			
			\node[shift={(0.45,0.02)}] at (L) {\small $\gamma$};
			\node[shift={(0.05,0.25)}] at (L) {\small $\gamma$};
			\node[shift={(-0.25,-0.03)}] at (L) {\small $\gamma$};
			\node[shift={(0.08,-0.3)}] at (L) {\small $\gamma$};

			
			\node[shift={(0.4,0.2)}] at (A3) {\small $\alpha$};
			\node[shift={(0.45,-0.26)}] at (A3) {\small $\delta$};
			\node[shift={(0.18,-0.55)}] at (A3) {\small $\delta$};
			
			\node[shift={(18:0.2)}] at (A4) {\small $\alpha$};
			
			\node[shift={(0,0.4)}] at (A5) {\small $\beta$};
			\node[shift={(0.4,0.1)}] at (A5) {\small $\beta$};
			
			\node[shift={(90:0.2)}] at (A6) {\small $\alpha$};
			
			\node[shift={(-0.65,0.01)}] at (A7) {\small $\delta$};
			\node[shift={(-0.02,0.3)}] at (A7) {\small $\delta$};

			\node[shift={(-0.4,0.02)}] at (B) {\small $\beta$};
			\node[shift={(0.02,-0.3)}] at (B) {\small $\beta$};
			\node[shift={(0.15,0.2)}] at (B) {\small $\alpha$};
			
			
			
			\node[shift={(-0.45,-0.04)}] at (R) {\small $\gamma$};
			\node[shift={(-0.05,-0.3)}] at (R) {\small $\gamma$};
			\node[shift={(0.25,0.01)}] at (R) {\small $\gamma$};
			\node[shift={(-0.08,0.25)}] at (R) {\small $\gamma$};

			
			\node[shift={(-0.4,-0.2)}] at (A8) {\small $\alpha$};
			\node[shift={(-0.45,0.26)}] at (A8) {\small $\delta$};
			\node[shift={(-0.18,0.55)}] at (A8) {\small $\delta$};
			
			\node[shift={(198:0.2)}] at (A9) {\small $\alpha$};
			
			\node[shift={(0,-0.4)}] at (A10) {\small $\beta$};
			\node[shift={(-0.4,-0.1)}] at (A10) {\small $\beta$};
			
			\node[shift={(270:0.2)}] at (A1) {\small $\alpha$};
			
			\node[shift={(0.65,-0.01)}] at (A2) {\small $\delta$};
			\node[shift={(0.02,-0.3)}] at (A2) {\small $\delta$};
			
			\node[shift={(0.4,-0.08)}] at (A) {\small $\beta$};
			\node[shift={(-0.02,0.25)}] at (A) {\small $\beta$};
			\node[shift={(-0.15,-0.2)}] at (A) {\small $\alpha$};

			\end{scope}
	\end{tikzpicture}
\caption{Time Zone Disk and Its Flip Modification}
\label{TZFlip}
\end{figure}
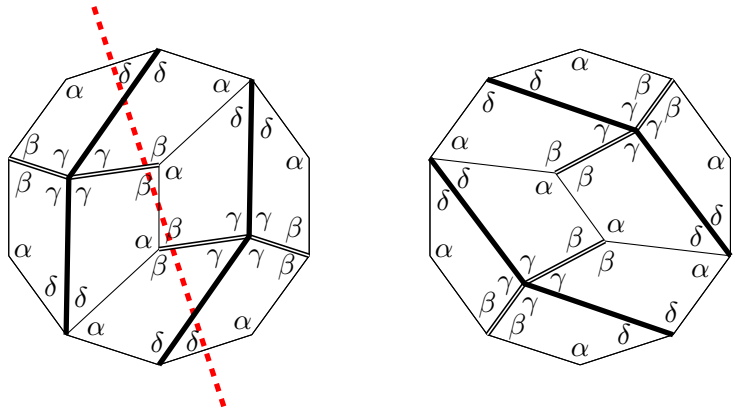

\begin{figure}[htp]
	\centering
	\begin{tikzpicture}[>=latex,scale=0.7]

	\foreach \b in {0,1,2}
	{
		\begin{scope}[xshift=0 cm, yshift=-7*\b cm]
		
		\foreach \a in {0,1,2,3,4,5,6,7,8,9}
		\draw[rotate=36*\a]
		(90:3) -- (126:3);
		
		\coordinate (O) at (0,0);
		\coordinate (A) at (90:0.8);
		\coordinate (B) at (270:0.8);
		\coordinate (L) at (162:1.8);
		\coordinate (R) at (342:1.8);
		
		\coordinate (A1) at (90:3);
		\coordinate (A2) at (126:3);
		\coordinate (A3) at (162:3);
		\coordinate (A4) at (198:3);
		\coordinate (A5) at (234:3);
		\coordinate (A6) at (270:3);
		\coordinate (A7) at (306:3);
		\coordinate (A8) at (342:3);
		\coordinate (A9) at (18:3);
		\coordinate (A10) at (54:3);
		
		\draw 
		(54:3) -- (A)--(B) -- (234:3)
		(L) -- (90:3)
		(L) -- (162:3)
		(L) -- (234:3)
		(L) -- (A)
		(R) -- (270:3)
		(R) -- (342:3)
		(R) -- (54:3)
		(R) -- (B);
		
		
		\draw[red]
		(A1) -- (A2) -- (A3) -- (A4) -- (A5) -- (A6);
		
		\draw[green]
		(A6) -- (A7) -- (A8) -- (A9) -- (A10) -- (A1);
		
		
		\draw[double, line width=0.6] 
		(A3) -- (L) -- (A)
		(A8) -- (R) -- (B);
		
		
		\draw[line width=2]
		(A1) -- (L) -- (A5)
		(A10) -- (R) -- (A6);
		
		
		\node[shift={(0.45,0.25)}] at (L) {\small $\gamma$};
		\node[shift={(-0.1,0.25)}] at (L) {\small $\gamma$};
		\node[shift={(-0.2,-0.25)}] at (L) {\small $\gamma$};
		\node[shift={(0.2,-0.25)}] at (L) {\small $\gamma$};
		
		
		\node[shift={(0.02,-0.3)}] at (A1) {\small $\delta$};
		\node[shift={(-0.45,-0.32)}] at (A1) {\small $\delta$};
		
		\node[shift={(306:0.2)}] at (A2) {\small $\alpha$};
		
		\node[shift={(0.3,0.15)}] at (A3) {\small $\beta$};
		\node[shift={(0.2,-0.35)}] at (A3) {\small $\beta$};
		
		\node[shift={(18:0.2)}] at (A4) {\small $\alpha$};
		
		\node[shift={(-0.2,0.65)}] at (A5) {\small $\delta$};
		\node[shift={(0.2,0.5)}] at (A5) {\small $\delta$};
		\node[shift={(0.4,0.08)}] at (A5) {\small $\alpha$};
		
		
		\node[shift={(0,0.22)}] at (A) {\small $\beta$};
		\node[shift={(-0.2,-0.3)}] at (A) {\small $\beta$};
		\node[shift={(0.2,-0.1)}] at (A) {\small $\alpha$};
		
		
		
		\node[shift={(-0.45,-0.3)}] at (R) {\small $\gamma$};
		\node[shift={(0.1,-0.3)}] at (R) {\small $\gamma$};
		\node[shift={(0.2,0.2)}] at (R) {\small $\gamma$};
		\node[shift={(-0.2,0.2)}] at (R) {\small $\gamma$};

		
		\node[shift={(-0.02,0.3)}] at (A6) {\small $\delta$};
		\node[shift={(0.45,0.32)}] at (A6) {\small $\delta$};
		
		\node[shift={(126:0.2)}] at (A7) {\small $\alpha$};
		
		\node[shift={(-0.3,-0.15)}] at (A8) {\small $\beta$};
		\node[shift={(-0.2,0.35)}] at (A8) {\small $\beta$};
		
		\node[shift={(198:0.2)}] at (A9) {\small $\alpha$};
		
		\node[shift={(0.2,-0.65)}] at (A10) {\small $\delta$};
		\node[shift={(-0.2,-0.5)}] at (A10) {\small $\delta$};
		\node[shift={(-0.4,-0.08)}] at (A10) {\small $\alpha$};

		
		\node[shift={(0,-0.25)}] at (B) {\small $\beta$};
		\node[shift={(0.2,0.25)}] at (B) {\small $\beta$};
		\node[shift={(-0.2,0.1)}] at (B) {\small $\alpha$};

		\end{scope}
	}

		\begin{scope}[xshift=6cm, yshift=0 cm]
		\foreach \a in {0,1,2,3,4,5,6,7,8,9}
		\draw[rotate=36*\a]
		(90:3) -- (126:3);
		
		\coordinate (O) at (0,0);
		\coordinate (A) at (90:0.8);
		\coordinate (B) at (270:0.8);
		\coordinate (L) at (162:1.8);
		\coordinate (R) at (342:1.8);
		
		\coordinate (A1) at (90:3);
		\coordinate (A2) at (126:3);
		\coordinate (A3) at (162:3);
		\coordinate (A4) at (198:3);
		\coordinate (A5) at (234:3);
		\coordinate (A6) at (270:3);
		\coordinate (A7) at (306:3);
		\coordinate (A8) at (342:3);
		\coordinate (A9) at (18:3);
		\coordinate (A10) at (54:3);
		
		\draw 
		(54:3) -- (A)--(B) -- (234:3)
		(L) -- (90:3)
		(L) -- (162:3)
		(L) -- (234:3)
		(L) -- (A)
		(R) -- (270:3)
		(R) -- (342:3)
		(R) -- (54:3)
		(R) -- (B);
		
		
		\draw[green]
		(A1) -- (A2) -- (A3) -- (A4) -- (A5) -- (A6);
		
		\draw[blue]
		(A6) -- (A7) -- (A8) -- (A9) -- (A10) -- (A1);
		
		
		\draw[double, line width=0.6] 
		(A3) -- (L) -- (A)
		(A8) -- (R) -- (B);
		
		
		\draw[line width=2]
		(A1) -- (L) -- (A5)
		(A10) -- (R) -- (A6);
		
		
		\node[shift={(0.45,0.25)}] at (L) {\small $\gamma$};
		\node[shift={(-0.1,0.25)}] at (L) {\small $\gamma$};
		\node[shift={(-0.2,-0.25)}] at (L) {\small $\gamma$};
		\node[shift={(0.2,-0.25)}] at (L) {\small $\gamma$};
		
		
		\node[shift={(0.02,-0.3)}] at (A1) {\small $\delta$};
		\node[shift={(-0.45,-0.32)}] at (A1) {\small $\delta$};
		
		\node[shift={(306:0.2)}] at (A2) {\small $\alpha$};
		
		\node[shift={(0.3,0.15)}] at (A3) {\small $\beta$};
		\node[shift={(0.2,-0.35)}] at (A3) {\small $\beta$};
		
		\node[shift={(18:0.2)}] at (A4) {\small $\alpha$};
		
		\node[shift={(-0.2,0.65)}] at (A5) {\small $\delta$};
		\node[shift={(0.2,0.5)}] at (A5) {\small $\delta$};
		\node[shift={(0.4,0.08)}] at (A5) {\small $\alpha$};
		
		
		\node[shift={(0,0.22)}] at (A) {\small $\beta$};
		\node[shift={(-0.2,-0.3)}] at (A) {\small $\beta$};
		\node[shift={(0.2,-0.1)}] at (A) {\small $\alpha$};
		
		
		
		\node[shift={(-0.45,-0.3)}] at (R) {\small $\gamma$};
		\node[shift={(0.1,-0.3)}] at (R) {\small $\gamma$};
		\node[shift={(0.2,0.2)}] at (R) {\small $\gamma$};
		\node[shift={(-0.2,0.2)}] at (R) {\small $\gamma$};

		
		\node[shift={(-0.02,0.3)}] at (A6) {\small $\delta$};
		\node[shift={(0.45,0.32)}] at (A6) {\small $\delta$};
		
		\node[shift={(126:0.2)}] at (A7) {\small $\alpha$};
		
		\node[shift={(-0.3,-0.15)}] at (A8) {\small $\beta$};
		\node[shift={(-0.2,0.35)}] at (A8) {\small $\beta$};
		
		\node[shift={(198:0.2)}] at (A9) {\small $\alpha$};
		
		\node[shift={(0.2,-0.65)}] at (A10) {\small $\delta$};
		\node[shift={(-0.2,-0.5)}] at (A10) {\small $\delta$};
		\node[shift={(-0.4,-0.08)}] at (A10) {\small $\alpha$};

		
		\node[shift={(0,-0.25)}] at (B) {\small $\beta$};
		\node[shift={(0.2,0.25)}] at (B) {\small $\beta$};
		\node[shift={(-0.2,0.1)}] at (B) {\small $\alpha$};

		\end{scope}

		\foreach \b in {1,2}
			{
			\begin{scope}[xshift=6 cm, yshift=-7*\b cm]
			\foreach \a in {0,1,2,3,4,5,6,7,8,9}
			\draw[rotate=36*\a]
			(90:3) -- (126:3);
			
			\coordinate (O) at (0,0);
			\coordinate (A) at (126:0.8);
			\coordinate (B) at (306:0.8);
			\coordinate (L) at (234:1.8);
			\coordinate (R) at (54:1.8);

			\coordinate (A1) at (90:3);
			\coordinate (A2) at (126:3);
			\coordinate (A3) at (162:3);
			\coordinate (A4) at (198:3);
			\coordinate (A5) at (234:3);
			\coordinate (A6) at (270:3);
			\coordinate (A7) at (306:3);
			\coordinate (A8) at (342:3);
			\coordinate (A9) at (18:3);
			\coordinate (A10) at (54:3);
			\draw 
			(A3) -- (A)--(B) -- (A8)
			(R) -- (A2)
			(R) -- (A10)
			(R) -- (A8)
			(R) -- (A)
			(L) -- (A7)
			(L) -- (A5)
			(L) -- (A3)
			(L) -- (B);
			
			
			\draw[green]
			(A1) -- (A2) -- (A3) -- (A4) -- (A5) -- (A6);
			
			\draw[blue]
			(A6) -- (A7) -- (A8) -- (A9) -- (A10) -- (A1);
			
			
			\draw[double, line width=0.6] 
			(A10) -- (R) -- (A)
			(A5) -- (L) -- (B);
			
			
			\draw[line width=2]
			(A3) -- (L) -- (A7)
			(A2) -- (R) -- (A8);
			
			
			
			\node[shift={(0.45,0.02)}] at (L) {\small $\gamma$};
			\node[shift={(0.05,0.25)}] at (L) {\small $\gamma$};
			\node[shift={(-0.25,-0.03)}] at (L) {\small $\gamma$};
			\node[shift={(0.08,-0.3)}] at (L) {\small $\gamma$};

			
			\node[shift={(0.4,0.2)}] at (A3) {\small $\alpha$};
			\node[shift={(0.45,-0.26)}] at (A3) {\small $\delta$};
			\node[shift={(0.18,-0.55)}] at (A3) {\small $\delta$};
			
			\node[shift={(18:0.2)}] at (A4) {\small $\alpha$};
			
			\node[shift={(0,0.4)}] at (A5) {\small $\beta$};
			\node[shift={(0.4,0.1)}] at (A5) {\small $\beta$};
			
			\node[shift={(90:0.2)}] at (A6) {\small $\alpha$};
			
			\node[shift={(-0.65,0.01)}] at (A7) {\small $\delta$};
			\node[shift={(-0.02,0.3)}] at (A7) {\small $\delta$};

			\node[shift={(-0.4,0.02)}] at (B) {\small $\beta$};
			\node[shift={(0.02,-0.3)}] at (B) {\small $\beta$};
			\node[shift={(0.15,0.2)}] at (B) {\small $\alpha$};
			
			
			
			\node[shift={(-0.45,-0.04)}] at (R) {\small $\gamma$};
			\node[shift={(-0.05,-0.3)}] at (R) {\small $\gamma$};
			\node[shift={(0.25,0.01)}] at (R) {\small $\gamma$};
			\node[shift={(-0.08,0.25)}] at (R) {\small $\gamma$};

			
			\node[shift={(-0.4,-0.2)}] at (A8) {\small $\alpha$};
			\node[shift={(-0.45,0.26)}] at (A8) {\small $\delta$};
			\node[shift={(-0.18,0.55)}] at (A8) {\small $\delta$};
			
			\node[shift={(198:0.2)}] at (A9) {\small $\alpha$};
			
			\node[shift={(0,-0.4)}] at (A10) {\small $\beta$};
			\node[shift={(-0.4,-0.1)}] at (A10) {\small $\beta$};
			
			\node[shift={(270:0.2)}] at (A1) {\small $\alpha$};
			
			\node[shift={(0.65,-0.01)}] at (A2) {\small $\delta$};
			\node[shift={(0.02,-0.3)}] at (A2) {\small $\delta$};
			
			\node[shift={(0.4,-0.08)}] at (A) {\small $\beta$};
			\node[shift={(-0.02,0.25)}] at (A) {\small $\beta$};
			\node[shift={(-0.15,-0.2)}] at (A) {\small $\alpha$};

			\end{scope}
			}
			
			\foreach \b in {0,1}
			{
			\begin{scope}[xshift=12 cm, yshift=-7*\b cm]			
			\foreach \a in {0,1,2,3,4,5,6,7,8,9}
			\draw[rotate=36*\a]
			(90:3) -- (126:3);
			
			\coordinate (O) at (0,0);
			\coordinate (A) at (90:0.8);
			\coordinate (B) at (270:0.8);
			\coordinate (L) at (162:1.8);
			\coordinate (R) at (342:1.8);
			
			\coordinate (A1) at (90:3);
			\coordinate (A2) at (126:3);
			\coordinate (A3) at (162:3);
			\coordinate (A4) at (198:3);
			\coordinate (A5) at (234:3);
			\coordinate (A6) at (270:3);
			\coordinate (A7) at (306:3);
			\coordinate (A8) at (342:3);
			\coordinate (A9) at (18:3);
			\coordinate (A10) at (54:3);
			
			\draw 
			(54:3) -- (A)--(B) -- (234:3)
			(L) -- (90:3)
			(L) -- (162:3)
			(L) -- (234:3)
			(L) -- (A)
			(R) -- (270:3)
			(R) -- (342:3)
			(R) -- (54:3)
			(R) -- (B);
			
			
			\draw[red]
			(A6) -- (A7) -- (A8) -- (A9) -- (A10) -- (A1);
			
			\draw[blue]
			(A1) -- (A2) -- (A3) -- (A4) -- (A5) -- (A6);

			
			\draw[double, line width=0.6] 
			(A3) -- (L) -- (A)
			(A8) -- (R) -- (B);
			
			
			\draw[line width=2]
			(A1) -- (L) -- (A5)
			(A10) -- (R) -- (A6);
			
			
			\node[shift={(0.45,0.25)}] at (L) {\small $\gamma$};
			\node[shift={(-0.1,0.25)}] at (L) {\small $\gamma$};
			\node[shift={(-0.2,-0.25)}] at (L) {\small $\gamma$};
			\node[shift={(0.2,-0.25)}] at (L) {\small $\gamma$};
			
			
			\node[shift={(0.02,-0.3)}] at (A1) {\small $\delta$};
			\node[shift={(-0.45,-0.32)}] at (A1) {\small $\delta$};
			
			\node[shift={(306:0.2)}] at (A2) {\small $\alpha$};
			
			\node[shift={(0.3,0.15)}] at (A3) {\small $\beta$};
			\node[shift={(0.2,-0.35)}] at (A3) {\small $\beta$};
			
			\node[shift={(18:0.2)}] at (A4) {\small $\alpha$};
			
			\node[shift={(-0.2,0.65)}] at (A5) {\small $\delta$};
			\node[shift={(0.2,0.5)}] at (A5) {\small $\delta$};
			\node[shift={(0.4,0.08)}] at (A5) {\small $\alpha$};
			
			
			\node[shift={(0,0.22)}] at (A) {\small $\beta$};
			\node[shift={(-0.2,-0.3)}] at (A) {\small $\beta$};
			\node[shift={(0.2,-0.1)}] at (A) {\small $\alpha$};
			
			
			
			\node[shift={(-0.45,-0.3)}] at (R) {\small $\gamma$};
			\node[shift={(0.1,-0.3)}] at (R) {\small $\gamma$};
			\node[shift={(0.2,0.2)}] at (R) {\small $\gamma$};
			\node[shift={(-0.2,0.2)}] at (R) {\small $\gamma$};

			
			\node[shift={(-0.02,0.3)}] at (A6) {\small $\delta$};
			\node[shift={(0.45,0.32)}] at (A6) {\small $\delta$};
			
			\node[shift={(126:0.2)}] at (A7) {\small $\alpha$};
			
			\node[shift={(-0.3,-0.15)}] at (A8) {\small $\beta$};
			\node[shift={(-0.2,0.35)}] at (A8) {\small $\beta$};
			
			\node[shift={(198:0.2)}] at (A9) {\small $\alpha$};
			
			\node[shift={(0.2,-0.65)}] at (A10) {\small $\delta$};
			\node[shift={(-0.2,-0.5)}] at (A10) {\small $\delta$};
			\node[shift={(-0.4,-0.08)}] at (A10) {\small $\alpha$};

			
			\node[shift={(0,-0.25)}] at (B) {\small $\beta$};
			\node[shift={(0.2,0.25)}] at (B) {\small $\beta$};
			\node[shift={(-0.2,0.1)}] at (B) {\small $\alpha$};

			\end{scope}
			}

			\begin{scope}[xshift=12 cm, yshift=-14 cm]
			\foreach \a in {0,1,2,3,4,5,6,7,8,9}
			\draw[rotate=36*\a]
			(90:3) -- (126:3);
			
			\coordinate (O) at (0,0);
			\coordinate (A) at (126:0.8);
			\coordinate (B) at (306:0.8);
			\coordinate (L) at (234:1.8);
			\coordinate (R) at (54:1.8);

			\coordinate (A1) at (90:3);
			\coordinate (A2) at (126:3);
			\coordinate (A3) at (162:3);
			\coordinate (A4) at (198:3);
			\coordinate (A5) at (234:3);
			\coordinate (A6) at (270:3);
			\coordinate (A7) at (306:3);
			\coordinate (A8) at (342:3);
			\coordinate (A9) at (18:3);
			\coordinate (A10) at (54:3);
			\draw 
			(A3) -- (A)--(B) -- (A8)
			(R) -- (A2)
			(R) -- (A10)
			(R) -- (A8)
			(R) -- (A)
			(L) -- (A7)
			(L) -- (A5)
			(L) -- (A3)
			(L) -- (B);
			
			
			\draw[red]
			(A6) -- (A7) -- (A8) -- (A9) -- (A10) -- (A1);
			
			\draw[blue]
			(A1) -- (A2) -- (A3) -- (A4) -- (A5) -- (A6);

			
			\draw[double, line width=0.6] 
			(A10) -- (R) -- (A)
			(A5) -- (L) -- (B);
			
			
			\draw[line width=2]
			(A3) -- (L) -- (A7)
			(A2) -- (R) -- (A8);
			
			
			
			\node[shift={(0.45,0.02)}] at (L) {\small $\gamma$};
			\node[shift={(0.05,0.25)}] at (L) {\small $\gamma$};
			\node[shift={(-0.25,-0.03)}] at (L) {\small $\gamma$};
			\node[shift={(0.08,-0.3)}] at (L) {\small $\gamma$};

			
			\node[shift={(0.4,0.2)}] at (A3) {\small $\alpha$};
			\node[shift={(0.45,-0.26)}] at (A3) {\small $\delta$};
			\node[shift={(0.18,-0.55)}] at (A3) {\small $\delta$};
			
			\node[shift={(18:0.2)}] at (A4) {\small $\alpha$};
			
			\node[shift={(0,0.4)}] at (A5) {\small $\beta$};
			\node[shift={(0.4,0.1)}] at (A5) {\small $\beta$};
			
			\node[shift={(90:0.2)}] at (A6) {\small $\alpha$};
			
			\node[shift={(-0.65,0.01)}] at (A7) {\small $\delta$};
			\node[shift={(-0.02,0.3)}] at (A7) {\small $\delta$};

			\node[shift={(-0.4,0.02)}] at (B) {\small $\beta$};
			\node[shift={(0.02,-0.3)}] at (B) {\small $\beta$};
			\node[shift={(0.15,0.2)}] at (B) {\small $\alpha$};
			
			
			
			\node[shift={(-0.45,-0.04)}] at (R) {\small $\gamma$};
			\node[shift={(-0.05,-0.3)}] at (R) {\small $\gamma$};
			\node[shift={(0.25,0.01)}] at (R) {\small $\gamma$};
			\node[shift={(-0.08,0.25)}] at (R) {\small $\gamma$};

			
			\node[shift={(-0.4,-0.2)}] at (A8) {\small $\alpha$};
			\node[shift={(-0.45,0.26)}] at (A8) {\small $\delta$};
			\node[shift={(-0.18,0.55)}] at (A8) {\small $\delta$};
			
			\node[shift={(198:0.2)}] at (A9) {\small $\alpha$};
			
			\node[shift={(0,-0.4)}] at (A10) {\small $\beta$};
			\node[shift={(-0.4,-0.1)}] at (A10) {\small $\beta$};
			
			\node[shift={(270:0.2)}] at (A1) {\small $\alpha$};
			
			\node[shift={(0.65,-0.01)}] at (A2) {\small $\delta$};
			\node[shift={(0.02,-0.3)}] at (A2) {\small $\delta$};
			
			\node[shift={(0.4,-0.08)}] at (A) {\small $\beta$};
			\node[shift={(-0.02,0.25)}] at (A) {\small $\beta$};
			\node[shift={(-0.15,-0.2)}] at (A) {\small $\alpha$};

			\end{scope}
\end{tikzpicture}
\caption{Relation Between Three Tilings}
\label{RelnViaFlipEg}
\end{figure}
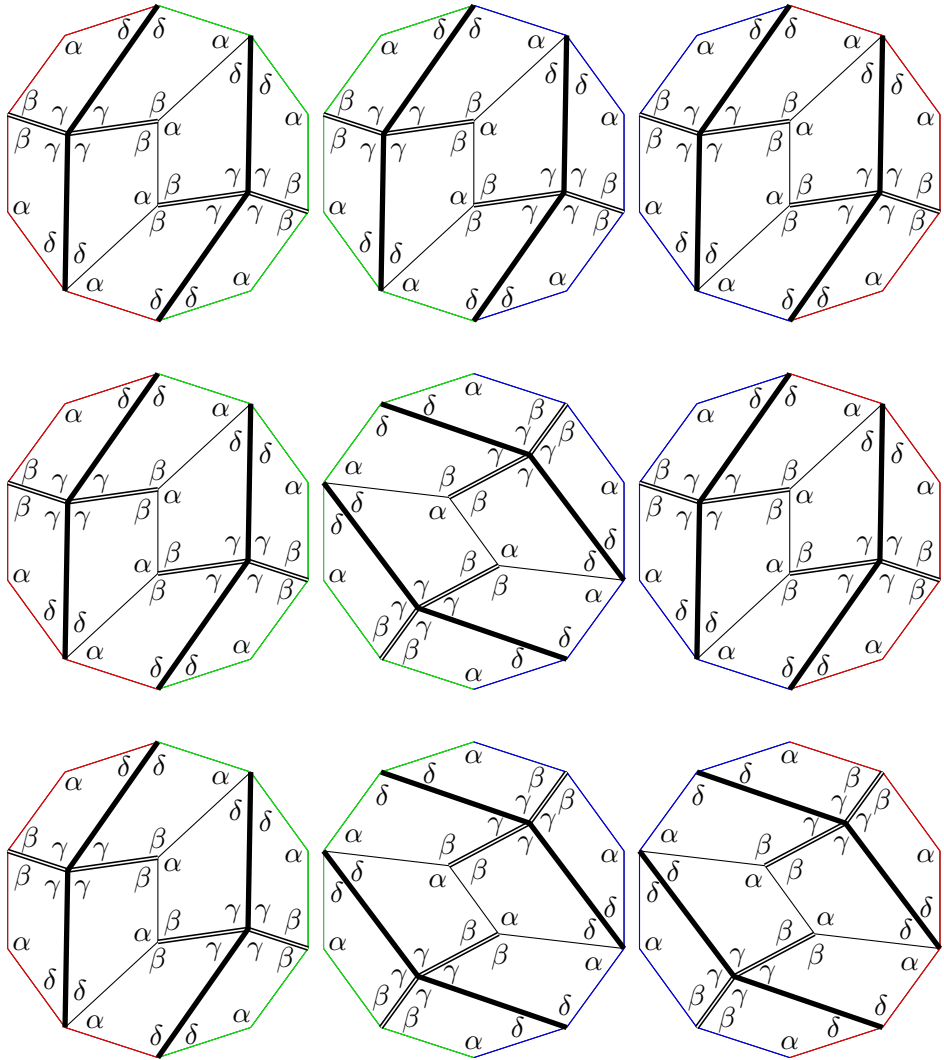

The new perspective can indeed be extended to explain the relation between the three families of tilings. First we have the corresponding disk of $\frac{f-8}{16}$ copies of time zones in the first picture of Figure \ref{TZCopiesFlip} and its flip modification along the dashed line given in the second picture of Figure \ref{TZCopiesFlip}. Next apply the same gluing method to $1$-time zone disk and two $\frac{f-8}{16}$-time zone disks, we get the $(\frac{f}{4},4)$-earth map tiling as illustrated in the first row of Figure \ref{RelnViaFlip}. Likewise the tilings of $\AVC \equiv \{ \alpha\beta^2, \alpha^2\delta^2, \gamma^4, \alpha\delta^{\frac{f+8}{8}} \}$ and $\AVC \equiv \{ \alpha\beta^2, \alpha^2\delta^2, \gamma^4, \alpha\delta^{\frac{f+8}{8}}, \beta^2\delta^{\frac{f-8}{8}} \}$ are obtained in the second and the third row of Figure \ref{RelnViaFlip}.

\begin{figure}[htp]
	\centering
	\begin{tikzpicture}[>=latex,scale=0.7]

\begin{scope}[]
			\draw[red, line width=2, dashed]
			(108:4) -- (288:4);
\end{scope}

	\foreach \a in {0,1,2,3,4,5,6,7,8,9}
	\draw[rotate=36*\a]
	(90:3) -- (126:3);
	
	\coordinate (O) at (0,0);
	\coordinate (A) at (90:0.8);
	\coordinate (B) at (270:0.8);
	\coordinate (L) at (162:1.8);
	\coordinate (R) at (342:1.8);
	
	\coordinate (A1) at (90:3);
	\coordinate (A2) at (126:3);
	\coordinate (A3) at (162:3);
	\coordinate (A4) at (198:3);
	\coordinate (A5) at (234:3);
	\coordinate (A6) at (270:3);
	\coordinate (A7) at (306:3);
	\coordinate (A8) at (342:3);
	\coordinate (A9) at (18:3);
	\coordinate (A10) at (54:3);

	\coordinate[shift={(-18:0.9)}] (A3L) at (A3);
	\coordinate[shift={(89:0.9)}] (A5L) at (A5);
	\coordinate[shift={(43:0.9)}] (A5B) at (A5);
	
	\coordinate[shift={(162:0.9)}] (A8R) at (A8);
	\coordinate[shift={(269:0.9)}] (A10R) at (A10);
	\coordinate[shift={(223:0.9)}] (A10A) at (A10);


	\draw
	(A5) -- (A5B)
	(A10) -- (A10A);

	
	\draw[double, line width=0.6] 
	(A3) -- (A3L)
	(A8) -- (A8R);
	
	
	\draw[line width=2]
	(A5) -- (A5L)
	(A10) -- (A10R);

	
	
	\node[shift={(0.02,-0.4)}] at (A1) {\small $\delta^{\frac{f-8}{8}}$};

	\node[shift={(306:0.2)}] at (A2) {\small $\alpha$};
	
	\node[shift={(0.3,0.15)}] at (A3) {\small $\beta$};
	\node[shift={(0.2,-0.35)}] at (A3) {\small $\beta$};
	
	\node[shift={(18:0.2)}] at (A4) {\small $\alpha$};
	
	\node[shift={(-0.2,0.65)}] at (A5) {\small $\delta$};
	\node[shift={(0.2,0.5)}] at (A5) {\small $\delta$};
	\node[shift={(0.4,0.08)}] at (A5) {\small $\alpha$};


	
	\node[shift={(-0.02,0.4)}] at (A6) {\small $\delta^{\frac{f-8}{8}}$};
	
	\node[shift={(126:0.2)}] at (A7) {\small $\alpha$};
	
	\node[shift={(-0.3,-0.15)}] at (A8) {\small $\beta$};
	\node[shift={(-0.2,0.35)}] at (A8) {\small $\beta$};
	
	\node[shift={(198:0.2)}] at (A9) {\small $\alpha$};
	
	\node[shift={(0.2,-0.65)}] at (A10) {\small $\delta$};
	\node[shift={(-0.2,-0.5)}] at (A10) {\small $\delta$};
	\node[shift={(-0.4,-0.08)}] at (A10) {\small $\alpha$};
	

\begin{scope}[xshift=8 cm, yshift=0 cm]

	\foreach \a in {0,1,2,3,4,5,6,7,8,9}
	\draw[rotate=36*\a]
	(90:3) -- (126:3);
	
	\coordinate (O) at (0,0);
	\coordinate (A) at (90:0.8);
	\coordinate (B) at (270:0.8);
	\coordinate (L) at (162:1.8);
	\coordinate (R) at (342:1.8);
	
	\coordinate (A1) at (90:3);
	\coordinate (A2) at (126:3);
	\coordinate (A3) at (162:3);
	\coordinate (A4) at (198:3);
	\coordinate (A5) at (234:3);
	\coordinate (A6) at (270:3);
	\coordinate (A7) at (306:3);
	\coordinate (A8) at (342:3);
	\coordinate (A9) at (18:3);
	\coordinate (A10) at (54:3);

	\coordinate[shift={(234:0.9)}] (A10R) at (A10);
	\coordinate[shift={(127:0.9)}] (A8R) at (A8);
	\coordinate[shift={(173:0.9)}] (A8B) at (A8);
	
	\coordinate[shift={(54:0.9)}] (A5L) at (A5);
	\coordinate[shift={(307:0.9)}] (A3L) at (A3);
	\coordinate[shift={(353:0.9)}] (A3A) at (A3);


	\draw
	(A8) -- (A8B)
	(A3) -- (A3A);

	
	\draw[double, line width=0.6] 
	(A10) -- (A10R)
	(A5) -- (A5L);
	
	
	\draw[line width=2]
	(A8) -- (A8R)
	(A3) -- (A3L);
	

	
	\node[shift={(0.4,0.2)}] at (A3) {\small $\alpha$};
	\node[shift={(0.45,-0.26)}] at (A3) {\small $\delta$};
	\node[shift={(0.18,-0.55)}] at (A3) {\small $\delta$};
	
	\node[shift={(18:0.2)}] at (A4) {\small $\alpha$};
	
	\node[shift={(0,0.4)}] at (A5) {\small $\beta$};
	\node[shift={(0.4,0.1)}] at (A5) {\small $\beta$};
	
	\node[shift={(90:0.2)}] at (A6) {\small $\alpha$};
	
	\node[shift={(-0.2,0.3)}] at (A7) {\small $\delta^{\frac{f-8}{8}}$};

	
	
	\node[shift={(-0.4,-0.2)}] at (A8) {\small $\alpha$};
	\node[shift={(-0.45,0.26)}] at (A8) {\small $\delta$};
	\node[shift={(-0.18,0.55)}] at (A8) {\small $\delta$};
	
	\node[shift={(198:0.2)}] at (A9) {\small $\alpha$};
	
	\node[shift={(0,-0.4)}] at (A10) {\small $\beta$};
	\node[shift={(-0.4,-0.1)}] at (A10) {\small $\beta$};
	
	\node[shift={(270:0.2)}] at (A1) {\small $\alpha$};
	
	\node[shift={(0.25,-0.3)}] at (A2) {\small $\delta^{\frac{f-8}{8}}$};
	

\end{scope}

\end{tikzpicture}
\caption{Disk of $\frac{f-8}{16}$ Copies of Time Zones and Its Flip Modification}
\label{TZCopiesFlip}
\end{figure}
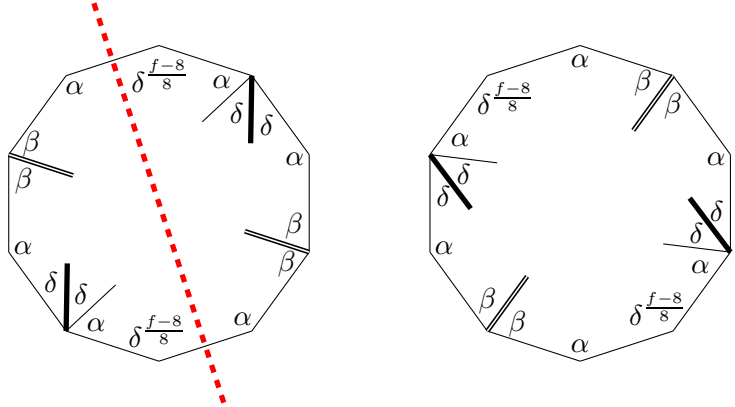

\begin{figure}[htp]
	\centering
	\begin{tikzpicture}[>=latex,scale=0.7]
	
	\foreach \b in {0,1,2}
	{
		\begin{scope}[xshift=0 cm, yshift=-7*\b cm]
		
		\foreach \a in {0,1,2,3,4,5,6,7,8,9}
		\draw[rotate=36*\a]
		(90:3) -- (126:3);
		
		\coordinate (O) at (0,0);
		\coordinate (A) at (90:0.8);
		\coordinate (B) at (270:0.8);
		\coordinate (L) at (162:1.8);
		\coordinate (R) at (342:1.8);
		
		\coordinate (A1) at (90:3);
		\coordinate (A2) at (126:3);
		\coordinate (A3) at (162:3);
		\coordinate (A4) at (198:3);
		\coordinate (A5) at (234:3);
		\coordinate (A6) at (270:3);
		\coordinate (A7) at (306:3);
		\coordinate (A8) at (342:3);
		\coordinate (A9) at (18:3);
		\coordinate (A10) at (54:3);
		
		\draw 
		(54:3) -- (A)--(B) -- (234:3)
		(L) -- (90:3)
		(L) -- (162:3)
		(L) -- (234:3)
		(L) -- (A)
		(R) -- (270:3)
		(R) -- (342:3)
		(R) -- (54:3)
		(R) -- (B);
		
		
		\draw[red]
		(A1) -- (A2) -- (A3) -- (A4) -- (A5) -- (A6);
		
		\draw[green]
		(A6) -- (A7) -- (A8) -- (A9) -- (A10) -- (A1);
		
		
		\draw[double, line width=0.6] 
		(A3) -- (L) -- (A)
		(A8) -- (R) -- (B);
		
		
		\draw[line width=2]
		(A1) -- (L) -- (A5)
		(A10) -- (R) -- (A6);
		
		
		\node[shift={(0.45,0.25)}] at (L) {\small $\gamma$};
		\node[shift={(-0.1,0.25)}] at (L) {\small $\gamma$};
		\node[shift={(-0.2,-0.25)}] at (L) {\small $\gamma$};
		\node[shift={(0.2,-0.25)}] at (L) {\small $\gamma$};
		
		
		\node[shift={(0.02,-0.3)}] at (A1) {\small $\delta$};
		\node[shift={(-0.45,-0.32)}] at (A1) {\small $\delta$};
		
		\node[shift={(306:0.2)}] at (A2) {\small $\alpha$};
		
		\node[shift={(0.3,0.15)}] at (A3) {\small $\beta$};
		\node[shift={(0.2,-0.35)}] at (A3) {\small $\beta$};
		
		\node[shift={(18:0.2)}] at (A4) {\small $\alpha$};
		
		\node[shift={(-0.2,0.65)}] at (A5) {\small $\delta$};
		\node[shift={(0.2,0.5)}] at (A5) {\small $\delta$};
		\node[shift={(0.4,0.08)}] at (A5) {\small $\alpha$};
		
		
		\node[shift={(0,0.22)}] at (A) {\small $\beta$};
		\node[shift={(-0.2,-0.3)}] at (A) {\small $\beta$};
		\node[shift={(0.2,-0.1)}] at (A) {\small $\alpha$};
		
		
		
		\node[shift={(-0.45,-0.3)}] at (R) {\small $\gamma$};
		\node[shift={(0.1,-0.3)}] at (R) {\small $\gamma$};
		\node[shift={(0.2,0.2)}] at (R) {\small $\gamma$};
		\node[shift={(-0.2,0.2)}] at (R) {\small $\gamma$};

		
		\node[shift={(-0.02,0.3)}] at (A6) {\small $\delta$};
		\node[shift={(0.45,0.32)}] at (A6) {\small $\delta$};
		
		\node[shift={(126:0.2)}] at (A7) {\small $\alpha$};
		
		\node[shift={(-0.3,-0.15)}] at (A8) {\small $\beta$};
		\node[shift={(-0.2,0.35)}] at (A8) {\small $\beta$};
		
		\node[shift={(198:0.2)}] at (A9) {\small $\alpha$};
		
		\node[shift={(0.2,-0.65)}] at (A10) {\small $\delta$};
		\node[shift={(-0.2,-0.5)}] at (A10) {\small $\delta$};
		\node[shift={(-0.4,-0.08)}] at (A10) {\small $\alpha$};

		
		\node[shift={(0,-0.25)}] at (B) {\small $\beta$};
		\node[shift={(0.2,0.25)}] at (B) {\small $\beta$};
		\node[shift={(-0.2,0.1)}] at (B) {\small $\alpha$};

		\end{scope}
	}
	
	\begin{scope}[xshift=6cm, yshift=0 cm]
	\foreach \a in {0,1,2,3,4,5,6,7,8,9}
	\draw[rotate=36*\a]
	(90:3) -- (126:3);
	
	\coordinate (O) at (0,0);
	\coordinate (A) at (90:0.8);
	\coordinate (B) at (270:0.8);
	\coordinate (L) at (162:1.8);
	\coordinate (R) at (342:1.8);
	
	\coordinate (A1) at (90:3);
	\coordinate (A2) at (126:3);
	\coordinate (A3) at (162:3);
	\coordinate (A4) at (198:3);
	\coordinate (A5) at (234:3);
	\coordinate (A6) at (270:3);
	\coordinate (A7) at (306:3);
	\coordinate (A8) at (342:3);
	\coordinate (A9) at (18:3);
	\coordinate (A10) at (54:3);

	\coordinate[shift={(-18:0.9)}] (A3L) at (A3);
	\coordinate[shift={(89:0.9)}] (A5L) at (A5);
	\coordinate[shift={(43:0.9)}] (A5B) at (A5);
	
	\coordinate[shift={(162:0.9)}] (A8R) at (A8);
	\coordinate[shift={(269:0.9)}] (A10R) at (A10);
	\coordinate[shift={(223:0.9)}] (A10A) at (A10);


	\draw
	(A5) -- (A5B)
	(A10) -- (A10A);

		
	\draw[green]
	(A1) -- (A2) -- (A3) -- (A4) -- (A5) -- (A6);
	
	\draw[blue]
	(A6) -- (A7) -- (A8) -- (A9) -- (A10) -- (A1);

	
	\draw[double, line width=0.6] 
	(A3) -- (A3L)
	(A8) -- (A8R);
	
	
	\draw[line width=2]
	(A5) -- (A5L)
	(A10) -- (A10R);

	
	
	\node[shift={(0.02,-0.4)}] at (A1) {\small $\delta^{\frac{f-8}{8}}$};

	\node[shift={(306:0.2)}] at (A2) {\small $\alpha$};
	
	\node[shift={(0.3,0.15)}] at (A3) {\small $\beta$};
	\node[shift={(0.2,-0.35)}] at (A3) {\small $\beta$};
	
	\node[shift={(18:0.2)}] at (A4) {\small $\alpha$};
	
	\node[shift={(-0.2,0.65)}] at (A5) {\small $\delta$};
	\node[shift={(0.2,0.5)}] at (A5) {\small $\delta$};
	\node[shift={(0.4,0.08)}] at (A5) {\small $\alpha$};


	
	\node[shift={(-0.02,0.4)}] at (A6) {\small $\delta^{\frac{f-8}{8}}$};
	
	\node[shift={(126:0.2)}] at (A7) {\small $\alpha$};
	
	\node[shift={(-0.3,-0.15)}] at (A8) {\small $\beta$};
	\node[shift={(-0.2,0.35)}] at (A8) {\small $\beta$};
	
	\node[shift={(198:0.2)}] at (A9) {\small $\alpha$};
	
	\node[shift={(0.2,-0.65)}] at (A10) {\small $\delta$};
	\node[shift={(-0.2,-0.5)}] at (A10) {\small $\delta$};
	\node[shift={(-0.4,-0.08)}] at (A10) {\small $\alpha$};

	
	\end{scope}

	\foreach \b in {1,2}
	{
	\begin{scope}[xshift=6 cm, yshift=-7*\b cm]
	\foreach \a in {0,1,2,3,4,5,6,7,8,9}
	\draw[rotate=36*\a]
	(90:3) -- (126:3);
	
	\coordinate (O) at (0,0);
	\coordinate (A) at (90:0.8);
	\coordinate (B) at (270:0.8);
	\coordinate (L) at (162:1.8);
	\coordinate (R) at (342:1.8);
	
	\coordinate (A1) at (90:3);
	\coordinate (A2) at (126:3);
	\coordinate (A3) at (162:3);
	\coordinate (A4) at (198:3);
	\coordinate (A5) at (234:3);
	\coordinate (A6) at (270:3);
	\coordinate (A7) at (306:3);
	\coordinate (A8) at (342:3);
	\coordinate (A9) at (18:3);
	\coordinate (A10) at (54:3);

	\coordinate[shift={(234:0.9)}] (A10R) at (A10);
	\coordinate[shift={(127:0.9)}] (A8R) at (A8);
	\coordinate[shift={(173:0.9)}] (A8B) at (A8);
	
	\coordinate[shift={(54:0.9)}] (A5L) at (A5);
	\coordinate[shift={(307:0.9)}] (A3L) at (A3);
	\coordinate[shift={(353:0.9)}] (A3A) at (A3);


	\draw
	(A8) -- (A8B)
	(A3) -- (A3A);
	
	
	\draw[green]
	(A1) -- (A2) -- (A3) -- (A4) -- (A5) -- (A6);
	
	\draw[blue]
	(A6) -- (A7) -- (A8) -- (A9) -- (A10) -- (A1);
	
	
	\draw[double, line width=0.6] 
	(A10) -- (A10R)
	(A5) -- (A5L);
	
	
	\draw[line width=2]
	(A8) -- (A8R)
	(A3) -- (A3L);
	

	
	\node[shift={(0.4,0.2)}] at (A3) {\small $\alpha$};
	\node[shift={(0.45,-0.26)}] at (A3) {\small $\delta$};
	\node[shift={(0.18,-0.55)}] at (A3) {\small $\delta$};
	
	\node[shift={(18:0.2)}] at (A4) {\small $\alpha$};
	
	\node[shift={(0,0.4)}] at (A5) {\small $\beta$};
	\node[shift={(0.4,0.1)}] at (A5) {\small $\beta$};
	
	\node[shift={(90:0.2)}] at (A6) {\small $\alpha$};
	
	\node[shift={(-0.2,0.3)}] at (A7) {\small $\delta^{\frac{f-8}{8}}$};

	
	
	\node[shift={(-0.4,-0.2)}] at (A8) {\small $\alpha$};
	\node[shift={(-0.45,0.26)}] at (A8) {\small $\delta$};
	\node[shift={(-0.18,0.55)}] at (A8) {\small $\delta$};
	
	\node[shift={(198:0.2)}] at (A9) {\small $\alpha$};
	
	\node[shift={(0,-0.4)}] at (A10) {\small $\beta$};
	\node[shift={(-0.4,-0.1)}] at (A10) {\small $\beta$};
	
	\node[shift={(270:0.2)}] at (A1) {\small $\alpha$};
	
	\node[shift={(0.25,-0.3)}] at (A2) {\small $\delta^{\frac{f-8}{8}}$};

	\end{scope}
	}

	\foreach \b in {0,1}
	{
	\begin{scope}[xshift=12 cm, yshift=-7*\b cm]
	\foreach \a in {0,1,2,3,4,5,6,7,8,9}
	\draw[rotate=36*\a]
	(90:3) -- (126:3);
	
	\coordinate (O) at (0,0);
	\coordinate (A) at (90:0.8);
	\coordinate (B) at (270:0.8);
	\coordinate (L) at (162:1.8);
	\coordinate (R) at (342:1.8);
	
	\coordinate (A1) at (90:3);
	\coordinate (A2) at (126:3);
	\coordinate (A3) at (162:3);
	\coordinate (A4) at (198:3);
	\coordinate (A5) at (234:3);
	\coordinate (A6) at (270:3);
	\coordinate (A7) at (306:3);
	\coordinate (A8) at (342:3);
	\coordinate (A9) at (18:3);
	\coordinate (A10) at (54:3);

	\coordinate[shift={(-18:0.9)}] (A3L) at (A3);
	\coordinate[shift={(89:0.9)}] (A5L) at (A5);
	\coordinate[shift={(43:0.9)}] (A5B) at (A5);
	
	\coordinate[shift={(162:0.9)}] (A8R) at (A8);
	\coordinate[shift={(269:0.9)}] (A10R) at (A10);
	\coordinate[shift={(223:0.9)}] (A10A) at (A10);


	\draw
	(A5) -- (A5B)
	(A10) -- (A10A);

	
	\draw[red]
	(A6) -- (A7) -- (A8) -- (A9) -- (A10) -- (A1);
	
	\draw[blue]
	(A1) -- (A2) -- (A3) -- (A4) -- (A5) -- (A6);
	
	
	\draw[double, line width=0.6] 
	(A3) -- (A3L)
	(A8) -- (A8R);
	
	
	\draw[line width=2]
	(A5) -- (A5L)
	(A10) -- (A10R);

	
	
	\node[shift={(0.02,-0.4)}] at (A1) {\small $\delta^{\frac{f-8}{8}}$};

	\node[shift={(306:0.2)}] at (A2) {\small $\alpha$};
	
	\node[shift={(0.3,0.15)}] at (A3) {\small $\beta$};
	\node[shift={(0.2,-0.35)}] at (A3) {\small $\beta$};
	
	\node[shift={(18:0.2)}] at (A4) {\small $\alpha$};
	
	\node[shift={(-0.2,0.65)}] at (A5) {\small $\delta$};
	\node[shift={(0.2,0.5)}] at (A5) {\small $\delta$};
	\node[shift={(0.4,0.08)}] at (A5) {\small $\alpha$};


	
	\node[shift={(-0.02,0.4)}] at (A6) {\small $\delta^{\frac{f-8}{8}}$};
	
	\node[shift={(126:0.2)}] at (A7) {\small $\alpha$};
	
	\node[shift={(-0.3,-0.15)}] at (A8) {\small $\beta$};
	\node[shift={(-0.2,0.35)}] at (A8) {\small $\beta$};
	
	\node[shift={(198:0.2)}] at (A9) {\small $\alpha$};
	
	\node[shift={(0.2,-0.65)}] at (A10) {\small $\delta$};
	\node[shift={(-0.2,-0.5)}] at (A10) {\small $\delta$};
	\node[shift={(-0.4,-0.08)}] at (A10) {\small $\alpha$};
	
	\end{scope}
	}

	\begin{scope}[xshift=12 cm, yshift=-14 cm]
	\foreach \a in {0,1,2,3,4,5,6,7,8,9}
	\draw[rotate=36*\a]
	(90:3) -- (126:3);
	
	\coordinate (O) at (0,0);
	\coordinate (A) at (90:0.8);
	\coordinate (B) at (270:0.8);
	\coordinate (L) at (162:1.8);
	\coordinate (R) at (342:1.8);
	
	\coordinate (A1) at (90:3);
	\coordinate (A2) at (126:3);
	\coordinate (A3) at (162:3);
	\coordinate (A4) at (198:3);
	\coordinate (A5) at (234:3);
	\coordinate (A6) at (270:3);
	\coordinate (A7) at (306:3);
	\coordinate (A8) at (342:3);
	\coordinate (A9) at (18:3);
	\coordinate (A10) at (54:3);

	\coordinate[shift={(234:0.9)}] (A10R) at (A10);
	\coordinate[shift={(127:0.9)}] (A8R) at (A8);
	\coordinate[shift={(173:0.9)}] (A8B) at (A8);
	
	\coordinate[shift={(54:0.9)}] (A5L) at (A5);
	\coordinate[shift={(307:0.9)}] (A3L) at (A3);
	\coordinate[shift={(353:0.9)}] (A3A) at (A3);


	\draw
	(A8) -- (A8B)
	(A3) -- (A3A);
	
	
	\draw[red]
	(A6) -- (A7) -- (A8) -- (A9) -- (A10) -- (A1);
	
	\draw[blue]
	(A1) -- (A2) -- (A3) -- (A4) -- (A5) -- (A6);
	
	
	\draw[double, line width=0.6] 
	(A10) -- (A10R)
	(A5) -- (A5L);
	
	
	\draw[line width=2]
	(A8) -- (A8R)
	(A3) -- (A3L);
	

	
	\node[shift={(0.4,0.2)}] at (A3) {\small $\alpha$};
	\node[shift={(0.45,-0.26)}] at (A3) {\small $\delta$};
	\node[shift={(0.18,-0.55)}] at (A3) {\small $\delta$};
	
	\node[shift={(18:0.2)}] at (A4) {\small $\alpha$};
	
	\node[shift={(0,0.4)}] at (A5) {\small $\beta$};
	\node[shift={(0.4,0.1)}] at (A5) {\small $\beta$};
	
	\node[shift={(90:0.2)}] at (A6) {\small $\alpha$};
	
	\node[shift={(-0.2,0.3)}] at (A7) {\small $\delta^{\frac{f-8}{8}}$};

	
	
	\node[shift={(-0.4,-0.2)}] at (A8) {\small $\alpha$};
	\node[shift={(-0.45,0.26)}] at (A8) {\small $\delta$};
	\node[shift={(-0.18,0.55)}] at (A8) {\small $\delta$};
	
	\node[shift={(198:0.2)}] at (A9) {\small $\alpha$};
	
	\node[shift={(0,-0.4)}] at (A10) {\small $\beta$};
	\node[shift={(-0.4,-0.1)}] at (A10) {\small $\beta$};
	
	\node[shift={(270:0.2)}] at (A1) {\small $\alpha$};
	
	\node[shift={(0.25,-0.3)}] at (A2) {\small $\delta^{\frac{f-8}{8}}$};

	\end{scope}
\end{tikzpicture}
\caption{Relation among Three Family Tilings.}
\label{RelnViaFlip}
\end{figure}
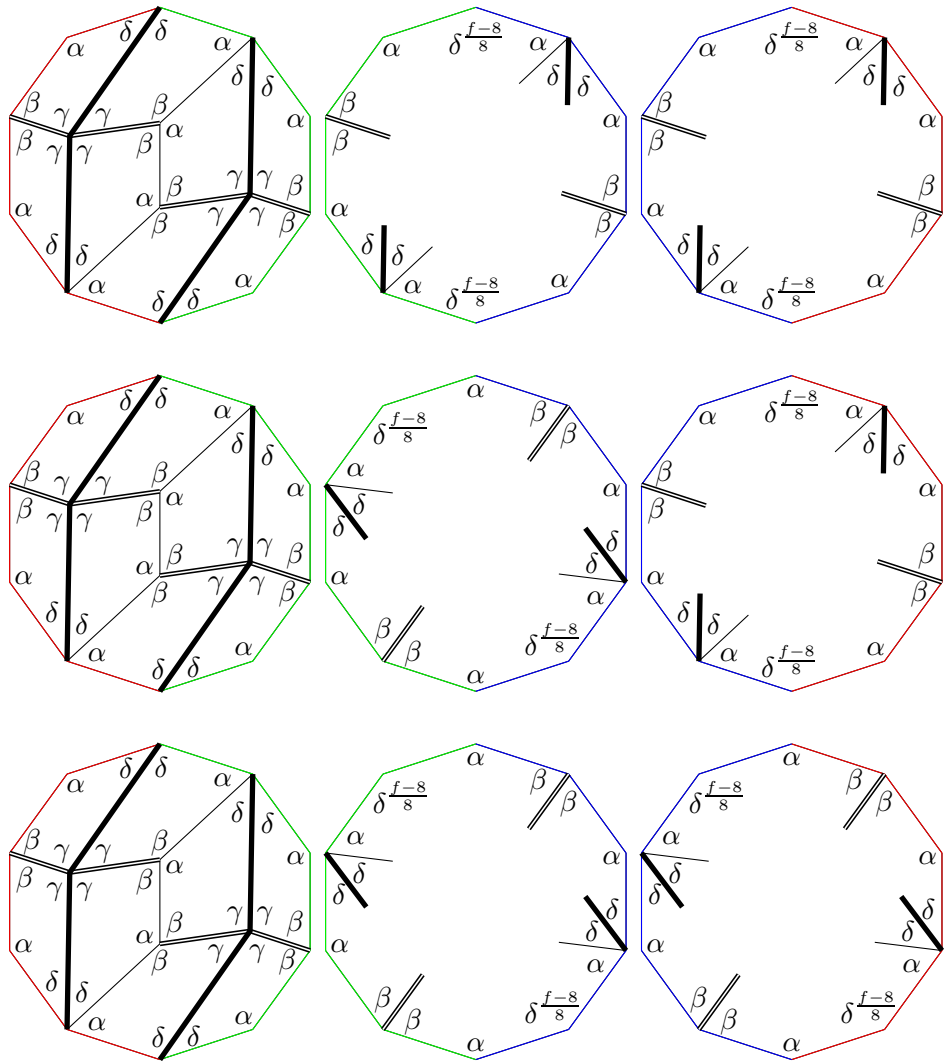

\section{Geometric Realisation}

In this section we show the geometric existence of the tile in Figure \ref{DefaultQuad} with respect to each of the $\AVC$s obtained in the previous sections. The well-known fact that the longer edge lying opposite to the bigger inner angle in a standard spherical triangle is used. That is, for inner angles $\theta_1, \theta_2$ and their corresponding opposite edges $e_1, e_2$, we have $\theta_1 \ge \theta_2$ if and only if $e_1 \ge e_2$.

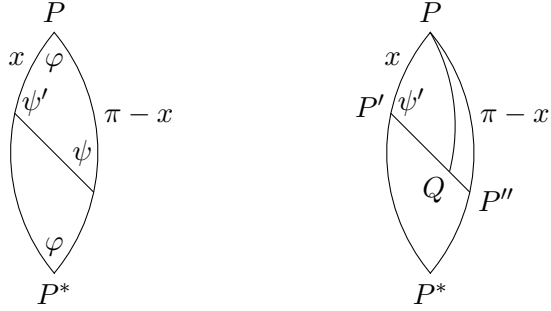
\begin{figure}[htp]
	\centering

\begin{tikzpicture}[>=latex,scale=2.5]

\draw[]
	([shift={(-40:1cm)}]-0.768,0) arc (-40:40:1cm);

\draw[]
	([shift={(-40+180:1cm)}]0.768,0) arc (-40+180:40+180:1cm);

\draw[]
	(-0.21,0.21) -- (0.21,-0.21);

\node at (0.0, 0.48) {\small $\varphi$};
\node at (0.0, -0.5) {\small $\varphi$};
\node at (-0.10, 0.27) {\small $\psi'$};
\node at (0.15, 0.0) {\small $\psi$};

\node at (0.45, 0.2) {\small $\pi - x$};
\node at (-0.2, 0.5) {\small $x$};

\node at (0.0, 0.75) {\small $P$};
\node at (0.0, -0.75) {\small $P^{\ast}$};

\begin{scope}[xshift=2cm, yshift=0cm] 

\draw[]
	([shift={(-40:1cm)}]-0.768,0) arc (-40:40:1cm);

\draw[]
	([shift={(-40+180:1cm)}]0.768,0) arc (-40+180:40+180:1cm);

\draw[]
	(-0.21,0.21) -- (0.21,-0.21);

\draw[]
	([shift={(40:1cm)}]-0.768,0) arc (30:-14:1cm) ;

\node at (0.0, 0.75) {\small $P$};
\node at (0.0, -0.75) {\small $P^{\ast}$};
\node at (0.02, -0.2) {\small $Q$};

\node at (-0.32, 0.25) {\small $P'$};
\node at (0.35, -0.25) {\small $P''$};

\node at (-0.10, 0.27) {\small $\psi'$};
\node at (0.45, 0.2) {\small $\pi - x$};
\node at (-0.2, 0.5) {\small $x$};

\end{scope}

\end{tikzpicture}

\caption{Half Lune Triangles}
\label{HalfLunTri}

\end{figure}

\begin{lem}\label{HalfLune} The area of a spherical triangle with edges $x, \pi-x$ and included angle $\varphi$ is $\varphi$. Let $\psi, \psi'$ be the angles opposite to $x, \pi-x$ respectively, then $\psi'=\pi - \psi$. When $\varphi < \pi$, we have $x \le \frac{\pi}{2} \le \pi - x$ if and only $\psi \le \frac{\pi}{2} \le \psi'$. One of the equalities holds if and only if another holds.
\end{lem}

\begin{proof} Piecing two copies of such triangle along the opposite edge to $\varphi$ we get a lune with antipodal angles $\varphi$ in Figure \ref{HalfLunTri} which has area $2\varphi$ and hence the area of the triangle is $\varphi$. By rotational symmetry, $\psi'$ has the same value as the supplementary angle of $\psi$ and hence $\psi'=\pi - \psi$. When $\varphi < \pi$, the longer edge lies opposite to the bigger inner angle, so $x \le \frac{\pi}{2} \le \pi - x$ if and only $\psi \le \frac{\pi}{2} \le \psi'$ and it is obvious that $x = \frac{\pi}{2}$ if and only if $\psi = \frac{\pi}{2}$.
\end{proof}

\begin{lem} \label{HalfLunTriIntEdge} In a standard spherical triangle with edges $x$, $\pi-x$ and one of their opposite angle $\psi' \ge \frac{\pi}{2}$, then any geodesic arc passing through the common vertex of these two edges dissecting the triangle into two triangles has length between $x$ and $\pi-x$. When $\psi' > \frac{\pi}{2}$, any such dissecting geodesic arc has length strightly larger than $\min (x, \pi - x)$ and smaller than $\max (x, \pi-x)$.
\end{lem}

\begin{proof} Suppose such spherical triangle is illustrated as $\triangle PP'P''$ in Figure \ref{HalfLunTri} and $PQ$ is any geodesic arc dissecting $\triangle PP'P''$. Then $\triangle PP'Q$ and $\triangle PP''Q$ are also standard triangles and $\angle PQP' + \angle PQP'' = \pi$. Without loss of generality, supppose $\pi - x \ge x$. Assume $PQ > \pi - x$. Then $\angle PQP' < \psi'$ and $\angle PQP'' < \pi - \psi'$ which implies $\angle PQP' + \angle PQP'' < \pi$, a contradiction. By the same token, $PQ < x$ implies $\angle PQP' + \angle PQP'' > \pi$, also a contradiction. Hence $\min (x, \pi - x) \le PQ \le \max (x, \pi-x)$. When $\psi' > \frac{\pi}{2}$, we have $\pi - x > x$. Assume $PQ = x$, then $\pi - x > PQ$ implies $\angle PQP'' > \pi - \psi'$ which gives $\angle PQP' < \psi'$ and hence $x < PQ$, a contradiction. So $PQ > x$. Assume $PQ = \pi - x$, then $PQ > x$ implies $\psi' > \angle PQP' = \pi - \angle PQP''$ which gives $\angle PQP'' > \pi - \psi'$ and hence $PQ < \pi - x$, a contradiction. So $PQ < \pi - x$. Hence we get $\min (x, \pi - x) < PQ < \max (x, \pi-x)$.
\end{proof}

\begin{lem}\label{ACLem} In a lune $L_{\alpha}$ defined by two geodesic arcs $\bm{a} \in (0, \pi)$ and $\alpha$ as the angle in between, let the antipodal vertices be $A, A^{\ast}$ and $AB, AD$ has lengths $=\bm{a} \in (0, \pi)$. For any $C$ inside $L_{\alpha}$, the quadrilateral $\square ABCD$ obtained has area $=\alpha$ if and only if $AC$ has length $=\pi - \bm{a}$. 
\end{lem}

\begin{proof} When $C$ is inside $L_{\alpha}$, it is obvious that the quadrilateral $\square ABCD$ is simple. Up to symmetry, let $\theta$ denote the angle between $AB$ and $AC$ be $\theta$. When $AC$ has length $\pi - \bm{a}$, Lemma \ref{HalfLune} implies the area of $\triangle ABC = \theta$ and area of $\triangle ADC = \alpha - \theta$. Then area of $\square ABCD$ is the sum of the area of $\triangle ABC$ and the area of $\triangle ADC$ which is $\alpha$. When $\square ABCD$ has area $\alpha$, assume $AC$ has length $<\pi - \bm{a}$, by Lemma \ref{HalfLune}, the area of $\triangle ABC < \theta$ and area of $\triangle ADC < \alpha - \theta$. Then the area of $\square ABCD < \theta +  \alpha - \theta = \alpha$, a contradiction. By the same token, if $AC$ has length $>\pi - \bm{a}$, we get the area of $\square ABCD > \alpha$, also a contradiction.
\end{proof}

\begin{proposition} \label{QuadDistabc} If $\alpha \ge \frac{4\pi}{f} $, for any $\bm{a} \in (0, \pi)$, there exist a $\theta$-family of simple default quadrilaterals in Figure \ref{DefaultQuad} with distinct edges $\bm{a},\bm{b},\bm{c}$ and area equals to $\frac{4\pi}{f}$.
\end{proposition}

\begin{proof} For every $\bm{a} \in (0, \pi)$ and $\alpha \ge \frac{4\pi}{f}$ where $f\ge6$, let $L_{\alpha}$ denote the lune defined by $\bm{a}$ and $\alpha$. We construct such $\theta$-family of quadrilaterals by locating vertex $C$ inside $L_{\alpha}$ and $\theta$ denotes the angle between $AB$ and $AC$. By Lemma \ref{ACLem}, when $AC$ has length $\pi - \bm{a}$, $\square ABCD$ is simple and has area $\alpha \ge \frac{4\pi}{f}$. Let $x$ denote the length of $AC$, the area of $\square ABCD$ is a function $\mathcal{A}(x, \theta)$. Then for every fixed $\theta$, $\mathcal{A}(x, \theta)$ is a continuous function in $x$ on $(0, \pi)$. In particular, $\mathcal{A}(0, \theta) = 0$ and $\mathcal{A}(\pi-\bm{a}, \theta) = \alpha \ge \frac{4\pi}{f}$. By Intermediate Value Theorem, there exists a value $l \in (0, \pi-\bm{a}]$ such that $\mathcal{A}(l, \theta) = \frac{4\pi}{f}$. In other words, there is a value $l \in (0, \pi-\bm{a}]$ such that when $AC$ has length $=l$, $\square ABCD$ has area $= \frac{4\pi}{f}$.  It remains to argue is that we can locate such a $C$ vertex so that $\bm{a},\bm{b},\bm{c}$ are distinct. This is done via a choice of $\theta \in [0, \alpha]$ as follows. We consider the $\bm{a}=\bm{b}, \bm{a}=\bm{c}, \bm{b}=\bm{c}$ respectively. 

If $\bm{b}=\bm{c}$, then triangles $\triangle ABC$ and $\triangle ACD$ are congruent $(SSS)$ and hence $\theta = \alpha - \theta = \frac{\alpha}{2}$. 

If $\bm{a}=\bm{b}$, then cosine law on $\triangle ABC$ gives 
\begin{align*}
\cos \bm{a} = \cos \bm{b} = \cos \bm{a} \cos l + \sin \bm{a} \sin l \cos \theta,
\end{align*}
which implies 
\begin{align*}
\cos \theta = \frac{\cos \bm{a} (1-\cos l)}{\sin \bm{a} \sin l }.
\end{align*}

If $\bm{a}=\bm{c}$, then cosine law on $\triangle ACD$ gives
\begin{align*}
\cos \bm{a} = \cos \bm{c} = \cos \bm{a} \cos l + \sin \bm{a} \sin l \cos(\alpha - \theta),
\end{align*}
which implies 
\begin{align*}
\cos( \alpha - \theta ) =  \frac{\cos \bm{a} (1-\cos l)}{\sin \bm{a} \sin l }.
\end{align*}

When $\left\vert  \frac{\cos a (1-\cos l)}{\sin a \sin l } \right\vert > 1$, there is no solution for $\theta$ in both of the discussion of $\bm{a}=\bm{b}$ and $\bm{a}=\bm{c}$ whereas when $\left\vert \frac{\cos a (1-\cos l)}{\sin a \sin l } \right\vert \le 1$ and $\theta \neq \cos^{-1} \left( \frac{\cos a (1-\cos l)}{\sin a \sin l } \right), \alpha - \cos^{-1} \left(  \frac{\cos a (1-\cos l)}{\sin a \sin l } \right)$, we have $\bm{a} \neq \bm{b}$ and $\bm{a} \neq \bm{c}$ respectively. Hence, for $a \in (0, \pi)$, when $\left\vert  \frac{\cos a (1-\cos l)}{\sin a \sin l } \right\vert > 1$, for any $\theta \in [0, \alpha] \backslash \{  \frac{\alpha}{2} \}$, there is a quadrilateral $\square ABCD$ with area $\frac{4\pi}{f}$ and distinct $\bm{a}, \bm{b}, \bm{c}$; when $\left\vert \frac{\cos a (1-\cos l)}{\sin a \sin l } \right\vert \le 1$, for any $\theta \in [0, \alpha] \backslash \{  \frac{\alpha}{2}, \cos^{-1} \left(  \frac{\cos a (1-\cos l)}{\sin a \sin l } \right), \alpha - \cos^{-1} \left(  \frac{\cos a (1-\cos l)}{\sin a \sin l } \right) \}$ there is a quadrilateral $\square ABCD$ with area $\frac{4\pi}{f}$ and distinct $\bm{a}, \bm{b}, \bm{c}$.
\end{proof}

\begin{cor} Moreover when $\alpha < \pi$, among these quadrilaterals constructed, we have
\begin{enumerate}
\item $0 < \bm{a} < \frac{\pi}{2}$ if and only if $0< \gamma < \pi$,
\item $\bm{a} = \frac{\pi}{2}$ if and only if $\gamma = \pi$,
\item $\frac{\pi}{2} < \bm{a} < \pi$ if and only if $\pi < \gamma < 2\pi$.
\end{enumerate}
\end{cor}

\begin{proof} When $0 < \bm{a} \le \frac{\pi}{2}$, $\triangle ABD$ is contained in half of $L_{\alpha}$ thus the area of $\triangle ABD$ is no more than half of the area of $L_{\alpha}$, which is $\alpha$. So $C$ must lie either outside of $\triangle ABD$ or on $BD$ and hence $0 < \gamma \le \pi$ such that $\bm{a} = \frac{\pi}{2}$ holds if and only if $\gamma = \pi$. When $\frac{\pi}{2} < \bm{a} < \pi$, the area of $\triangle ABD$ is bigger than $\alpha$ so $C$ must lie inside the triangle and hence $\gamma$ must be $>\pi$. 
\end{proof}

As proven by Cheung in \cite{ch}, the default quadrilateral in Figure \ref{DefaultQuad} satisfies the following matrix equation
\begin{align}\label{MatrixEq}
Y(\bm{b})Z(\pi - \beta)Y(\bm{a})Z(\pi - \alpha)Y(\bm{a})Z(\pi - \delta)Y(\bm{c})Z(\pi - \gamma) = I_3 
\end{align}
where 
\begin{align*}
Y(x) = 
\begin{bmatrix}
\cos x & 0 & \sin x \\
0 & 1 & 0 \\
-\sin x & 0 & \cos x
\end{bmatrix}, \ 
Z(\pi - \theta) = 
\begin{bmatrix}
-\cos \theta & -\sin \theta & 0 \\
\sin \theta & -\cos \theta & 0 \\
0 & 0 & 1
\end{bmatrix}, \ 
I_3 = 
\begin{bmatrix}
1 & 0 & 0 \\
0& 1 & 0 \\
0 & 0 & 1
\end{bmatrix},
\end{align*}
such that $x = \bm{a}, \bm{b}, \bm{c}$ and $\theta = \alpha, \beta, \gamma$. The following equations are obtained from \eqref{MatrixEq}. 

\begin{lem} In a tiling of the sphere by congruent quadrilaterals in Figure \ref{DefaultQuad}, we have the following trigonometry equations,
\begin{align}
\label{TrigEqca}&(\cos \alpha - 1) \sin \beta\sin\delta \cos^2 \bm{a}  + \sin\alpha(\cos\beta\sin\delta+\cos\delta\sin\beta)\cos \bm{a} \\ \notag
&+ \sin\beta\sin\delta+\cos\gamma-\cos\alpha\cos\beta\cos\delta = 0; \\
\label{TrigEqcb}&\cos \bm{b} \sin \beta \sin \gamma - \cos \bm{a} \sin \alpha \sin \delta - \cos \beta \cos \gamma+\cos \alpha \cos \delta = 0; \\
&\label{TrigEqcc} \cos \delta \sin \beta( \cos \alpha - 1)\cos^2 \bm{a} + \sin \alpha (  \cos \beta \cos \delta - \sin \beta \sin \delta )\cos \bm{a}  \\ \notag 
&+ \sin \gamma \cos \bm{c} +  \cos \alpha \cos \beta \sin \delta + \cos \delta \sin \beta = 0.
\end{align}
\end{lem}

\begin{proof} Rewriting \eqref{MatrixEq} as 
\begin{align}\label{MatrixEqDiff}
Z(\pi - \beta)Y(\bm{a})Z(\pi - \alpha)Y(\bm{a})Z(\pi - \delta) - Y^t(\bm{b}) Z^t(\pi - \gamma)Y^t(\bm{c}) = 0,
\end{align}
simply the $[2,1]$-entry and $[2,2]$-entry to give \eqref{TrigEqcc} and \eqref{TrigEqca} respectively. Rewriting \eqref{MatrixEq} as 
\begin{align*}
Z(\pi - \alpha)Y(\bm{a})Z(\pi - \delta)Y(\bm{c}) - Y^t(\bm{a}) Z^t(\pi - \beta)Y^t(\bm{b})Z^t(\pi - \gamma) = 0,
\end{align*}
simply the $[2,2]$-entry to gives \eqref{TrigEqcb}.
\end{proof}

\subsection{$\AVC \equiv \{ \beta\gamma\delta, \alpha^{\frac{f}{2}}\}$}

The existence of quadrilaterals $\square ABCD$ in Figure \ref{DefaultQuad} with area $\alpha = \frac{4\pi}{f}$ and distinct edges $\bm{a},\bm{b},\bm{c}$ is ensured in Proposition \ref{Distinctabc}. Note that by the quadrilateral angle sum, $\alpha = \frac{4\pi}{f}$ is equivalent to $\beta+ \gamma + \delta = 2\pi$.

\begin{proposition}\label{Distinctabc} For every $\bm{a} \in (0, \pi)$ and $\alpha = \frac{4\pi}{f}$, there exist a $\theta$-family of default quadrilaterals in Figure \ref{DefaultQuad} with distinct edges $\bm{a},\bm{b},\bm{c}$ which give tilings of $\AVC \equiv \{ \beta\gamma\delta, \alpha^{\frac{f}{2}}\}$. 
\end{proposition}

\begin{proof} For every $\bm{a} \in (0, \pi)$ and $\alpha = \frac{4\pi}{f}$ where $f\ge6$, let $L_{\alpha}$ denote the lune defined by $\bm{a}$ and $\alpha$. We provide the construction of such $\theta$-family of tiles in two cases, $C$ inside $L_{\alpha}$ and $C$ outside $L_{\alpha}$ where in the latter the range of values for $\bm{a}$ will be further specified when $\square ABCD$ is required to be simple. 

\begin{case*}[$C$ in $L_{\alpha}$] By Proposition \ref{QuadDistabc}, there is a $\theta$-family of default quadrilaterals with distinct edges $\bm{a}, \bm{b}, \bm{c}$ and when $l=\pi - \bm{a}$, the area is $\frac{4\pi}{f}$. It remains to discuss the range of $\theta$. 

If $\bm{b}=\bm{a}$, then by $l = \pi - \bm{a}$
\begin{align*}
\cos \theta = \frac{\cos \bm{a} }{1 - \cos \bm{a} }.
\end{align*}

If $\bm{c}=\bm{a}$, then by $l = \pi - \bm{a}$
\begin{align*}
\cos( \alpha - \theta ) = \frac{\cos  \bm{a} }{1 - \cos \bm{a} }.
\end{align*}

Hence, for $\bm{a} \in (0, \pi)$, when $\left\vert \frac{\cos \bm{a} }{1 - \cos \bm{a} } \right\vert > 1$, for any $\theta \in [0, \alpha] \backslash \{  \frac{\alpha}{2} \}$, there is a quadrilateral $\square ABCD$ with area $\alpha$; when $\left\vert \frac{\cos \bm{a} }{1 - \cos \bm{a} } \right\vert \le 1$, for any $\theta \in [0, \alpha] \backslash \{  \frac{\alpha}{2}, \cos^{-1} \left(  \frac{\cos \bm{a}}{1 - \cos \bm{a}} \right), \alpha - \cos^{-1} \left(  \frac{\cos \bm{a}}{1 - \cos \bm{a}} \right) \}$ there is a quadrilateral $\square ABCD$ with area $\alpha$.
\end{case*}

\begin{case*}[$C$ outside of $L_{\alpha}$] This implies one of $\beta, \delta$ has value $> \pi$. Up to symmetry, suppose $\beta > \pi$. By $\beta > \pi$ and we require $\beta\gamma\delta$ being a vertex, then $\gamma < \pi$. As a consequence of $\alpha, \gamma < \pi$, $BD$ is contained in $\square ABCD$ we are about to construct and hence $\triangle ABD$ is contained this $\square ABCD$. Since we require $\square ABCD$ to be simple, if $\pi - \bm{a} \le \bm{a}$, then we have $\bm{a}\ge \frac{\pi}{2}$ which implies the area of $\triangle ABD$ is bigger than half of that of $L_{\alpha}$. That is, the area of $\triangle ABD$ is bigger than $\alpha$, contradicting the area of $\square ABCD$ having to be $\alpha$ and hence the quadrilateral cannot be simple for $\bm{a} \ge \pi - \bm{a}$. So it remains to discuss $\pi - \bm{a} > \bm{a}$ which implies $\bm{a} \in (0, \frac{\pi}{2})$.

\begin{figure}[htp]
\centering
\begin{tikzpicture}[>=latex,scale=1.2]

\draw
	(-0.5,-1) -- (0,0) -- (0.25,-1.5)
	(0,0) -- (1,-2)
	(-0.5,-1) -- (1,-2);

\node at (0.0,0.2) {$A$};
\node at (1.1,-2.2) {$C$};
\node at (-0.7,-1.1) {$D$};
\node at  (0.2,-1.78) {$E$};

\node at (-0.06,-0.47) {\small $\alpha$};
\node at (0.25,-0.8) {\small $\theta$};

\node at (-0.5,-0.4) {$\bm{a}$};
\node at (1.1,-1.0) {\small $\pi - \bm{a}$};

\end{tikzpicture}
\caption{Triangles $\triangle ACD$ and $\triangle ACE$}
\label{LunExtC}
\end{figure}
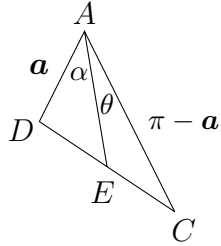

When $\theta \in ( 0, \pi - \alpha )$, by $\pi - \bm{a} > \bm{a}$, Lemma \ref{HalfLune} implies $\angle D > \frac{\pi}{2}$ in standard triangle $\triangle ACD$. In Figure \ref{LunExtC}, let $E$ be the intersection point of the geodesic arc inside $\triangle ACD$ at $CD$ of $\triangle ACD$ defined by $A, B$ and $\theta$. Then $AE$ is contained in $\triangle ACD$. Lemma \ref{HalfLunTriIntEdge} implies that $AE$ has length $>\bm{a}$ and hence $AB$ is contained in $\triangle ACD$. So $\square ABCD$ is contained in $\triangle ACD$ and hence $\square ABCD$ is simple. By Lemma \ref{HalfLune}, $\triangle ABC$ has area $\theta$ and $\triangle ACD$ has area $\alpha+\theta$ when $AC$ has length $\pi - \bm{a}$ and we obtain a quadrilateral $\square ABCD$ with area $\alpha+\theta - \theta = \alpha$.  

When $\theta \in [ \pi - \alpha, \pi )$, the complementary angle $2\pi - (\alpha + \theta) < \pi$ and the triangle $\triangle'$ with vertices $A, C, D$, $2\pi - (\alpha + \theta)$ as interior angle and the complementary geodesic arc of $CD$ is standard. Let $E^{\ast}$ be the antipodal point of $E$, then $\alpha, \theta < \pi$ implies that $E^{\ast}$ is on the complementary geodesic arc of $CD$ and hence is contained on the boundary of $\triangle'$. Meanwhile, $AE^{\ast} = \pi - AE$. Lemma \ref{HalfLunTriIntEdge} implies $\bm{a} < AE^{\ast} < \pi - \bm{a}$. The first inequality implies $AE < \pi - \bm{a}$ whereas the second implies $\bm{a} < AE$ and hence we get $\bm{a} < AE < \pi - \bm{a}$. So $AB$ is contained in $\triangle ACD$. So $\square ABCD$ is contained in $\triangle ACD$ and hence $\square ABCD$ is simple. Again by Lemma \ref{HalfLune}, $\triangle ABC$ has area $\theta$ and $\triangle ACD$ has area $\alpha+\theta$ when $AC$ has length $\pi - \bm{a}$ and we obtain a quadrilateral $\square ABCD$ with area $\alpha+\theta - \theta = \alpha$. Now it remains to argue whether we can locate such a $C$ vertex so that $\bm{a},\bm{b},\bm{c}$ are distinct.  

If $\bm{b}=\bm{a}$, then cosine law on $\triangle ABC$ gives 
\begin{align*}
\cos \bm{b} = \cos \bm{a} = \cos \bm{a} \cos ( \pi - \bm{a} ) + \sin \bm{a} \sin( \pi - \bm{a} ) \cos \theta = - \cos^2 \bm{a} + \sin^2 \bm{a} \cos \theta,
\end{align*}
which implies 
\begin{align*}
\cos \theta = \frac{\cos \bm{a} }{1 - \cos \bm{a} }.
\end{align*}

If $\bm{c}=\bm{a}$, then cosine law on $\triangle ACD$ gives
\begin{align*}
\cos \bm{c} = \cos \bm{a} = \cos \bm{a} \cos ( \pi - \bm{a} ) + \sin \bm{a} \sin( \pi - \bm{a} ) \cos( \alpha + \theta ) = - \cos^2 \bm{a} + \sin^2 \bm{a} \cos( \alpha + \theta ),
\end{align*}
which implies
\begin{align*}
\cos( \alpha + \theta ) = \frac{\cos \bm{a} }{1 - \cos \bm{a} }.
\end{align*}

If $\bm{b}=\bm{c}$, then combining formulae in the above we get
\begin{align*}
- \cos^2 \bm{a} + \sin^2 \bm{a} \cos \theta = \cos \bm{b} = \cos \bm{c} = - \cos^2 \bm{a} + \sin^2 \bm{a} \cos( \alpha + \theta )
\end{align*}
which, by $\sin^2 \bm{a} > 0$ for $\bm{a} \in (0, \pi)$, implies
\begin{align*}
\cos( \alpha + \theta ) = \cos \theta.
\end{align*}
Then we have
\begin{align*}
\alpha + \theta = \pm \theta + 2k\pi, \quad k \in \mathbb{Z}.
\end{align*}
For $f\ge6$, $0 <\alpha < \pi$, and we consider $\theta \in (0, 2\pi)$. When $\theta \ge \pi$, in $\triangle ABC$ and $\triangle ACE$ respectively, we have $BC$ and $CE$ having lengths $\ge \pi$. This implies $BC$ and $CE$ intersects at both of the antipodal points $C, C^{\ast}$ so $\square ABCD$ cannot be simple. Then we retrict our scope to  $\theta \in (0, \pi)$. As $C$ is outside of $L_{\alpha}$, we get
\begin{align*}
\theta = \pi - \frac{\alpha}{2}.
\end{align*}

Then when $\left\vert \frac{\cos \bm{a} }{1 - \cos \bm{a} } \right\vert > 1$, there is no solution for $\theta$ when $\bm{a}=\bm{b}$ and $\bm{a}=\bm{c}$ respectively whereas when $\left\vert \frac{\cos \bm{a} }{1 - \cos \bm{a} } \right\vert \le 1$ and $\theta \neq \cos^{-1} \left(  \frac{\cos \bm{a} }{1 - \cos \bm{a} } \right),  \cos^{-1} \left(  \frac{\cos \bm{a} }{1 - \cos \bm{a} } \right) - \alpha$, we have $\bm{a} \neq \bm{b}$ and $\bm{a} \neq \bm{c}$ respectively. Hence, for $\bm{a} \in (0, \frac{\pi}{2})$, when $\left\vert \frac{\cos \bm{a} }{1 - \cos \bm{a} } \right\vert > 1$, for any $\theta \in (0, \pi) \backslash \{  \pi - \frac{\alpha}{2} \}$, there is a quadrilateral $\square ABCD$ with area $\alpha$; when $\left\vert \frac{\cos \bm{a} }{1 - \cos \bm{a} } \right\vert \le 1$, for any $\theta \in (0, \pi) \backslash \{ \pi - \frac{\alpha}{2}, \cos^{-1} \left(  \frac{\cos \bm{a} }{1 - \cos \bm{a} } \right), \cos^{-1} \left(  \frac{\cos \bm{a} }{1 - \cos \bm{a} } \right) - \alpha \}$ there is a quadrilateral $\square ABCD$ with area $\alpha$.
\end{case*}
Therefore we complete the proof.
\end{proof}

\subsection{$\AVC = \{ \alpha\beta^2, \alpha^2\delta^2, \gamma^4, \alpha\delta^{\frac{f+8}{8}}, \beta^2\delta^{\frac{f-8}{8}}, \delta^{\frac{f}{4}} \}$}

By $\alpha=\pi - \frac{8\pi}{f}, \beta=\frac{\pi}{2}+\frac{4\pi}{f}, \gamma=\frac{\pi}{2}, \delta=\frac{8\pi}{f}$ and $f\ge16$, the quadrilateral is convex and hence it contains $BD$. Then Corollary to Proposition \ref{QuadDistabc} implies $0 < \bm{a} < \frac{\pi}{2}$. So we have $0 < \cos \bm{a} < 1$. 

Let $\mathcal{A} = \frac{4\pi}{f}$, by $\alpha=\pi - \frac{8\pi}{f}, \beta=\frac{\pi}{2}+\frac{4\pi}{f}, \gamma=\frac{\pi}{2}, \delta=\frac{8\pi}{f}$, we have
\begin{align*} 
\alpha=\pi - 2\mathcal{A}, \qquad 
\beta=\frac{\pi}{2}+\mathcal{A}, \qquad 
\gamma=\frac{\pi}{2}, \qquad 
\delta=2\mathcal{A}. 
\end{align*}
Substituting the above into \eqref{TrigEqca}, we obtain
\begin{align*} 
&(-4\sin^5 \mathcal{A} + 8\sin^3 \mathcal{A}- 4 \sin \mathcal{A} )\cos^2 \bm{a}+
(8\sin^5 \mathcal{A} -10 \sin^3 \mathcal{A}+2 \sin \mathcal{A}) \cos \bm{a} \\ 
&- 4\sin^5 \mathcal{A} + 2\sin^3 \mathcal{A} +\sin \mathcal{A} = 0,
\end{align*}
solving which we get 
\begin{align*}
\cos \bm{a} = \frac{4\sin^2 \mathcal{A} - \sqrt{5} - 1}{ 4\sin^2 \mathcal{A} - 4 }, \ \frac{4\sin^2 \mathcal{A} + \sqrt{5} - 1}{ 4\sin^2 \mathcal{A} - 4 },
\end{align*}
which is equivalent to 
\begin{align*}
\cos \bm{a} = \frac{4\cos^2 \mathcal{A} + \sqrt{5} - 3 }{4\cos^2 \mathcal{A}}, \ \frac{4\cos^2 \mathcal{A} - \sqrt{5} - 3 }{4\cos^2 \mathcal{A}}.
\end{align*}
As $\cos^2 \mathcal{A} \le 1$, 
\begin{align*}
\frac{4\cos^2 \mathcal{A} - \sqrt{5} - 3 }{4\cos^2 \mathcal{A}} = 1 - \frac{1}{4}\left( \frac{\sqrt{5} + 3}{\cos^2 \mathcal{A}} \right) < 0,
\end{align*}
so we have 
\begin{align*}
\cos \bm{a} = \frac{4\cos^2 \mathcal{A} + \sqrt{5} - 3 }{4\cos^2 \mathcal{A}}.
\end{align*}
Then by \eqref{TrigEqcb}, we get 
\begin{align*}
\cos \bm{b} = \frac{-(\sqrt{5}-3)\cos^2 \mathcal{A}+ \sqrt{5}-2}{ \cos \mathcal{A}}.
\end{align*}
Then by \eqref{TrigEqcc}, we get 
\begin{align*}
\cos \bm{c} = \frac{\sqrt{5}-1}{4\cos \mathcal{A}}.
\end{align*}

When $\bm{a} = \bm{b}$, then $\cos \bm{a}= \cos \bm{b}$ which implies one of the following,
\begin{align*}
 \frac{4\cos^2 \mathcal{A} + \sqrt{5} - 3 }{4\cos^2 \mathcal{A}} &= \frac{-(\sqrt{5}-3)\cos^2 \mathcal{A}+ \sqrt{5}-2}{ \cos \mathcal{A}}.
\end{align*}
Since $0 < \mathcal{A} < \frac{\pi}{2}$, we have $0 < \cos \mathcal{A} < 1$. Solving the above equations, the real solutions to $\mathcal{A}$ can only be 
\begin{align*}
\frac{4\pi}{f} = \mathcal{A} = 1.88495695649716,
\end{align*}
which implies
\begin{align*}
f = 6.666661841292876.
\end{align*}

When $\bm{a}=\bm{c}$, then $\cos \bm{a}= \cos \bm{c}$ which implies one of the following,
\begin{align*}
\frac{4\cos^2 \mathcal{A} + \sqrt{5} - 3 }{4\cos^2 \mathcal{A}} &= \frac{\sqrt{5}-1}{4\cos \mathcal{A}}.
\end{align*}
Since $0 < \mathcal{A} < \frac{\pi}{2}$, we have $0 < \cos \mathcal{A} < 1$. Solving the above equations, taking the positive solution to $\cos \mathcal{A}$, the real solutions to $\mathcal{A}$ can only be 
\begin{align*}
\frac{4\pi}{f} = \mathcal{A} = 0.9045568943023813,
\end{align*}
which implies
\begin{align*}
f = 13.89229433053042,
\end{align*}
which is not a valid value for $f$. 

When $\bm{b}=\bm{c}$, then $\cos \bm{b}= \cos \bm{c}$ which implies one of the following,
\begin{align*}
\frac{-(\sqrt{5}-3)\cos^2 \mathcal{A}+ \sqrt{5}-2}{ \cos \mathcal{A}} &= \frac{\sqrt{5}-1}{4\cos \mathcal{A}}.
\end{align*}
Since $\cos \mathcal{A} > 0$, solving the above we get 
\begin{align*}
\frac{4\pi}{f} = \mathcal{A} = \frac{2\pi}{5},
\end{align*}
which implies
\begin{align*}
f = 10.
\end{align*}
For every $f \ge 16$, as $\mathcal{A} = \frac{4\pi}{f}$, the following give solutions to distinct $\bm{a}, \bm{b}, \bm{c}$,
\begin{align*}
\bm{a} &= \cos^{-1} \left( \frac{4\sin^2  \frac{4\pi}{f} - \sqrt{5} - 1}{ 4\sin^2 \frac{4\pi}{f} - 4 } \right), \\
\bm{b} &= \cos^{-1} \left(  \frac{-(\sqrt{5}-3)\cos^2 \frac{4\pi}{f} + \sqrt{5}-2}{ \cos \frac{4\pi}{f} }  \right), \\
\bm{c} &= \cos^{-1} \left( \frac{\sqrt{5}-1}{4\cos \frac{4\pi}{f} }  \right).
\end{align*}
Lastly we substitute $\cos \bm{a}, \cos \bm{b}, \cos\bm{c}$ by the above and check whether the above solutions satisfy \ref{MatrixEq} which is equivalent to \ref{MatrixEqDiff}. Indeed we check the solutions with \ref{MatrixEqDiff}. Let
\begin{align*}
M := Z(\pi - \beta)Y(\bm{a})Z(\pi - \alpha)Y(\bm{a})Z(\pi - \delta) - Y^t(\bm{b}) Z^t(\pi - \gamma)Y^t(\bm{c}). 
\end{align*}
It is obvious that the $M[2,2]=0$ and we also get $M[1,2], M[2,1], M[3,3] = 0$ and 
\begin{align*}
M[1,3] &= \frac{(2\sqrt{5}-6)\sin \bm{a}\cos \mathcal{A} \sin \mathcal{A}+(\sqrt{5}-1)\sin \bm{b}}{4\cos \mathcal{A}}, \\
M[2,3] &= \frac{(\sqrt{5}+1)\sin \bm{a}\cos \mathcal{A}}{2} - \sin \bm{c}.
\end{align*}

Multiply $M[1,3]$ by $\frac{(2\sqrt{5}-6)\sin \bm{a}\cos \mathcal{A} \sin \mathcal{A}-(\sqrt{5}-1)\sin \bm{b}}{4\cos \mathcal{A}}$, we get 
\begin{align*}
\frac{(2\sqrt{5}-6)^2\sin^2 \bm{a}\cos^2 \mathcal{A} \sin^2 \mathcal{A}-(\sqrt{5}-1)^2 \sin^2 \bm{b}}{16\cos^2 \mathcal{A}} = 0
\end{align*}
when substituting $\sin^2 \mathcal{A} = 1 - \cos^2 \mathcal{A}$, $\sin^2 \bm{a} = 1 - \cos^2 \bm{a} = 1 - \left( \frac{ 4\cos^2 \mathcal{A} + \sqrt{5} - 3 }{4\cos^2 \mathcal{A}} \right)^2$, and $\sin^2 \bm{b} = 1 - \cos^2 \bm{b} = 1 - \left( \frac{-(\sqrt{5}-3)\cos^2 \mathcal{A}+ \sqrt{5}-2}{ \cos \mathcal{A}}  \right)^2$. Meanwhile, if $(2\sqrt{5}-6)\sin \bm{a}\cos \mathcal{A} \sin \mathcal{A}-(\sqrt{5}-1)\sin \bm{b} = 0$, we have
\begin{align*}
(2\sqrt{5}-6)\sin \bm{a}\cos \mathcal{A} \sin \mathcal{A} = (\sqrt{5}-1)\sin \bm{b}.
\end{align*}
However, $\sin \bm{a}, \cos \mathcal{A}, \sin \mathcal{A}, (\sqrt{5}-1), \sin \bm{b} > 0$ but $(2\sqrt{5}-6) < 0$, which implies the LHS $< 0$ while the RHS $>0$, a contradiction. So $\frac{(2\sqrt{5}-6)\sin \bm{a}\cos \mathcal{A} \sin \mathcal{A}-(\sqrt{5}-1)\sin \bm{b}}{4\cos \mathcal{A}} \neq 0$ which implies $M[1,3] = 0$. As a consequence, 
\begin{align*}
\sin \bm{b} = (\sqrt{5}-1)\sin \bm{a}\cos \mathcal{A} \sin \mathcal{A}.
\end{align*}

Multiply $M[2,3]$ by $\frac{(\sqrt{5}+1)\sin \bm{a}\cos \mathcal{A}}{2} + \sin \bm{c}$, we get 
\begin{align*}
\frac{(\sqrt{5}+1)^2\sin^2 \bm{a}\cos^2 \mathcal{A}}{4} - \sin^2 \bm{c} = 0
\end{align*}
when substituting $\sin^2 \bm{a} = 1 - \cos^2 \bm{a} = 1 - \left( \frac{ 4\cos^2 \mathcal{A} + \sqrt{5} - 3 }{4\cos^2 \mathcal{A}} \right)^2$, and $\sin^2 \bm{c} = 1 - \cos^2 \bm{c} = 1 - \left( \frac{\sqrt{5}-1}{4\cos \mathcal{A}} \right)^2$. Meanwhile, by $\sqrt{5}+1, \sin \bm{a}, \cos \mathcal{A}, \sin \bm{c} > 0$, then $\frac{(\sqrt{5}+1)\sin \bm{a}\cos \mathcal{A}}{2} + \sin \bm{c} > 0$ which implies $M[2,3]=0$. As a consequence, 
\begin{align*}
\sin \bm{c} = \frac{(\sqrt{5}+1)\sin \bm{a}\cos \mathcal{A}}{2}. 
\end{align*}

By $\sin \bm{b} = (\sqrt{5}-1)\sin \bm{a}\cos \mathcal{A} \sin \mathcal{A}, \sin \bm{c} = \frac{(\sqrt{5}+1)\sin \bm{a}\cos \mathcal{A}}{2}$ and $\cos^2 \mathcal{A} =1 - \sin^2 \mathcal{A}$, we get 
\begin{align*}
M[1,1] =\sin \bm{b} \sin \bm{c} + \frac{4(\sqrt{5} - 3) \sin^3 \mathcal{A} + (5 - \sqrt{5})\sin\mathcal{A}}{4\sin^2 \mathcal{A} - 4} = 0.
\end{align*}
By the same token, we have $M[3,1] = M[3,2] = 0$. Hence the solutions satisfy \eqref{MatrixEq} and the quadrilateral indeed exists.  

\subsection{$\AVC \supset \{ \alpha^3, \beta^2\delta^2, \gamma^4 \}$}

The discussion here includes $\AVC \equiv \{  \alpha^3, \beta^2\delta^2, \gamma^4 \}$ in Figure \ref{Tf24a3-1}, Figure \ref{Tf24CubeSD} and $\AVC \equiv \{ \alpha^3, \alpha\beta^2, \alpha^2\delta^2, \beta^2\delta^2, \gamma^4 \}$ in Figure \ref{Tf24a3-2}. By $\alpha=\frac{2\pi}{3}, \beta + \delta=\pi , \gamma=\frac{\pi}{2}$, the quadrilateral is convex and hence it contains $BD$. Then Corollary to Proposition \ref{QuadDistabc} implies $0 < \bm{a} < \frac{\pi}{2}$. So we have $0 < \cos \bm{a} < 1$. 

Substituting $\alpha=\frac{2\pi}{3}, \beta=\pi - \delta, \gamma=\frac{\pi}{2}$ into \eqref{TrigEqca}, we have
\begin{align*}
-\frac{3\sin^2 \delta \cos^2 \bm{a} }{2}+ \sin^2 \delta - \frac{\cos^2 \delta}{2} = 0.
\end{align*}
by $0 < \cos \bm{a} < 1$, solving the above equation we get
\begin{align*}
\cos \bm{a} = \frac{\sqrt{3\sin^2 \delta - 1}}{\sqrt{3}\sin \delta }.
\end{align*}
For $\bm{a} \in (0, \frac{\pi}{2})$, $\cos \bm{a} > 0$ so we must have $3\sin^2 \delta - 1 > 0$, that is $\sin^2 \delta > \frac{1}{3}$. Then $\cos^2\delta < \frac{2}{3}$. 

Next by \eqref{TrigEqcb}, we get 
\begin{align*}
\cos \bm{b}= \frac{\sqrt{3\sin^2 \delta-1}+\cos \delta }{2\sin \delta}.
\end{align*}
For $\bm{b} \in (0, \pi)$, likewise we have $\sin^2 \delta > \frac{1}{3}$ and $\cos^2\delta < \frac{2}{3}$. 

Next by \eqref{TrigEqcc}, we get 
\begin{align*}
\cos \bm{c} = \frac{\sqrt{3\sin^2 \delta - 1}-\cos \delta}{2\sin \delta}.
\end{align*}
For $\bm{c} \in (0, \pi)$, likewise we have $\sin^2 \delta > \frac{1}{3}$ and $\cos^2\delta < \frac{2}{3}$. 

The geometric existence of the tile requires $M=0$, in particular $M[1,1]=0$. By the above, we have $M[1,1]$ as follows,
\begin{align*}
M[1,1] = \frac{2\cos \delta^2\sin \bm{b}\sin \bm{c}-2\sin \bm{b}\sin \bm{c}-2\cos \delta^2+1}{2(\cos \delta-1)(\cos \delta+1)} = 0, 
\end{align*}
which implies $\sin \bm{b}\sin \bm{c} = \frac{ 2\cos^2\delta - 1}{2(\cos^2\delta - 1)}$. Since $\sin \bm{b}\sin \bm{c} > 0$ and $2(\cos^2\delta - 1) < 0$, we get $\cos^2 \delta < \frac{1}{2}$ which implies $-\frac{1}{\sqrt{2}} < \cos\delta< \frac{1}{\sqrt{2}}$ and hence we have $\frac{\pi}{4} < \delta < \frac{3\pi}{4}$.

When $\bm{a} = \bm{b}$, then $\cos \bm{a} =\cos \bm{b}$ which implies one of the following,
\begin{align*}
\frac{\sqrt{3\sin^2 \delta - 1}}{\sqrt{3}\sin \delta } &=  \frac{\sqrt{3\sin^2 \delta-1}+\cos \delta }{2\sin \delta}.
\end{align*}
The solutions to the above equations are 
\begin{align*}
\sin \delta =\pm \frac{\sqrt{5-2\sqrt{3}}}{\sqrt{2}\sqrt{ 6-3^{\frac{3}{2}} } }.
\end{align*}
As $\delta \in (0, \pi)$, we have $\sin \delta > 0$ so that
\begin{align*}
\delta &= \sin^{-1} \left( \frac{ \sqrt{ 5 - 2\sqrt{3} } }{ \sqrt{2} \sqrt{ 6- 3^{\frac{3}{2}} } } \right), \ \pi - \sin^{-1} \left( \frac{ \sqrt{ 5 - 2\sqrt{3} } }{ \sqrt{2} \sqrt{ 6- 3^{\frac{3}{2}} } } \right) \\ 
&= 0.4322221997677038\pi, \ 0.5677778002322962\pi.
\end{align*}

When $\bm{a} = \bm{c}$, then $\cos \bm{a} =\cos \bm{c}$ which implies one of the following,
\begin{align*}
\frac{\sqrt{3\sin^2 \delta - 1}}{\sqrt{3}\sin \delta } &= \frac{\sqrt{3\sin^2 \delta - 1}-\cos \delta}{2\sin \delta}.
\end{align*}
The solutions to the above equations are 
\begin{align*}
\sin \delta = \pm \frac{ \sqrt{ 5 - 2\sqrt{3} } }{ \sqrt{2} \sqrt{ 6- 3^{\frac{3}{2}} } }, 
\end{align*}
As $\delta \in (0, \pi)$, we have $\sin \delta > 0$ so that
\begin{align*}
\delta &=\sin^{-1} \left( \frac{ \sqrt{ 5 - 2\sqrt{3} } }{ \sqrt{2} \sqrt{ 6- 3^{\frac{3}{2}} } } \right), \ \pi - \sin^{-1} \left( \frac{ \sqrt{ 5 - 2\sqrt{3} } }{ \sqrt{2} \sqrt{ 6- 3^{\frac{3}{2}} } } \right) \\ 
&= 0.4322221997677038\pi, \ 0.5677778002322962\pi.
\end{align*}

When $\bm{b} = \bm{c}$, then $\cos \bm{b} =\cos \bm{c}$ which implies one of the following,
\begin{align*}
\frac{\sqrt{3\sin^2 \delta-1}+\cos \delta }{2\sin \delta} &= \frac{\sqrt{3\sin^2 \delta - 1}-\cos \delta}{2\sin \delta}.
\end{align*}
Solving the above equations, we get 
\begin{align*}
\cos \delta = 0,
\end{align*}
then we get $\delta = \frac{\pi}{2} = \beta$. Hence for $f=24$, the following give solutions to distinct $\bm{a}, \bm{b}, \bm{c}$ 
\begin{align*}
\bm{a} &=  \cos^{-1} \left(  \frac{\sqrt{3\sin^2 \delta - 1}}{\sqrt{3}\sin \delta } \right), \\
\bm{b} &= \cos^{-1} \left( \frac{  \sqrt{3\sin^2 \delta-1} + \cos \delta }{2\sin \delta} \right), \\
\bm{c} &=  \cos^{-1} \left( \frac{ \sqrt{3\sin^2 \delta - 1} - \cos \delta }{2\sin \delta} \right),
\end{align*}
when $\delta \in (\frac{\pi}{4}, \frac{3\pi}{4})$ such that 
\begin{align*}
\delta \neq \frac{\pi}{2}, \ \sin^{-1}  \left( \frac{ \sqrt{ 5 - 2\sqrt{3} } }{ \sqrt{2} \sqrt{ 6- 3^{\frac{3}{2}} } } \right), \ \pi - \sin^{-1} \left( \frac{ \sqrt{ 5 - 2\sqrt{3} } }{ \sqrt{2} \sqrt{ 6- 3^{\frac{3}{2}} } } \right).
\end{align*}

In $\AVC \equiv \{ \alpha^3, \alpha\beta^2, \alpha^2\delta^2, \beta^2\delta^2, \gamma^4 \}$, we have $\alpha = \beta = \frac{2\pi}{3}$ and $\delta = \frac{\pi}{3}$. Then
\begin{align*}
\cos \bm{a} = \frac{\sqrt{5}}{3}, \ \cos \bm{b} =  \frac{\sqrt{5}+1}{2\sqrt{3}}, \ \cos \bm{c} = \frac{\sqrt{5}-1}{2\sqrt{3}},
\end{align*}
which implies 
\begin{align*}
\bm{a} &= \cos^{-1} \left( \frac{\sqrt{5}}{3} \right), \\
\bm{b} &= \cos^{-1} \left(  \frac{\sqrt{5}+1}{2\sqrt{3}} \right), \\
\bm{c} &=  \cos^{-1} \left( \frac{\sqrt{5}-1}{2\sqrt{3}} \right).
\end{align*}
Lastly we check whether the $\cos \bm{a} =  \frac{\sqrt{3\sin^2 \delta - 1}}{\sqrt{3}\sin \delta }, 
\cos \bm{b} = \frac{ \sqrt{3\sin^2 \delta-1}  +  \cos \delta  }{2\sin \delta}, 
\cos \bm{c} =  \frac{ \sqrt{3\sin^2 \delta - 1} - \cos \delta }{2\sin \delta}$ satisfy \eqref{MatrixEqDiff}. It is obvious that the $M[2,2]=0$ and we also get $M[1,2], M[2,1], M[3,3] = 0$ and 
\begin{align*}
M[2,3] &= \frac{(3\cos \bm{a} \sin \delta+\sqrt{3}\cos\delta)\sin\bm{a}}{2} - \sin \bm{c}, \\
M[3,2] &= \frac{ ( 3\cos \bm{a} \sin \delta - \sqrt{3}\cos \delta )\sin \bm{a}  }{2} - \sin \bm{b}.
\end{align*}

Multiple $M[3,2]$ by $\frac{ ( 3\cos \bm{a} \sin \delta - \sqrt{3}\cos \delta )\sin \bm{a}  }{2} + \sin \bm{b}$, we get 
\begin{align*}
 \frac{ ( 3\cos \bm{a} \sin \delta - \sqrt{3}\cos \delta )^2 \sin^2 \bm{a}  }{4} - \sin^2 \bm{b} = 0
\end{align*}
when substituting  $\sin^2 \delta = 1 - \cos^2 \delta$, $\sin^2 \bm{a} = 1 - \cos^2 \bm{a} = 1 -  \left( \frac{\sqrt{3\sin^2 \delta - 1}}{\sqrt{3}\sin \delta } \right)^2$, $\sin^2 \bm{b} = 1 - \cos^2 \bm{b} = 1 - \left( \frac{\sqrt{3\sin^2 \delta-1}+\cos \delta }{2\sin \delta} \right)^2$. Meanwhile,
\begin{align*}
 3\cos \bm{a} \sin \delta - \sqrt{3}\cos \delta = \sqrt{3} \left( \sqrt{3\sin^2\delta - 1}-\cos\delta \right).
\end{align*}
When $\sqrt{3\sin^2\delta - 1}-\cos\delta \le 0$, we get $\cos^2\delta \ge \frac{1}{2}$, contradicting $\cos^2\delta < \frac{1}{2}$. So $\sqrt{3\sin^2 \delta-1}-\cos \delta > 0$ and $\cos \bm{c}> 0$. Combining with $\sin \bm{a}, \sin \bm{b} > 0$, we have $\frac{ ( 3\cos \bm{a} \sin \delta - \sqrt{3}\cos \delta )\sin \bm{a} }{2} + \sin \bm{b} > 0$, which implies $M[3,2]=0$. As a consequence, 
\begin{align*}
\sin \bm{b} = \frac{\sqrt{3}(\sqrt{2-3\cos^2 \delta} - \cos \delta)\sin \bm{a}}{2}.
\end{align*}

Multiple $M[2,3]$ by $\frac{(3\cos \bm{a} \sin \delta+\sqrt{3}\cos\delta)\sin\bm{a}}{2} + \sin \bm{c}$, we get 
\begin{align*}
\frac{(3\cos \bm{a} \sin \delta+\sqrt{3}\cos\delta)^2\sin^2\bm{a}}{4} - \sin^2 \bm{c} = 0
\end{align*}
when substituting $\sin^2 \delta = 1 - \cos^2 \delta$, $\sin^2 \bm{a} = 1 - \cos^2 \bm{a} = 1 -  \left( \frac{\sqrt{3\sin^2 \delta - 1}}{\sqrt{3}\sin \delta } \right)^2$, $\sin^2 \bm{c} = 1 - \cos^2 \bm{c} = 1 - \left( \frac{\sqrt{3\sin^2 \delta - 1}-\cos \delta}{2\sin \delta} \right)^2$. Meanwhile,
\begin{align*}
3\cos \bm{a} \sin \delta+\sqrt{3}\cos\delta = \sqrt{3} ( \sqrt{3\sin^2 \delta - 1 }+\cos \delta ).
\end{align*}
When $\sqrt{3\sin \delta^2  -1}+\cos \delta \le 0$, we have $\sqrt{2-3\cos \delta^2} \le \cos( \pi - \delta )$ which implies $\cos^2\delta \ge \frac{1}{2}$, contradicting $\cos^2\delta < \frac{1}{2}$. So $\sqrt{3\sin \delta^2 - 1}+\cos \delta > 0$ and $\cos \bm{b}> 0$. Combining with $\sin \bm{a}, \sin \bm{c} > 0$, we have $\frac{(3\cos \bm{a} \sin \delta+\sqrt{3}\cos\delta)\sin\bm{a}}{2} + \sin \bm{c} > 0$, which implies $M[2,3]=0$. As a consequence, 
\begin{align*}
\sin \bm{c} = \frac{\sqrt{3}(\sqrt{2-3\cos^2 \delta} + \cos \delta)\sin \bm{a}}{2}.
\end{align*}

By the solutions to $\cos \bm{a}, \cos \bm{b}, \cos\bm{c}$ and solutions to $\sin \bm{b}, \sin\bm{c}$, we get $M[1,1]=M[1,3]=M[3,1]=0$. Hence the solutions satisfy \eqref{MatrixEq} and the quadrilateral indeed exists.

\section{Symmetries}

We conclude our main theorem by presenting the symmetry group of each tiling. It is well known that the symmetry group of each tiling is a discrete subgroup of $O(3)$. What we need to do is simply to determine which discrete subgroup it is for each tiling. To do so we follow the decision tree on p313 of Cromwell \cite{cr}. Before applying the decision tree, we first establish some general facts about the reflectional and rotational symmetries which will be shown to be closely related the all the tiling vertices, namely $\alpha^3, \alpha\beta^2, \beta\gamma\delta, \alpha^2\delta^2, \gamma^4, \alpha^{\frac{f}{2}}, \delta^{\frac{f}{4}}, \alpha\delta^{\frac{f+8}{8}}, \beta^2\delta^{\frac{f-8}{8}}$. 

Reflectional symmetries of the tilings can only occur when their mirror planes do not cut through the interior of the tile since it does not have any reflectional symmetry. In other words, the mirror plane of any reflectional symmetry must cut along some edges and bisect the angle combinations at the end point vertices of each edge. We call a path \textit{vertex bisecting path} if it consists of edges bisecting vertices and if such path is also a cycle we call it a \textit{vertex bisecting cycle}. Then a mirror plane must contain vertex bisecting cycle. Only $\gamma^4, \alpha^2\delta^2, \alpha^{\frac{f}{2}}, \beta^2\delta^{\frac{f-8}{8}}, \delta^{\frac{f}{4}}$ can be on a vertex bisecting path or a vertex bisecting cycle.

Rotational symmetries of the tilings can only occur when their axes do not pass through the interior of the tile since it does not have any rotational symmetry. In other words, the axis of any rotational symmetry passes through the midpoint of some edge or some vertex. Moreover, $n$-fold rotations, where $n\ge3$, can only happen at a vertex. Meanwhile, by the AAD the vertices $\alpha\beta^2, \beta\gamma\delta, \alpha^2\delta^2, \beta^2\delta^{\frac{f-8}{8}}, \alpha\delta^{\frac{f+8}{8}}$ do not have rotational symmetry. Hence if a rotational axis passes through a vertex, then it must be one of $\alpha^3, \gamma^4,  \alpha^{\frac{f}{2}}, \delta^{\frac{f}{4}}$.

\begin{case*}[$\AVC \equiv \{ \beta\gamma\delta, \alpha^{\frac{f}{2}} \}$] There are $2$-fold symmetry axes between the midpoints of $\bm{b}$-edges and their antipodal points, likewise for $\bm{c}$-edges. There is also a $\frac{f}{2}$-fold ($\frac{f}{2}>2$) rotational symmetry with axis through the $\alpha^{\frac{f}{2}}$-pair and it is the only $n$-fold rotation symmetry with $n > 2$. Since every $\alpha^{\frac{f}{2}}$ is adjacent to some $\beta\gamma\delta$, there is no vertex bisecting cycle containing $\alpha^{\frac{f}{2}}$. Then there is no mirror plane. Hence the symmtry group is $D_{\frac{f}{2}}$. 
\end{case*}

\begin{case*}[$\AVC \equiv \{ \alpha^3, \beta^2\delta^2, \gamma^4 \}$] Since the tiling is a quadrilateral subdivision of the cube (or octahedron), there are more than one $3$-fold axes through the origin and the antipodal $\alpha^3$-vertices. The rotational axis of every $\gamma^4$ is at most $2$-fold. So there is no $4,5$-fold axis. Meanwhile, the origin is a point of inversion. Hence the symmetry group is $T_h$.
\end{case*}

\begin{case*}[$\AVC \equiv \{ \alpha^3, \alpha\beta^2, \alpha^2\delta^2, \beta^2\delta^2, \gamma^4  \}$] Since the tiling is a quadrilateral subdivision of the triangular prism, there are one $3$-fold axis and three $2$-fold axes. The $3$-fold axis passes through the origin and the pair of $\alpha^3$. Each of the $2$-fold axes passes through a midpoint of the common edge ($\bm{a}$-edge) between the pair of tiles in the middle of the side face of the prism and the midpoint of the middle side edge ($\bm{a}$-edge) on the opposite side. 

Since every $\alpha^2\delta^2, \beta^2\delta^2$ is adjacent to some $\gamma^4$ along a $\bm{b}$-edge or a $\bm{c}$-edge, it suffices to show that there is no vertex bisecting cycle containing $\gamma^4$. Up to the rotational symmetry in Figure \ref{Tf24a3-2}, it suffices to discuss one $\gamma^4$, for instance $\gamma^4$ at $T_1, T_4, T_9, T_{10}$. The vertex bisecting path of $\gamma^4$ along $\bm{b}$-edges terminates before $\alpha\beta^2$ and the vertex bisecting path along $\bm{c}$-edges terminates at the neighbours of $\alpha^2\delta^2$ ($\beta^2\delta^2, \alpha^2\delta^2, \alpha\beta^2$). Then there is no vertex bisecting cycle containing any of $\alpha^2\delta^2, \beta^2\delta^2, \gamma^4$. So there is no mirror plane and hence the symmetry group is $D_3$ (the rotational group of the triangular prism).
\end{case*}

\begin{case*}[$\AVC \equiv \{ \alpha\beta^2, \alpha^2\delta^2, \gamma^4, \delta^{\frac{f}{4}} \}$] Since $\frac{f}{4}$ is even, we have $f \equiv 0 \mod 8$. There is a $\frac{f}{8}$-fold rotational symmetry with axis through the $\delta^{\frac{f}{4}}$-pair. For $f\ge24$, this is the only $n$-fold rotational symmetry with $n>2$. Meanwhile, as $\frac{f}{4}$ is even, viewing the tiling as a collection of time zones in Figure \ref{pqEMT}, there is a $2$-fold rotational symmetry with axis through the midpoint of the time zones at the $\bm{a}$-edge between the $\alpha\beta^2$-pair and its antipodal point (also at the $\bm{a}$-edge between the $\alpha\beta^2$-pair counterpart). 

Up to symmetry, every vertex bisecting cycle cutting through both $\delta^{\frac{f}{4}}$ and the $\alpha^2\delta^2, \gamma^4$ in between correponds to a mirror plane of reflectional symmetry. For example, a vertex bisecting cycle is given by a vertex sequence $\delta^{\frac{f}{4}}, \gamma^4,  \alpha^2\delta^2,  \delta^{\frac{f}{4}}, \alpha^2\delta^2, \gamma^4$ in the first picture of Figure \ref{pqEMT1624} and a vertex sequence $\delta^{\frac{f}{4}}, \alpha^2\delta^2, \gamma^4, \delta^{\frac{f}{4}}, \alpha^2\delta^2, \gamma^4$ in the second picture of Figure \ref{pqEMT1624}.  However, there is no mirror plane perpendicular to $\frac{f}{8}$-fold rotational axis. So for $f\ge16$, the symmetry group is $D_{\frac{f}{8}v}$.
\end{case*}

\begin{case*}[$\AVC \equiv \{ \alpha\beta^2, \alpha^2\delta^2, \gamma^4, \beta^2\delta^{\frac{f-8}{8}},  \alpha\delta^{\frac{f+8}{8}} \}$] In Figure \ref{b2dFamily}, we can see that there is an axis of $2$-fold rotational symmetry through the midpoint of the $\bm{a}$-edge between $T_{11}, T_{16}$. This axis is indeed the only axis of rotational symmetry. It is because any rotation either fixes $\beta^2\delta^{\frac{f-8}{8}}, \alpha\delta^{\frac{f+8}{8}}$, which is not possible by the AAD of these vertices, or moves one pair of $\beta^2\delta^{\frac{f-8}{8}}, \alpha\delta^{\frac{f+8}{8}}$ to their counterpart, which has the same image as the $2$-fold rotation.  

Since every $\alpha^2\delta^2, \beta^2\delta^{\frac{f-8}{8}}$ is adjacent to some $\gamma^4$ along a $\bm{b}$-edge or a $\bm{c}$-edge, it suffices to show that there is no vertex bisecting cycle containing $\gamma^4$. The vertex bisecting path of $\gamma^4$ along $\bm{b}$-edges terminates before $\alpha\beta^2$ and the vertex bisecting path along $\bm{c}$-edges terminates before $\alpha\delta^{\frac{f+8}{8}}$. Then there is no vertex bisecting cycle containing any of $\alpha^2\delta^2, \beta^2\delta^{\frac{f-8}{8}}, \gamma^4$. So there is no mirror plane. 

Since the rotational symmetry is unique, a reflection in the plane perpendicular to the axis sends its end point, the midpoint of the $\alpha\beta^2$-pair and the midpoint of $\alpha^2\delta^2$-pair to each other. No rotation about the same axis preceding or succeeding the reflection is a symmetry. So there is no rotation-reflection. Hence the symmetry group is $C_2$.
\end{case*}

\begin{case*}[$\AVC\equiv\{\alpha\beta^2,\gamma^4,\alpha^2\delta^2,\alpha\delta^{\frac{f+8}{8}}\}$] From the $\AVC$, there is no $n$-fold rotations with $n>2$. The regions $R_1, R_2$ in the tiling in Figure \ref{Tf24a2d2ad4} can be expressed in as the two disks in Figure \ref{2FoldRot}. The $2$-fold rotation with axis passing through the midpoint of the $\bm{a}$-edge between $T_1, T_2$ is indicated by circles in Figure \ref{2FoldRot} on the boundary of two disks. Since $\frac{f-8}{8}, \frac{f+8}{8}$ are even, there is another $2$-fold rotation with axis passing through the midpoint of $\bm{a}$-edge, indicated as a dot in Figure \ref{2FoldRot}, between the $\alpha\beta^2$-pair at the centre of each disk. 

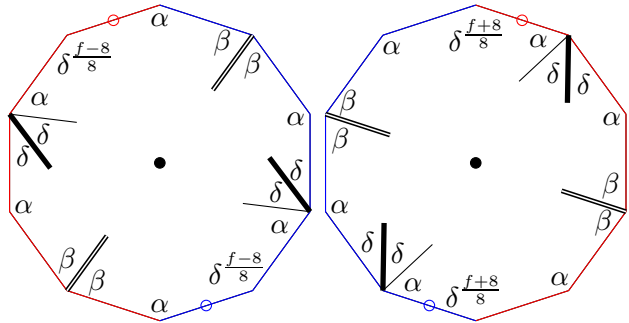
\begin{figure}[htp]
	\centering
	\begin{tikzpicture}[>=latex,scale=0.7]
	
	\begin{scope}[xshift=0 cm, yshift=0 cm]
	\foreach \a in {0,1,2,3,4,5,6,7,8,9}
	\draw[rotate=36*\a]
	(90:3) -- (126:3);
	
	\coordinate (O) at (0,0);
	\coordinate (A) at (90:0.8);
	\coordinate (B) at (270:0.8);
	\coordinate (L) at (162:1.8);
	\coordinate (R) at (342:1.8);
	
	\coordinate (A1) at (90:3);
	\coordinate (A2) at (126:3);
	\coordinate (A3) at (162:3);
	\coordinate (A4) at (198:3);
	\coordinate (A5) at (234:3);
	\coordinate (A6) at (270:3);
	\coordinate (A7) at (306:3);
	\coordinate (A8) at (342:3);
	\coordinate (A9) at (18:3);
	\coordinate (A10) at (54:3);

	\coordinate[shift={(234:0.9)}] (A10R) at (A10);
	\coordinate[shift={(127:0.9)}] (A8R) at (A8);
	\coordinate[shift={(173:0.9)}] (A8B) at (A8);
	
	\coordinate[shift={(54:0.9)}] (A5L) at (A5);
	\coordinate[shift={(307:0.9)}] (A3L) at (A3);
	\coordinate[shift={(353:0.9)}] (A3A) at (A3);


	\draw
	(A8) -- (A8B)
	(A3) -- (A3A);
	
	
	\draw[red]
	(A1) -- (A2) -- (A3) -- (A4) -- (A5) -- (A6);
	
	\draw[blue]
	(A6) -- (A7) -- (A8) -- (A9) -- (A10) -- (A1);

	
	\draw[double, line width=0.6] 
	(A10) -- (A10R)
	(A5) -- (A5L);
	
	
	\draw[line width=2]
	(A8) -- (A8R)
	(A3) -- (A3L);
	

	
	\node[shift={(0.4,0.2)}] at (A3) {\small $\alpha$};
	\node[shift={(0.45,-0.26)}] at (A3) {\small $\delta$};
	\node[shift={(0.18,-0.55)}] at (A3) {\small $\delta$};
	
	\node[shift={(18:0.2)}] at (A4) {\small $\alpha$};
	
	\node[shift={(0,0.4)}] at (A5) {\small $\beta$};
	\node[shift={(0.4,0.1)}] at (A5) {\small $\beta$};
	
	\node[shift={(90:0.2)}] at (A6) {\small $\alpha$};
	
	\node[shift={(-0.2,0.3)}] at (A7) {\small $\delta^{\frac{f-8}{8}}$};

	
	
	\node[shift={(-0.4,-0.2)}] at (A8) {\small $\alpha$};
	\node[shift={(-0.45,0.26)}] at (A8) {\small $\delta$};
	\node[shift={(-0.18,0.55)}] at (A8) {\small $\delta$};
	
	\node[shift={(198:0.2)}] at (A9) {\small $\alpha$};
	
	\node[shift={(0,-0.4)}] at (A10) {\small $\beta$};
	\node[shift={(-0.4,-0.1)}] at (A10) {\small $\beta$};
	
	\node[shift={(270:0.2)}] at (A1) {\small $\alpha$};
	
	\node[shift={(0.25,-0.3)}] at (A2) {\small $\delta^{\frac{f-8}{8}}$};
	
	\draw[red] (108:2.85) circle (0.1cm);
	\draw[fill,black] (0,0) circle (0.1cm);
	\draw[blue] (288:2.85) circle (0.1cm);
	

	\end{scope}

	\begin{scope}[xshift=6 cm, yshift=0 cm]
	\foreach \a in {0,1,2,3,4,5,6,7,8,9}
	\draw[rotate=36*\a]
	(90:3) -- (126:3);
	
	\coordinate (O) at (0,0);
	\coordinate (A) at (90:0.8);
	\coordinate (B) at (270:0.8);
	\coordinate (L) at (162:1.8);
	\coordinate (R) at (342:1.8);
	
	\coordinate (A1) at (90:3);
	\coordinate (A2) at (126:3);
	\coordinate (A3) at (162:3);
	\coordinate (A4) at (198:3);
	\coordinate (A5) at (234:3);
	\coordinate (A6) at (270:3);
	\coordinate (A7) at (306:3);
	\coordinate (A8) at (342:3);
	\coordinate (A9) at (18:3);
	\coordinate (A10) at (54:3);

	\coordinate[shift={(-18:0.9)}] (A3L) at (A3);
	\coordinate[shift={(89:0.9)}] (A5L) at (A5);
	\coordinate[shift={(43:0.9)}] (A5B) at (A5);
	
	\coordinate[shift={(162:0.9)}] (A8R) at (A8);
	\coordinate[shift={(269:0.9)}] (A10R) at (A10);
	\coordinate[shift={(223:0.9)}] (A10A) at (A10);


	\draw
	(A5) -- (A5B)
	(A10) -- (A10A);

	
	\draw[red]
	(A6) -- (A7) -- (A8) -- (A9) -- (A10) -- (A1);
	
	\draw[blue]
	(A1) -- (A2) -- (A3) -- (A4) -- (A5) -- (A6);
	
	
	\draw[double, line width=0.6] 
	(A3) -- (A3L)
	(A8) -- (A8R);
	
	
	\draw[line width=2]
	(A5) -- (A5L)
	(A10) -- (A10R);

	
	
	\node[shift={(0.02,-0.4)}] at (A1) {\small $\delta^{\frac{f+8}{8}}$};

	\node[shift={(306:0.2)}] at (A2) {\small $\alpha$};
	
	\node[shift={(0.3,0.15)}] at (A3) {\small $\beta$};
	\node[shift={(0.2,-0.35)}] at (A3) {\small $\beta$};
	
	\node[shift={(18:0.2)}] at (A4) {\small $\alpha$};
	
	\node[shift={(-0.2,0.65)}] at (A5) {\small $\delta$};
	\node[shift={(0.2,0.5)}] at (A5) {\small $\delta$};
	\node[shift={(0.4,0.08)}] at (A5) {\small $\alpha$};


	
	\node[shift={(-0.02,0.4)}] at (A6) {\small $\delta^{\frac{f+8}{8}}$};
	
	\node[shift={(126:0.2)}] at (A7) {\small $\alpha$};
	
	\node[shift={(-0.3,-0.15)}] at (A8) {\small $\beta$};
	\node[shift={(-0.2,0.35)}] at (A8) {\small $\beta$};
	
	\node[shift={(198:0.2)}] at (A9) {\small $\alpha$};
	
	\node[shift={(0.2,-0.65)}] at (A10) {\small $\delta$};
	\node[shift={(-0.2,-0.5)}] at (A10) {\small $\delta$};
	\node[shift={(-0.4,-0.08)}] at (A10) {\small $\alpha$};

	\draw[red] (72:2.85) circle (0.1cm);
	\draw[fill,black] (0,0) circle (0.1cm);
	\draw[blue] (252:2.85) circle (0.1cm);

	\end{scope}
	\end{tikzpicture}
	\caption{The $2$-fold Rotations of the Tiling of $\AVC \equiv \{\alpha\beta^2,\gamma^4,\alpha^2\delta^2,\alpha\delta^{\frac{f+8}{8}}\}$.}
	\label{2FoldRot}
\end{figure}

Since every $\alpha^2\delta^2$ is adjacent to some $\gamma^4$ along a $\bm{c}$-edge, it suffices to show that there is no vertex bisecting cycle containing $\gamma^4$. The vertex bisecting path of $\gamma^4$ along $\bm{b}$-edges terminates before $\alpha\beta^2$ and the vertex bisecting path along $\bm{c}$-edges terminates before $\alpha\delta^{\frac{f+8}{8}}$. Then there is no vertex bisecting cycle containing any of $\alpha^2\delta^2,\gamma^4$. So there is no mirror plane and hence the symmetry group is $D_2$.
\end{case*}

\end{document}